\documentclass[12pt,a4paper]{amsart}
\makeindex
\usepackage{amssymb}
\usepackage{enumitem}
\usepackage[all]{xy}
\usepackage{graphicx,amssymb,amsfonts,epsfig,amsthm,a4,amsmath,url,verbatim}
\usepackage[latin1]{inputenc}
\newtheorem{thm}{Theorem}[section]
\newtheorem{cor}[thm]{Corollary}
\newtheorem{lem}[thm]{Lemma}

\newtheorem{prop}[thm]{Proposition}

\theoremstyle{definition}
\newtheorem{defn}[thm]{Definition}
\newtheorem{sett}[thm]{Setting}

\newtheorem{que}[thm]{Question}

\newtheorem{exe}[thm]{Example}

\newtheorem{rem}[thm]{Remark}

\numberwithin{equation}{section}

\newcommand{\N}{\mathbf{N}}
\newcommand{\M}{\mathbf{M}}
\newcommand{\Z}{\mathbf{Z}}
\newcommand{\R}{\mathbf{R}}

\newcommand{\Q}{\mathbf{Q}}
\newcommand{\K}{\mathbf{K}}

\newcommand{\SL}{\textnormal{SL}}

\newcommand{\GL}{\textnormal{GL}}

\newcommand{\Ker}{\textnormal{Ker}}

\newcommand{\Aut}{\mathrm{Aut}}
\newcommand{\End}{\mathrm{End}}
\newcommand{\NAut}{\mathrm{NAut}}

\newcommand{\mk}{\mathfrak}

\newcommand{\eps}{\varepsilon}

\newcommand{\Hom}{\textnormal{Hom}}

\newcommand{\MN}{\mathrm{MN}}
\newcommand{\GN}{\mathrm{GN}}

\newcommand{\tu}{\bigtriangleup}

\newcommand{\PC}{\mathrm{PC}}

\newcommand{\mks}{\mathfrak{S}}

\newcommand{\mksh}{\mathfrak{S}^\star_0}
\newcommand{\mkst}{\mathfrak{S}^\star}

\newcommand{\pl}{\partial_{\mathsf{L}}}
\newcommand{\pr}{\partial_{\mathsf{R}}}
\newcommand{\lst}{\mathrm{lst}}
\newcommand{\rea}{\mathrm{rea}}
\newcommand{\mka}{\mathfrak{A}}
\newcommand{\tmksh}{\tilde{\mathfrak{S}}^\star_0}
\newcommand{\bt}{\widetilde{\beta_T}}
\newcommand{\bx}{\widetilde{\beta_X}}
\newcommand{\by}{\widetilde{\beta_Y}}
\newcommand{\bz}{\widetilde{\beta_Z}}
\newcommand{\tq}{\tilde{q}}
\newcommand{\smsm}{\!\smallsmile\!}
\newcommand{\hatwr}{\,\hat{\wr}\,}
\newcommand{\har}{\mathcal{H}}
\newcommand{\nhar}{\mathcal{NH}}
\newcommand{\WW}{\mathcal{W}}
\newcommand{\II}{\mathcal{I}}
\newcommand{\St}{\mathrm{St}}

\newcommand{\lclc}{[\![}
\newcommand{\rcrc}{]\!]^{\#}}
\newcommand{\lcc}{[\![}
\newcommand{\rcc}{]\!]}
\newcommand{\substar}{\subseteq^\star}
\newcommand{\mce}{\mathcal{E}}
\newcommand{\mces}{\mathcal{E}^\star}
\newcommand{\mke}{\mathfrak{E}}
\newcommand{\mkp}{\mathfrak{P}}
\newcommand{\mkpm}{\mathfrak{P}^{\mathrm{map}}}

\newcommand{\mkem}{\mathfrak{E}^{\mathrm{map}}}
\newcommand{\mkep}{\mathfrak{E}^{\mathrm{cfbij}}}

\newcommand{\mkpst}{\mathfrak{P}^\star}
\newcommand{\mkest}{\mathfrak{S}^\star}

\newcommand{\mkrm}{\mathfrak{R}^{\mathrm{map}}}
\newcommand{\mkrp}{\mathfrak{R}^{\mathrm{part}}}
\newcommand{\mkr}{\mathfrak{R}}
\newcommand{\mkrst}{\mathfrak{R}^\star}

\newcommand{\bmkrm}{\breve{\mathfrak{R}}^\mathrm{map}}

\newcommand{\bmkp}{\breve{\mathfrak{P}}}
\newcommand{\bmkpm}{\breve{\mathfrak{P}}^\mathrm{map}}
\newcommand{\bmkpst}{\breve{\mathfrak{P}}^\star}
\newcommand{\co}{\mathrm{c}_0}
\newcommand{\tot}{\hbox{$\to\kern-1.3em\div\;$}}
\newcommand{\tor}{\to\hspace{-1.25em}\div\hspace{.4em}}
\newcommand{\topm}{\to\hspace{-1.25em}\div\hspace{.4em}}
\newcommand{\HilH}{\mathrm{HilH}}

\newcounter{saveenum}
\begin{document}
\title{Near actions}
\author{Yves Cornulier}%
\address{CNRS and Univ Lyon, Univ Claude Bernard Lyon 1, Institut Camille Jordan, 43 blvd. du 11 novembre 1918, F-69622 Villeurbanne}
\email{cornulier@math.univ-lyon1.fr}

\subjclass[2010]{20B07, 20B27 (primary); 03E05, 03E15, 03G05, 05C63, 06E15, 18B05, 20B30, 20F65, 20M18, 20M20, 20M30, 22F05, 22F50, 37C85 (secondary)}
%03Exx		Set theory
%	03E05  	Other combinatorial set theory
%	03E15  	Descriptive set theory
%  03Gxx		Algebraic logic
%	03G05  	Boolean algebras
%03Gxx		Algebraic logic
%	03G05  	Boolean algebras 
%05 combinatorics
%	05C63  	Infinite graphs
%	18B05 	Category of sets, characterizations
%06-XX			Order, lattices, ordered algebraic structures
%  06Exx		Boolean algebras (Boolean rings)
%	06E15  	Stone spaces (Boolean spaces) and related structures
%20Bxx		Permutation groups
%   20B07  	General theory for infinite groups
%	20B27  	Infinite automorphism groups
%	20B30   Symmetric groups
%	20F65  	Geometric group theory
%	20M18  	Inverse semigroups
% 	20B30  	Symmetric groups
%22Fxx		Noncompact transformation groups
%   22F05  	General theory of group and pseudogroup actions
%	22F50  	Groups as automorphisms of other structures
%	20M20  	Semigroups of transformations, etc.
%	20M30  	Representation of semigroups; actions of semigroups on sets	
%	37C85  	Dynamics of group actions other than ${\bf Z}$ and ${\bf R}$, and foliations
	
\date{January 14, 2019}%{\today}

\begin{abstract}
A near permutation of a set is a bijection between two cofinite subsets, modulo coincidence on smaller cofinite subsets. Near permutations of a set form its near symmetric group. We define near actions as homomorphisms into this group, and perform a systematic study of this notion. Notably, we deal with the problem of realizability: how far is a near action from being induced by an action? Also, what invariants of group actions are actually invariants of the underlying near action? the most fundamental one is the space of ends.

We notably provide a classification of near actions of finitely generated abelian groups. There are several examples of exotic near actions of the free abelian group of rank two $\Z^2$, which we classify according to a positive integer, called winding number, and an integral vector, called additive holonomy. In contrast, we prove that there are no such exotic near actions in higher rank abelian groups. 

We also obtain results for other groups; for instance, every near action of $\SL_2(\Z)$ is realizable; this is not true for arbitrary amalgams of two finite cyclic groups. 

We introduce and study almost and near automorphism groups, which have been considered for some particular graphs or labeled graphs, in the general setting of relational structures.
\end{abstract}

\maketitle

%{\bf Warning.} This article is in construction. Apart from the table of contents, the current version of the introduction was written before part of the material was added, and thus only partially reflects the contents of this monograph.

\section{Introduction}
%\localtableofcontents
%\subsection{rien}

\subsection{Foreword}
%The purpose of this monograph is to introduce and develop the notion of near action. 

Group actions are omnipresent in group theory: indeed permutation groups were defined before groups were abstractly defined. The purpose of this monograph is to introduce and develop the concept of near actions of groups on sets: roughly speaking a near action of a group is like an action, except that for every group element, its action is defined ``modulo indeterminacy on some finite subset". Near actions are defined in the first subsections of this introduction. 

%Central notions are the realizability issues: how far is a near action from being induced by an action? Besides, which invariants of actions only depend on the underlying near action?

The length of this monograph is mainly due to the variety of concepts that can be attached to near actions, and to the multiplicity of examples. This is reflected in the various topics discussed in this introduction. Let us emphasize that once the bare definition of near action is understood, most chapters can essentially be read independently.

The extensive bibliography yields an idea of how this work connects to existing work. We therefore hope this monograph to be of interest, among others, to readers interested in the combinatorics of infinite group actions, in the abstract or symbolic dynamics of groups, and in the connections between set theory and group theory, notably related to the structure of infinite symmetric groups.

%A (non-comprehensive) account appears in the various topics discussed in this introduction.

\subsection{Naive attempt: balanced near actions}\label{nbna}
For illustration of the main concept and simplicity, we start with  the more naive notion of balanced near action\index{balanced near action}.
%\index{toto! tata}.

Let $X$ be a set. Recall that an action of a group $G$ on $X$ is the datum of a homomorphism $G\to\mks(X)$, where $\mks(X)$ is the symmetric group\index{symmetric group}\index{$\mks(X)$, symmetric group} on $X$, namely the group of all permutations of $X$. Endowed with this action\index{action}, $X$ is called a $G$-set.
%\index{titi}

Consider the quotient $\mksh(X)$\index{$\mksh(X)$, balanced near symmetric group} of $\mks(X)$ by its normal subgroup $\mks_{<\aleph_0}(X)$ of finitely supported permutations (the notation $\mksh$ will be explained later); call it the balanced\index{balanced} near symmetric group of $X$, and call its elements balanced near permutations of $X$. 

\begin{defn}
A balanced near action of the group $G$ on $X$ is the datum of a homomorphism $G\to\mksh(X)$. Endowed with this action, $X$ is called a balanced near $G$-set.
\end{defn}

\begin{exe}\label{firstex}
Examples appear throughout the paper. Here is a first list, some of them being detailed later.
\begin{enumerate}
\item Every group action induces a balanced near action.
\item The identity map of $\mksh(X)$ defines a balanced near action of $\mksh(X)$ on $X$.
\item If $X$ is a finite set, then the group $\mksh(X)$ is a trivial group. Hence, there is a single structure of near $G$-set on $X$, for every group $G$.
\item\label{x_minus-pt} For every balanced near action of $G$ on a set $X$ and $x_0\in X$, there is a well-defined balanced near action of $G$ on $X\smallsetminus\{x_0\}$.
\item\label{exzz2}Consider the set $X=\Z\times\{\pm 1\}$. Define $a(n,m)=(n+1,m)$ for all $n,m$; $b(n,m)=(n,(-1)^{1_\N(n)}m)$, where $1_\N$ is the characteristic function of $\N$:
%\bigskip
\[
\xymatrix{
& \cdots\ar@{=>}[r]& \bigcirc\ar@(dl,dr)\ar@{=>}[r]
& \bigcirc\ar@(dl,dr)\ar@{=>}[r]& \bigcirc\ar@(dl,dr)\ar@{=>}[r]_a
& \bigcirc\ar@{=>}[r]\ar@{<->}[d]^b& \bigcirc\ar@{=>}[r]\ar@{<->}[d]
& \bigcirc\ar@{=>}[r]\ar@{<->}[d]& \cdots&\\
& \cdots\ar@{=>}[r]& \bigcirc\ar@(ul,ur)\ar@{=>}[r]
& \bigcirc\ar@(ul,ur)\ar@{=>}[r]& \bigcirc\ar@(ul,ur)\ar@{=>}[r]
& \bigcirc\ar@{=>}[r]& \bigcirc\ar@{=>}[r]
& \bigcirc\ar@{=>}[r]& \cdots&
%& \bigcirc\ar@(ul,ur)\ar@{=>}[r]& \bigcirc\ar@(ul,ur)\ar@{=>}[r]& \bigcirc\ar@(ul,ur)\ar@{=>}[r]& \bigcirc\ar@(ul,ur)\ar@{=>}[r]& \bigcirc\ar@(ul,ur)\ar@{=>}[r]& \bigcirc\ar@(ul,ur)\ar@{=>}[r]
%& \bigcirc\ar@(ul,ur)\ar@{=>}[r]& \bigcirc\ar@(ul,ur)\ar@{=>}[r] & \bigcirc\ar@(ul,ur)  \ar@{=>}[r]  & \bigcirc \ar@{<->}[d] \\
%& \bigcirc\ar@(ul,ur)& \bigcirc\ar@(ul,ur)& \bigcirc \ar[r] & \bigcirc 
}
\]
Then $a,b$ are two permutations of $X$, with $b^2=1$. Moreover, the commutator $aba^{-1}b^{-1}$\index{commutator} is a finitely supported permutation. Hence, the images $\bar{a}$ and $\bar{b}$ of $a$ and $b$ in $\mksh(X)$ commute. Thus, this defines a balanced near action of the direct product $\Z\times(\Z/2\Z)=\langle a\rangle\times\langle b\rangle$ on $X$.
\item Thompson's group $V$ and Neretin's group are defined from balanced near action on the regular rooted binary tree, see \S\ref{s_neretin}.
\item Let $X$ be a Hausdorff topological space. Let $\mathrm{PC}_0(X)$ be the subgroup of $\mksh(X)$ defined as the subgroup of permutations of $X$ with finitely many discontinuity points, modulo the subgroup of finitely supported permutations; it naturally has a balanced near action on $X$. 

This group is particularly interesting when $X$ is a topological circle.
%The group of piecewise continuous self-transformations of the circle $T$ can be defined as the subgroup of permutations of $T$ with finitely many discontinuity points, modulo the subgroup of finitely supported permutations; it naturally has a balanced near action on $T$. 
%(This uses the fact that the number of discontinuity of $f$ and $f^{-1}$ coincide; this is not true for arbitrary topological spaces.)
\end{enumerate}
\end{exe}

A first natural question is to understand when a balanced near action is induced by an action. This is encapsulated in the following definition; denote by $p_X$ the quotient homomorphism $\mks(X)\to \mksh(X)$.

\begin{defn}
A balanced near action $\alpha:G\to\mksh(X)$ is realizable\index{realizable} if it lifts to an action on $X$, that is, if there exists a homomorphism $\beta:G\to\mks(X)$ such that $\alpha=p_X\circ\beta$.
\end{defn}

Thus, the very first question, namely whether every balanced near action is realizable, is equivalent to asking whether the surjective group homomorphism $p_X:\mks(X)\to\mksh(X)$ splits. The latter question was first raised by W. Scott \cite{Sco1} and later solved by himself in his book \cite[\S 11.5.8]{Sco}. Let us rewrite it, an then interpret it in the language used here.

\begin{exe}[W.\ Scott's example]\label{i_scotex}\index{Scott's example}
Consider the disjoint union $X=\bigsqcup_{n\ge 0}\Z/2^n\Z$. Let $f_0$ be the permutation acting by $+1$ on each $\Z/2^n\Z$. Then, modulo finitely supported permutations, $f_0$ has a $2^n$-root for every $n$. But there exists no single finitely supported permutation $s$ such that $fs$ admits in $\mks(X)$ a $2^n$-root for all $n$ (this is an exercise: use that a permutation has a $2^n$-root if and only if for every $m$ that is even or infinite, the number of $m$-cycles is multiple of $2^n$ or infinite). This shows that $p_X$ does not split.  

Let us now view this more explicitly as a non-realizability result.
 For $k\ge 0$ let $f_k$ be the permutation acting on $\Z/2^n\Z$ as the addition by $2^{n-k}$ for $n\ge k$ and as the identity on it for $n\le k$. Then $f_k^2$ is the identity on $\Z/2^n\Z$ for $n\le k$, and is the addition by $2^{n-k+1}$ for $n\ge k$. Thus, $f_k^2$ and $f_{k-1}$ only differ on $\Z/2^k\Z$, the first being the identity and the second being given by the addition by $2$. Hence, denoting by $\bar{f}_k$ the image of $f_k$ in $\mksh(X)$, we have $\bar{f}_{k}^2=\bar{f}_{k-1}$ for all $k\ge 1$. Thus this defines a balanced near action of the (infinitely generated abelian) group $\Z[1/2]$, which is not realizable by the above exercise.
\end{exe}

\begin{exe}\label{abminuspt}
Let $G$ be a finitely generated abelian group, $\alpha_0$ its action on itself by translation, and consider the induced balanced near action of $G$ on $G\smallsetminus\{1\}$ (as in Example \ref{firstex}(\ref{x_minus-pt})). 
Let us start with the obvious observation that this is realizable if and only if the balanced near action on $X$ is realizable as an action fixing $1$. 

The reader can check that this holds when $G$ is finite, and when $G=\Z$, namely by the cycle $\dots-2\mapsto -1\mapsto 1\mapsto 2\dots$, fixing $0$, and then more generally when $G$ is virtually cyclic. 

Let us show that conversely, if $G$ has rank $\ge 2$, then every realization $\alpha$ of the balanced near action of $G$ on itself by translation is a simply transitive action, which implies that the balanced near action on $G\smallsetminus\{1\}$ is not realizable.

%: this means that $\alpha$ is an action and that for every $g\in G$, the permutations $\alpha_0(g)$ and $\alpha(g)$ differ outside a finite subset (we call $\alpha_0$ a finite perturbation of $\alpha$). 

Define $f:G\to G$ by $f(g)=\alpha(g)(1)$. Then for every $g\in G$, the set of $h\in  G$ such that $f(gh)\neq gf(h)$ is finite (this uses that the action of $G$ on itself is free).  
Set $u(g)=g^{-1}f(g)$: then this rewrites as: for every $g$, the set of $h$ such that $u(gh)\neq u(h)$ is finite. We view $u$ as a function on the Schreier graph $\mathcal{G}$ (of $G$ acting on itself). Define another graph $\mathcal{G}'$ from $\mathcal{G}$ by removing the edges between $h$ and $sh$ whenever, for given $s$, $h$ is among the exceptions. Thus, $u$ is locally constant on the graph $\mathcal{G}'$. Now we use that $G$ has rank at least 2, to deduce that $G$ is 1-ended, which implies that $\mathcal{G}'$ has a component with finite complement. Hence, there exists $g_0\in G$ such that $u(g)=g_0$ for all $g$ but finitely many exceptions. That is, $f(g)=gg_0$ for all $g$ but finitely many exceptions, say for $g\in G\smallsetminus F$ with $F$ finite. In particular, the cardinal of $f^{-1}(\{f(g)\})$ is equal to 1 for all $g\in G$ but finitely many exceptions. Since $f$ is $G$-equivariant from $(G,\alpha_0)$ to $(G,\alpha)$, this cardinal $f^{-1}(\{f(g)\})$ is constant on $G$. Hence $f$ is injective. Now $f$ restricts to the translation $G\smallsetminus F\to G\smallsetminus Fg_0$, and hence by injectivity induces an injection from $F$ to $Fg_0$. Since these are finite subsets of the same cardinal, this is a bijection, and eventually we deduce that $f$ is bijective. 
Since $f$ is equivariant from $(G,\alpha_0)$ to $(G,\alpha)$, this implies that $(G,\alpha)$ is a simply transitive $G$-action.
\end{exe}

These motivations could lead to a systematic investigation of the notion of balanced near action. However, such an investigation would soon be clumsy, because these is a slightly larger setting giving rise to a much more ``stable" setting, that of near actions.

\subsection{Near symmetric group and index}\label{i_nsgi}

Near actions are defined in the same way as balanced near actions, replacing the group $\mksh(X)$ by a slightly larger group, the near symmetric \index{near symmetric group} group $\mkst(X)$.\index{$\mkst(X)$, near symmetric group} The latter was first defined by Wagoner \cite[\S 7]{Wa} in the context of algebraic topology, forgotten, later implicitly suggested by Shelah \cite{She} and explicitly by van Douwen \cite{vD}, and in spite of this has remained quite little known since then (see \S\ref{hire} for a much more detailed historical account).

Let $X,Y$ be sets. By cofinite subset,\index{cofinite} we mean a subset with finite complement. The set $\mkst(X,Y)$ of near bijections from $X$ to $Y$ is defined as the set of bijections from a cofinite subset of $X$ to a cofinite subset of $Y$, identify two such partial bijections whenever they coincide on a cofinite subset. Composition defines a well-defined map $\mkst(Y,Z)\times\mkst(X,Y)\to \mkst(X,Z)$, and every $f\in\mkst(X,Y)$ has an inverse in $\mkst(Y,X)$. In particular, the set $\mkst(X)=\mkst(X,X)$ is naturally a group under composition. 

Given $f\in\mkst(X,Y)$ with representative $\tilde{f}$, bijection between cofinite subsets $X\smallsetminus F,Y\smallsetminus F'$, the number $|F'|-|F|$ is not hard to see not to depend only on $f$, denoted by $\phi_{X,Y}(f)$, and called index\index{index} of $f$. Furthermore, this number is additive under composition. In particular, denoting $\phi_X=\phi_{X,X}$, the index map $\phi_X$ is a homomorphism $\mkst(X)\to\Z$. 

We actually have a short exact sequence 
\[1\to \mks(X)/\mks_{<\aleph_0}(X)\to \mkst(X)\stackrel{\phi_X}\to \Z (\to 1).\]

Here, the middle map is induced by the natural group homomorphism $\mks(X)\to\mkst(X)$, whose kernel is the subgroup of finitely supported permutations. Also, elements of index zero are clearly those that can be represented by a bijection, whence the exactness at the middle; this finally explains the choice of notation $\mksh$, which we now can view as the kernel of the index map. Finally, the surjectivity on the right holds as soon as $X$ is infinite, using that there exists a bijection from $X$ onto $X\smallsetminus \{x_0\}$. 

%Let us mention here a slight

\begin{rem}[Join with Hilbert's hotel]\label{hilbho} Let us mention a slightly different way to construct $\mkst(X)$: consider the group $\HilH_X$ of permutations of $X\sqcup\N$ eventually acting as a translation $n\mapsto n+\tau$ on $\N$. This group has a canonical homomorphism into $\Z$, which maps such a permutation to $\tau$. 
The reader can check that its quotient by the group of finitely supported permutations is naturally isomorphic to $\mkst(X)$, thus defining a balanced near action of $\mkst(X)$ on $X\sqcup\N$. The group $\HilH_X$ can be interpreted as an almost automorphism group (see \S\ref{i_anag}), and appears in \S\ref{s_omega}. 
\end{rem}
%%%%%%% MENTIONER QUE C'EST UN AAUT?

\subsection{Near actions and their index character}
We come to the main definition of this monograph. A near action\index{near action} of a group $G$ on a set $X$ is a homomorphism $\alpha:G\to \mkst(X)$. Endowed with $\alpha$, we call $X$ a near $G$-set. 
%\end{}

A near $G$-set is near free\index{near free} if for every $g\in G\smallsetminus\{1\}$, the set of fixed points of $g$ in $X$ (which is well-defined up to finite indeterminacy) is finite.

Given two near actions $\alpha,\beta$ of $G$ on sets $X,Y$, the disjoint union of the near actions, denoted $\alpha\sqcup\beta$, or $X\sqcup Y$ when the context is clear, is naturally defined. 

The index character\index{index character} of a near action $\alpha$ is the element $\phi_\alpha=\phi_X\circ\alpha$ of $\Hom(G,\Z)$. Its vanishing means that $\alpha$ maps $G$ into $\mksh(X)$, and thus precisely means that the near action is a balanced near action.

Note that the index character is additive under disjoint unions: $\phi_{\alpha\sqcup\beta}=\phi_\alpha+\phi_\beta$.

\begin{exe}\label{secondex}
Here are examples of near actions, in addition to the examples of balanced near actions of Example \ref{firstex}.
\begin{enumerate}
\item The group $\Z$ near acts on $\N$ by powers of the near permutation induced by the assignment $n\mapsto n+1$. The index character of this near action is $\mathrm{Id}_\Z$.
\item The group $\mkst(X)$ has a tautological near action on $X$, whise index character is the index homomorphism $\psi_X$.
\item For a Hausdorff topological space $X$, the group of near self-homeomorphisms $\PC(X)$\index{near self-homeomorphism group} of $X$ is the subgroup of $\mksh(X)$ having a representative that is a homeomorphism $X_1\to X_2$ between two cofinite subsets of $X$. (For $X$ discrete, this is just $\mksh(X)$ itself.) Depending on $X$, the index character is or is not zero: it is nonzero if and only if there exist finite subsets $F_1,F_2$ of $X$ with $|F_1|\neq |F_2|$ such that $X\smallsetminus F_1$ and $X\smallsetminus F_2$ are homeomorphic. For instance, for $X$ a topological circle, the index character vanishes on $\PC(X)$, while for $X$ infinite discrete, or a Cantor set, it does not vanish.

In all cases, the kernel of the index character coincides with the subgroup $\mathrm{PC}_0(X)$ of Example \ref{firstex}.
\end{enumerate}
\end{exe}

By the Hilbert hotel join (Remark \ref{hilbho}), every near action of a group $G$ on a set $X$ canonically extends to a balanced near action on $X\sqcup\N$.

% When $X$ is a topological circle, this group is particularly rich and interesting

\subsection{Near isomorphism}

First, for analogy, recall that two $G$-sets $(X,\beta)$ and $(X',\beta')$ are isomorphic as $G$-sets if there exists a $G$-equivariant bijection $X\to X'$. Thus, we say that two near $G$-sets $X$ and $X'$ are near isomorphic\index{near isomorphic} if there exists $f\in\mkst(X,X')$ that is near $G$-equivariant, in the sense that for every $g\in G$ and for all $x\in X$ but finitely many exceptions, we have $f(gx)=gf(x)$. Note that the latter equality is meaningful since for given $g$, both terms are defined up to finite indeterminacy.

Now let us warn about this notion: the notion of realizability is not a near isomorphism invariant among near $G$-sets. Indeed, it follows from the definition that for every $G$-set $X$ and $x\in X$, the inclusion of $X\smallsetminus\{x\}$ into $X$ is a near isomorphism of near $G$-sets; $X$ is realizable by definition, but we saw in Example \ref{abminuspt} that $X\smallsetminus\{x\}$ is not realizable in general.

This leads us now to, on the one hand, provide a stronger notion of isomorphism that makes realizability an invariant, and also provide weakenings of the notion of realizability that are near isomorphism invariants.

We say that two near $G$-sets $(X,\alpha)$ and $(X',\alpha')$ are balanceably near isomorphic\index{balanceably near isomorphic} if there exists $f\in\mkst_0(X,X')$ that is near $G$-equivariant. That is, we require the near isomorphism to be implementable by a bijection.

A trivial example to keep in mind is that all finite near $G$-sets are near isomorphic, while the cardinal is the (unique) balanced near isomorphism invariant of finite near $G$-sets. In general, two near $G$-sets $X,X'$ are near isomorphic if and only if there are finite subsets $F\subseteq X,F'\subseteq X'$ such that $X\smallsetminus F$ and $X\smallsetminus F'$ are balanceably near isomorphic, and this holds if and only if there are two finite sets $K,K'$ such that $X\sqcup K$ and $X'\sqcup K'$ are balanceably near isomorphic.

\subsection{Commensurated subsets}

In the theory of abstract group actions, the most basic notion is that of orbit decomposition, which reduces to studying transitive actions, and classifying transitive actions translates into classifying subgroups up to conjugation.

% Let us formulate this as a (folklore) theorem, so as to keep in mind as an analogy:

%\begin{thm}[Classifications of $G$-sets]
%Fix a group $G$. Consider a function $\mu$ from the set $\mathrm{CSub}(G)$ of subgroups of $G$ modulo isomorphism, into the collection of cardinals. For every element $H\in \mathrm{CSub}(G)$, choose $[H]\le G$ as a representative. Then define $X_\mu$ as the disjoint union $\bigsqcup \mu(H)G/H$, where $\mu(H)G/[H]$ means a disjoint copy of $\mu(H)$ copies of $G/H$, endowed with the obvious action of $G$. Then $X_\mu$ is a $G$-set, and every $G$-set is isomorphic to $X_\mu$ for a unique $\mu$.

%Up to isomorphism of $G$-sets (=$G$-equivariant bijections), $G$-sets are\end{thm}

In the context of near actions, there is no natural notion of $G$-invariant subset. However, there is a natural substitute: it is natural to ask whether a given near $G$-set splits as disjoint union of two ``smaller" near $G$-sets. The natural notion to consider is then the following:

\begin{defn}
Let $X$ be a near $G$-set. A subset $Y$ of $X$ is $G$-commensurated\index{commensurated subset} if for every $g\in G$ and for all but finitely many $x\in Y$, we have $gx\in Y$.
\end{defn}

Note the abuse of notation since $gx$ is not well-defined. However, for given $g$, when $x$ ranges over $X$, $gx$ is defined up to finitely many exceptions, and hence the above condition is meaningful. The intuition to keep is that a commensurated subset is a subset with ``finite boundary"; this intuition can be made more rigorous when $G$ is finitely generated, see \S\ref{i_nafg}.

If $X$ is a near $G$-set and $Y$ is a $G$-commensurated subset, then the near action naturally restricts to a near action on $Y$ and on its complement $X\smallsetminus Y$, and we can identify $X$ to the disjoint union of near $G$-sets $Y$ and $X\smallsetminus Y$. Let us emphasize that even if $X$ is balanced, $Y$ can fail to be balanced, and this is the whole point of introducing non-balanced near actions. 

%Let us also point out that ``can identify" in the previous paragraph is vague and now has to be made explicit. Namely, when should we identify near $G$-sets? The point is that there are two distinct and natural ways to do so.

A near $G$-set $X$ is called 0-ended if it is finite. It is called 1-ended\index{one-ended (near action)} if it is infinite and its only commensurated subsets are finite or cofinite.

For $X$ a near $G$-set, consider the set $\mathcal{P}_{(G)}(X)$ of $G$-commensurated subsets of $X$. It is saturated under the ``near equal" equivalence relation identifying two subsets of $X$ whenever they have a finite symmetric difference. Let $\mathcal{P}^\star_G(X)$ be the quotient by this equivalence relation. It naturally inherits the partial ordering induced by ``near inclusion" ($Y$ is near included\index{near included} in $Z$, denoted $Y\subseteq^\star Z$\index{$\subseteq^\star$ (near inclusion)}, if $Y\smallsetminus Z$ is finite).

Note that $G$ is 0-ended if and only $\mathcal{P}^\star_G(X)$ is a singleton, and is 1-ended if and only if $\mathcal{P}^\star_G(X)$ has cardinal $2$.

The near $G$-set is called finitely-ended\index{finitely-ended} if $\mathcal{P}^\star_G(X)$ is finite. In this case, considering those minimal elements in $\mathcal{P}^\star_G(X)\smallsetminus\{0\}$, it is not hard to see that there is a finite set $I$ and a partition $X=\bigsqcup_{i\in I}X_i$, with each $X_i$ being $G$-commensurated and $1$-ended. Then, up to finite symmetric difference, the $G$-commensurated subsets of $X$ are precisely the partial unions $\bigsqcup_{j\in J}X_j$ when $J$ ranges over subsets of $I$. Thus $\mathcal{P}^\star_G(X)$ has cardinal $2^n$. The number $n$ is called the number of ends of $X$.

When $\mathcal{P}^\star_G(X)$ is infinite, $X$ is called infinitely-ended, but then there is a natural notion of space of ends carrying considerably more information, which is naturally termed in terms of Stone duality. This is the approach to ends originally due to Specker \cite{Sp} in the case of groups acting on themselves, and which is adopted here, see \S\ref{s_soe}.

The above description reduces, to a certain extent, the study of finitely-ended near $G$-sets to 1-ended ones. This is particularly useful to classify near actions of finitely generated abelian groups. Importantly, the 1-ended ``components" of a balanced near action are not balanced in general. The implicit key feature used above about $\mathcal{P}^\star_{G}(X)$ (and $\mathcal{P}_{(G)}(X)$) is that it is a Boolean algebra, i.e., is stable under taking usual (finitary) operations of set theory. Its subset of balanced commensurated subsets is not well-behaved in this respect: it is stable neither under intersection, nor under symmetric difference, as the reader can check in the case of a free $\Z$-set with two orbits (which is a $4$-ended $\Z$-set).

The space of ends of the near $G$-set is thus defined as the Stone dual of the Boolean algebra $\mathcal{P}^\star_{G}(X)$: this is a compact Hausdorff, totally disconnected space whose Boolean algebra of clopen subsets is naturally isomorphic to $\mathcal{P}^\star_{G}(X)$. This space has notably been considered the context of group actions, and often in a restricted context (such as transitive actions of finitely generated groups). It is an important motivating observation that it only depends on the underlying near action. 

\subsection{Realizability notions}

A near action is realizable\index{realizable} if it is induced by an action. Clearly, this implies that it is balanced, and hence this is just the notion introduced in \S\ref{nbna}. Here are natural weakenings of this notion:

\begin{defn}A near $G$-set $X$ is
\begin{itemize}
%\item finitely stably realizable if there exists a finite set $F$ such that $X\sqcup F$ is realizable;
\item stably realizable\index{stably realizable} if there exists a set $Z$ with trivial near action, such that $X\sqcup Z$ is realizable;
\item completable\index{completable} if there exists a near $G$-set $Y$ such that $X\sqcup Y$ is realizable.
\end{itemize}
\end{defn}

%The first two notions about stable realizability are closely related, and for many groups, including abelian groups and finitely generated groups. Let us especially emphasize the other two ones.\Rightarrow\textnormal{fin.\ stably real. }

We have the list of implications, for a near action of a group $G$ 
\[\textnormal{realizable}\Rightarrow\textnormal{stably realizable}\Rightarrow
\textnormal{completable}\Rightarrow\textnormal{near action}.\]
In general, ($G$ not fixed), none of the implication can be reversed, and the failure of reversing each of the arrows deserves a specific discussion.

The failure of reversing the first arrow follows from Example \ref{abminuspt}: the near action of $\Z^2$ on $\Z^2\smallsetminus\{0\}$ is stably realizable but not realizable. This was somewhat already discussed, so let us pass to other implications.

There is a fundamental difference between stably realizable and completable near actions, namely that stably realizable implies the vanishing of the index character. Nevertheless, there also exist completable balanced near actions that are not stably realizable. A typical example is the near action of $\Z\times (\Z/2\Z)$ in Example \ref{firstex}(\ref{exzz2}). Indeed, although it is balanced, the set of fixed points of $b$ defines a commensurated subset (defined up finite symmetric difference) and hence a near action, which has a nonzero index character, while in a realizable action this would have finite symmetric difference with an invariant subset. This phenomenon is actually the only type of obstruction to stable realizability, for completable near actions of finitely generated abelian group (see Theorem \ref{complab}). 

%Next, there exist non-completable near actions, of suitable finitely generated groups on countable sets.

There exist non-completable near actions: 

\begin{exe}\label{introex}
The group $\Z^2$ admits non-completable near actions. Let us describe two ways to produce such near actions.

\begin{itemize}
\item[(a)] Consider its simply transitive action on the standard Cayley graph, add usual squares to obtain the plane tessellated by squares, remove one square, so that the resulting square complex has infinite cyclic fundamental group. Make a finite connected $m$-fold covering. Then we obtain a balanced near action, which for $m\ge 2$ is non-completable, described in detail in \S\ref{s_noncomp}, and denoted $X_{m,0}$ below. It is depicted in Figure \ref{fig1} below, in the case of a double covering.

\item[(b)] A second way is to start from the same Cayley graph, to remove one vertical half-strip of width $\ell$, and to glue; denote it $K_\ell$ (it will be denoted $X_{1,(\ell,0)}$ in the sequel). See also \S\ref{s_noncomp}, and Figure \ref{fig2}.
\end{itemize}
\end{exe}

\begin{figure}[h]
{\includegraphics[keepaspectratio=true,
width=13cm,clip=true,trim= 0cm 7.6cm 0cm 7.6cm]{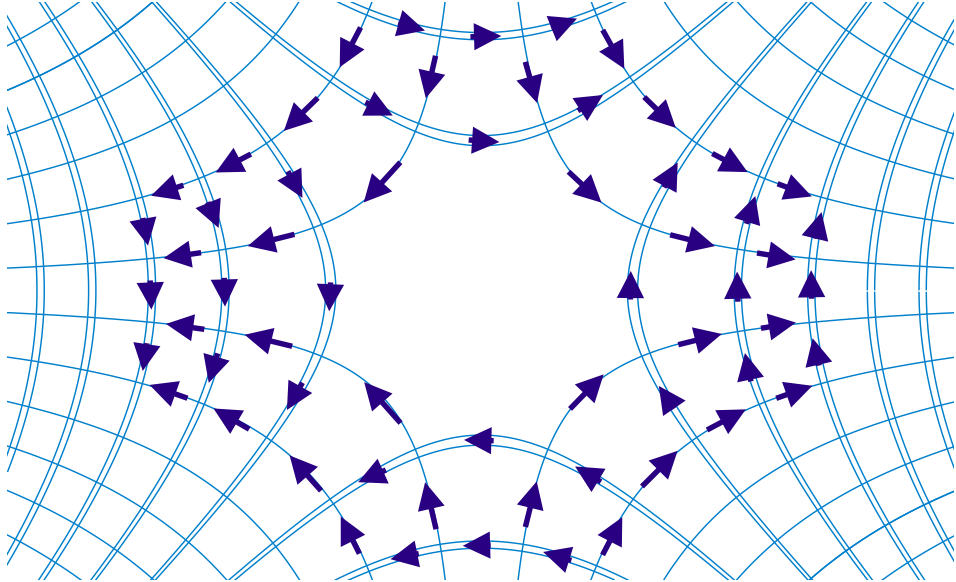}}
\caption{The near $\Z^2$-set $X_{2,0}$. Thin [resp.\ thick] edges represent the action of the first [resp.\ second] generator, in the direction indicated by the arrow.}\label{fig1}
\end{figure}

\begin{figure}[h]
{\includegraphics[keepaspectratio=true,
width=6cm]{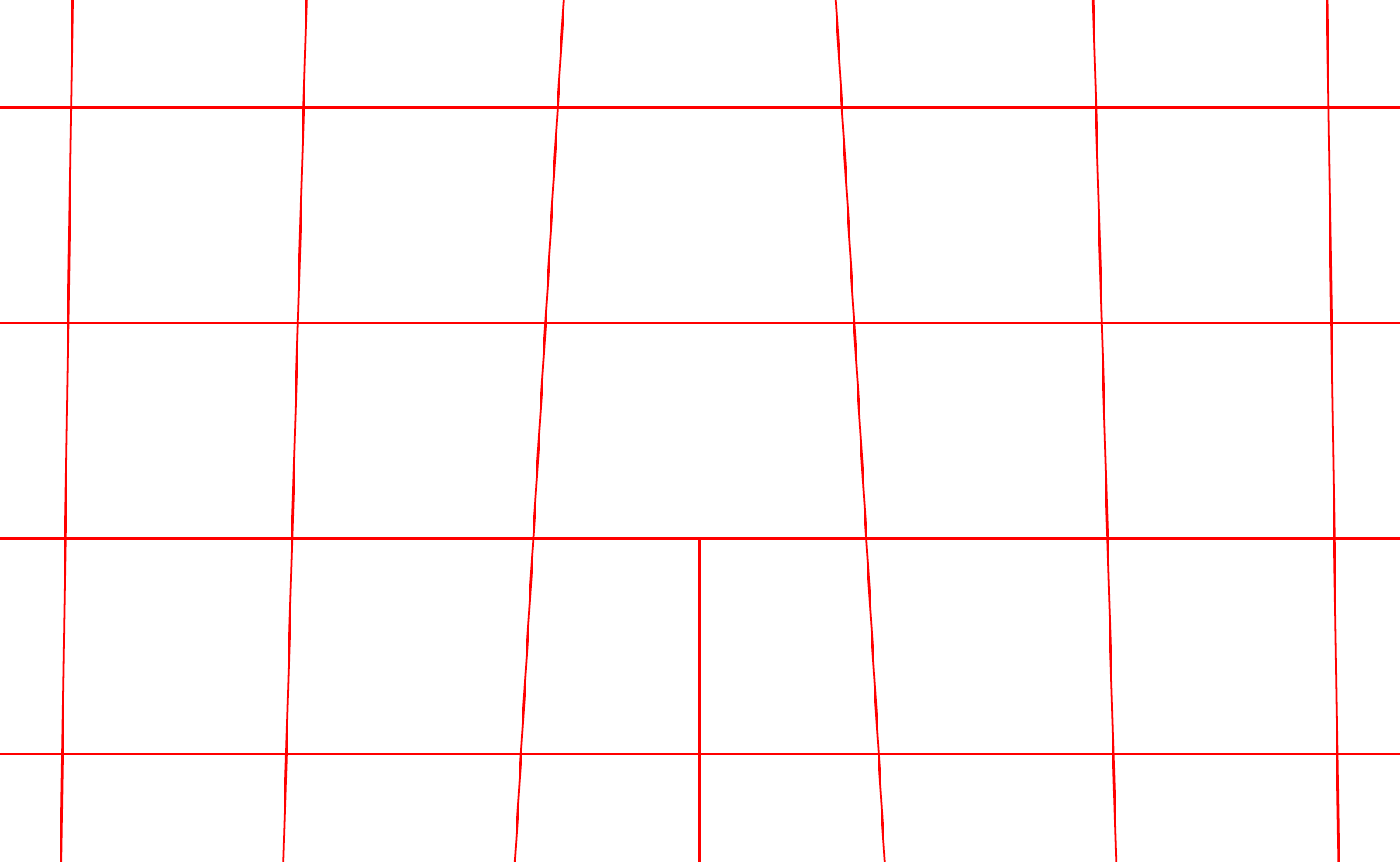}}
\caption{The near $\Z^2$-set $K_1$. When possible, the first [resp.\ second] generator acts by moving one step to the right [resp.\ upwards].}\label{fig2}
\end{figure}
 It formally follows that for $X$ infinite, the near action of $\mkst(X)$ on $X$ is not completable. The idea of creating near actions by such a cutting and pasting a planar Schreier graph along a ray can be recycled to some other groups, see for instance \S\ref{noncomplother}. Unrelated non-completable actions, of many infinitely presented, finitely generated groups can also be produced using Proposition \ref{nafpi} below. We come back to the classification of near actions of $\Z^2$ and finitely generated abelian groups in \S\ref{i_cfgag}.

 % Now let us provide 

\subsection{The Kapoudjian class}\label{i_kapou}

This topic is studied in detail in \S\ref{s_omega}. Let $X$ be an infinite set. The alternating group $\mka(X)$ has index 2 in the finitary symmetric group $\mks_{<\aleph_0}(X)$. Since a cyclic group of order 2 has a trivial automorphism group, the short exact sequence
\[1\to \mks_{<\aleph_0}(X)/\mka(X)\to \mks(X)/\mka(X)\to \mksh(X)\to 1\]
is a central extension. It thus canonically defines a cohomology class in $H^2(\mksh(X),\Z/2\Z)$, which we call the canonical class $\omega^X$ (and define $\omega^X=0$ for $X$ finite).

That the cohomology class $\omega^X$ is nonzero for $X$ infinite actually immediately follows from Vitali's early observation \cite{Vi} that there is no nonzero homomorphism $\mks(X)\to\Z/2\Z$. An immediate corollary is that the surjective group homomorphism $\mks(X)\to\mksh(X)$ is not split, which was later explicitly proved by W.\ Scott with another method described in Example \ref{i_scotex}.

Given a balanced near action $\alpha:G\to\mksh(X)$, we call the pullback $\alpha^*\omega^X$ the Kapoudjian class \index{Kapoudjian class} of $\alpha$. Indeed, it appears in the papers by Kapoudjian and Kapoudjian-Sergiescu \cite{Ka,KaS}, in the context of balanced near actions on trees, see Remark \ref{kapou2}. 

Here, we check that there is a unique cohomology class in $H^2(\mkst(X),\Z/2\Z)$ whose pull-back to $H^2(\mksh(X),\Z/2\Z)$ is $\omega^X$; we still denote it as $\omega^X$. This allows to define the Kapoudjian class for arbitrary near actions. We also provide a proof, communicated by Sergiescu, that $H_2(\mksh(X),\Z)$ is reduced to $\Z/2\Z$ for every infinite set $X$ (and hence $H^2(\mksh(X),\Z/2\Z)$ is reduced to $\Z/2\Z$). We notably provide a formula for the Kapoudjian class of a disjoint union: for a disjoint union of balanced near action this is plainly additive, but it is much more interesting (and involved) in the case with nonzero index, as the formula involves some cup-product involves the indices.

%The non-vanishing of $\omega^X$ is actually one way to prove that the surjective group homomorphism $\mks(X)\to\mksh(X)$ is not split, and this actually follows from Vitali's early observation \cite{Vi} that every element of $\mks(X)$ is a product of squares: indeed, if the above central exact sequence were split, it would imply the existence of a nontrivial homomorphism $\mks(X)\to\Z/2\Z$.

\subsection{Near actions of finitely generated groups}\label{i_nafg}

To be finitely-ended is a very practical finiteness condition, but is too strong in general, since for instance it generally does not encompass the left action of a group on itself. Here we generalize the notion of action with finitely many orbits; we have to restrict to finitely generated groups to obtain a meaningful definition.

Let $G$ be a group with a finite generating subset $S$. Given a $G$-set $X$, recall that the (unlabeled) Schreier graph $\Gamma(G,S,X)$\index{Schreier graph} is the graph with set of vertices $X$, where two distinct vertices $x,y$ are adjacent if there exists $s\in S^{\pm 1}$ such that $sx=y$.
The connected components of the Schreier graph are the $G$-orbits in $X$. This graph has finite valency, namely bounded above by $2|S|$.

%Write $E_s=\{\{x,y\}:x\neq y\in X,y=sx\}$. Then the set of edges is $\bigcup_{s\in S}E_s$.To enlarge this notion to near actions, let us restrict to the context when $S$ is finite. 

Consider now a near action of $G$ on a set $X$, and choose a partial bijection $\bar{s}$ representing $s$ for each $s\in S$. Similarly as above, we define the set of edges as $E=\bigcup_{s\in S}E_s$, with $E_s=\{\{x,y\}:x\neq y\in X,y=\bar{s}x\}$. For another choice of representatives, we obtain another set of edges $E'$, and $E$ and $E'$ have a finite symmetric difference. Hence the resulting graph, called near Schreier graph\index{near Schreier graph} is defined ``up to adding and removing finitely many edges". An important remark is that commensurated subsets are the same as subsets with finite boundary. It follows that the space of ends naturally coincides with the space of ends of the near Schreier graph.

Various graph properties are affected, or unaffected, by such a perturbation of the graph structure: for instance the number of components can be changed; however to have finitely many components is unchanged. Next, the graph changes when one replaces $S$ by another finite generating subset. The graphs obtained in this way are called the near Schreier graph of the near $G$-action.

\begin{defn}
A near action of a finitely generated group $G$ is called of finite type\index{finite type (near action)} if some/every near Schreier graph has finitely many components, and of infinite type otherwise.
\end{defn}

Thus, a $G$-action has finitely many orbits if and only if it induces a near action of finite type. An immediate observation is that a finitely-ended near action has finite type; the converse fails in general. Still, for finitely generated abelian groups we will establish the equivalence.

%There are indeed some finitely generated groups $G$-for which it is known that every transitive $G$-set is finitely-ended, or, equivalently, that every completable near action of finite type is finitely-ended:
%\begin{itemize}
%\item virtually polycyclic groups;
%\item finitely generated groups with Property FW, namely those for which every transitive $G$-set is $\le 1$-ended (this includes groups with Kazhdan's Property T);
%\item some extensions of such finitely generated groups (see \cite[Prop.\ 5.B.11]{CorFW})
%\item irreducible lattices in $\mathrm{SL}_2(\R)^k$, $k\ge 2$, for which Property FW is conjectural, but for which it is known that every transitive $G$-set is $\le 2$-ended \cite{NRa}.
%\end{itemize}

%We do not know any example of a finitely generated group $G$ such that every transitive $G$-set is finitely ended but such that there exist infinitely-ended near $G$-set of finite type.

% EVERY NEAR $G$-SET = REALIZABLE + FINITELY MANY ISOLATED?

The following result is not hard, but plays an essential role in the study of near actions.

\begin{thm}[Theorem \ref{nearactionfp}]\label{nafpi}
Let $G$ be a finitely presented group. Then every near $G$-set is disjoint union of a realizable near $G$-set and a near $G$-set of finite type.
\end{thm}

For an arbitrary finitely generated group $G$, such a decomposition does not hold in general, but holds for completable near $G$-sets. This provides an obstruction for near $G$-sets (of infinite type) to be completable, see \S\ref{s_rnfg}.

For finitely presented groups, Proposition \ref{nafpi} essentially reduces the study of realizability and its variants to the case of near actions of finite type. In the case of finitely generated abelian groups, this is a first step towards a classification, outlined in \S\ref{i_cfgag}.

%Many coarse invariants of graphs yield meaningful notions for near actions 

\subsection{Near actions of locally finite groups}\label{i_lofi}
Here is a sample of results:

\begin{itemize}
\item (Theorem \ref{coulofi}) Let $G$ be a countable locally finite group. Then every near $G$-set is completable.
\item (\S \ref{unlofi}) There exists an uncountable locally finite (abelian) group with a non-completable near action (on a countable set).
\item (Theorem \ref{coulofi}) Let $G$ be a countable restricted direct product of finite groups. Then every near $G$-set is realizable.
\item (\S\ref{s_qcy}) For $m\ge 2$, consider the quasi-cyclic group $C_{m^\infty}=\Z[1/m]/\Z$. To each near free near action $\rho$ of $C_{m^\infty}$, we can associate in a natural way an $m$-adic number $\nu_\rho\in\Z_m$. It satisfies the following:
\begin{itemize}
\item $\nu_\rho$ a balanced near isomorphism invariant of the near $C_{m^\infty}$-action;
\item the image of $\nu_\rho$ in $\Z_m/\Z$ is a near isomorphism invariant of the near $C_{m^\infty}$-action;
\item $\nu_{\rho_1\sqcup\rho_2}=\nu_{\rho_1}+\nu_{\rho_2}$;
\item $\nu_\rho=n$ for the near action on a finite set of cardinal $n$;
\item $\nu_\rho=0$ for a free action $\rho$;
\item $\rho$ is stably realizable if and only if $\nu_\rho\in\Z$;
\item $\rho$ is realizable if and only if $\nu_\rho\in\N$;
\item every $m$-adic number equals $\nu_\rho$ for some near action on some countable set. 
\end{itemize} 
In particular, $C_{m^\infty}$ admits near actions that are not stably realizable, and near actions that are stably realizable but not realizable. Note that this provides near isomorphism invariants for near actions of any group having a subgroup isomorphic to $C_{m^\infty}$, such as Thompson's group $T$ of the circle; see Corollary \ref{corthompt}.
\item For every infinite set and countable locally finite subgroup $G$ of $\mkst(X)$, the centralizer of $G$ in $\mkst(X)$ (that is, the automorphism group of the near $G$-set $X$) has cardinal $2^{|X|}$.

In particular, no countable maximal abelian subgroup of $\mkst(\N)$ is locally finite. Conversely, we prove that every countable abelian group that is not locally finite is isomorphic to a maximal abelian subgroup of $\mkst(\N)$. See \S\ref{morecen}.

%\item A group $G$ admits a 1-ended near $G$-set if and only if it is not (countable locally finite).
\end{itemize}

\subsection{Classification for finitely generated abelian groups}\label{i_cfgag}

%We obtain this for virtually abelian groups, as a first step to a classification:

The following theorem, together with Proposition \ref{nafpi}, reduces the classification of near $G$-sets to 1-ended ones.

\begin{thm}
Let $G$ be a finitely generated, virtually abelian group. Then every near $G$-set of finite type is finitely-ended.
\end{thm}

This is proved by immediate reduction to the case of a finitely generated abelian group, and then by reduction to $\Z^2$, where we use explicit cut-and-paste arguments in the near Schreier graph.

 We say that a near $G$-set $X$ is $\star$-faithful\index{faithful ($\star$-)} if the homomorphism $G\to\mkst(X)$ is injective (when it comes from an action, this is stronger than being faithful, so we avoid calling it ``near(ly) faithful"). For a 1-ended near $G$-set with $G$ abelian, this is the same as being near free.

%; we say that it is near free if for every $g\in G\smallsetminus\{1\}$, the set of fixed points of some/any representative of $g$ is finite. 

Before formulating a theorem for arbitrary finitely generated abelian groups, let us write it in the case of $\Z^2$, which is the richest case. In \S\ref{s_windholo}, we define, for every integer $m\ge 1$ and $v\in\Z^2$, a 1-ended, $\star$-faithful near $\Z^2$-set $X_{m,s}$; the number $m$ is called its winding number\index{winding number} and $s$ its additive holonomy\index{additive holonomy}. The simply transitive action of $\Z^2$ appears as $X_{1,0}$. In Example \ref{introex}, (a) refers to $X_{m,0}$ and (b) refers to $X_{1,(1,0)}$. 
%(\ref{exz2})

\begin{thm}
The near $\Z^2$-sets $X_{m,s}$, $m\ge 1$ and $s\in\Z^2$ are pairwise non-isomorphic; among them, only $X_{1,0}$ (the simply transitive action) is completable; every $\star$-faithful 1-ended near $\Z^2$-set is near isomorphic to $X_{m,s}$ for some $(m,s)$.
\end{thm}

\begin{thm}
Let $G$ be a finitely generated abelian group, and $d$ its $\mathbf{Q}$-rank (so $G$ is isomorphic to the direct product of $\Z^d$ with a finite abelian group).
\begin{enumerate}
\item\label{td3} If $d\ge 3$, then every 1-ended, $\star$-faithful near $G$-set $X$ is stably realizable, and more precisely is near isomorphic to the simply transitive action. If $X$ is completable, this also holds for $d=2$. 
\item\label{td1} If $d=1$ and $f$ is a surjective homomorphism $G\to\Z$ (so $f$ is unique up to sign), then there are exactly two 1-ended, $\star$-faithful near $G$-sets up to near isomorphism: $f^{-1}(\N)$ and $f^{-1}(-\N)$.
\item\label{td2} Suppose $d=2$ and we choose a decomposition $G=\Z^2\times F$. Let $X$ be a 1-ended, $\star$-faithful near $G$-set. Then there exists unique integers $m,n\ge 1$ (with $n$ dividing $|F|$) and a unique $s\in\Z^2$ such that in restriction to $\Z^2$, $X$ is isomorphic to the disjoint union of $n$ copies of $X_{m,s}$. In this case, $X$ is completable if and only if $(m,s)=(1,0)$.
\end{enumerate}
\end{thm}

In the case $d=2$, we need to be a little more precise to provide a full classification, but we rather postpone the general statement to \S\ref{s_windholo} (namely Theorem \ref{msuh}). 

%Let us only provide it in some significant cases. First, when $m=1$, or when $s$ is primitive, more generally when there exist no prime $p$ dividing simultaneously $m$ and $v$, and $|F|$. Then the only possibility is the plain product, namely the product of the near $\Z^2$-set $X_{m,s}$ and the simply transitive action of $F$ (so $n=|F|$). (In the completable case $(m,s)=(1,0)$, this shows that $X$ is near isomorphic to the simply transitive action, as claimed in (\ref{td3}).)

%For near $G$-sets of finite type over finitely generated groups, there is a notion of 

%\begin{cor}
%Let $G$ be a finitely generated abelian group. Let $X$ be a near $G$-set with no infinite commensurated subset of at most quadratic growth. Then $X$ is stably realizable. 
%\end{cor}

The following corollary, proved in \S\ref{s_abnors}, is a sample of how these results could be used in the context of possibly non-abelian groups, in the presence of normal abelian subgroups.

% to obtain structural results in non-abelian groups.

Recall that the $\Q$-rank\index{rank ($\Q$-)} of an abelian group $G$ is the dimension over $\Q$ of $G\otimes_Z\Q$. It is also the cardinal of a maximal $\Z$-free family in $G$.

\begin{cor}\label{abnors}
Let $G$ be a group. Let $A$ be a finitely generated, normal and abelian subgroup of $G$. For a finite index subgroup $H$ of $G$ including $A$, let $d_H$ be the $\Q$-rank of the image of $A$ in the abelianization of $H$; define $d=\sup_Hd_H$. Let $X$ be a near $G$-set.
\begin{enumerate}
\item If $d\le 1$, then $X$ is completable as near $A$-set;
\item Suppose that $d=0$ (that is, $[H,H]\cap A$ has finite index in $A$ for every finite index subgroup $H$ of $G$ including $A$). Assume in addition that one of the following holds:
\begin{itemize}
\item $A$ is central in $G$;
\item the action of $G$ on the set of hyperplanes of $A\otimes_\Z\Q$ has no finite orbit.
\end{itemize}
Then $X$ is stably realizable as near $A$-set.
\end{enumerate}
\end{cor}

\subsection{Realizability for some amalgams of finite groups}\label{i_amalgamfini}

The condition that every near $G$-set is realizable is satisfied by finite groups and is stable under free products; in particular it holds for free products of finite groups. Beyond the case of free products, the following result illustrates the results of \S\ref{s_amalgam}. Write $C_n=\Z/n\Z$.

\begin{thm}\label{ti_amalgam}
Fix a prime $p$. For $n\ge 1$, consider the amalgam $G_{p,n}=C_{np}\ast_{C_p} C_{p^2}$. Then:
\begin{enumerate}
\item if $p$ does not divide $n$, then every near $G_{p,n}$-set is realizable (this applies in particular to $\mathrm{SL}_2(\Z)\simeq G_{2,3}$);
\item if $p$ divides $n$, then there is a natural near isomorphism invariant $i_X\in\Z/p\Z$ for every near $G_{p,n}$-set $X$; it achieves all values in $\Z/p\Z$; it is additive under disjoint unions, and we have the equivalence: $i_X=0$ $\Leftrightarrow$ $X$ is realizable $\Leftrightarrow$ $X$ is stably realizable.
\end{enumerate}
\end{thm}

\subsection{Uniqueness of realizations: rigidity phenomena}\label{i_rig}

So far we have emphasized the problem of existence of realizations of near actions, but the uniqueness is also a very natural problem. Fixing a realization, it can be formulated as follows: given an action $\beta$ of a group $G$, what are the actions $\beta'$ such that $\beta$ and $\beta'$ define the same near action (that is, $\beta(g)^{-1}\beta'(g)$ is finitely supported for every $g\in G$). Such $\beta'$ is called a finite perturbation of $\beta$.\index{finite perturbation} Those ``trivial" finite perturbation are those of the form $\beta'(g)=f\circ \beta(g)\circ f^{-1}$, for some finitely supported permutation $f$. The simplest example of nontrivial finite perturbation is just changing a nontrivial action on a finite set to the trivial action. For a finitely generated group with a given finite generating subset, two actions on the same set are finite perturbations of each other if and only if the (labeled) near Schreier graphs differ by only finitely edges.

%We mostly obtain actions for free 

A sample theorem of \S\ref{s_rig}, whose proof was already given in one particular case in Example \ref{abminuspt}, is the following, contained in Theorem \ref{nonrea_1end}:

\begin{thm}
Let $G$ be a 1-ended group that is not locally finite. Then every perturbation of a free action of $G$ is conjugate to the original action by some finitely supported permutation.
\end{thm}

This applies, for instance, to all abelian groups that are not locally finite nor virtually cyclic, such as $\Z^2$ or $\Z[1/2]$. See \S\ref{s_equivariant} for more.

As in Example \ref{abminuspt}, it is based on the following particular case of Corollary \ref{ghgh}:

\begin{thm}
Let $G$ be a 1-ended group. Let $\alpha,\beta$ be the left and right near actions of $G$ on itself. Then $\beta$ induces an isomorphism from $G$ to the automorphism group of the near $G$-set $(G,\alpha)$. In other words, $\alpha(G)$ and $\beta(G)$ are the centralizer of each other in $\mkst(G)$.
\end{thm}

The same statement inside $\mks(G)$ is true (as a standard elementary exercise) for arbitrary groups, but the above one fails in general, and actually its truth characterizes $(\le 1)$-ended groups. For instance, the centralizer of $\Z$ in $\mkst(\Z)$ is free abelian of rank 2.

\subsection{Maximal abelian subgroups}\label{i_mas}

Every infinite countable abelian group $G$ is isomorphic to a maximal abelian subgroup of $\mks(\N)$: indeed, $G$ equals its own centralizer in $\mks(G)$.

Using results of \S\ref{s_lofi} in the locally finite case, and of \S\ref{s_equivariant}, we obtain the following (\S\ref{morecen}):

\begin{thm}
Let $G$ be a countable abelian group. Then $G$ is isomorphic to a maximal abelian subgroup of $\mkst(\N)$ if and only if $G$ is not locally finite.
\end{thm}

We also obtain a similar result in $\mksh(\N)$: in this case we have to also exclude groups that are (virtually cyclic but not cyclic). The maximal abelian subgroups of $\mksh(\N)$ were considered by Shelah and Stepr\={a}ns \cite{ShS}, but essentially excluding the countable case.

\subsection{Almost and near automorphism groups}\label{i_anag}

These groups are considered in \S\ref{s_neretin}, although they also implicitly appear at other places. The setting of \S\ref{s_neretin} is that of finitary relational structures; for simplicity we restrict here to graph structures. Namely, consider a set $X$, and a finite family $P=(P_i)_{i\in I}$ of subsets of $X^2$. We can think of it as a labeled graph structure on $X$, with oriented edges $(x,y)$ labeled by $i$, when $(x,y)$ ranges over $P_i$ and $i\in I$. We assume that $P$ is biproper\index{biproper}, in the sense that this graph is locally finite; however no connectedness assumption is made.

Note that such structures can encode non-oriented graphs (just assume that $P_i$ is symmetric), or vertices colorings (choose $P_i$ included in the diagonal).

The almost automorphism group\index{almost automorphism group} of $(X,P)$, introduced by Truss \cite{Tru87,Tru89} in the case of the Rado graph, is the group $\mks(X,\lcc P\rcc)$ of permutations $g$ of $X$ such that $gP_i\tu P_i$ is finite for all $i\in I$.

For instance, for $P$ empty, this is just $\mks(X)$. In general, it includes the stabilizer $\mks(X,P)$ of $P$ as a subgroup (the group of labeled automorphisms of the graph $(X,P)$). The almost automorphism group $\mks(X,\lcc P\rcc)$ admits a natural topology for which $\mks(X,P)$ is an open subgroup. Here $\mks(X,P)$ is endowed with its topology as a closed subgroup of $\mks(X)$; however, the resulting topology on $\mks(X,\lcc P\rcc)$ is usually finer than its topology as a (dense) subgroup of $\mks(X)$; for instance for $X$ countable, this is a Polish group; when $(X,P)$ has finitely many connected components, this is a locally compact group.

One example is the Hilbert hotel join, which we now interpret as follows: let $X$ be a set, viewed as a discrete graph, and $\N$, viewed as an oriented graph with edges $(n,n+1)$. Consider the disjoint union $X\sqcup\N$ as a graph. Then its almost automorphism group is equal to the group $\HilH_X$ of Remark \ref{hilbho}.

%This has notably been studied when $X$ is a locally finite tree

The near automorphism group\index{near automorphism group} $\mkst(X,\lcc P\rcc)$ of $(X,P)$ is defined, informally, as the set of near permutations of $X$ that ``preserve $P$ up to finite error". More rigorously, the group $\mkst(X)$ naturally acts on the set of biproper subsets of $X^2$ modulo finite symmetric difference, and $\mkst(X,\lcc P,\rcc)$ is defined as the subgroup of $\mkst(X)$ preserving the class $\lcc P_i\rcc$ of $P_i$ modulo finite symmetric difference, for each $i$.

There is a canonical homomorphism $\mks(X,\lcc P\rcc)\to \mkst(X,\lcc P\rcc)$; its kernel is not always closed (it is not closed when $P$ is empty, for instance); however it is closed in many interesting cases, notably when $(X,P)$ is tree with no vertex of valency~1. In this case, there is a natural Hausdorff group topology on $\mkst(X,\lcc P\rcc)$ for which the image of $\mks(X,P)$ is an open subgroup.

\begin{exe}
The near automorphism group is notably known when $(X,P)$ encodes a regular tree, and then known as Neretin's group.

When it encodes a regular binary tree in which edges have two colors, and each vertex has two successors, joined by edges with distinct colors, the near automorphism group is known as Thompson's group $V$.
\end{exe}

Many other natural examples can be considered. For instance, if $G$ is a finitely generated group and $X$ is a near $G$-set, the automorphism group of the near $G$-set $X$ (that is, the centralizer of the image of $G$ in $\mkst(X)$) can naturally be interpreted as the near automorphism group of any of its (labeled) near Schreier graphs.

\subsection{Link with cofinite-partial actions}

A cofinite-partial action of a group $G$ on a set $X$ is the datum, for each $g\in G$, of a bijection between two cofinite subsets, with a few natural axioms, saying that things ``behave like a group action, when defined". See \S\ref{s_partact} for details. This notion is due to Exel (without the cofinite restriction).

It is important to clarify the difference between this and near action: the crucial difference is that, in a cofinite-partial action, each $g\in G$ acts with a given domain of definition and target, which are part of the data. Unlike in the near symmetric group, we do not identify two cofinite-partial bijections when they coincide on a cofinite subset. Every cofinite-partial action defines a near action. This yields two natural questions: when does a near action arise from a cofinite-partial action, and in how many ways? The first question has a simple answer with surprisingly simple proof.

\begin{prop}
For any group $G$, a near $G$-set $X$ arises from a cofinite-partial action on $X$ if and only if it is a completable near $G$-set.
\end{prop}

This is particularly useful to show that an action is completable. Indeed, in practice the completion of the action, although possible to describe formally (as quotient of $G\times X$ by a suitable equivalence relation), has a quite complicated and non-practical form. An example for which this applies is the following:

\begin{thm}
For every perfect Hausdorff topological space, the near action of the group of near self-homeomorphisms $\PC(X)$ (see Example \ref{secondex}) on $X$, is completable.
\end{thm}

Indeed, it is checked in \cite[Proposition 2.16]{Corpcw} that it naturally arises from a cofinite-partial action.

\subsection{Related contexts}

%In spite of its simplicity, I am not aware of any previous work addressing  in general the problem of realizability in the current setting.

%previous works on the problem of realizability of actionsOn the other hand, r

The concept (and term) of realizability first appeared in a different setting, namely the realization of measure-preserving actions on probability spaces, where the issue is to lift from the automorphism group of a probability space (where one identifies automorphisms matching almost everywhere) to the group of measurable measure-preserving permutations with measurable inverse. This problem is trivial for countable group actions, so is essentially disjoint to our study; it is generally studied for Polish group actions. It was initiated by Mackey \cite{Ma}.

Realizability also appears in another setting, less related at first sight, but where some analogies are more striking, namely the context of Fredholm isometries. In this context, there is indeed a notion of $\Z$-valued index, and the realizability issues in this context have been considered by de la Harpe and Karoubi \cite{HK1,HK2,HK3}. For instance, they also considered the case of amalgams of finite groups. However, in contrast to the results of \S\ref{i_amalgamfini},
in their setting, realizability always holds for amalgams of two finite abelian groups, but not for amalgams of two arbitrary finite groups. A closely related work, in this context, is that of Brown, Douglas and Fillmore \cite{BDF}.

%In the context of Fredholm isometries, the realizability problem has been tackled by de la Harpe and Karoubi; 

%, which appeared simultaneously in several papers at the beginning of this century

Also, an analogue of the near symmetric group is the abstract commensurator of a group $G$, whose earliest appearance I could detect is in a 1994 paper of A'Campo and Burger \cite{ACBu}, and studied for various groups since then. One first defines a monoid of cofinite partial isomorphisms, consisting of all isomorphisms between two finite index subgroups, and then modding out by coincidence on a finite index subgroup, one obtains a group, called the abstract commensurator of the group $\mathrm{AComm}(G)$. There is an obvious homomorphism from $\Aut(G)$ to $\mathrm{AComm}(G)$, whose kernel is the set of automorphisms that are equal to the identity on a finite index subgroup (this kernel can be trivial or not). There is an analogue of the index homomorphism, valued in the multiplicative group of positive rational numbers. I am not aware of any study of realizability notions in this context.

%A'Campo, Norbert(CH-BASL); Burger, Marc(CH-LAUS)
%\bibitem[ACBu]{ACBu} N. A'Campo, M. Burger.
%R\'eseaux arithm\'etiques et commensurateur d'apr\`es G. A. Margulis. 
%Invent. Math. 116 (1994), no. 1--3, 1--25. 

%Farb, B.; Mosher, L. The geometry of surface-by-free groups. Geom. Funct. Anal. 12 (2002), no. 5, 915?963. (Reviewer: Athanase Papadopoulos) 20F65 (20F67 30F60 57M07)

%Review PDF Clipboard Journal Article 15 Citations
%MR1914566 Reviewed Farb, Benson; Mosher, Lee Convex cocompact subgroups of mapping class groups. Geom. Topol. 6 (2002), 91?152. (Reviewer: Luis Paris)

\subsection{Groups and monoids}

The sole definition of the near symmetric group shows that it is more naturally introduced in the framework of monoids. For this reason, we have made a detailed introduction to such monoids in \S\ref{s_rough}. Such a point of view is useful even with a primary interest in the group setting (which is the case of the author). The near symmetric group is naturally viewed as the subgroup of invertible elements in several larger monoids.

For instance, let us consider a free semigroup $S$ on $2$ generators; then $S$ naturally embeds into the monoid of self-maps of $S$, which itself naturally maps into the monoid $\mkrst(S)$ of near self-maps of $S$, whose group of invertible elements is $\mkst(S)$. Then the centralizer of $S$ in $\mkst(S)$ naturally isomorphic to Thompson's group $V$.

Extending the setting to a non-bijective one is not plain routine: one difficulty is that the equivalence relation of coincidence outside a finite subset is not compatible with composition in general.

\subsection{Outline and guidelines}

Most sections and subsections are independent: the common necessary background to read the paper has been defined in the introduction: the notion of near symmetric group and of near action, and the realizability notions. The notion of near action of finite type is also regularly used, whenever we consider near actions of finitely generated groups. Given this, the reader should be able to start reading any section and refer to others when punctually needed (using the index if necessary).

% for a deeper overview of near actions,
% I would still recommend in priority to read \S\ref{s_nearac}, \S\ref{s_rig}.

\S\ref{hire} gathers historical remarks about the appearance of infinite symmetric groups, their quotients, and the near symmetric groups. No survey on this vast theme seems to exist in the literature.

\S\ref{s_rough} introduces a categorical framework where near actions naturally appear, and includes many definitions. It is primarily addressed to readers with a set-theoretic or categorical point of view; others can content themselves to only punctually leap to this section when necessary.

\S\ref{s_nearac} introduces near actions, developing in detail notions given in the introduction, including proofs, and also addressing several notions that were skipped in the introduction.

\S\ref{s_examples} provides many examples motivating the definitions of \S\ref{s_nearac}. They are written in a separate section for the sake of coherence, but we recommend the reader to freely switch between these two sections.

The shorter \S\ref{s_nlofi}, summarized in \S\ref{i_lofi} is concerned with near actions of locally finite groups.

\S\ref{s_rig}, evoked in \S\ref{i_rig}, includes some crucial results of the rigidity of free actions of 1-ended groups, and some extensions.

\S\ref{s_endsri} considers further more elaborate aspects of near actions; its three subparts are independent. The first two are still related in spirit: it is the notion of amenability, and of growth. The last one is about the Kapoudjian class introduced in \S\ref{i_kapou}.

In \S\ref{s_ab}, summarized in \ref{i_cfgag}, we classify near actions of finitely generated abelian groups.

%In \S\ref{s_specific} and 
%\S\ref{s_ab}, we obtain specific results on near actions for some particular classes of groups. We have gathered the results concerning near actions of finitely generated groups in \S\ref{s_ab}.

In \S\ref{s_neretin}, we illustrate the notion of near action in the context of near automorphism group of graphs, and notably of trees, where it has already been studied (without this language and point of view).

\S\ref{cenesy} studies centralizers in near symmetric groups. Centralizers of finitely generated subgroups are particular instances of near automorphism groups. Centralizers of cyclic subgroups are fully described in \S\ref{i_mas}. Maximal abelian subgroups are considered in \S\ref{morecen} (briefly summarized in \S\ref{i_mas}).

%It also includes a subsection about centralizers in near permutations groups (briefly summarized in \S\ref{i_mas}).

 Another setting where the notion of near action is natural is the group of piecewise continuous transformations of the circle; it will be considered in a separate paper \cite{CIET}.

\tableofcontents

%%%%%%%%%%%%%%%%%%%%%%%%%%%%%%%%%%%%%%%%%%%%%%%%%%%%%%%%%%%
%%%%%%%%%%%%%%%%%%%%%%%%%%%%%%%%%%%%%%%%%%%%%%%%%%%%%%%%%%%
%%%%%%%%%%%%%%%%%%%%%%%%%%%%%%%%%%%%%%%%%%%%%%%%%%%%%%%%%%%

\section{Historical remarks}\label{hire}
The study of infinite symmetric groups possibly starts with two Italian articles, both in 1915. The first is a short note by Vitali \cite{Vi}, where he shows --- without the language of group theory --- a result which, conveniently rephrased, states that there is no nonzero homomorphism from the symmetric group on an infinite countable set into $\Z/2\Z$; the proof consists in proving that every permutation is a product of squares.
The second, by Andreoli \cite{And}, considers the monoid of injective self-maps of an arbitrary set and notably discusses the cycle decomposition. Later in 1927-1929, Onofri's series of articles \cite{On1,On2,On3} goes much deeper into the group structure of the symmetric group $\mks(X)$ on a countable set $X$, and notably establishes \cite[\S 141]{On3} that it has exactly 4 normal subgroups: the two obvious ones, the subgroup of finitely supported permutations $\mks_{<\aleph_0}(X)$, and its alternating subgroup of index two $\mka(X)$, consisting of even finitely supported permutations. Onofri's theorem was rediscovered soon after (1933) by J.\ Schreier and Ulam \cite{ScU}. This was generalized by Baer
\cite{Bae} in 1934 to a classification of normal subgroups of $\mks(X)$ for an arbitrary infinite set $X$, as stated in Theorem \ref{baerthm}. Onofri \cite[\S B]{On1} also introduced the natural pointwise convergence topology on $\mks(X)$, which was rediscovered more than 25 years later.

% \cite{Bae} in 1934 to a classification of normal subgroups of $\mks(X)$ for an arbitrary infinite set $X$: normal subgroups form a well-ordered chain under inclusion, and are $\{1\}$, $\mka(X)$, and, for infinite cardinals $\kappa$ at most equal to the smallest cardinal $|X|^+$ greater than $|X|$, the subgroup $\mks_{<\kappa}(X)$ of permutations with support of cardinal $<\kappa$. See \cite[\S 11.3]{Sco} or \cite[\S8.1]{DM}. 

Let us make an attempt to survey results about the objects in consideration in this chapter, and especially around the group of near permutations. Existing results roughly fall into:
\begin{itemize}
\item[(a)] results about symmetric groups $\mks(X)$, which mechanically yield results about its group quotients;
\item[(b)] specific results explicitly involving quotients of symmetric groups;
\item[(c)] definitions and/or results explicitly involving the group of near permutations, where the index character is made explicit.
%\item[(d)] definitions of the various monoids introduced in this chapter.
\end{itemize} 

Let us start with a few facts pertaining to (a):

\begin{enumerate}[leftmargin=*]
\item
The above Onifri-Baer results are typically in (a): for $X$ infinite countable, the group $\mksh(X)$ is simple, and more generally, for $X$ arbitrary, normal subgroups of $\mksh(X)$ correspond to normal subgroups of $\mks(X)$ including $\mks_{<\aleph_0}(X)$ and in particular form a well-ordered chain. While the alternating group $\mka(X)$ disappears from this list, it reappears in the central extension (for $X$ infinite)
\[1\to \Z/2\Z\to \mks(X)/\mka(X)\to\mksh(X)\to 1.\]
This gives a canonical cohomology class in $H^2(\mksh(X),\Z/2\Z)$, studied (in a specific context) by Kapoudjian and Kapoudjian-Sergiescu \cite{Ka,KaS}; we pursue its study and extend it to $\mkst(X)$ in \S\ref{s_omega}.

\item P.\ de la Harpe and D.\ McDuff proved that the group $\mks(X)$ is acyclic (has trivial homology with arbitrary trivial coefficients). A simple consequence, observed by Sergiescu, is the that $H_2(\mksh(X))\simeq\Z/2\Z$ (see Proposition \ref{sergi}), and hence $H^2(\mksh(X),\Z/2\Z)\simeq\Z/2\Z$. The higher homology of $\mksh(X)$ has not yet been studied, up to my knowledge.

\item Another fact pertaining to (a) is the result of Rabinovi\u c and Dixon-Neumann-Thomas \cite{Ra,DNT} that any subgroup of index $<2^{\aleph_0}$ in a symmetric group $\mks(X)$ is open. Here $\mks(X)$ is endowed with its standard group topology, which consists of pointwise convergence on $X$. Since $\mka(X)$ is a dense subgroup, an immediate consequence is that $\mksh(X)$ (and more generally, any proper quotient of $\mks(X)$) has no proper subgroup of index $<2^{\aleph_0}$.

\item A group $G$ is strongly distorted if there exists a function $u:\N\to\N$ such that for every sequence $(f_n)$ in $G$, there exists a finite subset $S\subseteq G$ such that for each $n$, we can write $f_n$ as a word of length $\le u(n)$ with respect to $S$. Strong distortion has a consequence, namely the property known as strong boundedness, or Bergman property: every subadditive function $f:G\to\R_{\ge 0}$ is bounded. In turn, strong boundedness implies the property called ``uncountable cofinality", which means that every $f:G\to\R_{\ge 0}$ such that $f(gh)\le\max(f(g),f(h))$ for all $g,h\in G$, is bounded. All these properties obviously pass to quotient groups.

Galvin proved (but did not state) in \cite{Ga} that $\mks(X)$ is strongly distorted, with $u(n)=36n+100$; his statement is that every subgroup of $\mks(X)$ is included in a finitely generated subgroup, with some refinements on the nature of $S$, but no claim about uniformity of $u$, which follows from the construction, and is essential to derive the strong boundedness corollary.

\item The cofinality of a group is the least cardinal $\kappa$ such that the group is union of a chain of $\le\kappa$ proper subgroups. This exists if and only if the group is infinitely generated. For finitely generated groups, the cofinality is understood as being greater than any cardinal, and ``uncountable cofinality" means cofinality $>\aleph_0$, and matches with the above ad-hoc definition.

MacPherson and Neumann \cite{MN} proved that $\mks(X)$ has cofinality $>|X|$. Thus it has uncountable cofinality; this particular case also follows from Galvin's above-mentioned result.

Cofinality increases when passing to quotients. In \cite[Prop.\ 3.2]{FS}, it is asserted that in the case of symmetric groups and nontrivial quotients, this is an equality, but the proof relies on \cite[Theorem 1.2]{MN}, which is false (as noticed in \cite{Bi}). It follows, however from \cite[Prop.\ 3.2]{Bi} that it is correct for the quotient by the subgroup of finitely supported permutations, thus $\mks(X)$ and $\mksh(X)$ have the same cofinality for every set $X$. In models of ZFC where the continuum hypothesis fails, the study of the cofinality of $\mks(\aleph_0)$ has notably been studied by Shelah and Thomas \cite{ShTh}.

\item\label{smallindex} Gaughan \cite{Gau} proved that for any infinite set $X$, the group $\mks(X)$ (as well as its subgroups $\mks_{<\kappa}(X)$ for $\kappa$ uncountable) has no proper subgroup of index $<|X|$. This property obviously passes to quotients.

Rabinovi\u c proved \cite{Ra} that every subgroup of countable index in the group of permutations of a countable set is open; Dixon-Neumann-Thomas \cite{DNT} generalized the statement to subgroups of index $<2^{\aleph_0}$. 
Since the subgroup of finitely supported permutations is dense, as a corollary, we deduce that $\mksh(\aleph_0)$ has no proper subgroup of index $<2^{\aleph_0}$.

Dixon-Neumann-Thomas obtained another generalization of Rabinovi\u c's theorem to arbitrary infinite sets $X$. Consider the topology $\mathcal{T}_\kappa$ on $\mks(X)$ for which a basis of open subsets consists of pointwise stabilizers of subsets of cardinal $<\kappa$; this is a group topology. The topology $\mathcal{T}_{\kappa}$, for $\kappa$ uncountable, is much larger than the usual one of pointwise convergence $\mathcal{T}_{\aleph_0}$; in particular the subgroup of finitely supported permutations is closed and non-discrete. Dixon-Neumann-Thomas proved that any subgroup of index $\le |X|$ in $\mks(X)$ is open in $(\mks(X),\mathcal{T}_{|X|})$. Shelah and Thomas \cite{ShTh1} studied whether this statement holds for subgroups of cardinal $<2^{|X|}$ and obtained a positive answer under the generalized continuum hypothesis, as well as the consistency of a negative answer assuming the consistency of the existence of some large cardinals.

% and seems to assume implicitly that they are indeed pairwise distinct; this is easy to check 

\item Rudin observed \cite[Theorem 1.6]{Rud} that stabilizers in $\mks(X)$ of ultrafilters are pairwise distinct subgroups when $X$ is countable (see Proposition \ref{stabmeandist} for a proof of a generalization); Richman \cite{Ric} proved that they are maximal subgroups. It immediately follows that stabilizers of non-principal ultrafilters in $\mks(X)/\mks_{<\aleph_0}(X)$ are pairwise distinct maximal subgroups of $\mksh(X)$. Ball and Richman \cite{Bal1,Bal2,Ric} constructed various other maximal subgroups in $\mks(X)$, most of which include $\mks_{<\aleph_0}(X)$ and thus define maximal subgroups of $\mksh(X)$. There have been various subsequent works about maximal subgroups, especially from the late eighties on, often in relation with the small index properties ((\ref{smallindex}) above); see notably the 1993 survey \cite{Macp} for early references.

\item We finish this account of (a) with a more speculative item. The usual (Polish) topology yields a notion of genericity in $\mks(\aleph_0)$. It was observed by Lascar and Truss \cite{La,Tr92} that $\mks(\aleph_0)$ has a (necessarily unique) $\mathrm{G}_\delta$-dense conjugacy class, which consists of elements with no infinite cycle and infinitely many cycles of each finite length. This conjugacy class is invariant by multiplication by finitely supported permutations, and therefore maps to a ``canonical" conjugacy class in $\mksh(\aleph_0)$, which we would like to think of as generic, in a sense which is unclear so far.

This extends to $d$-tuples \cite{HHLS}: for $F_d$ a free group on $d$ given generators, let $X_d$ be the (countable) $F_d$-set consisting of the disjoint union of $\aleph_0$ copies of each (isomorphism class of) finite $F_d$-set.  Then those $d$-tuples in $\mks(\aleph_0)^d$ for which the corresponding $F_n$-action is isomorphic to $X_d$ form a $\mathrm{G}_\delta$-dense class in $\mks(\aleph_0)^d$ under the conjugation action of $\mks(\aleph_0)$. It is clearly invariant by left and right multiplication by $n$-tuples of finitely supported permutations, and hence defines a canonical $\mksh(\aleph_0)$-orbit in $\mksh(\aleph_0)^d$, which again we would like to think of as generic. It is readily seen to be $\mkst(\aleph_0)$-invariant which, in view of Truss' theorem that $\Aut(\mksh(\aleph_0))=\mkst(\aleph_0)$, means this orbit in $\mksh(\aleph_0)^d$ is ``intrinsic" to the group $\mksh(\aleph_0)$.
%Also, for larger cardinals $\kappa$, the same definition, with each finite $F_d$-set occurring $\kappa$ times, should constitute some ``generic" $\mksh(X)$-conjugacy class in $\mksh(X)^d$.
\setcounter{saveenum}{\value{enumi}}
\end{enumerate}

% (one passes from countable $X$ to arbitrary $X$ by observing that $\mksh(X)$ is comment mksh(N) va dans mksh(X)? bof. mks(N) va dans mks(X=YxN) support fini= support verticalement borne $\le |X|$ in $\mks(X)$ are open. POUR AUTRE TOPOLOGIE OUVERTS SONT $<|X|$! As a corollary, proper quotients of $\mks(X)$ have no proper subgroup of index $\le |X|$. Dixon-Neumann-Thomas also proved that subgroups of index $<2^{\aleph_0}$ in $\mks(X)$ are open. Whether this can replaced with ``$<2^{|X|}$" The group of permutations of a countable set has no subgroup of index $<2^{\aleph_0}$: this directly follows from Rabinovi\u c and Dixon-Neumann-Thomas \cite{Ra,DNT} that any subgroup of index $<2^{\aleph_0}$ in a symmetric group is open. 

Let us now pass to (b), namely results about quotients of symmetric groups that are not obtained by ``passing to the quotient" properties of the symmetric group; we emphasize the case of the quotient by the subgroup of finitely supported permutations.

\begin{enumerate}[leftmargin=*]
\setcounter{enumi}{\value{saveenum}}
\item The first such studies concern the existence of isomorphisms or embeddings between the various subquotients of symmetric groups.

The story starts with the problem of embeddings into infinite symmetric groups. Notably, one original question was: what are obstructions for a group $G$ to embed into the symmetric group $\mks(X)$, beyond the obvious cardinality restriction $|G|\le 2^{|X|}$? Higman \cite{Hig} and independently De Bruijn \cite[Theorem 5.1]{DB}, see also \cite[Th.\ 11.2.5]{Sco}, proved that the alternating group $\mka(Y)$ does not embed into $\mks(X)$ whenever $|Y|>|X|$; when $|Y|\le 2^{|X|}$ this indeed does not follow from the cardinality restriction. McKenzie \cite{McK} obtained many further examples: notably the restricted direct product of $>|X|$ non-abelian groups cannot be embedded into $\mks(X)$ (thereby proving that \cite[Theorem 4.1]{DB} is false; the erratum is listed at the entry \cite{DB}), and he deduces that $\mks(X)$ does not include any isomorphic copy of its simple quotient $\mks(X)/\mks_{<|X|}(X)$. In particular, $\mksh(\aleph_0)$ does not embed into $\mks(\aleph_0)$.

% and in particular $\mks(X)$ does not  isomorphic copies of its nontrivial proper quotients 
In the other direction, it is an observation (\cite[\S 6]{Sco1}, \cite[Theorem 4.4]{DB}, \cite[11.5.2]{Sco}) that infinite symmetric groups can be embedded as subgroups of all their nontrivial quotients.

Acting by conjugation, one can naturally embed $\mkst(X)$ into $\Aut\big(\mkst_{<\aleph_1}(X)\big)$. Just forgetting the group structure of $\mkst_{<\aleph_1}(X)$ and retaining its cardinal $|X|^{\aleph_0}$, we deduce that for every set $X$, the group $\mkst(X)$ embeds into $\mks(X^{\aleph_0})$. When $|X|$ has the form $2^\alpha$, we have $|X|^{\aleph_0}=|X|$ and in particular $\mkst(X)$ embeds into $\mks(X)$. Under the generalized continuum hypothesis (GCH), this applies whenever $|X|$ is a cardinal of uncountable cofinality. (This remark is a variation on observations from \cite[\S 3]{FH}, with a reformulation avoiding assuming GCH straight away). This, by the way, contradicts \cite[Cor.\ 2.5]{FH} by the same authors, which erroneously claims that the existence of a set $X$ such that $\mksh(X)$ embeds into $\mks(X)$ is not a theorem of ZFC; the proof of \cite[Cor.\ 2.5]{FH} being blatantly flawed (relying on a misquotation of Cohen's results about independence of the continuum hypothesis in ZFC).

In general, the question whether $\mksh(X)$ embeds into $\mks(X)$ was raised by de Bruijn \cite{DB}, and has notably been addressed in \cite{McK,Cl,Ra,FH} As we have just seen, the answer is no when $|X|$ is countable \cite{McK} and yes when $X$ is a power set. Under GCH, the only remaining cases are uncountable cardinals of countable cofinality. It follows (in ZFC) from work of Felgner and Haug \cite[\S 3]{FH} that if $\kappa$ is a cardinal of uncountable cofinality and if ($*$) every subgroup of index $\le\kappa$ in $\mks(X)$ is open, then $\mksh(X)$ does not embed into $\mks(X)$; in addition, by work of Shelah and Thomas \cite{ShTh1}, ($\ast$) holds true under G\"odel's constructibility axiom, a strengthening of GCH.

The groups $\mksh(\kappa)$ are pairwise non-isomorphic, when $\kappa$ ranges over infinite cardinals, because if $\alpha$ is the ordinal type of the chain of nontrivial proper subgroups of $\mksh(\kappa)$, then $\kappa=\aleph_\alpha$. The classification up to elementary equivalence is tackled in \cite[Corollary 6.4]{ShT}.

Finally, the question whether every group of cardinal $2^\kappa$ embeds into $\mksh(\kappa)$ was raised by Clare, including in the case $\kappa=\aleph_0$. Felgner and Haug \cite{FH} proved the consistency of a negative answer for $\aleph_0$, namely that the negative answer holds in ZFC+PFA, where PFA refers to the ``Proper forcing axiom". Whether the negative answer is a theorem of ZFC remains an open problem.

\item\label{scolif} Scott \cite[\S 11.5.8]{Sco} checked that for any infinite set $X$, the quotient homomorphism $\mks(X)\to\mksh(X)$ does not split (he previously asked this in \cite{Sco1}). In our language, this means that $X$ is not realizable as a near $\mksh(X)$-set. His example is described in Example \ref{i_scotex}. 

%See also \S\ref{scottexample}.

% Scott's proof consists in considering a permutation $\sigma$ consisting of $2^{n-1}$ cycles of length $2^{n}$ for every $n\ge 1$, and an arbitrary number of fixed points: in $\mksh(X)$, such a permutation admits a $2^n$-root for every $n$, while none of its lifts to $\mks(X)$ can be indefinitely square rooted. In the language used here, this means that the given near action of $\Z$ (by powers of $\sigma$) can be extended to a near action of $\Z[1/2]$, and that any such extension is a non-realizable balanced near action of $\Z[1/2]$. 

Clare \cite[Corollary 2]{Cl} has obtained non-splitting of the quotient homomorphism $\mks(X)\to\mks(X)/\mks_{<|X|}(X)$.

Actually, Vitali's initial observation \cite{Vi} that any transposition of an infinite set is a product of two squares, implies that the quotient homomorphism $\mks(X)/\mka(X)\to\mksh(X)$ is not split, and even does not split in restriction to some finitely generated subgroup. Indeed, if the extension were split, modding out by $\mka(X)$ would provide a direct product with a cyclic group on two elements, in which the image of any transposition cannot be a product of squares. 

\item Last and not least, the determination by Truss \cite{Tru} of the automorphism group of $\mksh(X)$ as $\mkst(X)$, which was already discussed in the proof of Theorem \ref{autnx}. Its ancestor, namely that $\mks(X)$ can be identified to its own automorphism group for $X$ infinite, was established by J.\ Schreier and Ulam in 1937, see \cite[\S 11.4]{Sco} or \cite[\S 8.2]{DM} (the result being valid whenever $|X|\neq 6$). Truss' theorem is considerably harder; the proof notably involves a fine understanding of centralizers and conjugacy classes.

\setcounter{saveenum}{\value{enumi}}
\end{enumerate}

Let us now come to appearances of the near symmetric group and of the index character.

\begin{enumerate}[leftmargin=*]
\setcounter{enumi}{\value{saveenum}}
\item Andreoli's 1915 article \cite{And} defines the ``rank" of an injective self-map of an arbitrary set as the (possibly infinite) cardinal of the complement of its image, and observes that it is additive under composition. This seems to be the first (partial) appearance of the index character.

\item\label{wagind} The group of near permutations seems to first appear hidden in Wagoner's last section \cite{Wa}, with K-theoretic motivations. Namely, he defines a Fredholm permutation of a (countable) set $X$ as a self-map $f$ of $X$ such that there exists another self-map $g$ such that $f\circ g$ and $g\circ f$ coincide with the identity outside a finite subset (in the notation of \S\ref{mkpetc}, these are elements of the submonoid $\mkem(X)\subseteq X^X$). He calls a ``germ" of Fredholm permutation, a class modulo the coincidence on finite subsets. This defines the group of near permutations without ambiguity.

The index character also appears there, yet in a somewhat circuitous way.  Namely, Wagoner considers the subgroup $F_n$ of $\mkst(\N^2)$ that have a representative that is identity on $\N_{>n}\times\N$, and $F_\infty=\bigcup F_n$. The latter group $F_\infty$ seems to be the goal of the construction in \cite[\S 7]{Wa} (which itself is motivated by K-theoretic considerations). Indeed, $F_n$ is the group of near permutations of $\N_{\le n}\times\N$ while $F_{\infty}$ is not isomorphic to it (e.g., because the derived subgroup of $F_\infty$ has countable cofinality, unlike the derived subgroup of $\mkst(\N)$, by \cite{MN}). The (prodigal) index of $f\in F_\infty$ is defined as $|\N^2\smallsetminus V|-|f(V)|$, after choosing a cofinite subset $V$ of $\N^2$ on which $f$ is injective. Wagoner introduces it as ``index homomorphism" without further comment, and that it does not depend on the choice of $V$ is also not mentioned. He states its surjectivity on $F_\infty$, and that its kernel is the union $P_\infty=\bigcup P_n$, where $P_n$ is the subgroup of permutations of $\N_{\le n}\times\N$ modulo finitely supported permutations. This material is almost textually repeated in the subsequent paper \cite{Pri}. It then seems to have been forgotten, with the notable exception of \cite{GrS}.

%Note that this formally implies that in restriction to $F_n$, the kernel is $P_n$, and that it is surjective on $F_n$ for $n$ large enough. 

%, where it is explicitly mentioned that $F_\infty/[F_\infty,F_\infty]$.

%\item ????? The formulation defining $F_n$ is given by using proper maps with a proper inverse up to coincidence on a cofinite subset. In this formulation, he defines the index homomorphism (with this name), yet with this formulation the formula is less clean that using cofinite-partial bijections. Wagoner mentions its surjectivity on $F_\infty$, and that its kernel is the union $P_\infty=\bigcup P_n$, where $P_n$ is the subgroup of permutations of $\N_{\le n}\times\N$ modulo finitely supported permutations. Note that this formally implies that in restriction to $F_n$, the kernel is $P_n$, and that it is surjective on $F_n$ for $n$ large enough. 
%Although  it has then been widely forgotten (with the notable exception of \cite{GrS}) and rediscovered in other areas:

\item\label{rudinshelah} (Rudin-Shelah problem)\index{Rudin-Shelah problem}
The group $\mkst(X)$ naturally acts on the Stone-\u Cech boundary of $X$; whether the resulting injective homomorphism $\mkst(X)\to\mathrm{Homeo}(\partial_{\mathsf{SC}}(X))$ is surjective is actually undecidable (even for $X$ countable) in ZFC. There is a large literature in logic on this problem, initiated by Rudin \cite{Rud} and then Shelah \cite[\S IV.0 p.145, \S IV.5 p.171]{She}; I will not try to survey it since this would be beyond by competence in set theory. Let me stick to the early times of this problem. The usual formulation, due to Shelah of this problem (asking whether every self-homeom\-orphism of the Stone-Cech boundary is ``trivial") does not refer to the group structure of $\mkst(X)$ (the original formulation being whether every self-homeomorphism of the Stone-\u Cech boundary of $X$ induced by a bijection between two cofinite subsets of $X$. This group structure on $\mkst(X)$ becomes explicit, in relation to this problem, in van Douwen's posthumous article \cite{vD}, which includes a clear-cut exposition of the index character (with the banker index convention), rather starting from the monoid $\mkep(X)$ of cofinite-partial bijections (see (\ref{itsemi}) below). Wagoner's article (see (\ref{wagind}) above) in not cited in \cite{vD}; the existence of this homomorphism was then visibly considered as unexpected, and indeed, Shelah's result of consistency of the surjectivity of $\mkst(\aleph_0)\to\Aut(\mathcal{P}^\star(\aleph_0))$ is misquoted in \cite[\S 2.6]{vM} as consistency of the surjectivity of $\mks(\aleph_0)\to\Aut(\mathcal{P}^\star(\aleph_0))$. In his posthumous article, van Douwen's explains, in comments starting \cite[\S 5]{vD}, that some formulations of Shelah's problem, referring to ``trivial" self-homeomorphisms, overshadowed the fact that there is a distinction between those ``trivial ones" ($\mkst(X)$) and those ``very trivial ones" ($\mks(X)/\mks_{<\aleph_0}(X)$), and that at the time he realized the distinction, he thought of $\mks(X)/\mks_{<\aleph_0}(X)$ as ``the natural candidate for a nontrivial normal subgroup of $\mkst(X)$". Nevertheless the distinction between automorphisms induced by permutations, and automorphisms induced by bijections between cofinite subsets, is already explicit in 
\cite[\S IV.5 p.171]{She}.

\item 
The near permutation group and index characters have been occasionally reconstructed, with or without reference: 
\begin{itemize}
\item in \cite{GrS}, with reference to \cite{Wa}, but expurgated from its labyrinthine system to define the index character (and also using the prodigal index);

\item in \cite{ACM}, without reference. It starts there from the monoid of cofinite-partial bijections (called ``near bijections") on a set $X$, and the quotient group $\mkst(X)$ is called group of near symmetries of $X$. The (banker) index is introduced as ``index". 

\item in \cite{Rob}, without reference. It starts with the monoid of closely bijective self-maps (called ``near-bijections"), precisely defined as those self-maps $f$ whose image is cofinite such that the $W_f$ of those $x$ such that $f^{-1}(\{f(x)\}\neq\{x\}$ is finite. Then the (prodigal) index is defined as $(|W_f|-|f(W_f)|)-|X\smallsetminus f(X)|$. 

\item in \cite{GuR}, without reference, in the study of group quotients of the monoid of cofinite-partial self-maps $\mkem(X)$. We elaborate on this in the next item, (\ref{itsemi}).
\end{itemize}

%, e.g., in \cite{ACM} (with the banker convention and starting from ), in \cite{CG} and in \cite{Rob}.

%The main goal of this remark is to observe that the near permutation group naturally appears from a classical and canonical construction in semigroup theory.

\item \label{itsemi}
The set of partial bijections of a set $X$, under composition, is the most fundamental example of an inverse monoid (see Remark \ref{invsemi}), called symmetric inverse monoid (or symmetric inverse semigroup) of $X$, denoted $\mk{I}(X)$. Its inverse submonoid of cofinite-partial permutations, denoted by $\mkep(X)$ in \S\ref{mkpetc}, is less known, but mentioned in \cite[p.217]{Law}.

One important and well-known result on inverse semigroups is that every inverse semigroup has a canonical largest group quotient (modding out by an explicit congruence, called minimal group congruence, generated as a semigroup congruence by the relations $e^2=e$ when $e$ ranges over idempotents). In the case of $\mkep(X)$, this largest group quotient precisely yields $\mkst(X)$ (Proposition \ref{puni}(\ref{punivm})). Gutik and Repov\u s \cite{GuR} actually perform this construction and then study this group and recognize its basic properties. Further, they prove that every nontrivial congruence on the monoid $\mkep(X)$ is a group congruence: in other words, proper quotients of the monoid $\mkem(X)$ coincide with group quotients of $\mkst(X)$, which are described in Corollary \ref{normne}.

General monoids do not always have a unique largest group quotient; still it is the case for $\mkem(X)$ and $\mke(X)$ (which are not inverse monoids as soon as $|X|\ge 2$): they admit $\mkst(X)$ as their largest group quotient, also by Proposition \ref{puni}.

\item At my surprise, all the references I could find exclusively restrict to (partial) maps from a single set to itself, with no attempt of a categorical approach (which entails no particular difficulty). 
\end{enumerate}

Warning: I spent a considerable time at attempts to track early appearances of all these notions (especially, the group of near permutations); nevertheless it is quite possible, if not likely, that some further early references slipped through the net.

%%%%%%%%%%%%%%%%%%%%%%%%%%%%%%%%%%%%%%%%%%%%%%%%%%%%%%%%%%%
%%%%%%%%%%%%%%%%%%%%%%%%%%%%%%%%%%%%%%%%%%%%%%%%%%%%%%%%%%%
%%%%%%%%%%%%%%%%%%%%%%%%%%%%%%%%%%%%%%%%%%%%%%%%%%%%%%%%%%%

\section{Rough maps, near maps, near permutations}\label{s_rough}
\numberwithin{thm}{subsection}

A solid preparatory work involving various suitable categories of maps will be useful for the study of near actions. However, we recommend the reader {\it not} to start with this section, unless with a taste for categories and algebraic structures.

\subsection{Premaps, rough mappings}
\begin{defn}\label{defrelation}
Let $X,Y$ be sets. Define a relation\index{relation} $X\tot Y$ as a subset of $X\times Y$; the notation $X\tot Y$ emphasizes that we want to think of them as multivalued maps.
%this is usually called relation but we rather want to think of them as multivalued maps.
Given another set $Z$ and relations $f:X\tot Y$ and $g:Y\tot Z$, the composition (or comproduct)\index{comproduct} $g\circ f$ is defined as
\[\{(x,z)\in X\times Z:\exists y\in Y,(x,y)\in f,(y,z)\in g\},\]
and the flipped relation $f^{-1}:Y\tot X$, image of $f$ by the flip $(x,y)\mapsto (y,x)$. We also call $f^{-1}$ the flip of $f$.
\end{defn}

Where defined, composition is associative and identities $\mathrm{id}_X$ are neutral elements; this forms the classical category of relations $\mathbf{Rel}$ (see for instance \cite[\S I.7]{McL}), whose objects are sets and arrows $X\to Y$ are relations $X\tot Y$. Beware that $f^{-1}$ is usually not inverse of $f$: there is no containment between $f^{-1}\circ f$ and $\mathrm{id}_X$ holding in general.

\begin{defn}For a relation $f:X\tot Y$ and $x\in X$ or $A\subseteq X$, define $f[x]=\{y\in Y:(x,y)\in f\}$ and $f[A]=\bigcup_{x\in A}f[X]$.
The domain of definition $D_f$ of a relation $X\tot Y$ is the set of $x\in X$ such that $f[x]$ is a singleton, which we then write as $\{f(x)\}$ to get the usual functional notation.

If $X'\subseteq X$ and $Y'\subseteq Y$, the relation $X'\tot Y'$ induced by $f:X\tot Y$ is the intersection $f\cap (X'\times Y')$.
\end{defn}

Note that when $f:X\to Y$ is a genuine function and $A\subseteq X$, then $f[A]=f(A)$ (image of $A$); similarly for $B\subseteq Y$, then $f^{-1}[B]=f^{-1}(B)$. In this case, $B\mapsto f^{-1}[B]$ is a Boolean algebra (unital) homomorphism. Conversely, if for some relation $f$, the mapping $B\mapsto f^{-1}[B]$ is a Boolean algebra (unital) homomorphism, then $f$ is a function.

%$A\subseteq X$ and $B\subseteq Y$, the notation $f[A]$ and $f^{-1}[B]$ are the usa

\begin{defn}
Denote by $p_X$ and $p_Y$ the projections $X\stackrel{\;\;p_X}\longleftarrow X\times Y\stackrel{p_Y}\longrightarrow Y$. 
A relation $f:X\tot Y$ is
\begin{itemize}
\item a finitely-valued\index{finitely-valued (relation)} relation if $p_X|_f$ is proper, that is, $f[x]$ is finite for all $x\in X$;
\item proper\index{proper (relation)} if $p_Y|_f$ is proper, that is, $f^{-1}$ is a finitely-valued relation;
\item a rough mapping\index{rough mapping} if there exists a cofinite subset $Y'$ of $Y$ such that $p_X|_{f\cap (X\times Y')}$ is injective;
\item a partial mapping\index{partial mapping} if $p_X|_f$ is injective, that is, $f[x]$ is at most a singleton for all $x\in X$, or equivalently, $f\subseteq D_f\times Y$;
\item a map, or a function, if $p_X|_f$ is bijective, that is, $f[x]$ is a singleton for all $x\in X$, or equivalently, $D_f=X$;
\item bijective, or a bijection, if both $f$ and $f^{-1}$ are maps, or equivalently if $f$ is the inverse of $f$ (in the sense that $f^{-1}\circ f:\mathrm{id}_X$ and $f\circ f^{-1}=\mathrm{id}_Y$) 
\item a cofinite-partial bijection\index{cofinite-partial bijection} if both $X'=p_X(f)$ and $Y'=p_Y(f)$ are cofinite, and $f$ induces a bijection $X'\to Y'$.
\end{itemize}
\end{defn}

%Beware that finitely-valued relations, rough mappings, partial mappings are not maps in general.

These are the main definitions to keep in mind; of course the definitions of maps and of bijections are the usual ones. Here are some additional definitions, which can be skipped in a first reading.

\begin{defn}
A relation $f:X\tot Y$ is
\begin{itemize}
\item injective if $p_Y|_f$ is injective, that is, $f^{-1}[y]$ is at most a singleton for all $y\in Y$;
\item surjective if $p_Y|_f$ is surjective, that is, $f^{-1}[y]$ is non-empty for all $y\in Y$;
%\item bijective if it is simultaneously a map, injective and surjective;
%, that is, $f^{-1}[y]$ is finite for all $y\in Y$.
%\item a close map if $f[x]$ is finite for all $x\in X$ and a singleton for all but finitely many $x\in X$, that is, $f$ is a finitely-valued relation and $D_f$ is a cofinite subset of $X$;
\item closely injective\index{closely injective/surjective/bijective} if there exists a cofinite subset $Y'$ of $Y$ such that, denoting $f'=f\cap (X\times Y')$, the relation $f'$ is injective and $p_X(f')$ is cofinite in $X$;
% and $f^{-1}[y]$ is at most a singleton for all but finitely many $y\in Y$;
\item closely surjective if $f^{-1}[y]$ is non-empty for all but finitely many $y\in Y$;
\item closely bijective if it is a rough mapping, closely injective and closely surjective;
\item a cofinite-partial injection if it is an injective partial mapping with cofinite domain of definition; equivalently $f$ is a bijection from a cofinite subset of $X$ onto a subset of $Y$;
%\item a cofinite-partial bijection if it is a cofinite-partial injection with cofinite image.
%it is an injective and closely bijective partial mapping ; equivalently $f$ is a bijection from a cofinite subset of $X$ to a cofinite subset of $Y$.
\end{itemize}
\end{defn}

Note that for maps, injectivity and surjectivity are the classical definitions.

%\begin{rem}The composition of close maps can fail to be a close map: if we have $X\stackrel{f}\to Y\stackrel{g}\to Z$ with $f$ equal to a constant not in the domain of definition of $g$, then the domain of definition of $g\circ f$ is empty. For this reason (among others), we need to concentrate on rough mappings and proper close maps.\end{rem}

\begin{defn}\label{simdef}We say that two relations $f,f':X\tot Y$ are near equal, denoted $f\sim f'$\index{$\sim$ (near equal)}\index{near equal $\sim$}, if $f\bigtriangleup f'$ is finite ($\bigtriangleup$ denoting symmetric difference as subsets of $X\times Y$). 

We say that a subset of $X\times Y$ is horizontal if it is included in $X\times F$ for some finite subset $F$ of $X$. We say that two relations $f,f':X\tot Y$ are roughly equal\index{$\approx$ (roughly equal)}\index{roughly equal $\approx$}, denoted $f\approx f'$, if $f\bigtriangleup f'$ is horizontal.
\end{defn}

\begin{prop}\label{condsatu}
Both $\sim$ and $\approx$ are equivalence relations on the set of relations $X\tot Y$; $f\sim f'$ implies $f\approx f'$.

Each of the following five conditions is saturated under the equivalence relation $\approx$: being a finitely-valued relation, being a rough mapping, being closely injective, being closely surjective, being closely bijective.

In addition, being proper is saturated under the equivalence relation $\sim$ among relations.\qed
\end{prop}

\begin{prop}
Let $f,f'$ are proper rough mappings. Then $f\sim f'$ and $f\approx f'$ both mean that there exists a cofinite subset of $X$ included in $D_f\cap D_{f'}$ on which $f$ and $f'$ (viewed as functions) coincide.\qed
\end{prop}

% to the five previous ones, each of the following three conditions is saturated under the equivalence relation $\sim$ among relations: being proper.\qed

\begin{prop}
Let $f:X\tor Y$ be a proper rough mapping and $X'\subseteq X$ a cofinite subset, and  $Y'$ a cofinite subset of $Y$. Define $f'$ as the intersection $f\cap (X'\times Y')$.% If $f'$ is a close map, t
Then $f\sim f'$.
\end{prop}
\begin{proof}
The subset $f\smallsetminus f'$ is itself a proper rough mapping, and included in $((X\smallsetminus X')\times Y)\cup (X\times (Y\smallsetminus Y'))$; it intersects $(X\smallsetminus X')\times Y$ in a finite subset because it is a rough mapping, and $X\times (Y\smallsetminus Y')$ in a finite subset because it is proper.
\end{proof}

%By definition of the domain of definition, $((D_{f}\cap D_{f'})\times Y)\cap f=((D_{f}\cap D_{f'})\times Y)\cap f'$, which is thus included in $X\times Y'$. This just reflects the fact that for $x$ in both domains of definition, there is a unique image by $f$ and its restriction $f'$, which is the same and hence belongs to $Y'$. Thus, $f\bigtriangleup f'$ is contained in $F\times Y$, where $F$ is the complement of $D_{f}\cap D_{f'}$ in $X$. Since $f,f'$ are close maps, $F$ is finite, and moreover the projection $p_X$ is proper on $f\cup f'$, and hence $(f\cup f')\cap F\times Y$ is finite. Thus $f\bigtriangleup f'$ is finite.

In combination with Proposition \ref{condsatu}, this yields:

\begin{cor}
Let $f:X\tor Y$ be a proper rough mapping, $X'$ a cofinite subset of $X$ and $Y'$ a cofinite subset of $Y$. Each of the following properties is satisfied by $f$ if and only if it is satisfied by the restriction $f':X'\tor Y'$: proper, closely injective, closely surjective, closely bijective.
\end{cor}
\begin{proof}In all cases it is trivial (for arbitrary relations) that the given property passes from $f$ to $f'$. Conversely, suppose that it holds for $f'$; by the proposition, $f\sim f'$. Since all given properties are saturated under $\sim$, we deduce that $f'$, viewed as a relation $X\tor Y$ (included in $X'\times Y'$), satisfies the property. The last immediate observation is that for each of the given properties, for $g\in X'\tor Y'$, the property holds as a relation $X'\tor Y'$ if and only if it holds as a relation $X\tor Y$. 
\end{proof}

\begin{prop}\label{precompo}
%Finitely-valued relations are stable under composition, and c
Composition of finitely-valued relations yields finitely-valued relations, and is compatible with the relation of rough equivalence $\approx$ of Definition \ref{simdef}: for sets $X,Y,Z$ and finitely-valued relations $X\stackrel{f,f'}\longrightarrow Y\stackrel{g,g'}\longrightarrow Z$ such that $f\approx f'$ and $g\approx g'$, then $g'\circ f'\approx g\circ f$ and these are finitely-valued relations.
\end{prop}
\begin{proof}
We have $(g\circ f)[x]\subseteq \bigcup_{y\in f[x]}g[y]$; since both $f[x]$ and $g[y]$ are finite, this union is finite and we have stability under composition. 

Now write $h_g=g'\tu g$ and $h_f=f'\tu f$. Then $h_g,h_f$ are horizontal and  
\[(g\circ f)\tu (g'\circ f')\subseteq \big((g\cup g')\circ h_f\big)\cup \big(h_g\circ (f\cup f'\cup h_f)\big).\]

The set of horizontal relations is stable under right composition by arbitrary relations, and under left composition by finitely-valued relations. Hence $(g\circ f)\tu (g'\circ f')$ is horizontal, that is, $g\circ f\approx g'\circ f'$.
\end{proof}

\begin{prop}\label{stabcompo}
The following classes of relations are closed under composition, that is, if $X\stackrel{f}\tot Y\stackrel{g}\tot Z$ are relations with both $f,g$ satisfying the given property, so does $g\circ f:X\tot Z$: rough mapping; map; partial mapping; injective; surjective; bijective; proper; proper rough mappings; closely injective finitely-valued relations; closely surjective rough mappings; closely bijective.
\end{prop}
\begin{proof}
Verifications are routine, we only check it in a few cases:

For close surjectivity of rough mappings: Define $Y'=p_Y(f)$; it is cofinite in $Y$ by close surjectivity of $f$. Since $g$ is a finitely-valued relation, the set $F=g[Y\smallsetminus Y']$ is a finite subset of $Z$. If $z\in Z$ belongs to the cofinite subset $p_Z(g)\smallsetminus F$, then since $z\in p_Z(g)$, we have $(y,z)\in g$ for some $y\in Y$ and since $z\notin F$, we have $y\in Y'$. Hence there exists $x\in X$ such that $(x,y)\in f$ and thus $(x,z)\in g\circ f$.

As regards rough mappings, observe that a relation is a rough mapping if and only if is roughly ($\approx$) equivalent to a partial mapping, and hence the result follows from Proposition \ref{precompo}. 
\end{proof}

Since for proper rough mappings, $\approx$ and $\sim$ coincide, we deduce:
\begin{cor}\label{simcompo}
For proper rough mappings, composition is compatible with the relation $\sim$ of Definition \ref{simdef}.\qed
\end{cor}

%: for sets $X,Y,Z$ and proper rough mappings $X\stackrel{f,f'}\longrightarrow Y\stackrel{g,g'}\longrightarrow Z$ such that $f\sim f'$ and $g\sim g'$, then $g'\circ f'\sim g\circ f$.

%\begin{proof}It is enough to prove the result when either $f=f'$ or $g=g'$.When $g=g'$, $f$ and $f'$ are defined and coincide on a cofinite subset of $X$, on which $g\circ f$ and $g\circ f'$ therefore coincide.When $f=f'$, we need to use the assumption that $f$ is proper. This ensures that there is a cofinite subset $D$ of $D_f$ mapping into a cofinite subset of $D_g\cap D_{g'}$ on which $g$ and $g'$ coincide, and hence $g'\circ f$ and $g\circ f$ coincide on $D$.\end{proof}

%For partial mappings, stability under composition is clear. Write $g=g'\tu h_g$ and $f=f'\tu h_f$, with $g',f'$ partial mappings and $h_g,h_f$ horizontal. Then \[(g\circ f)\tu (g'\circ f')\subseteq \big((g\cup g')\circ h_f\big)\cup \big(h_g\circ (f\cup f'\cup h_f)\big).\]The set of horizontal relations is stable under right composition by arbitrary relations, and under left composition by finitely-valued relations. Hence $(g\circ f)\tu (g'\circ f')$ is horizontal.

\subsection{Category of rough mappings and near maps}\label{mkpetc}

For sets $X,Y$, denote by:
\begin{itemize}[leftmargin=*]
\item $\mkr(X,Y)$ the set of rough mappings $X\tor Y$;
%\index{$\mkr(X,Y)$, rough mappings}
\item $\mkrp(X,Y)$ the set of partial mappings $X\topm Y$;
%\index{$\mkrp(X,Y)$, partial mappings}
\item $\mkrm(X,Y)=Y^X$ the set of maps $X\to Y$;
\index{$\mkr(X,Y)$, $\mkrp(X,Y)$, $\mkrm(X,Y)$} 
\item $\mkp(X,Y)$ the set of proper rough mappings $X\tor Y$;
\index{$\mkp(X,Y)$, $\mkpm(X,Y)$}
%\item $\mkpi(X,Y)$ (resp.\ $\mkps(X,Y)$) its subset of closely injective (resp.\ surjective) proper rough mappings;
%$\mke(X,Y)=\mkpi(X,Y)\cap\mkps(X,Y)$
\item $\mke(X,Y)$ its subset of closely bijective proper rough mappings; 
\index{$\mke(X,Y)$, $\mkem(X,Y)$}
\item $\mkpm(X,Y)=\mkp(X,Y)\cap\mkrm(X,Y)$ the subset of $\mkp(X,Y)$ consisting of proper maps. 
%In other words, for $f\in\mkp(X,Y)$, we have $f\in\mkpm(X,Y)$ if and only $p_X$ restricts to a bijection from $f$ to $X$.
\item $\mkem(X,Y)$ (resp.\ $\mkep(X,Y)$) the two subsets of $\mke(X,Y)$ consisting of (everywhere defined) maps, respectively of cofinite-partial bijections. In other words, we have \[\mkem(X,Y)=\mkpm(X,Y)\cap\mke(X,Y)\] and for $f\in\mke(X,Y)$, we have $f\in\mkep(X,Y)$ if and only if both $p_X$ and $p_Y$ are injective on $f$.
\end{itemize}

For each $\mk{C}$ among these, we write $\mk{C}(X,X)=\mk{C}(X)$. All these are stable under composition, as particular cases of Proposition \ref{stabcompo}. Define the category $\mkr$ of rough mappings with sets as objects and rough mappings as arrows; this is a subcategory of the category $\mathbf{Rel}$ of relations. All the others are subcategories of $\mkr$, denoted $\mkrm$, $\mkrp$, $\mkp$, $\mke$, $\mkem$, $\mkep$. The category $\mke$ is called the category of close bijections. Of course $\mkrm$ is just the category of sets with maps, but here we especially view it as a subcategory of the category $\mkr$.

%$\mkpi$, $\mkps$,

\begin{rem}\label{invsemi}
A semigroup is called an inverse semigroup if for every $x$ there exists a unique $y$ such that $x=xyx$ and $y=yxy$; when it is a monoid (i.e., there is a unit element), it is called inverse monoid. 

One of the most proto-typical examples of an inverse monoid is the monoid of partial bijections of a set $X$. The monoid $\mkem(X)$ of cofinite-partial bijections is a inverse submonoid of the latter; it is mentioned in \cite[p.217]{Law} and studied in \cite{GuR}.
\end{rem}

%item The set of partial bijections of a set $X$, under composition, is the most fundamental example of an inverse semigroup, called symmetric inverse semigroup of $X$, denoted $\mk{I}(X)$. Its subset $\mk{I}_{\mathrm{cof}}(X)=\widehat{\mkst}(X)\cap\mk{I}(X)$ of cofinite-partial permutations is an inverse subsemigroup ().

%\begin{prop}For finitely-valued relations, and in particular for rough mappings, composition is compatible with the relation of rough equivalence $\approx$ of Definition \ref{simdef}: for sets $X,Y,Z$ and proper rough mappings $X\stackrel{f,f'}\longrightarrow Y\stackrel{g,g'}\longrightarrow Z$ such that $f\approx f'$ and $g\approx g'$, then $g'\circ f'\approx g\circ f$.\end{prop}

\begin{defn}
For $X,Y$ sets, define $\mkrst(X,Y)$ as the quotient $\mkr(X,Y)/\!\!\approx$, and denote by $f_\star$ the image of $f$, that is, the class of $f$ modulo $\approx$.
\index{$\mkrst(X,Y)$, $\mkrst(X)$}
Define $\mkpst(X,Y)$ as the quotient $\mkp(X,Y)/\!\!\approx$ and $\mkest(X,Y)$ as the quotient $\mke(X,Y)/\!\!\approx$; we have inclusions \[\mkest(X,Y)\subseteq \mkpst(X,Y)\subseteq \mkrst(X,Y).\]
\index{$\mkpst(X,Y)$, $\mkpst(X)$}
\index{$\mkest(X,Y)$}
\index{$\mkst(X)$, near symmetric group}
%, and denote by $f_\star$ the image of $f$. Define $\mkest(X,Y)$ as the quotient $\mke(X,Y)/\!\!\sim$ (which can be viewed as a subset of $\mkpst(X,Y)$).
% 
  Call elements of $\mkrst(X,Y)$, resp.\ $\mkpst(X,Y)$, resp.\ $\mkest(X,Y)$, near maps, proper near maps, resp.\ near bijections from $X$ to $Y$. Write $\mkrst(X,X)=\mkrst(X)$, $\mkpst(X,X)=\mkpst(X)$ and $\mkest(X,X)=\mkest(X)$, and call them near self-maps of $X$, proper near self-maps, and near permutations of $X$ respectively.
\end{defn}

%Also define $\mkpist(X,Y)$ and $\mkpsst(X,Y)$ as the images of respectively $\mkpi(X,Y)$ and $\mkps(X,Y)$ in $\mkpst(X,Y)$

In view of Corollary \ref{simcompo}, this defines a category $\mkrst$, with a functor, which is the identity on objects and mapping $f$ to $f_\star$. Each of $\mkpst$ and $\mkst$ is a subcategory of $\mkrst$, with the same objects.

\begin{prop}\label{qtxy}Let $X,Y$ be sets.
\begin{enumerate}
\item\label{monoepi} The monomorphisms (resp.\ epimorphisms) from $X$ to $Y$ in the category $\mkrst$ or $\mkpst$ are precisely the closely injective  (resp.\ closely surjective) elements.

%the elements of $\mkpist(X,Y)$, resp.\ $\mkpsst(X,Y)$.

\item\label{gfiso} The isomorphisms from $X$ to $Y$ in both $\mkrst$ and $\mkpst$ are precisely the elements of $\mkst(X,Y)$, and in particular coincide, in each of these categories, with homomorphisms that are both monomorphisms and epimorphisms. Accordingly, the near permutation group $\mkst(X)$ is the subgroup of invertible elements in the near self-map monoid $\mkrst(X)$ as well as in its submonoid $\mkpst(X)$.

 \item\label{gfis} For proper rough mappings $X\stackrel{f}\tor Y\stackrel{g}\tor Z$, if $g\circ f$ is closely injective and $f$ is closely surjective, then $g$ is closely injective and $f$ is closely bijective; if $g\circ f$ is closely surjective and $g$ is closely injective, then $f$ is closely surjective and $g$ is closely bijective.
\end{enumerate}
\end{prop}
\begin{proof}
(\ref{monoepi}) Consider $g\in\mkr(Y,Z)$. Suppose that $g$ is closely injective and consider $f,f'\in\mkr(X,Y)$ such that $g\circ f\approx g\circ f'$. There exists a cofinite subset $Z'$ of $Z$ such that
\begin{itemize}
\item for some cofinite subset $Y'$ of $Y$, the rough mapping $g$ restricts to an injective map $Y'\to Z'$;
\item $(g\circ f)\cap (X\times Z')=(g\circ f')\cap (X\times Z')$.
\end{itemize}
Define $Y''=Y'\smallsetminus g^{-1}[g[Y\smallsetminus Y']]$. Since $g\cap (Y'\times Z)$ is proper, this is a cofinite subset. Let us prove that $f\cap (X\times Y'')=f'\cap (X\times Y'')$. Indeed, for $(x,y)\in f\cap (X\times Y'')$, just using that $y\in Y'$, we have $(x,g(y))\in (g\circ f)\times (X\times Z')$, and hence belongs to $(g\circ f')\times (X\times Z')$. So there exists $y'\in Y$ such that $(x,y')\in f'$ and $(y',g(y))\in g$. If by contradiction $g'\in Y\smallsetminus Y'$, then we have $y\in g^{-1}[g[Y\smallsetminus Y']]$, which is excluded, so $y'\in Y'$. Since $g\times (Y'\times Z')$ is injective, we deduce that $y=y'$, and hence $(x,y)\in f'$. This proves $f\cap (X\times Y'')\subseteq f'\cap (X\times Y'')$, and the opposite inclusion follows by switching roles.

%$(x,y)\in f$ $y\in Y'$ donc $(y,g(y))\in g$ donc $(x,g(y))\in g\circ f=g\circ f'$. donc existe $y'\in Y$ tq $(x,y')\in f$, $(y',g(y))\in g$ $y'\notin Y'$

%be a cofinite subset of $Z$ and $Y'$ a cofinite subset of $Y$ such that $g$ restricts to an injection $Y'\to Z'$. Let $X''$ be a cofinite subset of $X'$ on which $g\circ f$ and $g\circ f'$ coincide. For all $x\in X''$, we have $g\circ f(x)=g\circ f'(x)$. Injectivity of $g$ on $f(X')\cup f'(X)$ ensures that $f(x)=f'(x)$, so  $f$ and $f'$ coincide on $X'$ and thus $f\sim f'$. 

Conversely, suppose that $g$ is not closely injective. If $Y\neq\emptyset$, we can suppose that $g$ is a map and then either it has an infinite fiber, or infinitely many fibers with at least two elements, and hence there exists an infinitely supported permutation $\sigma$ of $X$ preserving fibers of $g$, and hence $g\circ\sigma=g\circ\mathrm{id}_Y$, while $\sigma\nsim\mathrm{id}_Y$, and thus $f_\star$ is not a monomorphism. If $Y$ is empty, that $g$ is not closely injective means that $X$ is infinite and any infinitely supported permutation of $X$ does the same job.

Consider $f\in\mkr(X,Y)$. Suppose that $f$ is closely surjective.
Let $g,g':Y\tor Z$ be rough mappings such that $g\circ f\approx g'\circ f$.
Let $Y'$ be a cofinite subset of $Y$ such that $f\cap (X\times Y')$ is a surjective partial mapping. Let $Z'$ be a cofinite subset of $Z$ included in $g(Y')$ and such that $(g\circ f)\cap (X\times Z')=(g'\circ f)\cap (X\times Z')$. Define $Z''=Z'\smallsetminus g[Y\smallsetminus Y']$.
Let $(y,z)$ belong to $g\cap (Y\times Z'')$. Since $z\in Z''$, we have $y\in Y'$. So there exists $x\in X$ such that $(x,y)\in f$. So $(x,z)\in (g\circ f)\cap (X\times Z'')$. Hence $(x,z)\in (g'\circ f)\cap (X\times Z'')$. So there exists $y'\in Y$ such that $(x,y')\in f$ and $(y',z)\in g'$. Again, $z\in Z''$ forces $y'\in Y'$. So $(x,y)$ and $(x,y')$ both belong to $f\cap (X\times Y')$, which is a partial mapping, and thus $y'=y$. So $(y,z)\in g'$. This proves $g\cap (Y\times Z'')\subseteq g'\cap (Y\times Z'')$ and the reverse inclusion follows by switching roles.

%  $Y$ in which $g,g'$ are maps, and a cofinite subset $X'$ of $X$, included in $D_f$ and mapping into $Y'$ (using properness of $f$). Then $g$ and $g'$ coincide on the cofinite subset $f(Y')$ and hence $g\sim g'$.

Conversely, if $f$ is not closely surjective, choose a permutation $\sigma$ of $Y$ with support included in the complement of the ``image of $g$", namely $p_Y(g)$. Then $\sigma\circ f=\mathrm{id}_Y\circ f$ while $\sigma\nsim\mathrm{id}_Y$, and thus $f_\star$ is not an epimorphism in any of the given categories.

(\ref{gfiso}) Consider $f\in\mkr(X,Y)$. If $f_\star$ is an isomorphism, then it is both a monomorphism and an epimorphism, and hence by the above characterization, $f$ is both closely injective and closely surjective, thus is closely bijective. Conversely, if $f$ is closely bijective, choose an injective partial mapping $g\in\mkep(X,Y)$ with $g\approx f$. Then we readily check that $g^{-1}\circ g$ is a cofinite subset of $\mathrm{id}_X$ and $g\circ g^{-1}$ is a cofinite subset of $\mathrm{id}_Y$ (and, if we work in $\mkp(X,Y)$, we automatically have $g^{-1}\in\mkp(X,Y)$). 

We now prove (\ref{gfis}), freely using (\ref{monoepi}). Since $g\circ f$ is a monomorphism in $\mkpst$, $f$ is a monomorphism in $\mkpst$; since $f$ is supposed to be an epimorphism in $\mkpst$, it is therefore an isomorphism in $\mkpst$ (by (\ref{gfiso})) and so is closely bijective; it also follows that $g$ is a monomorphism in $\mkpst$, hence $g$ is closely injective. The other assertion is similar.
\end{proof}

\begin{lem}\label{carasim}
For $g,g'\in\mkr(X,Y)$, the following are equivalent:

\begin{enumerate}
\item\label{gsimg} $g\approx g'$;
\item\label{hggp} there exists $h\in\mke(Y)$ such that $h\circ g=h\circ g'$;
\item\label{hggpii} there exists $h\in\mkep(Y)$ with $h\circ h=h$ (so $h\sim\mathrm{id}_Y$) such that $h\circ g=h\circ g'$;
\setcounter{saveenum}{\value{enumi}}
\end{enumerate}
if $g,g'\in\mkrm(X,Y)$, these are also equivalent to:
\begin{enumerate}
 \setcounter{enumi}{\value{saveenum}}
%\item\label{gsimg2} $g\approx g'$;
\item\label{hggpi} there exists $h\in\mkem(Y)$ with $h\circ h=h$, with $h\sim\mathrm{id}_Y$, such that $h\circ g=h\circ g'$;
\setcounter{saveenum}{\value{enumi}}
\end{enumerate}
if $g,g'\in\mkp(X,Y)$, these are also equivalent to:
\begin{enumerate}
 \setcounter{enumi}{\value{saveenum}}
\item\label{gfgf} there exists $f\in\mke(X)$ such that $g\circ f=g'\circ f$;
\item\label{gfgfii} there exists $f\in\mkep(X)$ with $f\circ f=f$ (so $f\sim\mathrm{id}_X$) such that $g\circ f=g'\circ f$;
\setcounter{saveenum}{\value{enumi}}
\end{enumerate}
if $g,g'\in\mkp(X,Y)$ and $X$ is infinite, these are also equivalent to:
\begin{enumerate}
 \setcounter{enumi}{\value{saveenum}}
\item\label{gfgfi} there exists $f\in\mkem(X)$ with $f\circ f=f$ (so $f\sim\mathrm{id}_X$) such that $g\circ f=g'\circ f$.
\end{enumerate}
\end{lem}
\begin{proof}
Any of the condition implies (\ref{gsimg}): indeed, if a closely bijective rough mapping $f$ or $h$ exists as in any of them, then modulo $\approx$, the finitely-valued relation $f$ or $h$ induces an isomorphism (Proposition \ref{qtxy}(\ref{gfiso})), and (\ref{gsimg}) follows.

It remains to prove that (\ref{gsimg}) implies, under the possible additional specifications, the other conditions. Trivially (\ref{gfgfii}) implies (\ref{gfgf}) and (\ref{hggpii}) implies (\ref{hggp}), so we can skip (\ref{gfgf}) and (\ref{hggp}). 

%If closely bijective rough mapping $f$ or $h$ exists as in (\ref{hggp}) or (\ref{gfgf}), since modulo $\approx$, the map $f$ or $h$ is an isomorphism (Proposition \ref{qtxy}(\ref{gfiso})), we deduce in both cases that $g\approx g'$ and thus obtain (\ref{gsimg}). Similarly, each of (\ref{hggpi}), (\ref{gfgfi}) implies (\ref{gsimg}).

%the map $h$ is a monomorphism and $f$ is an epimorphism by Proposition \ref{qtxy}, we deduce in both cases $g\sim g'$ and thus obtain (\ref{gsimg}). 

%If an idempotent $f$ or $h$ exists as in (\ref{hggp}) or (\ref{gfgf}), since modulo $\sim$, the map $h$ is a monomorphism and $f$ is an epimorphism by Proposition \ref{qtxy}, we deduce in both cases $g\sim g'$ and thus obtain (\ref{gsimg}). 

We thus suppose (\ref{gsimg}), namely $g\approx g'$. So there exists a cofinite subset $Y'$ of $Y$ such that $g\cap (X\times Y')=g'\cap (X\times Y')$. Define $h=\mathrm{id}_{Y'}$. Then $h\circ h=h$ and $h\circ g=h\circ g'$; we have $h\in\mkep(Y)$ and we obtain (\ref{hggpii}). To obtain (\ref{hggpi}), assume in addition that $g,g'$ are maps. If $Y'=Y$, $h$ already works. Otherwise, choose $y_0\in Y\smallsetminus Y'$ and extend $h$ to $Y\smallsetminus Y'$ as constant equal to $y_0$. Then $h$ is a map an $h\circ h=h$. For $x\in X$, we have $g(x)\in Y'$ if and only if $g'(x)\in Y'$, in which case then $g'(x)=g(x)$ and then $h(g(x))=h(g'(x))$. Otherwise, $g(x),g(x')$ both belong to $Y\smallsetminus Y'$ and then $h(g(x))=h(g'(x))=y_0$, proving (\ref{hggpi}).

Now suppose that $g,g'$ are proper rough mappings. Let $F$ be the set of $x\in X$ such that either $x\notin D_g\cap D_{g'}$ or $g[x]\neq g'[x]$. Since $g,g'$ are proper, $F$ is finite. Let $f$ be the identity on the complement $X'$ of $F$. If we leave $f$ undefined on $F$, then $f\circ f=f$ and $g\circ f=g'\circ f$, and we obtain (\ref{gfgfii}). If we assume $X$ infinite, then $X'$ is nonempty and hence we can extend $f$ as equal to some constant element of $X'$ on $F$, and thus we obtain (\ref{gfgfi}).
\end{proof}

\begin{prop}\label{surjca}
Let $X,Y$ be sets.
\begin{enumerate}
\item\label{erst} The quotient map $\mkr(X,Y)\to\mkrst(X,Y)$ is surjective in restriction to $\mkrp(X,Y)$, and is surjective in restriction to $\mkrm(X,Y)$ unless $X$ is nonempty and $Y$ is empty.

\item\label{canosur} The canonical map from $\mkpm(X,Y)$ to $\mkpst(X,Y)$ is surjective, except when $X$ is finite nonempty and $Y$ is empty.

\item\label{eest}
The quotient map $\mke(X,Y)\to\mkest(X,Y)$ is surjective in restriction to $\mkep(X,Y)$, and is surjective in restriction to $\mkem(X,Y)$ at the exception when $X$ is finite and nonempty while $Y$ is empty.
\end{enumerate}
\end{prop}
\begin{proof}
(\ref{erst}) this reflects the fact that every $\approx$-class of a rough mapping contains a partial mapping, and contains a map at the given exception. The case of (\ref{canosur}) is similar; note then the difference that when $X$ is infinite and $Y$ is empty, both terms are empty and there is no failure of surjectivity.

(\ref{eest}) For $f\in\mke(X,Y)$, let $X'$ be a cofinite subset of $D_f$ on which $f$ is injective, and define $f'=f\cap (X'\times f(X'))$. Then $f\sim f'$ and $f'$ is a cofinite-partial bijection. This proves surjectivity of $\mkep(X,Y)\to\mkst(X,Y)$. Also, if $Y$ is nonempty, define $f''$ as equal to $f$ on $X'$ and to some given constant elsewhere. This proves surjectivity of $\mkem(X,Y)\to\mkst(X,Y)$ if $Y$ is nonempty. If $Y$ is empty and $X$ is infinite, then both $\mkem(X,Y)$ and $\mkst(X,Y)$ are empty, so we have $\mathrm{id}_\emptyset$, which is surjective; if both $X,Y$ are empty, this map is the unique map between two singletons and hence is surjective.
\end{proof}

\begin{prop}\label{canosur0}
Let $X,Y$ be sets.
\begin{enumerate}
\item\label{canosur2} Let $\mk{C}$ be a full subcategory of $\mkr$ or $\mkrp$, or a full subcategory of $\mkpm$ not containing the empty set as an object. Let $\mk{C}^\star$ be the corresponding subcategory of $\mkrst$, with the same objects as $\mk{C}$.

Let $u$ be a functor from $\mk{C}$ into another category $\mathcal{D}$. Suppose that for every object $X$ of $\mk{C}$ and every $f\in\mk{C}(X)$ such that $f\approx\mathrm{id}_X$, we have $u(f)=\mathrm{id}_{u(X)}$. Then the functor $u$ factors (uniquely) through $\mk{C}^\star$.

%Let $u$ be a functor from a full subcategory $\mk{C}$ of $\mkr$ or $\mkrp$, or from a subcategory of $\mkpm$ not containing the empty set as an object, into another category $\mathcal{D}$. Suppose that for every object $X$ of $\mk{C}$ and every $f\in\mk{C}(X)$ such that $f\approx\mathrm{id}_X$, we have $u(f)=\mathrm{id}_{u(X)}$. Then the functor $u$ factors (uniquely) through the corresponding full subcategory of $\mkpst$ (with the same objects).
\item\label{canosur3} The same statement holds with $\mkr$ replaced with $\mkp$ (and $\mkrp$ erased), $\mkrm$ replaced with $\mkpm$, and $\mkrst$ replaced with $\mkpst$.
\item\label{canosur4} The same statement holds with $\mkr,\mkrp$ replaced by $\mke,\mkep$, with $\mkrm$ replaced by $\mkem$, and $\mkrst$ replaced by $\mkst$.
\item\label{canomo} Let $X$ be a set. Every homomorphism $u$ from $\mkr(X)$, $\mkrp(X)$, or $\mkrm(X)$ (resp.\ $\mkp(X)$ or $\mkpm(X)$) (resp. $\mke(X)$, $\mkep(X)$, $\mkem(X)$) into a monoid such that $u(g)=1$ whenever $g\approx\mathrm{id}_X$ factors (uniquely) through $\mkrst(X)$ (resp.\ through $\mkpst(X)$, resp.\ through $\mkst(X)$).
\end{enumerate}
\end{prop}
\begin{proof}
(\ref{canomo}) If $X$ is empty, this is trivial. Otherwise, this follows from the various previous items, applied to the corresponding full subcategory with a single object $X$.

(\ref{canosur2}) Denote by $\mkr'$ the left-hand category ($\mkr$, $\mkrp$ or $\mkrm$). For $X,Y$ sets in the full subcategory, $\mkr'(X,Y)\to\mkrst(X,Y)$ being surjective (since we have excluded the empty set in the case of $\mkrm$), the uniqueness of the factored map $\mkrst(X,Y)\to\mathcal{D}(u(X),u(Y))$ is clear.
If $f,f'\in\mkr(X,Y)$ and $f\approx f'$, then there exists, by Lemma \ref{carasim} (using characterization (\ref{hggpii}) in the case of $\mkr$ or $\mkrp$, and (\ref{hggpi}) in the case of $\mkrm$), an element $h\in\mk{C}(Y)$ such that $h\approx\mathrm{id}_Y$ and $h\circ f=h\circ f'$. Hence, by assumption, $u(h)=\mathrm{id}_{u(Y)}$, and therefore $u(f)=u(f')$, whence the factorization.

The proofs of (\ref{canosur3}) and (\ref{canosur4}) work in the same way, Lemma \ref{carasim} applying equally (using characterization (\ref{hggpii}) in the case of $\mkp$, $\mke$ or $\mkep$, and (\ref{hggpi}) in the case of $\mkem$).
\end{proof}

Let us provide a variant of Proposition \ref{canosur0}(\ref{canosur4}), replacing the assumption referring to maps $\approx$-equivalent to the identity, with the bare assumption that the target category consists of isomorphisms.

% We provide the proof for concreteness, and because the proof immediately adapts the general case.

\begin{prop}\label{puni}
Let $X,Y$ be sets.
\begin{enumerate}
\item\label{punivc} More generally, every functor from a full subcategory of the category $\mke$, $\mkep$ or $\mkem$, into a category in which all arrows are invertible, uniquely factors though the category $\mkest$. 

\item\label{punivm} In particular, the group $\mkest(X)$ is the largest group quotient of each of the monoids $\mke(X)$, $\mkep(X)$, $\mkem(X)$: every monoid homomorphism of any of these monoids into a group uniquely factors through its canonical homomorphism onto $\mkest(X)$.
\end{enumerate}
\end{prop}
\begin{proof}
(\ref{punivm}) This is the particular case of (\ref{punivc}), when we consider a subcategory with a single object; for simplicity we start proving it, because it highlights the main ideas of the proof. Let $E$ denote any of the groups $\mke(X)$, $\mkep(X)$, $\mkem(X)$. So we have the equivalence relation $\sim$ on $E$, and the quotient of $E$ by $\sim$ is equal to the group $\mkest(X)$, by Proposition \ref{surjca}(\ref{eest}). 
The monoid $E$ has the property that for $g,g'\in E$, $g\circ g'$ if and only if there exists $h\in E$ such that $h^2=h$ and $hg=hg'$, by Lemma \ref{carasim}. Let $u:E\to G$ be a homomorphism into a group. If $g,g'\in G$ and $g\sim g'$, consider $h$ as above; then $u(h)^2=u(h)$ and $u(h)u(g)=u(h)u(g')$. Since $G$ is a group, we deduce $u(h)=1$ and $u(g)=u(g')$. This shows that $u$ factors (uniquely) through $(E/\!\!\sim)=\mkest(X)$.

The general case (\ref{punivc}) for $\mke$ and $\mkep$ is proved exactly in the same way: the only difference is that $g,g'$ are maps $X\to Y$ with $g\sim g'$; Lemma \ref{carasim} equally applies. In $\mkem$, 
this works up to one bad joke, due to the the non-surjectivity of $\mkem(X,Y)\to\mkst(X,Y)$ when $X$ is nonempty and finite and $Y$ is empty (call this a bad $(X,Y)$). Let us formulate more precisely what is proved and what remains. Denote the functor by $u$ and the target category by $\mathcal{C}$. The previous argument yields that for any $(X,Y)$ not bad, the map $f\mapsto u(f)$ from $\mkem(X,Y)$ to $\mathcal{C}(X,Y)$ factors uniquely through a map $\bar{u}$ from $\mkst(X,Y)$ to $\mathcal{C}(X,Y)$, and this is functorial (one has functoriality of $\bar{u}$) for all triples $(X,Y,Z)$ such that none of $(X,Y)$ and $(Y,Z)$ is bad). It remains to extend this to bad pairs and check uniqueness and functoriality. If $(X,\emptyset)$ is a bad pair, $X$ is finite and $\mkst(X,\emptyset)$ is a singleton $\kappa_{X\emptyset}$, which is an isomorphism whose inverse is the image of the unique element $i_{\emptyset X}$ of $\mkem(\emptyset,X)$. Hence the only possible candidate for $\bar{u}(\kappa_{X\emptyset})$ is $u(i_{\emptyset X})^{-1}$, which exists since $\mathcal{C}$ has only isomorphisms. It remains to check functoriality; we have to consider triples $X\stackrel{f_\star}\to Y\to\emptyset$ or $Y\to\emptyset\to Z$ with $Y$ finite nonempty. In the first case, we need to check $\bar{u}(\kappa_{X\emptyset})=\bar{u}(\kappa_{Y\emptyset})\circ\bar{u}(f)$, that is, $u(i_{\emptyset X})^{-1}=u(i_{\emptyset Y})^{-1}\circ u(f)$, or equivalently $u(f)\circ u(i_{\emptyset X})=u(i_{\emptyset Y})$, which holds by functoriality of $u$ and since $f\circ i_{\emptyset X}=i_{\emptyset Y}$. In the second case, when $Z$ is empty we conclude immediately; otherwise denote by $g_\star$ the composite map $Y\to Z$; we need to check $\bar{u}(g_\star)=\bar{u}({i_{\emptyset Z}}_\star)\circ\bar{u}(\kappa_{Y\emptyset})$, that is, $u(g)\circ u(i_{\emptyset Y})=u(i_{\emptyset Z})$, which in turn holds by functoriality of $u$ and since $g\circ i_{\emptyset Y}=i_{\emptyset Z}$.
\end{proof}

Let us mention the coproduct in these ``near" categories.
Consider sets $X_1$, $X_2$, $Y_1$, $Y_2$, and $Z_i=X_i\sqcup Y_i$. We have inclusions of $X_1\times X_2$ and $Y_1\times Y_2$ into $Z_1\times Z_2$. This defines an injective map from $\mathbf{Rel}(X_1,X_2)\times\mathbf{Rel}(Y_1,Y_2)$ into $\mathbf{Rel}(Z_1,Z_2)$, mapping $(f_1,f_2)$ to $f_1\sqcup f_2$, which is bifunctorial (i.e., functorial in each variable). This is actually the coproduct in the category of relations. It restricts to an injective map from $\mkr(X_1,X_2)\times\mkr(Y_1,Y_2)$ into $\mkr(Z_1,Z_2)$ and similarly for $\mkrp$, $\mkrm$, $\mkpm$, $\mke$, $\mkem$, $\mkep$. In particular, this is the coproduct in each of these subcategories. (At the level of $\mathbf{Rel}$, this extends to infinite disjoint unions and thus defines infinite coproducts; however each of the given subcategories is not stable under infinite disjoint unions, as the reader can check.)

The above disjoint union construction is compatible with the relation $\approx$, and hence defines a map $\mkpst(X_1,X_2)\times\mkpst(Y_1,Y_2)\to\mkpst(Z_1,Z_2)$, which is also injective. This is the coproduct in the category $\mkpst$, as well as in its subcategory $\mkst$. In particular, we have a canonical injective monoid homomorphism of $\mkpst(X)\times\mkpst(Y)$ into $\mkpst(X\sqcup Y)$, and a canonical injective group homomorphism of $\mkst(X)\times\mkst(Y)$ into $\mkst(X\sqcup Y)$.

\begin{rem}
Let us mention a rough substitute for infinite disjoint unions of near actions. Let $X$ be a set and $(X_i)$ a family of subsets, with $X_i\cap X_j$ finite for all $i\neq j$. Let $G_i$ be a group with a near action $\alpha_i$ on $X_i$ for all $i$, and extend this as a near action $\alpha_i$ on $X$, with trivial near action outside $X_i$. Then $[\alpha_i(G_i),\alpha_j(G_j)]=1$ for all $i\neq j$. Hence this defines a near action of the restricted direct product $\bigoplus G_i$ near acts on $X$. If each $\alpha_i$ is $\star$-faithful on $G_i$, then the resulting near action is $\star$-faithful on $\bigoplus G_i$. Such a construction is used in \ref{unlofi}, and actually already appears in \cite[\S 2]{ShS}.
\end{rem}

\subsection{Index}

Let $\mathsf{card}^\pm$ be the collection of all signed cardinals, that is, cardinals or negative of cardinals; here the negative of a cardinal is just a cardinal decked out with a sign, with $-0=0$, and we sort cardinals and their negatives with the obvious ordering. It has a partial addition, namely we can add any two signed cardinals unless they are infinite and opposite, in the only reasonable way (for instance $-\aleph_1+\aleph_0=-\aleph_1$). We use subscripts to mean intervals in $\mathsf{card}^\pm$; for instance $\mathsf{card}_{\ge 0}$ is just the collection of cardinals (we omit $\pm$ when no confusion incurs).

To measure how is a relation far to be a bijective map, we introduce, for a relation $f:X\tot Y$ the defect
\[\delta_x(f)=|f[x]|-1\in\mathsf{card}_{\ge 0}\;(x\in X);\]
%;\quad \delta^y(f)=|f^{-1}[y]|-1\;(y\in Y).\]
here $|\cdot|$ means cardinal.
The expression 
\[\delta_\bullet(f)=\sum_{x\in X}\delta_x(f);\]
is well-defined when there are no infinitely many nonzero contributions $\delta_x(f)$ of both signs, namely:
\begin{itemize}
\item
if $f$ is a proper rough mapping, in which case it is valued in $\Z$;
\item when $f^{-1}$ is closely injective, in which case $\delta_\bullet(f)\in\mathsf{card}^\pm_{<\aleph_0}$;
\item when $f^{-1}$ is closely surjective, in which case $\delta_\bullet(f)\in\mathsf{card}^\pm_{>-\aleph_0}$.
\end{itemize}

\begin{defn}\label{defindex}
For a proper rough mapping $f$, supposed closely injective or closely surjective, define its (banker) index
\[\phi(f)=\phi_{X,Y}(f)=\delta_\bullet(f)-\delta_\bullet(f^{-1})\in\mathsf{card}^\pm\]
and its prodigal index $\psi(f)=\psi_{X,Y}(f)=-\phi(f)$. 
\index{index (banker/prodigal)}
We say that $f$ is balanced, receptive, strictly receptive, bounteous, strictly bounteous if $\phi(f)=0$, $\phi(f)\ge 0$, $\phi(f)>0$, $\phi(f)\le 0$, $\phi(f)<0$ respectively.\index{bounteous}\index{receptive}
\end{defn}

This is well-defined, takes a value in $\Z$ if and only if $f$ is closely bijective, and otherwise takes an infinite positive value for $f$ closely injective and an infinite negative value for $f$ closely surjective.

\begin{rem}
Both the banker index and the prodigal index appear as ``index" in the literature. It seems there is no definite objective reason to impose either choice. In the sequel, we follow, unless explicitly stated, the convention to work with the banker index character and abridge it as ``index". To memorize this convention of sign, it may be useful to retain that for $X=\N$ and $\sigma(n)=n+1$, we have $\phi(\sigma)=1$, which illustrates how $\N$ claims a positive balance when collecting one Gold Louis from its devoted complement. %and thereupon claims a positive balance.
\end{rem}

\begin{rem}
When $f$ is a map, $\delta_\bullet(f)=0$ and hence 
\[\phi(f)=-\sum_{y\in Y}(|f^{-1}(\{y\})|-1).\]
When $f$ is an injective partially defined map (i.e., $f$ is equal to a bijection from the cofinite subset $D_f$ of $X$ to the subset $D_{f^{-1}}$ of $Y$), the numbers $|f[x]|-1$ and $|f^{-1}[y]|-1$ can only take the values $0,-1$ and we have the restatement
\[\phi(f)=|D_{f^{-1}}^c|-|D_f^c|,\]
where $c$ means complement. 
%The latter is probably the most intuitive way to think of the index.
\end{rem}

It satisfies the following two fundamental properties:
\begin{prop}
The index is invariant in $\sim$-classes of proper rough mappings that are either closely injective or closely surjective. It is additive under composition of two closely injective, or two closely surjective proper rough mappings.
\end{prop}
\begin{proof}
Consider two sets $X_1,X_{-1}$ and $f\sim f':X_1\tor X_{-1}$. It is enough to prove the result when $f'=f\sqcup\{(x_1,x_{-1})\}$. In this case, for each $\varepsilon\in\{\pm 1\}$ and $x\in X_\varepsilon\smallsetminus \{x_\varepsilon\}$, we have $\delta_{x}(f^\varepsilon)=\delta_{x}({f'}^\varepsilon)$, while $\delta_{x_\varepsilon}({f'}^\varepsilon)=\delta_{x_\varepsilon}(f^\varepsilon)+1$. Therefore, $\delta_\bullet({f'}^\varepsilon)=\delta_\bullet({f}^\varepsilon)+1$ for both $\varepsilon$ equal to 1 or $-1$, and this difference is therefore annihilated in the index.

Consider proper rough mappings $X\stackrel{f}\tor Y\stackrel{g}\tor Z$, both closely surjective or both closely injective.

Start with the case when $g,f$ are both closely bijective. By invariance of the index on $\sim$ classes, we can suppose that both $f$ and $g$ are cofinite-partial bijections and that the domain of definition of $g$ is equal to the domain of definition of $f^{-1}$. Then in this case, $g\circ f$ is a bijection from $D_f$ to $D_{g^{-1}}$ and, denoting by $D^c_h$ the complement of $D_h$,
\[\phi(g\circ f)=|D^c_{g^{-1}}\!|\!-\!|D^c_f|=|D^c_{g^{-1}}|\!-\!|D^c_g|\!+\!|D^c_{f^{-1}}\!|\!-\!|D^c_f|=\phi(g)\!+\!\phi(f).\]

Now consider the case when $g,f$ are not both closely bijective. By Lemma \ref{qtxy}(\ref{gfis}), $g\circ f$ is not closely bijective, and hence $\phi(g\circ f)$ is infinite (positive or negative). Since $\delta_{\bullet}(g\circ f)$ is finite we have $\phi(g\circ f)=\delta_\bullet((g\circ f)^{-1})$. Since the sum of two cardinals including an infinite one is given by the maximum, we have to prove that $\delta_\bullet((g\circ f)^{-1})$ is the maximum of $\delta_\bullet(g^{-1})$ and $\delta_\bullet(f^{-1})$ when both $f,g$ are closely surjective, and the same with negative signs when both $f,g$ are closely injective.

First suppose $g,f$ are both closely surjective; we can suppose that they are both surjective (adding finitely many points to $f$ and $g$). Observe that for $h:X\tor Y$ surjective proper rough mapping, we have $\delta_\bullet(h^{-1})=|D^c_{h^{-1}}|$, where $c$ means complement. Then $D_{(g\circ f)^{-1}}^c=D_{g^{-1}}^c\cup g[D_{f^{-1}}^c]$. Since the cardinal of $g[D_{f^{-1}}^c]$ is the same as the cardinal of $D_{f^{-1}}^c$ (unless both are finite), we deduce that $|D_{(g\circ f)^{-1}}^c|=\max(|D_{g^{-1}}^c|,|D_{f^{-1}}^c|)$. Hence $\delta_\bullet((g\circ f)^{-1})=\max(\delta_\bullet(g^{-1}),\delta_\bullet(f^{-1}))$.

The case when $g,f$ are both closely injective is similar: we can suppose that they are both injective (removing finitely many points to $f$ and $g$). We then observe that for $h:X\tor Y$ injective proper rough mapping, we have $-\delta_\bullet(h^{-1})=|D^c_{h^{-1}}|$. Then again we have $D_{(g\circ f)^{-1}}^c=D_{g^{-1}}^c\cup g[D_{f^{-1}}^c]$ and we conclude in the same way
that $|D_{(g\circ f)^{-1}}^c|=\max(|D_{g^{-1}}^c|,|D_{f^{-1}}^c|)$, and hence $-\delta_\bullet((g\circ f)^{-1})=\max(-\delta_\bullet(g^{-1}),-\delta_\bullet(f^{-1}))$
\end{proof}

We can therefore define the index at the level of near bijections. For reference we write:

\begin{defn}
The index of a near bijection $u\in\mkst(X,Y)$ with a representative that is a bijection between cofinite subsets $X_1\smallsetminus F_1$ and $X_2\smallsetminus F_2$ is the value
\[\phi(u)=\phi_{X,Y}(u)=|F_2|-|F_1|\in\Z.\] 

The definitions of balanced, (strictly) receptive (strictly) bounteous (see Definition \ref{defindex}), make sense for near bijections. We denote by $\mksh(X,Y)$ the set of balanced near bijections from $X$ to $Y$, and $\mksh(X)=\mksh(X,X)$, the group of balanced near permutations. More generally we can use subscripts with obvious meaning; for instance receptive near permutations of $X$ form the submonoid $\mkst_{\ge 0}(X)$.
\end{defn}

\begin{cor}
The index yields a functor from the category $\mkst$ of near maps to the group $\Z$ (viewed as a category with a single object). In particular, for every set $X$ it yields a homomorphism $\mkst(X)\to\Z$, which is surjective for $X$ infinite and trivial otherwise.
\end{cor}
\begin{proof}
Only the last statement requires an additional argument: if $X$ is finite, $\mkst(X)$ is a trivial group and hence the index vanishes; otherwise we can write $X$ as $Y\sqcup\N$ for some subset $Y$, and then $\mathrm{id}_Y\sqcup\sigma$, where $\sigma(n)=n+1$ is an element of $\mke(X)$ with index 1.
\end{proof}

\begin{prop}\label{indbas}
Given sets $X,Y$ and $f\in\mke(X,Y)$, we have:
\begin{enumerate}
\item $f$ is near equal to some element in $\mkep(X,Y)$ that is either a map or the flip of a map (see Definition \ref{defrelation});
\item $f$ is near equal to some element in $\mkep(X,Y)\cap\mkem(X,Y)$ (i.e., an injective map) if and only $\phi_{X,Y}(f)\ge 0$ (i.e., $f$ is receptive);
\item $f^{-1}:Y\tor X$ is near equal to some injective map, if and only if $f$ is near equal to some surjective map, if and only if $\phi_{X,Y}(f)\le 0$ (i.e., $f$ is bounteous);
\item $f$ is near equal to some bijection if and only if $\phi_{X,Y}(f)=0$ (i.e., $f$ is balanced).\qed
\end{enumerate}
\end{prop}

\begin{cor}
For any set $X$ we have an exact sequence
\[1\to \mks_{<\aleph_0}(X)\to \mks(X)\to\mkst(X)\to\Z,\]
where the right hand map is the index character and other maps are the canonical ones; if $X$ is infinite the right-hand map is also surjective. In particular, the quotient $\mks(X)/\mks_{<\aleph_0}(X)$ is canonically isomorphic with the group $\mksh(X)$ of balanced near permutations.
\end{cor}

The classification of normal subgroups of $\mkst(X)$ easily follows from Baer's theorem below. For a cardinal $\kappa$, denote by $\kappa^+$ the next smallest cardinal. 

\begin{thm}[Baer \cite{Bae}, see {\cite[\S 11.3]{Sco}} or {\cite[\S8.1]{DM}}]\label{baerthm}
For any infinite set, the normal subgroups of $\mks(X)$ form a chain consisting of $\{1\}$, the alternating subsetoup $\mka(X)$, and, for infinite cardinals $\kappa$ at most equal to $|X|^+$, the subgroup $\mks_{<\kappa}(X)$ of permutations with support of cardinal less than $\kappa$.

In addition, for any two of these subgroups, the smaller one is characteristic in the larger one, and in particular these are the only subnormal subgroups, see \cite[\S 11.3]{Sco}. 
\end{thm}

%\cite{Bae} in 1934 to a classification of normal subgroups of $\mks(X)$ for an arbitrary infinite set $X$, as stated in Theorem \ref{baerthm}
%: normal subgroups form a well-ordered chain under inclusion, and are $\{1\}$, $\mka(X)$, and, for infinite cardinals $\kappa$ at most equal to the smallest cardinal $|X|^+$ greater than $|X|$, the subgroup $\mks_{<\kappa}(X)$ of permutations with support of cardinal $<\kappa$. See \cite[\S 11.3]{Sco} or \cite[\S8.1]{DM}

\begin{cor}\label{normne}
For every infinite set $X$, the distinct normal subgroups of $\mkst(X)$ are:
\begin{itemize}
\item the trivial subgroup;
\item for every $n\in\N$ and cardinal $\kappa$ with $\aleph_0<\kappa\le |X|^+$, the subgroup $\mkst_{n\Z,<\kappa}(X)$ of near permutations with support of cardinal $<\kappa$ and index in $n\Z$. The corresponding quotient canonically splits as the direct product $\Z/n\Z\times\mks(X)/\mks_{<\kappa}(X)$.
%
%\item the normal subgroups with cyclic quotient: $\phi^{-1}(n\Z)$, for $n\in\N$;
%\item the chain of normal subgroups of uncountable index $\mkst_{<\kappa}(X)$, where $\kappa$ ranges over infinite cardinals $\le\alpha$, of near permutations with support of cardinal $<\kappa$ (for $\kappa=\aleph_0$ this is the trivial subgroup).
\end{itemize}
%In particular, if $X$ is uncountable, normal subgroups of $\mkst(X)$ do not form a chain. pour X non den non plus a cause de Z... , and that for $X$ uncountable, there is no containment between $\mksh(X)$ and $\mkst_{<\aleph_1}(X)$.
\end{cor}
The groups $\mkst(X)$ and $\mkst(Y)$ are isomorphic if and only if $X,Y$ have the same cardinal.
\begin{proof}
It is clear that these are normal subgroups. To show that these are distinct, clearly $\mkst_{n\Z,<\kappa}(X)$ determines $\kappa$, and it also determines $n$ when $\kappa$ is uncountable, since one can find countably supported permutations of arbitrary index.

If $N$ is a normal subgroup, then by Baer's theorem (Theorem \ref{baerthm}), the intersection $N\cap\mksh(X)$ is equal to $\mks_{<\kappa}(X)/\mks_{<\aleph_0}(X)$ for some infinite cardinal $\kappa\le |X|^+$. Let $n\Z$ be the image of the index map in restriction to $N$. Then $\mkst_{0,<\kappa}(X)\subseteq N\subseteq \mkst_{n\Z,<\kappa}(X)$. Since the index map onto $n\Z$ is surjective, the reverse inclusion $\mkst_{n\Z,<\kappa}(X)\subseteq N$ immediately follows.

If $\kappa$ is uncountable, it is immediate that the quotient $\mkst(X)/\mkst_{n\Z,<\kappa}(X)$ splits as the direct product of $\mkst(X)/\mkst_{n\Z}(X)$, which is isomorphic to $\Z/n\Z$, and $\mkst(X)/\mkst_{<\kappa}(X)$, which is canonically isomorphic to $\mks(X)/\mks_{<\kappa}(X)$.

The last observation follows from the fact that nontrivial proper normal subgroups of the derived subgroup $\mksh(\aleph_\alpha)$ form a chain of length $\alpha$.
\end{proof}

% If $\kappa=\aleph_0$, then $N\cap\mksh(X)=\{1\}$. Since $\mksh(X)$ has trivial centralizer in $\mkst(X)$, it follows that $N=\{1\}$. 

% Now suppose that $\aleph_0<\kappa\le |X|$. Since $\mkst_{<\kappa}(X)\cap\mksh(X)\subseteq N\subseteq \mkst_{<\kappa}(X)$, to prove that it is enough to 

%If $\kappa=|X|^+$, it immediately follows that $N$ is among the groups in the first list. 

\begin{thm}\label{autnx}
Let $X$ be a set.
\begin{enumerate}
\item\label{item_tru} (Truss) the canonical homomorphism $\mkst(X)\to\Aut(\mksh(X))$ is an isomorphism.
\item\label{itau} The group of near permutations $\mkst(X)$ is complete, i.e., the canonical homomorphism to its automorphism group is an isomorphism.
\item\label{rebo} (Receptive vs bounteous) The monoids $\mkst_{\ge 0}(X)$ and $\mkst_{\le 0}(X)$ are not isomorphic. 
\end{enumerate}
\end{thm}
\begin{proof}
The first result is proved by Truss in \cite{Tru}; the method requires a solid luggage of logical background. I am not aware of an approach without logic; in the case of $X$ countable, it is done in \cite{ACM} but the proof has a gap (see Remark \ref{remacm}); I do not know if it can easily be fixed.

In (\ref{item_tru}), the injectivity part is the easy one, namely, the centralizer of $\mksh(X)$ in $\mkst(X)$ is trivial. It is not hard to deduce the (\ref{itau}) from (\ref{item_tru}). The first observation is that $\mksh(X)$ is a characteristic subgroup of $\mkst(X)$. This can be viewed in many ways, e.g., it follows immediately from the classification of normal subgroups of $\mks(X)$ (Theorem \ref{baerthm}) that $\mks(X)$ is perfect and hence $\mksh(X)$ is the derived subgroup of $\mkst(X)$; however a very elementary approach consists in using the fact that every element of $\mks(X)$ is the product of two elements of order $\le 2$, and hence it follows that $\mksh(X)$ is the intersection of all kernels of homomorphisms $\mkst(X)\to\Z$. 

%, which follows from the Schreier-Ulam-Baer theorem \cite[\S8.1]{DM} classifying normal subgroups of $\mks(X)$, implying in particular that every proper normal subgroup of $\mks(X)$ contains $\mks_{<|X|}(X)$, so that any proper quotient of $\mks(X)$ admits the nonabelian group $\mks(X)/\mks_{<|X|}(X)$ as a quotient. This implies that $\mksh(X)$ is a characteristic subgroup of $\mkst(X)$. 

Therefore, taking the restriction to $\mksh(X)$ yields a natural homomorphism $r:\Aut(\mkst(X))\to\Aut(\mksh(X))$. Then $r$ is injective. This follows from the following general easy claim:

In every semidirect product $\Lambda\rtimes\Z$ in which $\Lambda$ has a trivial centralizer, every automorphism $f$ that is the identity on $\Lambda$ is the identity everywhere.
To prove this claim, let $t$ be a generator of $\Z$. Then \[t^{-1}\lambda t=f(t^{-1}\lambda t)=f(t)^{-1}f(\lambda )f(t)=f(t)^{-1}\lambda f(t),\quad\forall\lambda\in\Lambda,\]
and hence $tf(t)^{-1}$ centralizes $\Lambda$, whence $f(t)=t$ and thus $f$ is the identity, so the claim is proved.

Finally, the surjectivity of $r$ is precisely the contents of (\ref{item_tru}), i.e., Truss' theorem.

(\ref{rebo}) Given a semigroup $M$, a group hull is an injective semigroup homomorphism from $M$ to a group, universal for injective homomorphisms from $M$ to groups. If $G$ is a group and $u$ a homomorphism $G\to\Z$, it is straightforward that $u^{-1}(\N)\to G$ is a universal hull. Thus an isomorphism $\mkst_{\ge 0}(X)\to\mkst_{\le 0}$ would extend to an index-reversing automorphism of $\mkst(X)$, which does not exist by (\ref{itau}).
\end{proof}
%%%%%%
%%%%% t_g, g\in G_{\ge 0} avec relateurs t_gh=t_g.t_h donne pres de G
%%%%%petit exo!
%%%%%%

Another feature of the group $\mksh(X)$, for $X$ infinite, is the central extension by $\Z/2\Z$ coming from the quotient by the alternating group $\mks(X)/\mka(X)$. See \S\ref{s_omega}.

\subsection{Power sets and cartesian products}\label{powset}\label{nearwreath}

\subsubsection{Near power sets}\label{nepose}
Denote by $\breve{\mk{C}}$ the opposite category of a category $\mk{C}$. So $\breve{\mk{C}}(X,Y)=\mk{C}(Y,X)$ for all objects $X,Y$. Note that $\mke$ can be viewed as a subcategory of both $\mkp$ and $\bmkp$ (mapping $f$ to $f$ and $f^{-1}$ respectively).

For a set $X$, we denote by $\mathcal{P}(X)$ its power set\index{$\mathcal{P}(X)$ (power set)}. This is a Boolean algebra. The set $\mathcal{P}_{<\aleph_0}(X)$ of finite subsets of $X$ forms an ideal and the quotient Boolean algebra is denoted by $\mathcal{P}^\star(X)$, and called near power set of $X$\index{near power set}\index{$\mathcal{P}^\star(X)$ (near power set)}.

%Start from the obvious observation that for a relation $f:X\to Y$, the mapping $f_*:\mathcal{P}(X)\to\mathcal{P}(Y)$, $A\mapsto f\lcc A\rcc$ is a Boolean algebra homomorphism if and only if it commutes with complementation, if and only if $f^{-1}$ is a map, that is, $f\in\bmkrm(X,Y)$. Denote by $\mathrm{BoA}$ the category of Boolean algebras. In this way we get functor $\bmkrm\to\mathbf{BoA}$ (the well-known inverse image functor). 

%\begin{lem}
%\end{lem}
%\begin{proof}
%This is clear, since it is contained in $f(A\smallsetminus B)\cup f(B\smallsetminus A)$, which is finite by properness of $f^{-1}$.\end{proof}

If $f\in\bmkp(X,Y)$ and $A,B\subseteq X$ with $A\tu B$ finite, then $f[A]\tu f[B]$ is finite. Indeed, it is included in $f[A\smallsetminus B]\cup f[B\smallsetminus A]$, which is finite by properness of $f^{-1}$. Therefore, $f$ induces a map $f_*$ from $\mathcal{P}^\star(X)$ to $\mathcal{P}^\star(Y)$. This is a Boolean algebra homomorphism: if $X$ is empty this is clear, and otherwise there exists a map $g$ such that $f\sim g$, and hence $f_*=g_*$, and $g_*$ is a Boolean algebra homomorphism because $g^{-1}$ is a map. Therefore this defines a functor $\bmkp\to\mathbf{BoA}$, constant on $\sim$-equivalences classes, which by Proposition \ref{canosur0} factors through a functor $\bmkpst\to\mathbf{BoA}$, and hence, by restriction, from $\mkest$ to $\mathbf{BoA}$. 
\begin{prop}\label{injboola}
This functor is faithful: for any sets $X,Y$, the canonical map
\[\bmkpst(X,Y)\to\Hom_{\mathbf{BoA}}(\mathcal{P}^\star(X),\mathcal{P}^\star(Y))\] is injective.
\end{prop}
\begin{proof}
If $X$ is empty, the left-hand set is at most a singleton and injectivity follows. Otherwise, injectivity amounts to show that whenever $f,g\in\mkpm(Y,X)$ and $f_*=g_*$, then $f\sim g$. Let us prove the contraposition. Suppose that $f,g\in\mkpm(Y,X)$ and $f\nsim g$. So there exists an infinite subset $I$ of $Y$ such that $f(y)\neq g(y)$ for all $y\in I$. Using that $f,g$ are proper, we can, by an obvious induction, construct an infinite countable subset $K$ of $J$ such that $A=f(J)$ and $g(J)$ are disjoint. Then $J\subseteq f^{-1}(A)\smallsetminus g^{-1}(A)$. Since $J$ is infinite, this implies that $f_*(A)\neq g_*(A)$.
\end{proof}

In particular, this yields, for every set $X$, a canonical injective monoid homomorphism $\bmkpst(X)\to\mathrm{End}(\mathcal{P}^\star(X))$; it restricts to a canonical injective group homomorphism $\mkst(X)\to\Aut(\mathcal{P}^\star(X))$. Surjectivity of this homomorphism is undecidable in ZFC, see (\ref{rudinshelah}) in \S\ref{hire}.

If $A\subseteq X$, and $\lcc A\rcc$ denotes its image in $\mathcal{P}^\star(X)$, the commensurator $\bmkpst(X,\lcc A\rcc)$ of $A$ in $\bmkpst(X)$ is the stabilizer of $\lcc A\rcc$ for the action of $\bmkpst(X)$ on $\mathcal{P}^\star(X)$; its inverse image in $\bmkp(X)$ is called its commensurator in $\bmkp(X)$, and denoted by $\bmkp(X,\lcc A\rcc)$. Similarly, we can define the commensurator $\mkst(X,\lcc A\rcc)=\bmkpst(X,\lcc A\rcc)\cap\mkst(X)$ of $A$ in $\mkst(X)$ as the stabilizer of $\lcc A\rcc$ for the action of $\mkst(X)$ on $\mathcal{P}^\star(X)$; its inverse image in $\mks(X)$ is called its commensurator in $\mks(X)$, and denoted by $\mks(X,\lcc A\rcc)$, it equals $\mkp(X,\lcc A\rcc)\cap\mks(X)$.

The subgroup $\mks(X,\lcc A\rcc)$ appears in a obscured way in \cite[Def.\ 2.1]{Bal1} in the restriction when either $A$ is either countable or has countable complement and both $A$ and its complement are infinite.\footnote{For (and only for) the interested reader: Ball starts \cite[\S 1]{Bal1} starts with a set $M$ partitioned into two infinite subsets $P\sqcup Q$. He denotes their cardinals as $X$ and $Z$, but I write $|P|$ and $|Q|$ instead; he assumes $|P|\ge |Q|$. He defines subgroups of normal subgroups of $\mks(M)$ and I will stick to $\mks(M)$ itself (so the cardinal he denotes as $Y$ should be understood as $|P|^+$). He then defines a subgroup of $\mks(M)$ he denotes as $L(Y,Z)$, which I thus denote (only here!) as $L(|P|^+,|Q|)$, which is the set of $s\in\mks(M)$ such that $|P\smallsetminus s^{-1}P|<|Q|$ and $|Q\smallsetminus s^{-1}Q|<|Q|$. Hence, when $Q$ is infinite countable, this is the same as $\mks(M,\lcc Q\rcc)$. This choice of inconveniently rigid notation is due to the fact that Ball was only interested in exhibiting maximal subgroups. Similar conventions are used in \cite{Ric}. The latter defines a subset $A\subseteq M$ as almost invariant for a given permutation $s$ if $|A\tu sA|<|A|$. This is not convenient, for instance because the set of such subsets is not a Boolean algebra of $\mathcal{P}(X)$ (while $\{A:|A\tu sA|<\kappa\}$ for fixed $\kappa$ is a Boolean subalgebra.)}

There is a canonical monoid homomorphism $\mkp(X,\lcc A\rcc)\to\mkp(A)$ mapping $f$ to $f\cap (A\times A)$, inducing a monoid homomorphism $\mkpst(X,\lcc A\rcc)\to\mkpst(A)$. They restrict to a monoid homomorphisms $\mkep(X,\lcc A\rcc)\to\mkep(A)$, $\mkem(X,\lcc A\rcc)\to\mkem(A)$, and a group homomorphism $\mkst(X,\lcc A\rcc)\to\mkst(A)$; the latter plays an important role throughout this paper.

\subsubsection{Bounded functions modulo $\co$-functions}\label{boco}

Let $\K$ be a normed field. Let $\ell_\K^\infty(X)$ be the $\K$-vector space of bounded functions $X\to\K$, endowed with the supremum norm, and $\co(X)$ its closed subspace of functions vanishing at infinity, that is, functions $f$ such that for every $\varepsilon>0$, the set of $x$ such that $\|f(x)\|\le\varepsilon$ is finite; this is the closure of the subspace of finitely supported functions.

For $f\in\mkrm(X,Y)$, composition yields a map $f^*:\ell^\infty_\K(Y)\to\ell^\infty_\K(X)$, with $\|f^*(u)\|\le\|u\|$ for all $u\in\ell^\infty_\K(Y)$. This gives a covariant functor from $\bmkrm$ to the category of normed $\K$-vector spaces with 1-Lipschitz linear maps. 

If $f\in\mkpm(X,Y)$, the map $f^*$ maps $\co(Y)$ into $\co(X)$, and hence passes to a linear map between quotients $\ell^\infty_\K(Y)/\co(Y)\to\ell^\infty_\K(X)/\co(X)$. Moreover, when $f\sim g$ and $u\in\ell^\infty_\K(Y)$, $f^*u$ and $g^*u$ coincide outside a finite subset. Hence the $\ell^\infty_\K(Y)/\co(Y)\to\ell_\K^\infty(X)/\co(X)$ only depends on the $\sim$-class of $f$. 

This gives a covariant functor from $\bmkpm$ to the category of normed $\K$-vector spaces with 1-Lipschitz linear maps. By Proposition \ref{canosur0}(\ref{canosur3}), this yields a functor from $\bmkpst$ (restricted to nonempty sets) to normed $\K$-vector spaces with 1-Lipschitz linear maps, mapping $X$ to $\ell^\infty_\K(X)/\co(X)$. Restricting to $\{0,1\}$ valued functions, we recognize the power set, and more precisely, we obtain injective maps $\mathcal{P}^\star(X)\to\ell^\infty_\K(X)/\co(X)$, equivariant under the $\mkp$-action. In particular, the above functor is faithful, as a consequence of Proposition \ref{injboola}.

By restriction, it yields a functor from $\mkst$ to normed $\K$-vector spaces with isometric bijective linear maps (we can use Proposition \ref{puni} to avoid excluding the empty set). In particular, for every set $X$, we have a natural linear isometric action of $\mkst(X)$ on the Banach space $\ell^\infty_\K(X)/\co(X)$.

\subsubsection{Near powers}\label{s_nearpower}
 Let $X,A$ be sets. Define the set of close maps $\mk{M}(X,A)$\index{close map} as the set of relations $f:X\tot A$ such that $p_X|_f$ has finite fibers, and such that there exists a cofinite subset $X'$ such that $f\cap (X'\times A)$ is a map. Define the near power\index{near power} as its quotient $A^{)X(}$ by the near equality relation $\sim$. Beware that close maps are not stable under composition: typically, $g\circ f$ can fail to be a close map if $f$ is constant. Also beware that while we have an inclusion $\mk{M}(X,A)\subseteq \mkr(X,A)$, the resulting map $A^{)X(}\to\mkrst(X,A)$, while well-defined, is not injective, because the restriction of $\approx$ to $\mk{M}(X,A)$ is larger than $\sim$. If $X$ is infinite, $A^{)X(}$ is known as the reduced power of $A$ with respect to the filter of cofinite subsets: it is the quotient of $A^X$ by the relation of coincidence on a finite subset; if $X$ is finite, $\mk{M}(X,A)$ is a singleton (even if $A$ is empty). For more on reduced powers, see \cite[Chap.\ 4]{CK}.

Nonetheless, being a close map is stable under precomposition with proper rough mappings, and moreover this is compatible with the relations $\sim$ on both sides. Therefore, for every set $A$, every $f\in\bmkp(X,Y)$ yields a map $\mk{M}(X,A)\to\mk{M}(Y,A)$, which in turn induces a map $f_*:A^{)X(}\to A^{)Y(}$. This defines a functor $u_M$ from $\bmkpst$ to sets. If $A=H$ is a group, the sets $H^{)X(}$ inherit a group structure and $u_H$ is actually a functor from $\bmkpst$ to the category of groups. The same holds for other structures of universal algebra; notably when $A=\Z/2\Z$ is considered as a Boolean algebra, we retrieve the construction $\mathcal{P}^\star(X)=(\Z/2\Z)^{)X(}$.

\begin{prop}
For every set $A$ with at least two elements, the functor $u_A$ is injective: the canonical map 
\[\bmkpst(X,Y)\to\mkrm(A^{)X(}, A^{)Y(})\] is injective for all sets $X,Y$.
\end{prop}
\begin{proof}
Let $B$ be a subset of $A$ of cardinal 2. If $f,g$ have the same image, then they also have the same image in $\mkrm(B^{)X(}, B^{)Y(})$. This reduces to the case when $A=\Z/2\Z$, in which case $A^{)X(}$ can be identified to $\mathcal{P}^\star(X)$. Hence this follows from Proposition \ref{injboola}.
\end{proof}

\subsubsection{Near products}\label{s_nearproduct}
 Near powers are particular instances of the more general notion of near product\index{near product}, which is also classical and a particular case of the notion of reduced product \cite[Chap.\ 4]{CK}. (However, the action of $\mkst(X)$ is specific to near powers.)

Given a set $I$ and a family of sets $(A_i)_{i\in I}$, the near product $\prod^\star_{i\in I}A_i$ can be defined the set of partial mappings, with cofinite definition domain, mapping each element $i$ of the definition domain to some element of $A_i$, and then identifying any two such partial mappings whenever they coincide outside a finite subset. 

If $A_i=A$ for all $i$, this is the near power $A^{)I(}$. If each $A_i$ is a group, the near product is canonically a group.

The issue of the empty set, already apparent in Proposition \ref{canosur0}, is more serious here: the canonical map $\prod_{i\in I}A_i\to\prod^\star_{i\in I}A_i$ is surjective when all $A_i$ are nonempty, but the product is empty as soon as one of the $A_i$ is empty, while the near product is empty if and only if infinitely many of the $A_i$ are empty.

\subsubsection{Near wreath products}\index{near wreath product}
Turning back to near powers, they allow to define a natural notion, naturally occurring (e.g., when studying centralizers in near permutation groups, see \S\ref{nearautz}), namely the notion of near wreath product.

Let us first recall that for a group $G$ and $G$-set $X$, and other group $H$ the unrestricted wreath product\index{wreath products} is the semidirect product $H\hat{\wr}\wr_X G=H^X\rtimes G$. The restricted wreath product $H\wr_X G$ is its subgroup $H^{(X)}\rtimes G$. When $X$ is implicit (which can mean that $X=G$ with left translation), it is sometimes omitted from the notation.

\begin{defn}
Given a group $H$ and a set $X$, the near wreath product $H\wr^\star\mkst(X)$ is the semidirect product $H^{)X(}\rtimes\mkst(X)$. 
\end{defn}

More generally, given a group $G$ and a homomorphism $G\to\mkst(X)$ (that is, anticipating on \S\ref{s_nearac}, a near action of $G$ on $X$), the near wreath product $H\wr^\star_X G$ is the semidirect product $H^{)X(}\rtimes G$. The notion of near wreath product naturally appears in the description of centralizers in near symemtric groups, see \S{nearautz}.

%%%%%%%%

%%%

%\subsection{Near products, near wreath products, proper subsets}\label{nearwreath}

%We start with a well-known procedure, to which we only add a name: near products and near powers. Let $(A_i)_{i\in I}$ be a family of sets. Define $\prod^\star_{i\in I}A_i$ as the quotient of the product $\prod_iA_i$ by the equivalence relation of coincidence outside a finite subset; we call it near product. When all $A_i$ are groups, note that this is a group as well. When $A_i=A$ is fixed, we write this as $A^{)I(}$ and call it near power of $A$ indexed by the set $I$.

%Let $A,X$ be sets. Define $A^{)X(}$ as the set of functions $X\to A$ modulo coincidence outside a finite subset. If $A$ is a group, this is the quotient $A^X$ by its subgroup $A^{(X)}$ of finitely supported elements.

%Let us now introduce near wreath products. If $X$ is a set, the symmetric group $\mks(X)$ acts on $A^X$ by $\sigma\cdot f(x)=f(\sigma^{-1}(x))$. The same formula defines an action of the group of near permutations on $A^{)X(}$; if $A$ is a group, this is an action by group automorphisms. Hence, if $X$ is a near $G$-set for some group $G$, this defines an action of $G$ on $A^{)X(}$. 
%\begin{defn}For groups $G,A$ and a near $G$-set $X$, the semidirect product $A^{)X(}\rtimes G$ is called the near wreath product of $A$ and $(G,X)$, and denoted $A\wr^\star_XG$. When $X$ is implicit, (notably, for $G=\mkst(X)$) we simply write it as $A\wr^\star G$. \end{defn}

%The near wreath product naturally appears when we need to describe some centralizers in $\mkst(X)$, see \S\ref{nearautz}.

\subsubsection{Multiproper subsets in cartesian products}
The cartesian product of two near $G$-sets does not naturally carry a structure of near $G$-set: the problem is that the product of two cofinite proper subsets is not cofinite. Nevertheless, some important near actions can be constructed using cartesian products. 

%Nevertheless, we can produce near actions using products.

\begin{defn}
Let $X_1,\dots,X_d$ be sets. Consider the product $W=\prod_{i=1}^dX_i$, and $\pi_i:X\to X_i$ the projection. Say that a subset $P$ of $W$ is multiproper\index{multiproper} (or biproper\index{biproper} when $d=2$) if $\pi_i|_P$ is a proper map (i.e., has finite fibers) for every $i$. Denote by $\mathcal{PP}(X_1,\dots,X_d)$ the set of multiproper subsets of the product set $W$, and $\mathcal{PP}^\star(X_1,\dots,X_d)$ its image in $\mathcal{P}^\star(W)$. If all $X_i$ are equal, denote them as $\mathcal{PP}_d(X)$ and $\mathcal{PP}^\star_d(X)$.
\end{defn}

Note that $\mathcal{PP}(X_1,\dots,X_d)$ is an ideal of $\mathcal{P}^\star(\prod X_i)$. For $d=1$, every subset is ``multi"-proper.

%A = proper + 1. A inclus dans P(X2/I)   Spec(A) <-- Spec(X2/I)

%SC(X2)-->>SC(X)
%P(X2) <-- P(X)
%  X2 -->> X

\begin{prop}\label{mullion}
The product $\prod_{i=1}^d\bmkpst(X_i)$, and hence $\prod_{i=1}^d\mkst(X_i)$, naturally acts on $\mathcal{P}_{\mathrm{pro}}^\star(X_1,\dots,X_d)$.
\end{prop}
\begin{proof}
If some $X_i$ is empty, the product is empty and $\mathcal{P}_{\mathrm{pro}}^\star(X_1,\dots,X_d)$ is a singleton. Hence we can assume that all $X_i$ are nonempty.

We consider the action of the monoid $\prod_{i=1}^d\bmkpm(X_i)$ on $W=\prod_{i=1}^d(X_i)$, and the induced action of $\prod_{i=1}^d\bmkpm(X_i)$ on $\mathcal{P}(W)$. An important issue is that, if at least two of the $X_i$ are infinite, this action does not factor through $\prod_{i=1}^d\bmkpst(X_i)$: typically a product of nontrivial finitely supported permutations can have infinite support. However, in restriction to some subsets of this power set, the action indeed factors.

Say that a subset of $W=\prod_{i=1}^dX_i$ is a window\index{window} if it includes a product of cofinite subsets, and a mullion\index{mullion} otherwise. The set of mullions is an ideal $I$ of $\mathcal{P}(W)$. This ideal is preserved by the action of $\prod_{i=1}^d\bmkpm(X_i)$: one way to see this is to observe that the Boolean algebra $\mathcal{P}(W)/I$ is canonically isomorphic to the tensor product of the Boolean algebras $\mathcal{P}^\star(X_i)$ and hence carries a natural action of $\prod_i\mkst(X_i)$, so that the composed action of $\prod_{i=1}^d\bmkpm(X_i)$ is the same as the previous one. Hence it indeed factors. 

The intersection between the ideals of multiproper subsets and mullions is reduced to the ideal of finite subsets.
This means that among multiproper subsets, the equivalence relations of coincidence modulo a mullion, and modulo a finite subset, coincide. Therefore, the action of $\prod_{i=1}^d\bmkpm(X_i)$ on $\mathcal{P}_{\mathrm{pro}}^\star(X_1,\dots,X_d)$ factors through $\prod_{i=1}^d\bmkpm(X_i)$.
\end{proof}

%Then the Boolean algebra $\mathcal{P}(W)/I$ is canonically the tensor product of the Boolean algebras $\mathcal{P}^\star(X_i)$ and hence carries a natural action of $\prod_i\mkst(X_i)$.

%If $P,P'$ are proper subsets of $X$ that coincide modulo $I$, then $P\triangle P'$ is both a mullion and proper, and hence is finite. This means that coincidence modulo the ideal of mullions or modulo finite subsets is the same, for proper subsets. This yields the desired action.

In particular, embedding diagonally $\mkst(X)$ into $\mkst(X)^d$ yields an action of $\mkst(X)$ on the set $\mathcal{PP}^\star_d(X)$ of proper subsets of $X^d$ modulo finite subsets. Many interesting subgroups of $\mkst(X)$ can be defined by considering stabilizers for such actions, see \S\ref{s_products}.

\section{Near actions}\label{s_nearac}

\subsection{Near actions and realizability}\label{nearaction}
We introduce here the notion of near action. 

%I am not aware of earlier appearance of this very basic notion.

\begin{defn}A near action of a group $G$ on a set $X$ is a homomorphism $\alpha:G\to\mkst(X)$. We then call $X$ endowed with a near action of $G$, a near $G$-set.\end{defn}

A near action is called $\star$-faithful\index{faithful ($\star$-)} if $\alpha$ is injective. It is called near free\index{near free} if for every $g\in G\smallsetminus\{1\}$, the set of fixed points of some/any lift of $\alpha(g)$ in $\mkem(X)$ is finite. Clearly when $X$ is infinite, near free implies being $\star$-faithful.

\begin{defn}\label{defequi}
Consider near $G$-sets $X_1,X_2$, defined by homomorphisms $\alpha_i:G\to\mkst(X_i)$. A near homomorphism\index{near homomorphism} (resp.\ proper near homomorphism, resp.\ near isomorphism) between the near $G$-sets $X_1,X_2$ is an element $f$ in $\mkrst(X_1,X_2)$ (resp.\ $\mkpst(X_1,X_2)$, resp.\ $f\in\mkst(X_1,X_2)$) such that $f\circ\alpha_1=\alpha_2$.

We denote by $\mkrst_G(X_1,X_2)$, resp.\ $\mkpst_G(X_1,X_2)$, resp.\ $\mkst_G(X_1,X_2)$ the set of near homomorphisms, resp.\ proper near homomorphisms, resp.\ near isomorphisms from $X_1$ to $X_2$.
\index{near isomorphic}
We denote by $\mkrst_G(X)=\mkrst_G(X,X)$, $\mkpst_G(X)=\mkpst_G(X,X)$ and $\mkst_G(X)=\mkst_G(X,X)$ the monoids of endomorphisms and proper endomorphisms and group of near automorphisms of the near $G$-set $X$. Near $G$-sets with near homomorphisms, and proper near homomorphisms, form categories $\mkrst_G$ and $\mkpst_G$, whose invertible arrows form the subcategory $\mkst_G$ in both cases.

Also, denote by $\mkrm_{(G)}(X_1,X_2)$ and $\mkpm_{(G)}(X_1,X_2)$ the inverse images of $\mkrst_G(X_1,X_2)$ and $\mkpst_G(X_1,X_2)$.
\end{defn}

(In non-ambiguous phrases such as ``an isomorphism between near $G$-sets", we mean a near isomorphism.) See also \S\ref{s_equivariant1}.

%A homomorphism (resp.\ isomorphism) between near $G$-sets $X_1,X_2$ (defined by homomorphisms $\alpha_i:G\to\mkst(X_i)$) is an element $f\in \mk{E}^*(X_1,X_2)$ (resp.\  $f\in\mkst(X_1,X_2)$) such that $f\circ\alpha_1=\alpha_2$. If clarification is necessary, we call them near proper homomorphisms and near isomorphisms. 

For sets $X,Y$, there is a canonical embedding of $\mkst(X)\times\mkst(Y)$ into $\mkst(X\sqcup Y)$ (see the end of \S\ref{mkpetc}). Thus if $X,Y$ are near $G$-sets, the disjoint union $X\sqcup Y$ is naturally a near $G$-set. This readily extends to finite disjoint unions; however beware that we cannot define infinite disjoint unions of near $G$-sets in general.

If $X$ is a finite set, then $\mkst(X)$ is a trivial group and thus $X$ is a near $G$-set in a unique canonical way. We implicitly use this throughout. 

A balanced isomorphism\index{balanced isomorphism} between near $G$-sets $X_1,X_2$ is an isomorphism $f$ with index zero, i.e., that admits a bijective representative. If it exists, we say that the near $G$-sets $X_1,X_2$ are balanceably near isomorphic\index{balanceably near isomorphic}.

\begin{prop}The near $G$-sets $X_1,X_2$ are near isomorphic if and only if there exist finite sets $F_1,F_2$ such that $X_1\sqcup F_1$ and $X_2\sqcup F_2$ are balanceably isomorphic; moreover, we can require $\min(|F_1|,|F_2|)=0$.\end{prop}
\begin{proof}
Since the embedding of $X_i$ into $X_i\sqcup F_i$ is a near isomorphism, the condition is sufficient. Conversely, consider a representative $f$ in $\mkem(X_1,X_2)$ of near isomorphism from $X_1$ to $X_2$. So $f$ is a bijection between two cofinite subsets $X_i\smallsetminus M_i$. Then $f$ extends to a bijection from $X_1\sqcup M_2$ to $X_2\sqcup M_1$, mapping the added $M_2$ by the identity and the added $M_1$ by the identity. So the condition is sufficient. Moreover, by Proposition \ref{indbas}, we can choose $f$ so that either $M_1$ or $M_2$ is empty.
\end{proof}

Given\index{index character}
 a near action $\alpha:G\to\mkst(X)$ and the (banker) index character $\phi_X:\mkst(X)\to\Z$, we call the resulting homomorphism $\phi_\alpha=\phi_X\circ\alpha:G\to\Z$ the {\em index character} of the near action (or ``index" for short); we often denote it $\phi_X$ when it is not ambiguous. The non-negative generator of the image in $\Z$ of the index map $\phi_\alpha$ is called the index number\index{index number} of $\alpha$ (or of the near $G$-set $X$), and denoted by $\iota_G(X)$, or $\iota(X)$; in general the index character carries more information. Nevertheless, ``zero index" indifferently refers to the index character or the index number: if this is the case we say that $X$ is a balanced (or $G$-balanced) near $G$-set.
Isomorphic near $G$-sets have the same index character, and hence the same index number.

%We call almost action of a group $G$ on a set $X$ the datum of a homomorphism $\alpha:G\to\mksh(X)$. Thus almost actions are precisely the near actions with zero index; then we call $X$ a balanced near $G$-set.

Every $G$-set is naturally a balanced near $G$-set. If the near action is $\star$-faithful (in which case we can call the action $\star$-faithful) then the action is faithful, but the converse does not hold: for instance the action of finitely supported permutations is faithful, but is a trivial near action.

For instance, finite sets form a single near isomorphy class of near $G$-sets; finite sets of a given cardinality form a single class of near $G$-sets up to balanced isomorphy.

%Let us now study the realizability notions of Definition \ref{d_rea}

Let us now introduce the realizability notions.

\begin{defn}\label{d_rea}
A near action $G\to\mkst(X)$ is
\begin{itemize}
\item realizable\index{realizable} if it can be lifted to an action;
\item finitely stably realizable\index{stably realizable}\index{finitely stably realizable} (resp.\ stably realizable) if there exists a finite set $Y$ (resp.\ a set $Y$ with trivial near action), such that $X\sqcup Y$ is realizable;
% with $Y$ required to be finite.
%, i.e., an homomorphism $G\to\mks(X)$;
\item completable\index{completable} if there exists a near $G$-set $Y$ such that $X\sqcup Y$ is realizable.
%\item {\em stably realizable} if there exists a trivial $G$-set $Y$ (i.e., $Y$ is endowed with the trivial action) such that $X\sqcup Y$ is realizable;
\end{itemize}
\end{defn}

%Hence (realizable) $\Rightarrow$ (finitely stably realizable) $\Rightarrow$ (stably realizable) $\Rightarrow$ (completable); also (stably realizable) $\Rightarrow$ (balanced).

%All these implications are strict. More precisely, apart from the middle implication, we give examples showing that even in the case of $G=\Z^2$, these implications are strict in general. The middle implication, namely (finitely stably realizable) $\Rightarrow$ (stably realizable), is also strict (Proposition \ref{nonfistre}), but when $G$ is finitely generated or abelian, it is an equivalence (Propositions \ref{streafg}(\ref{streafg1}) and Corollary \ref{abstre}).

For any set $X$ with $x\in X$, the inclusion map $X\smallsetminus\{x\}\to X$ is a closely bijective map, and thus induces an isomorphism $\mkst(X\smallsetminus\{x\})\to\mkst(X)$. Therefore, for any group $G$ with a near action on $X$, we obtain a near $G$-set structure on $X\smallsetminus\{x\}$.
If we start from a realizable near action on $X$, the latter near action can fail to be realizable (see Proposition \ref{nonrea_1end} and Corollary \ref{qcynr}); however, it is stably realizable, by construction.

%For instance, by Proposition \ref{injmkst}, for every group $G$ and $G$-set $X$ with $x\in X$, we obtain by restriction a (balanced) near action of $G$ on $X\smallsetminus\{x\}$, which is by construction finitely stably realizable. 

\begin{prop}
Fix a group $G$. For near $G$-sets, being completable, stably realizable, finitely stably realizable are invariant under near isomorphism. Being realizable is invariant under balanced near isomorphism.

Each of these notions is stable under taking finite disjoint unions.
\end{prop}
\begin{proof}
The last assertion (on disjoint unions) is immediate. It applies in particular to taking the disjoint union with a finite set.

The assertion on realizability is immediate, since we can transport the action through a bijection realizing the near isomorphism. Now let $X\to Y$ be a near isomorphism.

Suppose that $X$ is stably realizable. So $X\sqcup F\to Y\sqcup F$ (extended by the identity of $F$, where $F$ is endowed with the trivial near action) is a near isomorphism. Add finite subsets $F_1,F_2$ so that $X\sqcup F\sqcup F_1\to Y\sqcup F\sqcup F_2$ is a balanced near isomorphism. Then $X\sqcup F\sqcup F_1$ being realizable, $Y\sqcup F\sqcup F_2$ is realizable. Hence $Y$ is stably realizable. If $X$ is finitely stably realizable, then we can choose $F$ to be finite, and we deduce that $Y$ is finitely stably realizable. If instead $X$ is completable, we perform the same proof, now with $F$ an arbitrary near $G$-set, and deduce that $Y$ is completable.
\end{proof}

\begin{lem}\label{completablek}
Let $G$ be a group, with a finitely generated and normal subgroup $N$; view every near $G/N$-action, by composition, as a near $G$-action. Then a near $G/N$-action is completable (resp.\ stably realizable, resp.\ finitely stably realizable) if and only if it satisfies the same condition as near $G$-action.
\end{lem}
\begin{proof}
The forward implication is trivial (for an arbitrary normal subgroup). Conversely suppose that a near $G/N$-set $X$ is completable as near $G$-set: consider a $G$-set of the form $Z=X\sqcup Y$. Since $N$ is finitely generated, we have $Z^N\cap X$ cofinite in $X$. Thus $Z^N$ completes $Z^N\cap X$ as near $G/N$-set, which is near isomorphic to $X$. 

If in addition $Y$ is a trivial near $G$-set, resp.\ is finite, this also holds for $Y\cap Z^N$ and hence we obtain the similar implications with stable realizability and finite stable realizability.
\end{proof}

For instance, when we deal with finitely generated abelian groups, or more generally virtually polycyclic groups, completability of a near $G$-set $X$ only depends on the image of $G$ in $\mks(X)$. Beware that this is not true in general, since given a non-completable near $G$-set $X$, if one chooses a free group $F$ with quotient $G$, $X$ is realizable, and hence completable as near $F$-set.

\begin{rem}
There is an immediate generalization of the notion of near actions to groupoids (recall that a group is a group with a single object): a near action of a groupoid\index{near action (groupoid)} is a functor into the category $\mkst$. The notions also extend to a noninvertible setting. A near action of a monoid\index{near action (monoid)} is a homomorphism into $\mkr^\star(X)$ for some set $X$. It would then often be useful to restrict to near action by near proper maps, i.e., consider homomorphisms into $\mkp^\star(X)$. This restriction is automatic in the case of groups. Again, this has immediate categorical generalizations, considering a given small category and its functors into the category $\mke^\star$ or $\mkp^\star$. 
\end{rem}

\begin{rem}
If a group $G$ acts freely on a nonempty set $X$, we obviously have $|G|\le |X|$. This conclusion does not hold if we only assume the action to be near free. Indeed, suppose that $G$ admits a decreasing sequence of normal subgroups $(G_n)$ of countable (possibly finite) index with trivial intersection: plenty of groups of cardinal $2^{\aleph_0}$, including all abelian ones, admit such sequences. Then the action of $G$ on the disjoint union $\bigsqcup G/G_n$ (which is countable) is near free.
\end{rem}

\begin{prop}\label{kapparea}
Let $\kappa$ be an infinite cardinal, and let $G$ be a group of cardinal $\le\kappa$. Then every near action of $G$ is balanceably near isomorphic to the disjoint union of a realizable action and a near action on a set of cardinal $\le\kappa$.
\end{prop}
\begin{proof}
Let $X$ be a near $G$-set with near action $\alpha$. For every $g\in G$ choose $f_g:G\to G$ a representative of $\alpha(g)$. Let $Y$ be the set of elements $x\in X$ such that, for all $g,h\in G$ we have $f_{gh}(x)= f_g(f_h(x))$. For $x\in Y$ and $g,h,k\in G$ we have
\[f_g(f_h(f_k(x))=f_g(f_{hk}(x))=f_{ghk}(x)=f_{gh}(f_k(x)),\]
and hence $f_k(x)\in Y$. Thus $g\mapsto f_g$ defines a homomorphism from $G$ to the monoid of self-maps of $Y$, and this has to be group homomorphism into the group of permutations of $Y$. This implies that $Y$ is a commensurated, realizable subset of $X$; clearly its complement has cardinal $\le\kappa$.
\end{proof}

There is an elaboration of the latter argument, which intuitively corresponds to the case when $\kappa$ is finite, namely Theorem \ref{nearactionfp}.

\subsection{Commensurated subsets and ends}\label{csends}

The fundamental indecomposability notion for near actions is the following:

\begin{defn}
A near $G$ set $X$ is 1-ended\index{one-ended} if $X$ is infinite and $X$ is not isomorphic (as a near $G$-set) to the disjoint union of two infinite near $G$-sets.
\end{defn}

To reinterpret this condition, we need the following definition:

\begin{defn}Given a near action $\alpha:G\to\mkst(X)$, and denote by $\pi$ the projection $\mkem(X)\to\mkst(X)$ (see \S\ref{mkpetc}). We say that $Y\subseteq X$ is $G$-commensurated if for every $f\in\pi^{-1}(\alpha(G))$, the set $f^{-1}(Y)\triangle Y$ is finite.
\end{defn}

The condition is equivalent to the requirement $\alpha(G)\subseteq \mkst(X,\lcc Y\rcc)$, where the commensurator $\mkst(X,\lcc Y\rcc)$ is defined in \S\ref{nepose}. This notion is essential, because the canonical homomorphism $\mkst(X,\lcc Y\rcc)\to\mkst(Y)$ yields a near action on $Y$.

\begin{prop}
A near $G$-set $X$ is 1-ended if and only if $X$ has exactly two $G$-commensurated subsets up to near equality ($\sim$): $\emptyset$ and $X$.
\end{prop}
\begin{proof}The condition $\emptyset\sim X$ means that $X$ is finite, in which case none of the condition hold. So we can suppose that $X$ is infinite.

Suppose that $X$ is not 1-ended and infinite. Then $X$ is near isomorphic to a disjoint union $Y\sqcup Z$ with two infinite near $G$-sets $Y,Z$. If $f\in\mkem_G(Y\sqcup Z,X)$ (which is not empty by assumption), then $f(Y)\subseteq X$ is $G$-commensurated, infinite and not cofinite.

Conversely, suppose that $X$ is infinite does not satisfy the given condition. Then it has a $G$-commensurated subset $Y$ with both $Y$ and its complement $Z$ infinite. Then $X$ is near isomorphic to the near $G$-set $Y\sqcup Z$, and hence is not 1-ended.
\end{proof}

%for every $g\in G$ and some/every lift $\tilde{g}:X\smallsetminus F_1\to X\smallsetminus F_2$ of $\alpha(g)$ in $\widehat{\mkst}(X)$, the set $\tilde{g}(Y\smallsetminus F_1)\triangle Y$ is finite (for given $g$, this does not depend on the choice of representative).

%Note that this depends only on the image $\alpha(G)$ of $G$ in $\mkst(X)$.

%Commensurated subsets are thus the ``near subactions".
%
% In other words, $X$ is 1-ended if $X$ has exactly 2 commensurated subsets up to finite perturbation: the empty set and $X$.

\begin{defn}
The near $G$-set $X$ is said to be finitely-ended\index{finitely-ended} if the number of commensurated subsets modulo near equality is finite; then the number of 1-ended commensurated subsets up to near equality is called the number of ends. Otherwise, $X$ is said to be infinitely-ended.
\end{defn}

More generally, there is a notion of space of ends, see \S\ref{s_soe}.

For reference, we state the following straightforward proposition.

\begin{prop}\label{fgfended}
Let $G$ be a group. Then any finitely-ended near $G$-set $X$ is the near disjoint union of finitely many 1-ended commensurating subsets $X_1,\dots,X_k$; $k$ is the number of ends and $X_1,\dots,X_k$ are uniquely defined up to near equality and reordering. In particular, $X$ is (balanceably) near isomorphic to a disjoint union $X_1\sqcup\dots \sqcup X_k$ where $(X_i)_{1\le i\le k}$ is uniquely defined up to near equality and reordering.\qed
\end{prop}

% which in the infinitely-ended case is much finer

%the infinitely-ended case

%To say that a near action of a group on a set $X$ is 1-ended means that $X$ is infinite and that it is not isomorphic to the action on a disjoint union of two infinite near actions. This is thus an indecomposability condition. 

\begin{rem}It is tempting to define things in an analogous way for balanced near actions: the ``balanced subactions" would be given by the commensurating subsets with zero index. But beware that this is ill-behaved: indeed the intersection of two commensurated subsets of zero index can have nonzero index (Example \ref{illint} below). The central importance of commensurated subsets in the theory of near actions and this simple remark explains why the theory of balanced near actions can only be developed within the broader framework of near actions.
\end{rem}

\begin{exe}\label{illint}
Let $\Gamma=\mathbf{Z}$ act on $\Z\times\{\pm 1\}$ by $n\cdot (m,e)=(n+m,e)$. Then $\Z\times\{1\}$ and $(\Z_{\ge 0}\times\{1\})\cup (\Z_{<0}\times\{-1\})$ are both commensurated subsets with zero index, while their intersection $\Z_{\ge 0}\times\{1\}$ has index number equal to 1.
% (for $\Gamma=\Z$, index $n$ means: whose index character is the homomorphism $m\mapsto nm$).  
\end{exe}

Although it is a worse-behaved notion, let us provide a terminology for its weaker analogue for a balanced near action: we say that a balanced near action is balanceably indecomposable if it is infinite and not isomorphic to the disjoint union of two infinite balanced near actions. One-ended near actions are examples; the action of $\Z$ on itself is an example of a 2-ended balanceably indecomposable near action. For examples with an arbitrary finite number of ends, see the Houghton near action with 3 ends (Example \ref{e_hough}) and more generally with $n$ ends (\S\ref{s_hou}). For examples with infinitely many ends, see \S\ref{windec}.

That balanced near actions naturally split into simpler unbalanced near actions is part of the richness of the theory.

%%%%%%%%%%%%%%%%%%%%%%%%%%%%%%%%%%%%%%%%%%%%%%%%%%%%%%%%%%%%

\subsection{Space of ends}\label{s_soe}
We start with a very short reminder of Stone duality\index{Stone duality}; see \cite[\S I.7]{Je} for proofs (which are also instructive exercises!). A Stone space\index{Stone space} is by definition a compact, totally disconnected Hausdorff topological space. They form a category $\mathbf{Sto}$, whose arrows are continuous maps.

For a Stone space $K$, let $\mathcal{C}(K)$ be the Boolean algebra of continuous maps $K\to\Z/2\Z$. It separates points of $X$. Then $K\mapsto \mathcal{C}(K)$ is a naturally a contravariant functor from $\mathbf{Sto}$ to the category $\mathbf{BoA}$ of Boolean algebras (with unital ring homomorphisms).

In the other direction, let $A$ be a Boolean algebra. Let $\mathrm{Spec}(A)$ be the set of prime ideals of $A$. This is a closed subset of $\mathcal{P}(A)$ (endowed with the product topology under its identification with $(\Z/2\Z)^A$), and thus is a Stone space. Actually prime ideals of $A$ are kernels of unital ring homomorphisms $A\to\Z/2\Z$. The space $\mathrm{Spec}(A)$ is called the Stone space of $A$. Then $A\mapsto\mathrm{Spec}(A)$ is a contravariant functor from $\mathbf{BoA}$ to $\mathbf{Sto}$.

Stone duality states that for every Boolean algebra $A$, the canonical homomorphism from $A$ to $\mathcal{C}(\mathrm{Spec}(A))$ is an isomorphism, and that for every Stone space $K$, the canonical continuous map from $K$ to $\mathrm{Spec}(\mathcal{C}(K))$ is a homeomorphism. Thus, the above functors are essentially inverse to each other, and the categories of Boolean algebras and Stone spaces are contravariantly equivalent. Better, it follows that this is compatible with monoid actions: for any monoid $M$ (with opposite monoid $\breve{M}$), it yields a contravariant equivalence of categories between Boolean algebras endowed with an $\breve{M}$-action and Stone spaces endowed with an $M$-action. For groups $G$, we encode, as usual, $\breve{G}$-actions as $G$-actions by precomposing with the inversion map of the group and get a contravariant equivalence of categories between Boolean algebras endowed with a $G$-action and Stone spaces endowed with a $G$-action.

We now define the space of ends of a near action: it will be a singleton precisely in the one-ended case. For power sets, we use notation of \S\ref{powset}: for any set $X$, the power set $\mathcal{P}(X)$ is a Boolean algebra, and $\mathcal{P}^\star(X)$ is its quotient by the ideal of finite subsets of $X$. Recall from \S\ref{powset} that we have a canonical action of the group $\mkst(X)$ by ring automorphisms on $\mathcal{P}^\star(X)$.

%For a set $X$, let $\mathcal{P}(X)$ be the Boolean algebra of subsets of $X$, and $\mathcal{P}^\star(X)$ its quotient by the ideal of finite subsets.

%In the language of \S\ref{nearwreath}, this is the near power $(\Z/2\Z)^{)X(}$. In particular, there is a canonical action of the group $\mkst(X)$ by ring automorphisms on $\mathcal{P}^\star(X)$.

%, where given $g$ and $Y\in\mathcal{P}^\star(X)$, $gY$ is the class of $\tilde{g}\tilde{Y}$ modulo finite subsets, where $\tilde{g}\in\widehat{\mkst}(X)$ and $\tilde{Y}\in\mathcal{P}(X)$ are lifts of $g$ of $Y$; the resulting class does not depend on the choices of lifts.

%\begin{rem}Note that by Stone duality, the problem of surjectivity of the homomorphism $\mkst(X)\to\Aut_{\mathrm{ring}}(\mathcal{P}^\star(X))$ is a restatement of the Rudin-Shelah problem evoked in \S\ref{hire}, (\ref{rudinshelah}). This will play no role here.\end{rem}

%Now consider a group with a near action on $X$. This canonically defines an action of $G$

\begin{defn}Let $G$ be a group and consider a near $G$-set $X$. Let 
$\mathcal{P}_G^\star(X)$ be the Boolean subalgebra of fixed points of the action of $G$ on $\mathcal{P}^\ast(X)$. Let $\mathcal{E}_G^\star(X)$ be its Stone space (i.e., its spectrum); it is called the space of ends of the $G$-set $X$.\index{space of ends}\index{$\mathcal{P}_G^\star(X)$}\index{$\mathcal{E}_G^\star(X)$ (space of ends)}

Denote by $\mathcal{P}_{(G)}(X)$\index{$\mathcal{P}_{(G)}(X)$} the inverse image of $\mathcal{P}_G^\star(X)$ in $\mathcal{P}(X)$; its elements are the $G$-commensurated subsets of $X$.\index{commensurated subset}\index{$\mathcal{E}_{(G)}(X)$ (end compactification)}
Let $\mathcal{E}_{(G)}(X)$ be the Stone space of $\mathcal{P}_{(G)}(X)$; it is call the end $G$-compactification of $X$\index{end compactification}.
\end{defn}

%The following lemma is standard and left to the reader.
%\begin{lem}Let $X$ be a set, and let $I$ be an ideal of $\mathcal{P}(X)$ includiing \end{lem}

%By definition, $\mathcal{P}_{(G)}(X)$ is the set of $G$-commensurated subsets of $X$ and $\mathcal{P}_G^\star(X)$ is its quotient by the ideal of finite subsets.

The quotient map $\mathcal{P}_{(G)}(X)\to\mathcal{P}_G^\star(X)$ induces a continuous injection $\mathcal{E}_{G}^\star(X)\to\mathcal{E}_{(G)}(X)$; the complement of its image is discrete open and can be identified to $X$ by mapping $x\in X$ to the set of $Y\in \mathcal{P}_{(G)}(X)$ such that $x\in Y$, whence the term compactification.

By definition, $G$ acts trivially on $\mathcal{P}_G^\star(X)$ and hence its natural action on the space of ends is the trivial action. When $X$ is a $G$-set, the $G$-action on $X$ extends to the end compactification, the action on $\mathcal{E}_{G}^\star(X)=\mathcal{E}_{(G)}(X)\smallsetminus X$ being trivial. In general, $G$ does not act on $X$, but we can interpret this as a ``germ of action" of $G$ at infinity. 

The space of ends has a good functoriality behavior with respect to the near action. Namely, given $f\in\mkpst_G(X,Y)$, the induced map $f^*:\mathcal{P}^\star(Y)\to\mathcal{P}^\star(X)$ (see \S\ref{nepose}) maps $\mathcal{P}_{G}^\star(Y)\to\mathcal{P}_G^\star(X)$. Thus this defines a contravariant functor from $\mkpst_G$ to $\mathbf{BoA}$, and hence, by Stone duality, a covariant functor from $\mkpst_G$ to the $\mathbf{Sto}$. Similarly, $f\in\mkpm_{(G)}(X,Y)$ induces functorially a Boolean algebra homomorphism $\mathcal{P}_{(G)}(Y)\to\mathcal{P}_{(G)}(X)$ and $\mathcal{E}_{(G)}(X)\to\mathcal{E}_{(G)}(Y)$. In particular, the space of ends (as Stone space, modulo homeomorphism) is a near isomorphism invariant of a near $G$-set.

%For instance, in the Rudin-Shelah problem evoked in \S\ref{hire}, (\ref{rudinshelah}), 

%Then by definition, the set $\mathcal{P}_G^\star(X)$ of fixed points of the action of $G$ on $\mathcal{P}^\ast(X)$ is the set of (classes of) $G$-commensurated subsets of $X$. In particular, this is a Boolean subalgebra of $\mathcal{P}(X)$. 

% Note that it only depends on the image $\alpha(G)\subseteq \mkst(X)$.
%The definition of the space of ends is often written for a group action.

For group actions, this definition is classical in the transitive case (mostly in the language of coset spaces $G/H$). For the left action of a group on itself, it was introduced by Specker \cite{Sp}, taking advantage of Stone duality (which was discovered after Freudenthal's seminal paper \cite{Fr} introduced ends of groups). For coset spaces (i.e., for transitive actions, with a language emphasizing a base-point), it was introduced by Houghton \cite{Ho1}, but with an approach by hand instead of the powerfulness of Stone duality. Then the further developments (notably the influential \cite{Sco3} and subsequent papers) did not emphasize the space of ends, focussing on its cardinal (finite or $\infty$).

Once generalized to arbitrary actions, one interest of this (straightforward) generalization is to see that it only depends on the corresponding near action, and even only on the image $\alpha(G)$ of $G$ in $\mkst(X)$. Note that in general when one passes from a group action to a finite perturbation (i.e., another action defining the same near action), one can change orbits, and even change the number of infinite orbits (Example \ref{dinfty}); in particular, defining the space of ends under a transitivity assumption is inconvenient. (It is also inconvenient in the context of group actions, because restricting an action to a subgroup does not preserve transitivity.)

\begin{exe}\label{dinfty}
\begin{enumerate}
\item The space of ends $\mathcal{E}^\star_G(X)$ is empty if and only if $X$ is finite.
\item For an action of $\Z$ with finitely many orbits, the space of ends is finite, and its cardinal is twice the number of infinite orbits.
\item For an action of a trivial group on a set $X$, the end compactification is the Stone-\u Cech compactification\index{Stone-\u Cech compactification} of $X$, and the space of ends is its boundary. The Stone-\u Cech compactification of $X$ has a canonical bijection with the set of ultrafilters on $X$, and has cardinal $2^{2^{|X|}}$ when $X$ is infinite (see \cite[Theorem 7.6]{Je}).
\item For an action of a finitely generated group with only finite orbits (``sparse", see \ref{s_sna}), one can identify the space of ends to the Stone-\u Cech boundary of the set of orbits.
\item For $X$ infinite, the near action of $\mks(X)$ or $\mkst(X)$ on $X$, the end compactification is the 1-point compactification, and the space of ends is a singleton.
%\item For an action of the infinite dihedral group $D_\infty$ with finitely many orbits, the space of ends is finite, with cardinal $n_1+2n_2$, where $n_2$ is the number of free orbits and $n_1$ is the number of non-free infinite orbits. It is not hard to see that any two such actions with the same positive number of ends are finite perturbations of each other. FAUX
\end{enumerate}
\end{exe}

%Clearly, the set $\mathcal{P}^\star_G(X)$ of $G$-commensurated subsets of $X$ is a Boolean subalgebra of the power set: indeed, it consists of the fixed points by $, including the ideal of finite subsets. Let $\mathcal{M}^\star_G(X)$ the quotient Boolean algebra, and $\mathcal{E}_G(X)$ its spectrum; it is called the space of ends of $X$. This is a compact Hausdorff totally disconnected space.

\begin{prop}
Let $X$ be a near $G$-set and $Y$ a commensurated subset. Then the ``intersecting with $Y$" Boolean algebra homomorphism $\mathcal{P}_{(G)}(X)\to\mathcal{P}_{(G)}(Y)$ induces clopen inclusions $\mathcal{E}_{(G)}(Y)\to\mathcal{E}_{(G)}(X)$ and $\mathcal{E}_{G}^\star(Y)\to\mathcal{E}_{G}^\star(X)$.\qed
\end{prop}

%When $\mces_G(X)$ is finite, we can write $X=\bigsqcup_{i=1}^nX_i$ with $X_i$ $G$-commensurated and 1-ended and there is a canonical bijection between $\{1,\dots,n\}$ and $\mces_G(X)$. This extends to the following proposition:

\begin{prop}\label{isolend}
Let $G$ be a group and $X$ a near $G$-set. The map mapping a 1-ended commensurated subset of $X$ (modulo near equality) to its end is a bijection onto the set of isolated ends of the $G$-set $X$, that is, the set of isolated points in $\mces_G(X)$.
\end{prop}
\begin{proof}
By Stone duality, isolated points correspond to ideals of cardinal 2 in a Boolean algebra. In the case of $\mathcal{P}_G^\star(X)$, we see that indeed ideals of cardinal 2 are precisely 1-ended subsets (modulo near equality).
\end{proof}

Let $\alpha:G\to\mkst(X)$ be a near action. The normalizer of $\alpha(G)$ in $\mkst(X)$ preserves $\mces_G(X)$ and hence naturally acts on it. This action is trivial in restriction to $\alpha(G)$. Also, it restricts to a action of the centralizer $\mkst_G(X)$. 

%We retain the following:

\begin{prop}\label{endsubgroup}
Let $G,H$ be groups and $u:H\to G$ a group homomorphism (e.g., $u$ is the inclusion of a subgroup $H$). Let $X$ be a near $G$-set, and consider it as an $H$-set by composition. Then
\begin{enumerate}
\item\label{endi1} we have the inclusion $\mathcal{P}_{(G)}(X)\subseteq \mathcal{P}_{(H)}(X)$, inducing natural continuous surjective maps $\mce_{(H)}(X)\to\mce_{(G)}(X)$ and $\mces_H(X)\to\mces_G(X)$;
\item\label{endi2} if $u(H)$ is normal of finite index in $G$, the space of ends $\mces_G(X)$ is the quotient of $\mces_H(X)$ by the action of $G/u(H)$;
\item\label{endi3} if $u(H)=G$, then $\mathcal{P}_{(G)}(X)=\mathcal{P}_{(H)}(X)$, inducing identifications $\mce_{(H)}(X)=\mce_{(G)}(X)$ and $\mces_H(X)=\mces_G(X)$.
\end{enumerate}
\end{prop}
\begin{proof}
(\ref{endi1}) is immediate from Stone duality. For (\ref{endi2}), Stone duality implies that $\mces_G(X)$ is the quotient of $\mces_H(X)$ by the action of $G/H$ {\it in the category of Stone spaces}. Now since $G/H$ is finite, it is straightforward to check that the quotient of $\mces_H(X)$ by the action of $G/u(H)$, in the category of topological spaces, is a Stone space. The result follows. (\ref{endi3}) is immediate (and also follows from (\ref{endi2})).
%By Stone duality, the first statement reflects the inclusion $\mathcal{P}^\star_G(X)\subseteq \mathcal{P}^\star_H(X)$, and the second amounts to checking that $\mathcal{P}^\star_G(X)$ is the set of fixed points of the $G/H$-action on $\mathcal{P}^\star_H(X)$. The latter fact is immediate from the definition.
\end{proof}

\begin{prop}\label{fendind}
Let $G$ be a group and $H$ a subgroup of finite index. Let $X$ be a near $G$-set. Then $X$ is finitely-ended as near $G$-set if and only if $X$ is finitely-ended as near $H$-set. In the infinitely-ended case, they have the same number of ends.
\end{prop}
\begin{proof}
We have $\mathcal{P}^\star_G(X)\subseteq \mathcal{P}^\star_H(X)$. In particular, if $\mathcal{P}^\star_H(X)$ is finite, so is $\mathcal{P}^\star_G(X)$, and hence if $X$ has finitely many ends as near $H$-set, it does as near $G$-set.

For the converse, and in view of this first trivial implication, we can suppose that $H$ is normal in $G$. The result then follows from Proposition \ref{endsubgroup}, since the quotient of a infinite compact space by a finite group action is infinite. (In terms of Boolean algebras, this translates into the possibly less intuitive fact that for any finite group action on an infinite Boolean algebra, the subalgebra of fixed points is infinite.)
% Hence infiniteness of $\mathcal{P}^\star_H(X)$ implies that $\mathcal{P}^\star_G(X)$, proving the converse.
\end{proof}

\begin{defn}\label{mended}
Let $G$ be a group and $X$ a near $G$-set.

We say that $X$ is metrizably-ended\index{metrizably-ended} if its space of ends is metrizable, or equivalently if the Boolean algebra $\mathcal{P}_G^\ast(X)$ is countable.

Given a set $I$, a function $f:X\to I$ is called closely $G$-invariant\index{closely $G$-invariant} if for every $g\in G$, the set of $x$ such that $f(gx)=f(x)$ is cofinite (this is meaningful). We say that $X$ is tamely-ended\index{tamely-ended} if every near $G$-invariant function on $X$ has a finite image.
\end{defn}

\begin{rem}\label{rem_taend}
If a near $G$-set $X$ is not tamely-ended, then there exists a surjective near $G$-invariant function $f:X\to\N$ with infinite fibers. In particular, $\mathcal{P}^\star_G(X)$ has a subalgebra isomorphic to $2^\N$, and hence $\mathcal{E}^\star_G(X)$ has a continuous surjective map onto the Stone-\u Cech compactification of $\N$, and in particular has $\ge 2^{2^{\aleph_0}}$ ends, with equality if $X$ is countable.
%
%It also follows that for $G$ finitely generated, any near $G$-set of finite type is tamely-ended.
\end{rem}

\begin{rem}
Let $G$ be a countable locally finite group and $X$ an infinite near $G$-set. It follows from Theorem \ref{coulofi}(\ref{coulofi1}) that $X$ is not tamely-ended. When $X=G$ with left action, this is basically folklore and follows from the proof of \cite[Theorem 4]{ScS}, see the discussion in \cite{CoEn}.
\end{rem}

For the space of ends of near actions of finitely generated groups, see Proposition \ref{finite_type_ends} and Remark \ref{rem_ends_fg}.

\begin{rem}
I do not know if there exists a countable group $G$ and countable near $G$-set that is tamely-ended but not metrizably-ended. I do not know even in the case of $G$ itself with left-action. It is observed in \cite{CoEn} that a free group of infinite countable rank is non-tamely-ended, but there exist infinitely generated countable metrizably-ended groups with infinitely many ends.

% for examples of infinitely generated countable groups both metrizably-ended and non-tamely-ended.

Nevertheless, it holds (in ZFC) that for every countable group $G$ and countable near $G$-set $X$, the Boolean algebra $\mathcal{P}^\star_G(X)$ is either countable or has cardinal $2^{\aleph_0}$. Indeed, let us check that the Boolean algebra $\mathcal{P}_{(G)}(X)$ of $G$-commensurated subsets is a countable intersection of a countable union of closed subsets (in descriptive set theory this is called a ``$\Pi^0_3$-subset", see \cite[\S 11]{Je}) of the compact metrizable space $\mathcal{P}(X)$. Indeed, in the case of an action, it can be written as
\[\mathcal{P}_{(G)}(X)=\bigcap_{g\in G}\bigcup_{\;\mathrm{finite}\,F\subseteq X}K_{A,F},\quad K_{A,F}=\{A\in\mathcal{P}(X):A\smallsetminus gA\subseteq F\},\]
using that the action of $G$ on $\mathcal{P}(X)$ is continuous to infer that $K_{A,F}$ is closed.
For an arbitrary balanced near action $\alpha$, this works by considering the inverse image of $\alpha(G)$ in $\mks(X)$. For a non-balanced near action, this works by first completing it as a balanced near action. Finally, we use the fact that Borel subsets of Polish spaces are either countable or have continuum cardinal \cite[Corollary 11.20]{Je}.
\end{rem}

\begin{rem}[End compactifications and coarse structures]\index{coarse structure}
A coarse structure \cite[\S 2]{Roe} on a set $X$ is a subset $\mathcal{E}$ of the set of subsets of $X$ satisfying the conditions:
\begin{itemize}
\item the diagonal $D_X\subseteq X\times X$ belongs to $\mathcal{E}$;
\item $E_1\subseteq E_2$ and $E_2\in\mathcal{E}$ implies $E_1\in\mathcal{E}$;
\item $E\in\mathcal{E}$ implies $E^{-1}\in\mathcal{E}$ where $E^{-1}=\{(y,x):(x,y)\in E\}$;
\item $E_1,E_2\in\mathcal{E}$ implies $E_1E_2\in\mathcal{E}$.
\end{itemize}
It is proper if in addition the only subsets $Y\subseteq X$ such that $Y^2\in\mathcal{E}$ are the finite subsets. It is coarsely transitive if $\bigcup\mathcal{E}=X^2$.

If $\mathcal{E}$ is a coarse structure, denote by $\mathcal{E}^\sharp$ the set of $E\cup S$ where $E\in\mathcal{E}$ and $S$ ranges over finite subsets of $X^2$. This is the smallest coarsely transitive coarse structure containing $\mathcal{E}$. If $\mathcal{E}$ is proper, then $\mathcal{E}^\sharp$ is also proper.

Let $X$ be a set and $\pi$ be the projection $\mkep(X)\to\mkst(X)$ (see \S\ref{mkpetc} for the notation).
Let $G$ be a group with a near action $\alpha$ on $X$. Define a coarse structure $\mathcal{E}_\alpha$ on $X$ by requiring $E\in\mathcal{E}$ if there exists a finite subset $F$ of $\pi^{-1}(\alpha(G))$ such that $E\subseteq \bigcup F$. Here elements of $\pi^{-1}(\alpha(G))$ are viewed as subsets of $X^2$. This is a coarsely transitive, proper coarse structure.

If $X$ is a $G$-set defined by $\beta:G\to\mks(X)$, the coarse structure $\mathcal{E}_\beta$ induced by the $G$-action defined in \cite[Example 2.13]{Roe} is defined in the same way, but with $F\subseteq \beta(G)$ instead of $F\subseteq \pi^{-1}(\alpha(G))$.

Let $\beta^\star$ be the induced near action. Then $\mathcal{E}_{\beta^\star}=\mathcal{E}_\beta^\sharp$. In particular, $\mathcal{E}_{\beta^\star}=\mathcal{E}_\beta$ coincide if and only if $\beta$ is transitive. Note that $\mathcal{E}_\beta^\sharp$ is the coarse structure induced by the action of the subgroup of $\mks(X)$ generated by $\beta(G)$ and the group of finitely supported permutations of $X$.

The end compactification makes sense for a set $X$ endowed with a proper coarse structure $\mathcal{E}$. Namely, say that a subset $M\subseteq X$ is $\mathcal{E}$-commensurated if for every $E\in\mathcal{E}$, the subset $(M\times M^c)\cap E$ is finite. These form an ideal $\mathcal{P}_{(\mathcal{E})}(X)$ of the power set $\mathcal{P}(X)$. Note that $\mathcal{P}_{(\mathcal{E})}(X)=\mathcal{P}_{(\mathcal{E}^\sharp)}(X)$. Taking the Stone dual yields the end compactification of $(X,\mathcal{E})$.

Then it is straightforward that for a near action $\alpha$, the $\alpha(G)$-commens\-urated subsets coincides with the $\mathcal{E}_\alpha$-commens\-urated subsets.
\end{rem}

%For more on space of ends, see \S\ref{tameend}.

% We call this ``widely indecomposable".

% Then $G$ acts on the space of ends of $X$ as near $H$-set, and this action factors through $G/H$. 

%\begin{prop}
%Let $G$ be an abelian group. Let $X$ be a near $G$-set. If we let $i(X)$ be the nonnegative generator of the image of the index character of $X$, then $\lst(X)$ divides $i(X)$. In particular, if $X$ has nonzero index character, then $\lst(X)>0$.
%\end{prop}
%\begin{proof}
%PAS CLAIR OK $G$ abelien, mais ailleurs!
%\end{proof}

%\begin{prop}\label{lim1e}Let $G$ be a group in which every finite subset is included in a 1-ended subgroup. Then the $G$-set $G$ (for left translation) is 1-ended.\end{prop}
%\begin{proof}This is well-known and easy. Indeed, fix a 1-ended subgroup $\Lambda_0$. Then any commensurated subset has to be a finite perturbation of a $\Lambda_0$-invariant subset $Y$; if $Y$ is neither empty or $G$, it should contain some right coset $\Lambda_0g$ and have empty intersection with some other right coset $\Lambda_0h$. Now considering a 1-ended subgroup containing $\Lambda_0\cup\{g,h\}$, we get a contradiction.\end{proof}

\subsection{Realizability degree and lest}\label{lest}

We introduce here some invariants relevant to the classification of near actions modulo balanced near isomorphism, notably in comparison to the classification modulo near isomorphism.

Let $G$ be a group. Let $X$ be a near $G$-set. For $n\in\Z$, by $X+n$ we mean the disjoint union of $X$ and $n$ points if $n\ge 0$, and $X$ minus $-n$ points if $n\le 0$. This is a well-defined balanced near isomorphism class of near $G$-sets (unless $X$ is finite and $n<-|X|$; in this case we say that $X+n$ is not realizable; it can still be thought of as a near action on a set with a finite negative number of elements). Note that if $X+n$ is realizable, then $X+n+1$ is realizable.

\begin{defn}\index{realizability degree}
We define the realizability degree of $X$ as the supremum $\rea_G(X)\in\Z\cup\{\pm\infty\}$ of the set of $n\in\Z$ such that $X-n$ $(=X+(-n))$ is realizable. (Here $\sup(\emptyset)=-\infty$.)
\end{defn}

Thus:
\begin{itemize}
\item for $n\in \Z$, $\rea_G(X)=n$ if and only if $X-n$ is realizable but not $X-n-1$;
\item $\rea_G(X)=+\infty$ if and only if $X-n$ is realizable for all $n\in\Z$;
\item $\rea_G(X)>-\infty$ if and only if $X$ is finitely stably realizable;
\item $\rea_G(X)=|X|$ for $X$ finite.
\end{itemize}

The following proposition is immediate.

\begin{prop}
The invariant $\rea_G(X)$ is a balanced near isomorphism invariant of the near $G$-set $X$. Whether it is $-\infty$, $+\infty$, or an element of $\Z$, is a near isomorphism invariant. If $H\to G$ is a group homomorphism, then $\rea_H(X)\ge\rea_G(X)$.\qed
\end{prop}

Let us pass the notion of lest. The following lemma is immediate:

\begin{lem}
The set of $n\in\Z$ such that $X$ is balanceably near isomorphic to $X+n$ is a subgroup of $\Z$.\qed
\end{lem}

\begin{defn}\index{lest}
The unique non-negative generator of this subgroup is called the lest of the near $G$-set $X$, and denoted by $\lst_G(X)$, or $\lst(X)$ when $G$ is understood. (``Lest" translates into English as ``ballast".)
\end{defn}

\begin{exe}Let $G$ be a group and $X$ a near $G$-set.
\begin{enumerate}
\item If $X$ is finite, we have $\lst(X)=0$.
\item If $X$ has the form $Y\sqcup Z$ with $Z$ an infinite trivial near $G$-set, then $\lst(X)=1$.
\item Let $F$ be a finite group and $X$ an infinite $F$-set. Then $\lst_F(X)$ divides $|F|$, with equality if $X$ is infinite and $F$-free, by a simple argument. 
%Then $\lst_F(X)=|F|$, by a simple argument. 
\end{enumerate}
Thus for a given group $G$, the values 0 and 1 for the lest are always achieved by some $G$-set, and all possible values of the lest are achieved by some group $G$ and $G$-set.
\end{exe}

%Clearly, when $H\subseteq G$, the number $\lst_H(X)$ divides $\lst_G(X)$.

\begin{prop}\label{basiclest}
The lest satisfies the following basic properties:
\begin{enumerate}
\item\label{bl1} $\lst_G(X)$ is a near isomorphism invariant of the near $G$-set $X$;
\item\label{bl2} let $H$ be a group with a homomorphism $H\to G$, and let $X$ is a near $G$-set and hence a near $H$-set by composition; then $\lst_H(X)$ divides $\lst_G(X)$; if $H\to G$ is surjective, this is an equality: $\lst_H(X)=\lst_G(X)$;
%\item\label{bl3} if $L$ is a group having $G$ as quotient, then $\lst_L(X)$ divides $\lst_G(X)$.
\item\label{bl4} $\lst_G(X)=\iota_{\mkst_G(X)}(X)$: in words, $\lst_G(X)$ is the index number of the near action on $X$ of the group $\mkst_G(X)$ of near automorphisms of the near $G$-set $X$.
\end{enumerate}
\end{prop}
\begin{proof}
(\ref{bl4}) is essentially a restatement. If $X$ is balanceably near isomorphic to $X-n$, choose a balanced  isomorphism of near $G$-sets between $X$ and $X-n$: this is an element of $\mkst_G(X)$ with index equal to $n$. Hence $\iota_{\mkst_G(X)}(X)$ divides $\lst_G(X)$. Conversely, let $f$ be an element of $\mkst_G(X)$ with index equal to $\iota_{\mkst_G(X)}(X)$. Then $f$ has a representative that is an injective map onto the complement of a subset with $n$ elements, and this yields a balanced isomorphism of near $G$-sets between $X$ and $X-n$. Hence $\lst_G(X)$ divides $\iota_{\mkst_G(X)}(X)$.

The divisibility part of (\ref{bl2}) is immediate as well, and the equality follows from (\ref{bl4}) since $\mkst_G(X)$ only depends on the image of $G\to\mkst(X)$.

For (\ref{bl1}), it is obvious that the lest is invariant under balanced near isomorphism. Since it does not change when passing from $X$ to $X+1$ (for every $X$), we deduce that it is a near isomorphism invariant.
\end{proof}

%The number $\lst(X)$ is an invariant of the near isomorphy class of the near $G$-set $X$, since it is an  .

\begin{prop}
Let $X$ be a near $G$-set with $\lst(X)>0$. Then its realizability degree $\rea_G(X)$ belongs to $\{-\infty,+\infty\}$. Thus, if $X$ is finitely stably realizable, then it is realizable.
\end{prop}
\begin{proof}
Indeed, the set of $n$ such that $X+n$ is realizable is then $\lst(X)$-periodic, and hence either empty or equal to $\Z$. The second statement immediately follows.
%Since $X$ is stably realizable, $X+n$ is realizable for some $n\ge 0$. Since realizability is an  isomorphism invariant, it follows $X+(n+m\lst(X))$ is realizable for every $m\in\Z$. Since $\lst(X)>0$, we deduce that there exists $k\le 0$ such that $X+k$ is realizable, and hence $X$ is realizable.
\end{proof}

\begin{prop}\label{lestunion}
Let $X,Y$ be near $G$-sets. Then $\lst(X\sqcup Y)$ divides $\gcd(\lst(X),$ $\lst(Y))$. In the following two cases, there is equality:
\begin{itemize}
\item if $Y$ is finitely-ended.
\item if $Y$ is a near characteristic subset\index{nearly characteristic subset}, in the sense that every near automorphism of $X$ commensurates $Y$.
\end{itemize}
\end{prop}
\begin{proof}
The divisibility assertion is an immediate game using a B\'ezout relation.

Suppose that $Y$ is finitely-ended. Arguing by induction on the number of ends of $Y$, we can reduce to the case when $Y$ is 1-ended. Then the case when $X$ is finite is trivial, so let us assume that $X$ is infinite. Consider an a bijection $f$ representing a balanced near $G$-isomorphism between $X\sqcup Y$ and $X\sqcup Y+n$, with $n\ge 0$. Since $Y$ is 1-ended, $f(Y)$ is either near disjoint or near included in $Y$. Suppose that it is near disjoint. We can remove finitely many points to $Y$ and add them to $X$, to assume that $Y\cap f(Y)$ is empty. Then we can consider the permutation $g$ of $X\sqcup Y+n$ acting as $f$ on $Y$, as $f^{-1}$ on $f(Y)$, and identity elsewhere. Then $g\circ f$ is an balanced near $G$-isomorphism and maps $Y$ onto $Y$. In other words, we can suppose that $f(Y)$ has finite symmetric difference with $Y$. If $Y\smallsetminus f(Y)$ and $f(Y)\smallsetminus Y$ are both nonempty, we can post-compose with a transposition to reduce the size of $Y\bigtriangleup f(Y)$. Thus we can suppose that there is an inclusion between $Y$ and $f(Y)$. 

If $f(Y)\subseteq Y$, let $k$ be the cardinal of $Y\smallsetminus f(Y)$. So $k$ is a multiple of $\lst(Y)$. Then $f$ induces a balanced near $G$-isomorphism between $X$ and $X\sqcup (k+n)$. So $\lst(X)$ divides $k+n$. Hence $\gcd(\lst(X),\lst(Y))$ divides $(k+n)-k=n$. 

If $f(Y)\supset Y$, let $k$ be the cardinal of $f(Y)\smallsetminus Y$. Then $f$ induces a balanced near $G$-isomorphism between $X+k$ and $X+n$, and we conclude similarly. 

When $Y$ is near characteristic, the previous proof directly applies.
\end{proof}

See Remark \ref{noteqlst} for an example where we do not have equality in Proposition \ref{lestunion}. 
%I do not know if equality always holds for near $G$-sets of finite type.

\subsection{Near actions of finitely generated groups}\label{nfg}

Consider now a group $\Gamma$ generated by a finite subset $S$. Let $f:\Gamma\to\mkst(X)$ be a near action on a set $X$. Given any lift $u(s)\in\mke(X)$ of $f(s)$ (see \S\ref{mkpetc} for the notation), for each $s\in S$, we can make a graph whose set of vertices is $X$, and with an directed edge from $x$ to $u(s)x$, for each $x\in X$ and $s\in S$, labeled by $s$. We call this a near Schreier graph\index{near Schreier graph}. If we change the choice of lift, then the resulting graph coincides up to finitely many edges.

If for $s,s'\in S$ such that $ss'\in S$, we define $S'=S\cup\{ss'\}$ and define $u(ss')=u(s)u(s')$. Then the resulting near Schreier graph is obtained from the original by adding possibly infinitely many edges, but, with finitely many exceptions, these new edges join two points at distance at most 2. Since we can pass from any generating subset to any other by such steps (or adding inverses), forward or backwards, we obtain that many features of the Schreier graph are independent of the choice of $S$ and lift $u$.

%First, the finiteness of the number of components (or its cardinal, when infinite), is invariant. 

%Second, the space of ends of the Schreier graph is invariant, and coincides with the space of ends as previously defined, without restriction to finitely generated groups.

Now let us concentrate on the most interesting case, namely the case when there are finitely many components:

\begin{defn}
We say that a near action of a finitely generated group is of finite type\index{finite type (near action)} if some/any near Schreier graph has finitely many connected components.
\end{defn}

\begin{rem}Let $\Gamma$ be a finitely generated group.
\begin{enumerate}
\item A $\Gamma$-action induces a near action of finite type if and only if it has finitely many orbits. 

\item Every commensurated subset of a near $\Gamma$-set of finite type is also of finite type.
\end{enumerate}
\end{rem}

\begin{rem}In Theorem \ref{loficom}(\ref{lofi2}), we will see that there are (infinitely generated) countable groups satisfying the following: for every action on any infinite set, there exists another action defining the same near action, that has infinitely many orbits (and only finite orbits, if the set is countable). Hence the notion of action of finite type cannot be naively generalized to all countable groups. 
\end{rem}

Given a graph with finitely many connected components, adding finitely many edges yields a connected graph; call this a key fob completion\index{key fob completion} of the initial graph.

Given two choices of near Schreier graphs and two choices of key fob completions, the identity map induces a quasi-isometry between the resulting graphs. This is a consequence of the following lemma.

\begin{lem}\label{kfqi}
Let $X$ be a set. Consider two connected graph structures on $X$, such that we can pass from one to another by a finite number of the following operations:
\begin{itemize}
\item adding finitely many edges (possibly between different components);
\item removing finitely many edges (possibly breaking some components);
\item adding, for some $k$, edges between some family of pairs of vertices, each pair being at distance $\le k$ in the initial distance.
\item removing, for some $k$, edges between some family of pairs of vertices, each pair being at distance $\le k$ in the final distance.
\end{itemize}
% (in either direction): adding one edge, adding, for some $k$, edges between points at distance $\le k$ for the graph distance.
 Then the identity map is a quasi-isometry.
\end{lem}
\begin{proof}
First assume that the graph remains connected at any step. Then this is clear, since it is clear at each step. In general, consider $n$ steps, where there can be steps where connectivity fails. Choose a vertex $x_0$, and choose for each step a finite subset containing a vertex in each component, and consider the union $F$ of all these vertices ($F$ is finite). Then first add extra edges between $x_0$ and each element of $F$, and perform all the operations without touching these extra-edges (we allow multiple edges in these intermediate steps). After the last step, remove the extra-edges. Then this new process (in $n+2$ steps) preserves connectivity at each step, and hence we are done by the first case.
\end{proof}

So, on a near Schreier graph with finitely many components, all possible quasi-isometry invariants of connected graphs with bounded valency are relevant: growth, amenability, etc.

Let us mention the following fact, whose proof is easy but a little tricky.

\begin{prop}\label{finitetypeind}
Let $H,G$ be a finitely generated groups with a homomorphism $u:H\to G$. Let $X$ be a near $G$-set. If $X$ is of finite type as near $H$-set, then it is of finite type as near $G$-set. If $u(H)$ has finite index in $G$, the converse holds.
% and $H$ a subgroup of finite index. Let $X$ be a near $G$-set. Then $X$ is of finite type as near $G$-set if and only if $X$ is of finite type as near $H$-set.
\end{prop}
\begin{proof}
The first assertion is clear. The converse is clear when $u$ is surjective; hence, for the converse, we can now suppose that $H$ is a subgroup of finite index of $G$.
Let $F$ be a finitely generated free group with a quotient homomorphism $F\to G$, and let $F'\subseteq F$ be the inverse image of $H$; it has finite index in $F$. Realize the near $G$-action as an $F$-action; this has finitely many orbits, hence also has finitely many orbits as an $F'$-action; in turn, this implies that the near $H$-action is of finite type.
\end{proof}

%Clearly if $X$ is of finite as near $H$-set, then it is of finite type as near $G$-set. Conversely, suppose that $X$ is of finite type as near $G$-set.

\begin{rem}Beware that the quasi-isometry type of the near Schreier graph can change when restricting to a finite index subgroup. This phenomenon already occurs for actions: the Schreier graph of the action of the infinite dihedral group $D_\infty=\mathrm{Isom}(\Z)$ on $\Z$ is 1-ended (for a suitable choice, it consists of $\Z$ with $n$ connected to $n\pm 1$ and $-n$), while for the subgroup $\Z$ of index 2, it is 2-ended.\end{rem}

Here is a characterization of near $\Gamma$-sets of finite type in terms of the space of ends (\S\ref{s_soe}).

\begin{prop}\label{finite_type_ends}
Let $\Gamma$ be a finitely generated group and $X$ a near $\Gamma$-set. The following are equivalent:
\begin{enumerate}
\item $X$ is of finite type;
\item $X$ is metrizably-ended;
\item $X$ is tamely-ended.
\end{enumerate}
\end{prop}
\begin{proof}
Metrizably-ended implies tamely-ended in general, as observed in Remark \ref{rem_taend}. Fix a near Schreier graph.

Suppose that $X$ is not of finite type. Let $I$ be the set of components, so $I$ is infinite by hypothesis. Let $f:X\to I$ map any element of $X$ to its component. Then $f$ is closely $\Gamma$-invariant, and surjective. Hence $X$ is not tamely-ended.

If $X$ is of finite type, then $X$ is countable. Its commensurated subsets are the subsets of the near Schreier graph with finite boundary. Since there are finitely many components and the valency is finite, for a given finite subset there are only finitely many subsets with it as boundary. Hence the set of commensurated subsets is countable. So $X$ is metrizably-ended.
%Let $F$ be a finitely generated free group with a surjective homomorphism onto $\Gamma$. Since all properties are equivalent when $X$ is viewed as $\Gamma$-set or as $
\end{proof}

\begin{rem}\label{rem_ends_fg}
Let a finitely generated group $G$ near act on a set $X$ and fix a near Schreier graph $\mathcal{G}$ for the $G$-action on $X$. Let $I\subseteq X$ be the set of isolated ends of the near $G$-set $X$, and $C$ the set of components of $\mathcal{G}$. The reader can check the following assertions.

\begin{enumerate}
%\item if the near action is of finite type (i.e., if $C$ is finite) then $X$ is metrizably-ended, and conversely otherwise $X$ is non-tamely-ended.
\item suppose that all components of $\mathcal{G}$ are finite. Then $\mathcal{P}^\star_G(X)$ is canonically isomorphic to $\mathcal{P}^\ast(C)$, and $\mathcal{E}^\star_G(X)$ is canonically homeomorphic to the Stone-\u Cech boundary of $C$ (and $I=\emptyset$); 
\item suppose that all components are finitely-ended, and that all but finitely many are 1-ended. Then $\mathcal{P}^\star_G(X)$ is canonically isomorphic to $\mathcal{P}^\star(I)$, and $\mathcal{E}^\star_G(X)$ is canonically homeomorphic to the Stone-\u Cech boundary of $I$.
%\item Proposition \ref{finite_type_ends}
\end{enumerate}
\end{rem}

Let us finally mention the following easy but important lemma, which cannot be formulated in terms of near actions.

\begin{lem}\label{fgcomen}
Let $G$ be a finitely generated group and $X$ a $G$-set. Let $Y$ be a commensurated subset. Then for all but finitely many orbits $\Omega\subseteq X$, we have $\Omega\cap X\in\{\emptyset,\Omega\}$. 
\end{lem}
\begin{proof}
Indeed, in the Schreier graph with respect to some finite generating subset, $Y$ has a finite boundary. Hence all but finitely many orbits do not meet the boundary, and thus are disjoint to $Y$ or included in $Y$.
\end{proof}

I am not sure of an initial reference for Lemma \ref{fgcomen}. An important particular case (the left action of $G$ on a larger group) appears as \cite[Lemma 2.3]{Coh}.

\subsection{Realizability of near actions of finitely generated groups}\label{s_rnfg}

\begin{defn}\label{deftame}
A near action of a finitely generated group is neat\index{neat (near action)} if it is near isomorphic to the disjoint union of a realizable action and an action of finite type.
\end{defn}

\begin{prop}\label{streafg}
For a finitely generated group $\Gamma$,
\begin{enumerate}
\item\label{streafg1} every stably realizable near action is finitely stably realizable;
\item\label{streafg2} a near $\Gamma$-set $X$ of finite type is completable if and only if it is finitely completable, in the sense that there exists a near $G$-set of finite type $Y$ such that $X\sqcup Y$ is realizable.
\item\label{codirefi} every completable near action is neat.
%near isomorphic to the disjoint union of a realizable action and an action of finite type.
\end{enumerate}
\end{prop}
\begin{proof}
Let $X\sqcup Y$ be realizable, with trivial near action on $Y$. Then for the realized action, each generator of $\Gamma$ fixes all but finitely many elements of $Y$. Hence the realized action fixes a cofinite subset of $Y$, which we can remove.

Suppose that $X$ is completable and of finite type. Let $X\sqcup Y$ be realizable. For the realized action, $X$ is included in a finite union $X\sqcup Y'$ of orbits. Then $X\sqcup Y'$ is of finite type.

Let $X$ be a completable $\Gamma$-set. Let $X\sqcup Y$ be a realizable completion, and fix a realization. Since $\Gamma$ is finitely generated, only finitely many of the $\Gamma$-orbits meet both $X$ and $Y$. Hence, when write $X=X_1\sqcup X_2$ where $X_1$ is the union of $\Gamma$-orbits contained in $X$, we have $X_1$ realizable and $X_2$ of finite type.
\end{proof}

%\begin{prop}\label{comple_fin}
%Let $\Gamma$ be a finitely generated group, and let $X$ be a near action of $\Gamma$, of finite type. Then $X$ is completable if and only if there exists a near action of finite type $Y$ such that $X\sqcup Y$ is realizable.
%\end{prop}
%\begin{proof}
%To prove the non-trivial implication, suppose that $X$ is completable: $X\sqcup Y$ is realizable. Write the Schreier graph of a realization. Then the subgraph $X$ is a finite perturbation of a near Schreier graph; in particular it has finitely many components, and hence is included in a finite union of components of the whole Schreier graph. This precisely means that there exists $Y'\subseteq Y$, commensurated, such that $X\cup Y'$ is realizable and of finite type. Hence $Y'$ has finite type.
%\end{proof}

For finitely presented groups, realizability problems can essentially be reduced to the case of finite type, thanks to the following theorem, which resembles Proposition \ref{streafg}(\ref{codirefi}), without the completability assumption but at the cost of finite presentability.

\begin{thm}\label{nearactionfp}
For a finitely presented group, every near action is neat.
That is, every near action on a set $X$ is near isomorphic to the disjoint union of a realizable action and an action of finite type.
\end{thm}
\begin{proof}
Denote by $\Gamma$ the group. We can choose a finite subset $S$ such that $\Gamma$ has, as a monoid, a presentation with generating subset $S$ and only relators of the form $s_1s_2=s_3$.

%$s_1s_2=1$ or $s_1s_2=s_3$.

Choosing a lift $\tilde{s}\in\mkem(X)$ for each $s\in S$, we can consider the cofinite subset $Z$ of elements $x\in X$ such the following two conditions hold: for every relator $s_1s_2=s_3$, we have $\tilde{s_1}\tilde{s_2}x= \tilde{s_3}x$.

%\begin{itemize}
%\item $x$ has exactly one preimage by $\tilde{s}$ for every $s\in S$;
%\item for every relator $s_1s_2=1$, we have $\tilde{s_1}\tilde{s_2}x= x$:
%\item for every relator $s_1s_2=s_3$, we have $\tilde{s_1}\tilde{s_2}x= \tilde{s_3}x$.
%\end{itemize}
Let $Y$ be the union of all components of the corresponding near Schreier graph that are included in $Z$. Then $\tilde{s}$ is a self-map of $Y$, on which, when $s$ ranges over $S$ these satisfy all relators. Hence this defines a homomorphism $\Gamma\to Y^Y$, that is, a $\Gamma$-action on $Y$.
Since $Z$ is cofinite, the complement of $Y$ consists of finitely many components, and hence is a commensurated subset of finite type.
\end{proof}

%Proposition \ref{notfp} yields many infinitely presented finitely generated groups for which the conclusion of Theorem \ref{nearactionfp} fails. However, t the conclusion of Theorem \ref{nearactionfp} still holds. 

We see in \S\ref{s_sna} that the finitely presentability assumption is crucial in Theorem \ref{nearactionfp}. However, it is not optimal, in the sense that, as we see below, there are a few infinitely presented finitely generated groups for which every near action is neat.
Given a group $G$ and a normal subgroup $N$, say that $(G,N)$ is finitarily rigid\index{finitarily rigid} if for any action of $G$ such that $N$ acts by finitely supported permutations, $N$ acts trivially. When $N=G$, simply say that $N$ is finitarily rigid.

\begin{cor}\label{nofpfp}
Let $G$ be a finitely presented group and $N$ a normal subgroup; suppose that $(G,N)$ is finitarily rigid. Then every near action of $G/N$ is neat. 
\end{cor}
\begin{proof}
Given a near $G/N$-action, use Theorem \ref{nearactionfp} to split it, as a near $G$-action, as $X\sqcup Y$ with $X$ realizable and $Y$ of finite type. Then both are near $G/N$-actions, and $Y$ is of finite type. We claim that  $X$ is realizable as a $G/N$-action. Realize it as a $G$-action: then $N$ acts by finitely supported permutations. Since $(G,N)$ is finitarily rigid, $N$ acts trivially and hence this is a $G/N$-action.
\end{proof}

%There are a few infinitely presented groups for which the conclusion of Theorem \ref{nearactionfp} still holds. Given a group $G$ and a normal subgroup $N$, say that $(G,N)$ is finitarily rigid if for any action of $G$ such that $N$ acts by finitely supported permutations, $N$ acts trivially. When $N=G$, simply say that $N$ is finitarily rigid. 
\begin{exe}
If $N$ has no nontrivial locally finite quotient, then $N$ is finitarily rigid. Also, it is not hard to check that the quasi-cyclic group $C_{p^\infty}=\Z[1/p]/\Z$ ($p$ prime) is finitarily rigid, and hence so is the group $\Z[1/p]$. See \cite[\S 8.3]{DM} for various further result of finitary rigidity.

This indeed yields cases where Corollary \ref{nofpfp} applies to infinitely presented groups $G/N$. This occurs if $N$ is isomorphic to $C_{p^\infty}$ and is a normal subgroup in finitely presented group $G$; Abels \cite{Ab79} constructed now well-known examples of finitely presented groups of this form (with $N$ central).
\end{exe}
%If $(G,N)$ is finitarily rigid and $G$ is finitely presented, then $G/N$ satisfies the conclusion of Theorem \ref{nearactionfp}. 

%Indeed, given a near $G/N$-action, split it, as a near $G$-action, as $X\sqcup Y$ with $X$ realizable and $Y$ of finite type. Then both are near $G/N$-actions, and $Y$ is of finite type. The point is to check that $X$ is realizable as a $G/N$-action. Realize it as a $G$-action: then $N$ acts by finitely supported permutations. Since $(G,N)$ is finitarily rigid, $N$ acts trivially and hence this is a $G/N$-action.

%Then $\tilde{s}$ induces a permutation of $Y$ for every $S$. Since these satisfy all relators, the subgroup they generate realizes $\Gamma$ on $Y$. 

%We will now see that many infinitely presented, finitely generated groups fail to satisfy the conclusion of Theorem \ref{nearactionfp}, i.e., admit non-neat near actions. Let us start with an observation.APRES 4F5:

We next see in \S\ref{s_sna} that many infinitely presented, finitely generated groups fail to satisfy the conclusion of Theorem \ref{nearactionfp}, i.e., admit non-tame near actions.

\subsection{Sparse near actions}\label{s_sna}

\begin{defn}
We say that a near action of a group $G$ on a set $X$ is sparse\index{sparse} if for every finitely generated subgroup $\Gamma$ of $G$, some/every near Schreier graph of the $\Gamma$-set $X$ has only finite components.
\end{defn}

\begin{defn}
A near $G$-set $X$ is sparsely-ended\index{sparsely-ended} if there exists a set $I$ and a finite-to-one function $f:G\to I$ such that for every $g\in G$, for all but finitely many $x\in X$, we have $f(gx)=f(x)$. That is to say, $\mkpst_G(X,I)$ is nonempty (see Definition \ref{defequi}), for some set $I$ viewed a trivial $G$-set.
\end{defn}
%near $G$-equivariant $X\to I$. That is to say, for every $g\in G$, for all but finitely many $x\in X$, we have $f(gx)=f(x)$.

Note that if a near action on an infinite set is sparsely-ended, then it is not tamely-ended (Definition \ref{mended}). 
Here are some basic properties:

\begin{prop}
Let $X$ be a sparse (respectively sparsely-ended) near $G$-set. Then
\begin{enumerate}
\item\label{spa_restr} for every group $H$ and homomorphism $H\to G$, the $H$-set $X$ is sparse (resp.\ sparsely-ended);
\item for every normal subgroup $N$ included in the kernel of the near action, the $G/N$-set $X$ is sparse (resp.\ sparsely-ended);
\item being sparse (resp.\ sparsely-ended) is a near isomorphism invariant of near $G$-sets;
\item\label{spa_cosu} being sparse (resp.\ sparsely-ended) passes to $G$-commensurated subsets;
\item\label{spa_bal} $X$ is balanced;
\item\label{spa_spae} sparsely-ended implies sparse;
\item\label{spae_spa} the converse (sparse implies sparsely-ended) holds if both $G$ and $X$ are countable.
\end{enumerate}
\end{prop}
\begin{proof}
The first three facts are immediate. 

(\ref{spa_spae}): Suppose that $X$ is sparsely-ended. Let $H$ be a finitely generated subgroup of $X$, and $Y$ an $H$-commensurated subset of $X$, that is of finite type as $H$-set. Then $Y$ is sparsely-ended as well (by \ref{spa_cosu}); since $Y$ is tamely-ended (Remark \ref{rem_taend}), it follows that $Y$ is finite. This shows that all components of near Schreier graphs of $H$ are finite. Hence $X$ is a sparse near $G$-set.

(\ref{spa_bal}): in view of (\ref{spa_spae}), it is enough to show that sparse implies balanced. By (\ref{spa_restr}), we can suppose that $G=\Z$. By Theorem \ref{nearactionfp}, one can write $X=Y\sqcup Z$ with $Y$ of finite type and $Z$ realizable. Then $Y$ is a sparse near $\Z$-set of finite type, hence is finite, hence is realizable. So $X$ is a realizable $\Z$-set and hence its index is zero.

(\ref{spae_spa}): let $F$ be a countable free group with a surjective homomorphism onto $G$. Let $X$ be a countable sparse near $G$-set. By (\ref{spa_restr}), it is a sparse near $F$-set. By (\ref{spa_bal}) and since $F$ is free, it is realizable as a $F$-set; we fix such a realization. The near action being sparse, all finitely generated subgroups of $F$ act on $X$ with only finite orbits. Therefore, we can write $X$ as an ascending union of finite subsets $X_n$ and $F$ as an ascending union of subgroups $F_n$ such that $X_n$ is $G_n$-invariant for all $n$. For $x\in X$, write $f(x)=\min\{n:x\in X_n\}$. Then for all $g\in G_n$ and $x\notin X_n$, since both $X_{f(x)}$ and $X_{f(x)-1}$ are $g$-invariant, we obtain $f(gx)=f(x)$. This shows that $f$ is closely $G$-invariant. Hence $X$ is sparsely-ended. 
\end{proof}

Since obviously every near action of a locally finite group is sparse, this yields:
\begin{cor}
Every near action of a countable locally finite group on a countable set $X$ is sparsely-ended.
\end{cor}
A finer result is given in Theorem \ref{coulofi}(\ref{coulofi2}).

\begin{rem}
Let $G$ be an infinite countable locally finite group acting on a set $X$ with uncountably many infinite orbits. Then every closely $G$-invariant function on $G$ is constant on uncountably many infinite orbits. Hence $X$ is not sparsely-ended, although it is sparse.
%Holt proved that every uncountable locally finite group is 1-ended (the abelian case was previously done in \cite[Theorem 1]{ScS}). Hence sparse does not imply sparsely-ended in general.
\end{rem}

%%%%ENDNEW

Let us now study sparse near actions of finitely generated groups: this just means those whose near Schreier graphs have only finite components.

\begin{prop}\label{comreafin}
Let $\Gamma$ be a finitely generated group. Let $X$ be a sparse near $\Gamma$-set. Equivalences:
% with only finite orbits (i.e., for which some/every near Schreier graph has only finite components). 
%The following are equivalent
\begin{enumerate}
\item\label{ofo2} $X$ is realizable;
\item\label{ofo1} $X$ is completable;
\item\label{ofo3} $X$ is neat.
%near isomorphic to the disjoint union of a realizable and a finite type near $\Gamma$-sets.
\end{enumerate}
%Then $X$ is completable if and only it is realizable. If $\Gamma$ is finitely presented, these conditions are satisfied.
\end{prop}
Note that when $\Gamma$ is finitely presented, these conclusions are satisfied, as a consequence of Theorem \ref{nearactionfp}. 

%Also note that the assumption ``with only finite orbits" only makes sense for a finitely generated group.

\begin{proof}
(\ref{ofo2})$\Rightarrow$(\ref{ofo1}) is trivial. (\ref{ofo1})$\Rightarrow$(\ref{ofo3}) holds for arbitrary near $\Gamma$-sets (Proposition \ref{streafg}(\ref{codirefi})). Let us check (\ref{ofo3})$\Rightarrow$(\ref{ofo2}): there exists a near isomorphism $f:Y\sqcup Z\to X$ with $Y$ of finite type and $Z$ realizable. Then $Y$ is finite and $Z$ has only finite orbits. After removing finitely many orbits in $Z$, we can suppose that the restriction of $f$ to $Z$ has nonnegative index, i.e., has a representative that is an injective map. This implies that $X$ is realizable.
\end{proof}

%$X=Y\sqcup Y=Z$ be realizable; choose a realization. Only finitely many orbits meet both $X$ and $Y$. Since $X$ only has finite orbits, every orbit has finite intersection with $X$. Hence we can modify the action to be trivial on all these finitely many orbits. This does not affect the near action on $X$, and $X$ is invariant in the new action. Hence $X$ is realizable.

%For the nontrivial implication, suppose that $X$ is completable and let $X\sqcup Y=Z$ be realizable; choose a realization. Only finitely many orbits meet both $X$ and $Y$. Since $X$ only has finite orbits, every orbit has finite intersection with $X$. Hence we can modify the action to be trivial on all these finitely many orbits. This does not affect the near action on $X$, and $X$ is invariant in the new action. Hence $X$ is realizable.

\begin{defn}\label{profipre}
Let us say that a finitely generated group $\Gamma$ is profinitely finitely presented\index{profinitely finitely presented} if there exists a finitely presented group $\Lambda$ with a surjective homomorphism $\Lambda\to\Gamma$ inducing an isomorphism of profinite completions. 
\end{defn}

See \S\ref{dig_pfp} for general remarks about this notion. Note that finitely presented groups are obviously profinitely finitely presented.

%In terms of marked groups, this means that $\Gamma$ has a neighborhood in which every finite group is a marked quotient of $\Gamma$. 

\begin{prop}\label{notfp}
Let $\Gamma$ be a finitely generated group. Equivalences:
\begin{enumerate}
\item\label{gamprfp} $\Gamma$ is profinitely finitely presented;
\item\label{spa_rea} every sparse near $\Gamma$-action is realizable;
\item\label{spa_com} every sparse near $\Gamma$-action is completable;
\item\label{spa_tam} every sparse near $\Gamma$-action is neat.
%\item\label{spa_rea} every near $\Gamma$-action with only finite orbits is realizable;
%\item\label{spa_com} every near $\Gamma$-action with only finite orbits is completable;
%\item\label{spa_tam} every near $\Gamma$-action with only finite orbits is neat.
\end{enumerate}
In particular, if $\Gamma$ is not profinitely finitely presented, then it admits a non-neat near action.
% the conclusion of Theorem \ref{nearactionfp} fails for $\Gamma$.
% that is not profinitely finitely presented. Then the conclusion of Theorem \ref{nearactionfp} fails. More precisely, $\Gamma$ admits a near action with only finite orbits that is not completable.
\end{prop}
\begin{proof}(\ref{gamprfp})$\Rightarrow$(\ref{spa_rea})
Suppose that $\Gamma$ is profinitely finitely presented, and let $\Lambda\to\Gamma$ be as in the definition. Let $X$ be a near $\Gamma$-set with only finite orbits. By Theorem \ref{nearactionfp} and Proposition \ref{comreafin}, $X$ is realizable as a $\Lambda$-action; fix a realization. Since it has only finite orbits, it induces an action of the profinite completion of $\Lambda$, and hence of $\Gamma$. Thus it is realizable as near $\Gamma$-action.

(\ref{spa_rea})$\Rightarrow$(\ref{gamprfp}) Conversely, suppose that $\Gamma$ is not profinitely finitely presented. Let $F$ be a finitely generated free group with $\Gamma$ a quotient of $F$. Consider a sequence $(\Gamma_n)$ of successive finitely presented quotients of $F$, tending to $\Gamma$. By assumption, every $\Gamma_n$ has a finite quotient $F_n$ that is not quotient of $\Gamma$ (for the given marking). Consider the action of $F$ on the disjoint union $\bigsqcup F_n$. For every $x$ in the kernel of $F\to\Gamma$, the support of $x$ for this action is included in a finite sub-union, hence is finite. Therefore, the corresponding near action factors through $\Gamma$. If it were realizable on $\Gamma$, the realization would coincide with the given action on $F_n$ for large $n$, which by assumption is not the case. This near action is not realizable and is sparse.

This proves the equivalence (\ref{gamprfp})$\Leftrightarrow$(\ref{spa_rea}). The equivalence between the last three assertions is ensured by Proposition \ref{comreafin}.
\end{proof}

\begin{rem}
The idea of the construction of non-neat near actions, as in the proof of
Proposition \ref{notfp}, appears in B.~Neumann's construction in \cite{Neu}. Namely, he produces a group $\Gamma$ on 2 generators acting faithfully on a disjoint union of finite sets $F_n$, and preserving each of the $F_n$. In his construction, the subgroup of elements acting as a finitely supported permutation is the restricted product of all the alternating groups $\mka(F_n)$, and the quotient is isomorphic to the subgroup $\mka(\Z)\rtimes\Z$ of permutations of $\Z$ generated by the translation $n\mapsto n+1$ and the alternating group $\mka(\Z)$, which is now one of the most classical examples of a LEF finitely generated group that is not residually finite. This construction, in the current language, yields a non-completable sparse near action of $\mka(\Z)\rtimes\Z$.
\end{rem}

\begin{rem}
Restricting a neat near action of a finitely generated group to a finitely generated subgroup does not always yield a neat near action. Indeed, let $H$ be a finitely generated group with a non-near near action on a countable set $X$ (e.g., $H$ is profinitely finitely presented, using Proposition \ref{notfp}). Consider a fully-supported cycle on this set, and deduce a near action of the free product $G=H\ast\Z$ on $X$. The resulting near action of $G$ is of finite type, hence neat, but its restriction to $H$ is not neat.
\end{rem}

\subsection{Digression: profinitely finitely presented groups}\label{dig_pfp}
\index{profinitely finitely presented (group)}
This subsection is a digression on the notion of profinitely finitely presented groups (Definition \ref{profipre}). Recall that it was motivated by the problem to understand which finitely generated groups satisfy the conclusion of Theorem \ref{nearactionfp}, which states that all near actions of a given finitely generated group split as disjoint union of a finite type near action and a realizable near action.

To further illustrate that this is a natural definition, let us provide a few restatements. Recall the notion $\prod^\star$ of near product from \S\ref{s_nearproduct}, which for groups means the quotient of the direct product by the restricted (=finitely supported) product.

\begin{prop}\label{profp_cara}
Let $\Gamma$ be a finitely generated group. Equivalent statements:
\begin{enumerate}
\item\label{pfp1} $\Gamma$ is profinitely finitely presented;
%\item\label{pfp2} in the space of marked groups, $\Gamma$ has a neighborhood in which every finite group is a marked quotient of $\Gamma$;
\item\label{pfp3} for every family $(F_n)$ of finite groups, every homomorphism $\Gamma\to\prod_n^\star F_n$ lifts to a homomorphism $\Gamma\to\prod_nF_n$ (where $\prod^\star F_n=\prod F_n/\bigoplus F_n$).
\end{enumerate}
\end{prop}
\begin{proof}
The proof being routine, we only sketch it.
% (\ref{pfp2})$\Rightarrow$(\ref{pfp1}) is immediate. NON c'etait faux il faut des voisinages satures par passage au quotient

Suppose (\ref{pfp1}). Let $f:\Gamma\to\prod^\star F_n$ be a homomorphism. Choose a surjective homomorphism $p:\Lambda\to\Gamma$ as in the definition and a surjective homomorphism $\pi:L\to\Lambda$ from some finitely generated free group $L$. Then one can lift $f\circ p\circ\pi$ to a homomorphism $u:L\to\prod_n F_n$. Define a normal subgroup of $L$ by $M_n=u^{-1}(\bigoplus_{i\le n}F_i)$. Then the kernel of $\pi$ is included in $\bigcup_n M_n$; since it is finitely generated as a normal subgroup, it is included in $M_n$ for some $n$. This means that $f\circ p$ lifts to a homomorphism $v:\Lambda\to\prod F_n$. Since $\prod F_n$ is profinite, $v$ factors through the profinite completion of $\Lambda$. Hence, by (\ref{pfp1}), one can lift $f$ on $\Gamma$ and (\ref{pfp3}) is proved.

Now suppose (\ref{pfp3}). Choose a finitely presented group $\Lambda_0$ with a surjective homomorphism onto $\Gamma$, and enumerating elements of the kernel and modding out, we obtain a sequence of surjective homomorphisms $\Lambda_0\to\Lambda_1\dots$ with inductive limit $\Gamma$. For every $n$, let $\Lambda_n\to F_n$ be a finite quotient of $\Lambda_n$. This defines by composition a homomorphism $u_n:\Lambda_0\to F_n$, which together form a homomorphism $u:\Lambda_0\to\prod F_n$. Since every element of the kernel of $\Lambda_0\to\Gamma$ is in the kernel of $\Lambda_0\to F_n$ for $n$ large enough, the resulting homomorphism $\Lambda_0\to\prod^\star F_n$ factors through $\Gamma$. By (\ref{pfp3}), it lifts to a homomorphism $\Gamma\to\prod F_n$, which by composition lifts to a homomorphism $v=(v_n):\Lambda_0\to\prod F_n$. Since $u$ and $v$ coincide after composition by the projection to $\prod^\star F_n$, we obtain that there exists $n_0$ such that $u_n=v_n$ for all $n\ge n_0$. This means that $u_n$ factors through $\Gamma$ for all $n\ge n_0$. Since this holds for all possible choices of quotients $(F_n)$, we deduce that there exists $n$ such that every homomorphism from $\Lambda_n$ to any finite group factors through $\Gamma$, which yields (\ref{pfp1}).
\end{proof}

%Modifying $u_n$ to be the constant homomorphism for $n<n_0$, we 

%Choose a finitely generated free group $L$ with a surjective homomorphism onto $\Gamma$ to define the space of marked groups. Let $(F_n)$ be a sequence of marked groups converging to $\Lambda$. Putting together all given homomorphisms $L\to F_n$, we obtain a homomorphism $L\to\prod F_n$; since $F_n$ converges to $\Gamma$, the composite homomorphism $L\to\prod^\star F_n$ factors through $\Gamma$. By assumption, it lifts to a homomorphism $\Gamma\to\prod F_n$, and by composition this yields a homomorphism $L\to\prod F_n$ with the same composition to $\prod^\star F_n$ as the original one. Hence for some $n_0$ and all $n\ge n_0$, both homomorphisms $L\to F_n$ coincide. This implies that for $n\ge n_0$, $F_n$ is a marked quotient of $\Gamma$.

\begin{exe}\label{LEFRF}
Recall that a countable group is locally embeddable into finite groups\index{LEF (group)}, abbreviated LEF, if and only if it embeds as a subgroup of $\prod^\star F_n$ for some sequence $(F_n)$ of finite groups.

%In particular, if $\Gamma$ is LEF but not residually finite, t
It follows from Proposition \ref{profp_cara} that if a finitely generated group $\Gamma$ is LEF but not residually finite, then it is not profinitely finitely presented.

Beware that having a finitely presented profinite completion does not imply being profinitely finitely presented. Indeed, there exist infinite finitely generated simple LEF groups: these have a trivial profinite completion but are not profinitely finitely presented.
\end{exe}

\begin{exe}
It is not hard to check that for any nontrivial finite group $F$ and $d\ge 1$, the lamplighter group $F\wr\Z^d$ is not profinitely finitely presented. Indeed, let $\Gamma$ be a finitely presented group with a quotient homomorphism onto $F\wr\Z^d$. Then, by \cite{CGD}, there exists a finite index subgroup of $\Gamma$ with a quotient homomorphism onto the free product $F\ast\Z^d$. In particular, the profinite completion of $\Gamma$ does not satisfy any nontrivial group identity, unlike that of $F\wr\Z^d$. (When $F$ is not abelian, $F\wr\Z^d$ is not residually finite and hence the result also follows from Example \ref{LEFRF}.)
\end{exe}

\subsection{Stable realizability and finite stable realizability}
%Recall form  that 
%
In general, stable realizability does not always imply finite stable realizability (Proposition \ref{nonfistre}). Here we provide a result showing that this implication holds in some cases, beyond the finitely generated case.

\begin{lem}\label{strfix}
Let $G$ be a countable group and $X$ a near $G$-set. Then exactly one of the following holds:
\begin{itemize}
\item there exists a finitely generated subgroup $H$ of $G$ such that $X^H$ is finite ($X^H$ is the set of fixed points of some lift of a finite generating subset of $H$: this is well-defined up to near equality; see also Remark \ref{ultramenable});
\item $X$ is balanceably near isomorphic, as near $G$-set, to a disjoint union $Y\sqcup F$ with $F$ infinite with trivial action.
\end{itemize}
In the second case, the lest of $X$ is 1; in particular if $X$ is stably realizable, then it is realizable.
\end{lem}
\begin{proof}
These conditions are clearly incompatible (without restriction on $G$).

Suppose that the first condition does not hold. Fix a set-wise lifting map $i:G\to X^X$. Write $G=\bigcup S_n$ (ascending union) with $S_n$ finite subset. Let $Y_n$ be the set of fixed points of $i(S_n)$. By assumption, $Y_n$ is infinite for all $n$. Choose an injective sequence $(x_n)$ with $x_n\in Y_n$ for all $n$. Fix $g\in G$. Then $g\in S_n$ for some $n=n_g$, and hence belongs to $S_k$ for all $k\ge n_g$. Hence $i(g)$ fixes $x_k$ for all $k\ge n_g$. Hence $I=\{x_n:n\ge 1\}$ is commensurated with trivial near action. 

If $X$ is balanceably near isomorphic to $Y\sqcup I$ with infinite $I$ with trivial near action (we can choose $I$ countable), and is stably realizable, then $X\sqcup J$ is realizable for some $J$ with trivial near action; since $G$ is countable, we can choose $J$ countable, by a simple argument. So $X\sqcup J=Y\sqcup I\sqcup J$ is balanceably near isomorphic to $Y\sqcup I$, which itself is balanceably near isomorphic to $X$ as near $G$-set. Hence $X$ is balanceably near isomorphic to a realizable $G$-set, hence is realizable.
\end{proof}

\begin{rem}
The statement of Lemma \ref{strfix} does not always hold for actions of uncountable groups. Indeed, if $X$ is uncountable, and $G$ is the group of countably supported permutations of $X$, then it satisfies none of the two alternatives.

It even does not always hold for actions of uncountable groups on countable sets. Indeed, let $K$ be a finite field, and $V$ a vector space over $K$ of infinite countable dimension, say with basis $(e_n)_{n\in\Z}$, and $G$ be the set of automorphisms $f$ of $V$ that are identity on some subspace of finite codimension (depending of $f$). We consider the action of $G$ on $V$. Clearly it does not satisfy the first alternative, since any finitely generated subgroup of $G$ acts as the identity on some subspace of finite codimension. 

Let us show that $G$ does not satisfy the second alternative. To check this, consider any infinite subset of $V$ and let us find $f\in G$ moving infinitely many elements of $V$. Passing to a smaller subset, we can suppose that it is free and generates a subspace of infinite codimension, and we write it as $(e_n)_{n\ge 1}$; complete it as a basis $(e_n)_{n\in\Z}$. Defined an endomorphism $f(e_i)=e_0+e_i$ for $i\neq 0$ and $f(e_0)=e_0$. It is clearly an automorphism, and is the identity on the hyperplane of those $\sum _{n\in\Z}a_ne_n$ such that $\sum_{n\neq 0}a_n=0$. So $f\in G$, and $f$ fixes no $e_n$, $n\ge 1$. So $G$ does not satisfy the second alternative.
\end{rem}

\begin{prop}\label{abstre0}
Let $G$ be a countable group in which every finite subset is included in a normal and finitely generated subgroup. Then every stably realizable near action of $G$ is finitely stably realizable.
\end{prop}
\begin{proof}
By Lemma \ref{strfix}, if the near action is stably realizable but not realizable, then $G$ has a finitely generated subgroup $H$ with $X^H$ finite. Let $N$ be a finitely generated and normal subgroup of $G$ including $H$.
 Realize the $G$-action on $Z=X\sqcup I$, with trivial near action on $I$. Then $Z^H$ has finite symmetric difference with $I$. Moreover, since $N$ is normal, $Z^H$ is $G$-invariant, and so is its complement. So the $G$-action on $X$ is near isomorphic to a realizable near action, hence is finitely stably realizable.
\end{proof}

%(Note: the same argument works if every finitely generated subgroup of $G$ is contained in a normal and finitely generated subgroup. The latter also covers finitely generated groups.)

\begin{cor}\label{abstre}
Let $G$ be a countable virtually abelian group. Then every stably realizable near action of $G$ is finitely stably realizable.\qed
\end{cor}

\subsection{Analogues of the Burnside ring for near actions}\label{s_bur}

Fix a finitely generated group $G$. The set of balanced near isomorphism classes of near $G$-sets of finite type is a commutative monoid under the operation of disjoint union, denoted $\MN(G)$. Its quotient by the near isomorphism relation is denoted by $\MN'(G)$\index{$\MN(G)$, $\MN'(G)$, $\MN_0(G)$, $\MN^{\bowtie}(G)$}. The corresponding associated groups (``Grothendieck groups") are denoted by $\GN(G)$ and $\GN'(G)$.\index{$\GN(G)$, $\GN'(G)$, $\GN_0(G)$, $\GN^{\bowtie}(G)$} When we restrict to balanced near actions, we add an index 0 and write $\MN_0(G)$, etc.

When $G$ is an arbitrary group, we can similarly define such objects considering arbitrary near $G$-sets, and we denote this adding the exponent $\bowtie$, that is, $\MN^{\bowtie}(G)$, and so on. We can also restrict to finitely-ended near $G$-set and obtain similar objects; then the function ``number of ends" is a homomorphism from the corresponding monoid into $\N$, and induces a homomorphism of the corresponding group into $\Z$.

Beware that the product of near $G$-sets is not defined and we do not obtain a ring. We have a little substitute: we can form the product of a near $G$-set with a finite $G$-set. If $\mathrm{MB}(G)$ is the semiring of finite $G$-sets up to isomorphism and $\mathrm{GB}(G)$ is the associated ring (this is one natural way to define the Burnside ring of $G$ when $G$ is not necessarily finite), then we obtain a map $\mathrm{MB}(G)\times\MN^{\bowtie}(G)\to\MN^{\bowtie}(G)$, which is additive with respect to each variable, and is multiplicative in the first variable. It preserves all submonoids specified above, and passes to the associated groups (e.g., defining a map $\mathrm{GB}(G)\times\GN(X)\to\GN(X)$ when $G$ is finitely generated), making them a module over the Burnside ring.

The map mapping $X$ to its index character yields a semigroup homomorphism (also called index character) from any of these monoids into $\Hom(G,\Z)$, and passes to the corresponding groups.

%Note that any stably realizable near $G$-set $X$ vanishes in these associated groups,

\begin{prop}\label{gnb}~
\begin{enumerate}
\item\label{gnb1} For every stably realizable near $G$-set $X$, the class of $X$ vanishes in the group $\GN^{\bowtie}(G)$. 

\item\label{gnb2} For every group $G$, the index character $\MN^{\bowtie}(G)\to\Hom(G,\Z)$ is a surjective semigroup homomorphism, and hence induces a surjective group homomorphism $\GN^{\bowtie}(G)\to\Hom(G,\Z)$.

\item\label{gnb3} Let $G$ be a group such that every balanced near action of $G$ is stably realizable. Then the index character $\GN^{\bowtie}(G)\to\Hom(G,\Z)$ is a group isomorphism.
\end{enumerate}
\end{prop}
\begin{proof}
1) Consider a trivial $G$-set $Y$ such that $X\sqcup Y$ is realizable, and if we consider a huge $G$-set $Z$ having $\kappa$ orbits isomorphic to $G/H$ for each subgroup $H$, with $\kappa$ an infinite cardinal larger than the cardinal of $X\sqcup Y$, thein in $\MN^{\bowtie}(X)$ we have $Y+Z=Z$ and $X+Y+Z=Z$, which implies that $X=0$ in $\GN^{\bowtie}(X)$.

2) For any homomorphism $f:G\to\Z$ defines a near action $V_f$ of $G$ on $\N$ with index character equal to $f$. 

3) It remains to prove injectivity. Consider an element in the kernel of the index character. It can be written as $X-Y$ for two near $G$-sets $X,Y$, which thus have the same index character $f$. Then it can also be written as $(X+V_{-f})-(Y+V_{-f})$. Since $X+V_{-f}$ and $Y+V_{-f}$ are balanced, they are, by assumption, stably realizable, and hence are zero in $\GN^{\bowtie}(G)$ by (\ref{gnb1}). So the kernel is zero.
\end{proof}

%Consider near $G$-sets $X,Y$ given by homomorphisms $\alpha:G\to\mkst(X)$, $\beta:G\to\mkst(Y)$, we say that $f\in \mk{E}^\star(X,Y)$ is near $G$-equivariant if $f\circ\alpha(g)$ and $\beta(g)\circ f$ are equal in $\mk{E}^\star(X,Y)$. Near $G$-equivariant elements in $\mk{E}^\star(X,Y)$ form a subset $\mk{E}^\star_G(X,Y)$, or $N_{\alpha,\beta}(X,Y)$ if we need to specify the action. We thus have a category of near $G$-sets.

%Invertible elements in $\mk{E}^\star_G(X,Y)$ are called near $G$-equivalence, and then $X,Y$ are called near $G$-equivalent. If we can find an invertible element in $\mk{E}^\star_G(X,Y)$ with a bijective representative, then we call it a balanced near $G$-equivalence, and we say that $X$ and $Y$ are  $G$-isomorphic. For instance, two finite $G$-sets with different cardinal are near $G$-equivalent, but not almost $G$-equivalent.

%\subsection{Perturbations of actions}\label{s_perta}

%\subsection{Perturbations of actions of 1-ended groups}\label{p1e}

%exists a subset $X_{(g)}\subseteq X$ with finite complement such that $f(gx)=gf(x)$ for all $x\in X$.

%Two maps are equivalent (or: finite perturbations of each other) if they coincide outside a finite subset. Clearly, this preserves the condition of being $G$-AE.

%The $G$-equivariant maps $G\to X$ are the $g\mapsto gx_0$ for $x_0\in X$. These are called translations.

\subsection{Link between near actions and partial actions}\label{s_partact}

%The material of this subsection will be used in \S\ref{nearwreath}, and also in \cite{CIET}.

We recall the notion of partial action\index{partial action}, due to Exel \cite{E}. Let $G$ a group and $X$ a set. A partial group action of $G$ on $X$ is a partially defined function $G\times X\to X$, denoted $(g,x)\mapsto gx$ whenever defined, satisfying the following axioms:

\begin{itemize}
\item for every $x\in X$, $1x$ is defined and equal to $x$;
\item for all $g\in G$ and $x\in X$ such that $gx$ is defined, $g^{-1}(gx)$ is defined and equal to $x$;
\item for all $g,h\in G$ and $x\in X$ such that $g(hx)$ is defined, $(gh)x$ is defined and equal to $g(hx)$.
\end{itemize}

Every cofinite-partial action\index{cofinite-partial action} of $G$ on a set $X$ defines a near action on $X$. The main difference between partial actions and near actions is that in a partial action, the domain of definition is part of the datum. This is an essential difference, as illustrated by Proposition \ref{partcompletable} below.

Let $Y$ be a $G$-set and $X$ an arbitrary subset of $Y$. Then we obtain a partial action of $G$ on $X$ by requiring $gx$ to be defined and equal to $gx$ (in the sense of the action on $Y$) whenever $gx\in X$. In this case, the $G$-set $Y$ is called a {\it globalization}\index{globalization} of the partial action; let us also call $X$ a {\it partialization} of the action on $Y$.

It is a theorem of Abadie \cite{AT,A}, and later independently Kellendonk and Lawson \cite{KL}, that every partial action arises this way, i.e., admits a globalization (actually, both prove the existence of a ``universal globalization"). 

Now let us relate this to the notion in consideration here. First, let call cofinite-partial action, a partial action such that for every $g\in G$, $gx$ is defined for all but finitely $x\in X$. Given a $G$-set $Y$ and a subset $X$, the partialization on $X$ is cofinite-partial if and only if $X$ is a $G$-commensurated subset of $Y$.

The above Abadie and Kellendonk-Lawson result implies the following:

\begin{prop}\label{partcompletable}
A near $G$-set $X$ arises from a cofinite-partial action if and only if it is a completable near $G$-set.
\end{prop}
\begin{proof}
Suppose that $X$ is completable. Let $Y$ be another near $G$-set such that $X\sqcup Y$ is realizable; fix a realization. Then $X$ is a commensurated subset, and the corresponding partialization induces the given near $G$-action on $X$.

Conversely, suppose that $X$ admits a cofinite-partial $G$-action inducing the given near $G$-action. Let $Z=X\sqcup Y$ be a globalization (as predicted by the above-mentioned result). Then $Z=X\sqcup Y$ as near $G$-set and thus $X$ is a completable $G$-set.
\end{proof}

As we will see, there exist non-completable near actions (e.g., of $\Z^2$), see \S\ref{s_noncomp}. Hence they do not arise from cofinite-partial actions. still, the study of completable near actions is not just a restatement of the study of cofinite-partial actions, because a partial action is an enrichment of the structure. (And also the study of realizable near actions is not the same as the study of actions, for a similar reason.)

\begin{rem}\label{glothomp}
The globalization theorem also implies that for every cofinite-partial action, the function $g\mapsto |F_g|+|F_{g^{-1}}|$ is cardinal-definite, where $F_g$ is the (finite) set of $x$ such that $gx$ is undefined. 
%This immediately implies results, such as the existence of a proper cardinal definite function on Thompson's group $V$ (see Remark \ref{thompson}), a result of Farley-Hughes \cite{Hug}.
\end{rem}

\begin{rem}
Let $X$ be a topological space. Let $\PC(X)$ be the subgroup of $\mkst(X)$ consisting of those element having a representative that is a homeomorphism between two cofinite open subsets.

If $X$ is perfect and Hausdorff, then the near action of $\PC(X)$ on $X$ naturally arises from a cofinite-partial action \cite[Prop.\ 2.16]{Corpcw}; in particular, it is completable (by Proposition \ref{partcompletable}).
\end{rem}

Recall that a group $G$ has Property FW if for every $G$-set $Y$ and every $G$-commensurated subset $X\subseteq Y$ is $G$-transfixed, in the sense that there exists a $G$-invariant subset $X'$ such that $X\bigtriangleup X'$ is finite.

\begin{prop}
A group $G$ has Property FW\index{Property FW} if and only if for every cofinite-partial action on a set $X$, there exists a cofinite subset $X'\subseteq X$ such that the induced cofinite-partial action on $X'$ is cofinite in its universal globalization.
\end{prop}
\begin{proof}
Suppose that $G$ has Property FW and consider a cofinite-partial action on a set $X$. Let $X\subseteq Y$ be a universal globalization. By Property FW, there exists a $G$-invariant subset $X'\subseteq Y$ such that $X\bigtriangleup X'$ is finite. Then the universal globalization of $X\cap X'$ is included in $X'$, and contains $X\cap X'$ as a cofinite subset.

Conversely, suppose that the condition is satisfied. Let $Y$ be a $G$-set with a $G$-commensurated subset $X$. There exists a cofinite subset $X'$ of $X$ whose universal globalization includes $X'$ as a cofinite subset; this universal globalization $X''$ can be viewed as a subset of $Y$; hence $X$ is transfixed.
\end{proof}

\begin{prop}\label{FWcomp}
Let $G$ be a group with Property FW. Then every completable near $G$-set is finitely stably realizable.
\end{prop}
\begin{proof}
Let $X$ be a completable near $G$-set. Realize the action on a set $Y$ including $X$. So there exists a $G$-invariant subset $X'$ with $X\bigtriangleup X'$ finite. We can conjugate the action on $Y$ by a finitely supported permutation, so as to ensure that either $X\subseteq X'$, or $X'\subseteq X$. If $X'\subseteq X$, this shows that the near action on $X$ is realizable (with $|X\smallsetminus X|$ fixed points). If $X\subseteq X'$, this shows that the near action on $X$ is finitely stably realizable.
\end{proof}

Note that the converse is not true: the infinite dihedral group satisfies the condition that every completable near action is realizable (since all its near actions are realizable), but does not have Property FW.

%%%%%%%%%%%%%%%%%%%%%%%%%%%%%%%%%%%%%%%%%%%%%%%%%%%%%%%%%%%%%%%%%%

\section{Examples around realizability for finitely generated groups}\label{s_examples}
The subsections of this section are independent. They illustrate the notions of \S\ref{s_nearac}, notably the various notions related to realizability, namely from \S\ref{nearaction} to \S\ref{s_rnfg}, and also the analogues of the Burnside ring of \S\ref{s_bur} specifically in \S\ref{s_bure}. Most of the subsections here deal with groups that are abelian or close to abelian (notable exceptions being \S\ref{windec}, \S\ref{s_amalgam}).

%, and finitely generated (with the exceptions of the last two subsections).
% subsections, \S\ref{s_qcy} to \S\ref{s_nofista}).   punctually the Kapoudjian class from \S\ref{s_omega},

\subsection{Realizability for finite groups}

We start from the most basic example:

\begin{prop}\label{fini_rea}
Every near action of a finite group is realizable.
\end{prop}

We prove it in a more general context, for perspective. It turns out that our case ($\kappa=\aleph_0$) is slightly more complicated than the other ones.

\begin{lem}\label{liftkappa}
Let $X$ be a set and $\kappa$ an infinite cardinal. Let $\mks_{<\kappa}$ be the normal subgroup of $\mks(X)$ consisting of permutations of $X$ whose support has cardinal $<\kappa$. Let $\Gamma$ be a subgroup of $\mks(X)/\mks_{<\kappa}$, of cardinal $<\kappa$. Then $\Gamma$ can be lifted to $\mks(X)$.
\end{lem}
\begin{proof}
We can lift $\Gamma$ to a subset $S$ of $\mks(X)$, denote by $g\mapsto \tilde{g}$ this lifting map. For all $g,h\in \Gamma$, the set $X_{g,h}$ of $x$ such that $\widetilde{gh}x\neq\tilde{g}\tilde{h}x$ has cardinal $<\kappa$. Hence the union $X'$ over all pairs $(g,h)$ has cardinal $<\kappa$.

We claim that $\langle S\rangle$ has cardinal $<\kappa$; note that this is immediate when $\kappa>\aleph_0$. This being granted, we deduce that the union $X''$ of $\langle S\rangle$-orbits meeting $X'$ has cardinal $<\kappa$. Hence, we can change the lift by requiring the action to be trivial on $X''$ and unchanged elsewhere. This yields a lift which is a group homomorphism.

It remains to check the claim when $\kappa=\aleph_0$, namely that $\langle S\rangle$ is finite. In this case $S$ is finite, so $\langle S\rangle$ is finitely generated; hence the kernel $\langle S\rangle\to\Gamma$ has finite index in $\langle S\rangle$ and is finitely generated. But it is included in the group of finitely supported permutations, which is locally finite. Hence this kernel is finite, and $\langle S\rangle$ is finite.
\end{proof}

\begin{proof}[Proof of Proposition \ref{fini_rea}]
Observe that since $\Hom(F,\Z)=0$, this is a balanced near action, i.e., is given by a homomorphism $F\to\mksh(X)=\mks(X)/\mks_{<\aleph_0}(X)$. Hence we conclude by Lemma \ref{liftkappa}.
\end{proof}

Another special case of Lemma \ref{liftkappa} is that every countable permutation group modulo permutations with countable support can be lifted. We do not consider such a setting elsewhere in the paper; it sounds as a quite disjoint study since here we focus on actions of countable groups, especially finitely generated groups.

\subsection{Direct product with an infinite cyclic group}

All balanced near actions of free and of finite groups are realizable. By the following, this does not extend to the direct product of a nontrivial finite group with $\Z$. 

\begin{prop}\label{hznore}
For every nontrivial finite group $H$, the product $H\times\Z$ admits a $\star$-faithful balanced near action that is not stably realizable.

More generally, let $H$ be a group admitting a non-trivial action by finitely supported permutations on a set (e.g., $H$ is finite, or more generally $H$ has a non-trivial finite quotient). Then the group $H\times\Z$ admits a balanced near action that is not stably realizable, and which is $\star$-faithful if the initial $H$-action is faithful.
\end{prop}
\begin{proof}
Let $H$ act non-trivially by finitely supported permutations on a set $Y$.
Let $H$ act on $Y\times\Z$ by: $h(y,n)=(hy,n)$ for $n\ge 0$, and acting trivially on $Y\times\Z_{<0}$. Let $\Z$ act on $Y\times\Z$ by $n(y,m)=(y,n+m)$. Using that $H$ by finitely supported permutations, one easily checks that the two near actions commute. It thus defines a near action of $H\times\Z$ whose kernel is $W\times\{0\}$, where $W$ is the kernel of the $H$-action on $Y$. 

Fix $h\in H$ not in the kernel of the near action. Then $h$ has one lift $h_1$ given by the above action of $H$. Also, let $s_1$ be the above generator of $\Z$. Suppose by contradiction that the near action of $H\times\Z$ is stably realizable, and for some realization (on $(Y\times\Z)\sqcup V$, with trivial near action on $V$), let $h_2$ be the action of $h$ and $s_2$ a permutation, near equal to $s_1$ and commuting with $h_2$. Let $P$ be the (finite, non-empty) support of $h_1$. Then the support of $h_1$ is $P\times\Z_{\ge 0}$, and the support $K$ of $h_2$ thus has finite symmetric difference with $P\times\Z_{\ge 0}$. Since $h_1(P\times\Z_{\ge 0})$ is properly included in $P\times\Z_{\ge 0}$, there is no permutation $h_2$ near equal to $h_1$ leaving $K$ invariant (in the sense that $h_2(K)=K$), so we have a contradiction.
\end{proof}

\begin{rem}
It follows from Proposition \ref{nonrea_1end} that whenever a group $H$ is not locally finite, $H\times\Z$ admits a non-realizable (finitely stably realizable) action. Proposition \ref{hznore} is concerned with some cases where $H$ is locally finite. 
\end{rem}

Near actions of the direct product of a finite group with $\Z$ are considered again in \S\ref{difi}.

\subsection{Examples of realizability degree and lest}\label{redele}

\begin{prop}\label{lest_forbit}
Let $G$ be a finitely generated group. Let $X$ be a $G$-set with only finite orbits. Let $O_X$ be the the set of $n\in\N_{>0}$ such that there are infinitely many orbits of cardinal $n$. Then the lest $\lst_G(X)$ of $X$ is equal to $\gcd(O_X)$ (where $\gcd(\emptyset)=0$).
\end{prop}
\begin{proof}
We use the characterization of lest given in Proposition \ref{basiclest}(\ref{bl4}).

For every finite index subgroup $H$ of $F$, let $X^{[H]}$ be the union of orbits isomorphic to $G/H$. By assumption, $X$ is the disjoint union of all $X^{[H]}$. For every $H$ such that $X^{[H]}$ is infinite, there exists an injective endomorphism of the $G$-set $X$ (that is the identity outside $X^{[H]}$), whose image is the complement of one orbit isomorphic to $G/H$ as a $G$-set. Hence the lest $\lst_G(X)$ divides the index $|G/H|$. Since $G$ has finitely many subgroups of each given finite index, for every $n\in O_X$ there exists $H$ with $|G/H|=n$ and $X^{[H]}$. It follows that the lest $\lst_G(X)$ divides $\gcd(O_X)$.

We have to show that conversely, $\gcd(O_X)$ divides $\lst_G(X)$, that is, divides $n$ for every $n\in O_X$. We start with the case when $X=X^{[H]}$. Let $Y$ be the quotient set (the set of orbits). Then the quotient map $X\to Y$ is near $G$-equivariant, and the index satisfies $\phi_X(g)=|G/H|\phi_Y(g)$ for all $g\in G$. Hence it takes values in $|G/H|\Z$.

In general, for every $g\in G$ and representative $\tilde{g}$ of $g$, the element $\tilde{g}$ preserves $X^{[H]}$ for all but finitely many $H$, and commensurates all $X^{[H]}$. Hence its index belongs to $\sum_H|G/H|\Z=\gcd(O_X)\Z$, where $H$ ranges over finite index subgroups of $G$. This precisely means that $\gcd(O_X)$ divides $\lst_G(X)$.
\end{proof}

%Let us determine the near automorphism group. For every finite index subgroup $H$ of $F$, let $X^{[H]}$ be the union of orbits isomorphic to $G/H$. By assumption, $X$ is the disjoint union of all $X^{[H]}$. Given a cofinite partial bijection representing a near automorphism, $g$ is defined and preserves $X^{[H]}$ for all but finitely many $H$, and commensurates all $X^{[H]}$. Hence the index of $g$ is equal to the (finitely supported) sum over all $H$ of the index of $g$ as a near permutation of $X^{[H]}$. When $X^{[H]}$ is finite, this is zero. When $X^{[H]}$ is infinite, there exists an automorphism $g$ for which the index on $X^{[H]}$ is $|G/H|$ and for which the index on all others $X^{[H']}$ is 0. Hence, the index degree of $X$ as near set over its near automorphism group is equal to $\gcd(P)$, where $P$ is the set of $n$ such that there exists $H$ such that $X^{[H]}$ is infinite and $|G/H|=n$. Since there are finitely many subgroups of each given index, we have $P=O_X$. So the above index is $\gcd(O_X)$, and this is equal to the lest of $X$ by Proposition \ref{basiclest}(\ref{bl4}).

\begin{lem}\label{lestcopy}
Let $G$ be a group and $X$ a 1-ended $G$-set, and $Z$ a nonempty set with a trivial $G$-action. Then $\lst(X\times Z)=\lst(X)$.
\end{lem}
\begin{proof}
We first claim that the near automorphism group is the (restricted) permutational wreath product $\mkst_G(X)\wr\mks(Z)$. This is clearly a subgroup of the near automorphism group.
Since $X$ is 1-ended, one can identify $Z$ with the set of isolated ends of $X\times Z$ (by Proposition \ref{isolend}). This yields a homomorphism into $\mks(Z)$, which is the identity in restriction to $\mks(Z)$. Hence it is enough to check that its kernel is restricted to the restricted power $\mkst_G(X)^{(Z)}$. Indeed, any element of this kernel acts as an automorphism of the $G$-action on all but finitely many copies of $X$, and commensurates all copies of $X$, so the result follows.

Then we use that $\Hom(\mks(Z),\Z)=0$, and hence every homomorphism from the above wreath product to $\Z$ factors through the projection $\mkst_G(X)\wr\mks(Z)\to\mkst_G(X)$. Hence, since $Z$ is nonempty, the lest is equal to the index number of the action of the given copy of $\mkst_G(X)$, which equals $\lst(X)$.
\end{proof}

\begin{prop}
Given an action of $\Z$ on a set $X$, the lest $\lst(X)$ is equal to 1 if there exists an infinite orbit, and otherwise is given by Proposition \ref{lest_forbit}.
\end{prop}
\begin{proof}
In the case of $\Z$, the infinite cycle has lest 1. Hence it follows that the lest is 1 whenever there is an infinite orbit (by the obvious part of Proposition \ref{lestunion}).
\end{proof}

\begin{prop}\label{lestdinfty}
Given an action of the infinite dihedral group on a set $X$, if there are no infinite orbits, the lest is given by Proposition \ref{lest_forbit}. If is at least one infinite orbit, it is equal to 1 or 2, namely 1 if and only if one of the two following holds:
\begin{itemize}
\item there exists an odd number $n$ such that there are infinitely many orbits of cardinal $n$. 
\item there are infinitely many 1-ended orbits of at least two of the three possible types.
\end{itemize}
\end{prop}
\begin{proof}
First recall that it has three subgroups of infinite index up to conjugation: the trivial group and two subgroups of order 2, yielding 3 types of transitive actions: the simply transitive one (2-ended), and the two types of 1-ended action (which are near isomorphic). Consider a 1-ended action. Its Schreier graph is a ray with edges of two alternating colors. By a simple argument, the near automorphism group is cyclic, generated by a translation of length 2 (so as to be color-preserving). Since the simply transitive action is near isomorphic to a disjoint union of two such 1-ended actions, its lest is 2 as well.

Suppose that there is at least one infinite orbit, so the lest is 1 or 2. If there are infinitely many orbits of cardinal $n$, then there are infinitely many of the same type and hence the lest divides $n$. If there are infinitely many orbits of both 1-ended type, removing the end vertex in one changes it to the other type and adds one point. If there are infinitely many simply transitive orbits and infinitely many of one of the two 1-ended type, removing one suitable point in a simply transitive orbit yields two orbits of the same 1-ended type and adds one point. This proves that if the second condition is satisfied, then the lest is 1.

It remains to show that the lest is 2 when none of these conditions is satisfied. In view of Proposition \ref{lestunion}, we can suppose that $X$ only admits infinite orbits of one given type (if there are finitely many infinite orbits of one given type, one ``removes" them). Also, decomposing into the union of finite orbits and the union of infinite orbits, Proposition \ref{lestunion} also applies to treat both cases separately. The case of only finitely many orbits is already proved, and hence we are reduced to the case of only infinite orbits of a given type. Then Lemma \ref{lestcopy} applies.
\end{proof}

\begin{rem}\label{noteqlst}
The above yields an example of two $D_\infty$-sets $X,Y$ such that $\lst(X\sqcup Y)$ is a proper divisor of $\gcd(\lst(X),\lst(Y))$ (see Proposition \ref{lestunion}). Namely $X=\kappa e$ and $X=\kappa'e'$ for any infinite cardinals $\kappa,\kappa'$. Here $\lst(X)=\lst(Y)=2$ and $\lst(X\sqcup Y)=1$.
\end{rem}

\subsection{Houghton near actions}\label{s_hou}

Proposition \ref{hznore} provides stably non-real\-izable balanced near actions of $\Z^2$, which are not $\star$-faithful.\index{Houghton near action}
It also admits stably non-realizable balanced near actions that are also $\star$-faithful. 

%One is given in Example \ref{e_hough}; this very example provides one proof that it is not stably realizable: it has a non-trivial Kapoudjian class.

%which is well-known: it is given by Houghton's group action on a graph consisting of three rays (the inverse image of $\Z^2$ in the permutation group is indeed Houghton's group). The action is not stably realizable because $\Z^2$ does not admit a (possibly non-connected) Schreier graph of linear growth with three ends. Let us redefine this using our formalism.

Let us generalize this to arbitrary rank. First, start with the near action of $\Z$ on $\N=\{0,1,\dots\}$ where the positive generator of $\Z$ acts by $n\mapsto n-1$. Its index character (as defined in \S\ref{s_nearac}) is the identity map. Now let $\Gamma$ be any group with a homomorphism $f:\Gamma\to\Z$. Then this endows $\N$ with a near action of $\Gamma$, by composition, denoted $\N_f$. Let now $f_1,\dots,f_k$ be homomorphisms $\Gamma\to\Z$ and define $\N_{(f_1,\dots,f_k)}$ as the disjoint union $\bigsqcup \N_{f_i}$. Its index character is equal to $\sum f_i$. In particular, if $\sum f_i=0$, then this is a balanced near action.

For $\Z^d$, choose $k=d+1$, $f_1,\dots,f_d$ being the coordinate projections and $f_{d+1}=-\sum_{i\le d}f_i$. The resulting balanced near action $\N_{(f_1,\dots f_{d+1})}$ (on $d+1$ copies of $\N$) is called the Houghton near action of $\Z^d$. By construction, it has $d+1$ ends.

\begin{prop}\label{Hou_not}
Suppose that $d\ge 2$. The Houghton near action of $\Z^d$ is $\star$-faithful and not stably realizable. Actually, it is not stably realizable in restriction to any finite index subgroup. It is completable.
\end{prop} 
\begin{proof}
$\star$-faithfulness is immediate.

%We give two proofs that it is not stably realizable.

%First proof: by Corollary \ref{omega_nunion}, its Kapoudjian class is the ``canonical" cohomology class $\sum_{i<j}\pi_i\smsm \pi_j$ of $\Z^d$, where $\pi_i$ is the canonical projection (this element is invariant under the action of $\GL_d(\Z)$ and hence does not depend on the choice of basis.Second proof: l

Let us prove that it is not finitely stably realizable; by Proposition \ref{streafg}(\ref{streafg1}), this implies that it is finitely stably realizable. Suppose by contradiction that it is finitely stably realizable, namely $X\sqcup F=Z$, with $Z$ realizable. Let $T$ be a 1-ended commensurated subset of $X$. Then there exists a unique orbit $U$ including some cofinite subset of $T$. Then the index character of $T$ is equal to $\pi_i$ for some $i$. So $U$ has at least two ends. Hence (by the classification of transitive $\Z^d$-sets) $U$ has 2 ends, the second end having index $-\pi_i$. Since $X$ has no 1-ended commensurated subset with index character equal to $\pi_i$, we obtain a contradiction. 

%This second proof 

This adapts to any finite index subgroup $\Lambda$ of $\Z^2$: indeed, on $\Lambda$, it is still true that each 1-ended subset has index character equal to $\pi_i|_\Lambda$.

%(On the other hand, the Kapoudjian class of the restricted action is nonzero if and only if $\Lambda$ has odd index, by a simple argument.)

It is completable: just embed each copy of $\N$ in a copy of $\Z$.
\end{proof}

\begin{rem}Note that $\Z^2$ also admits a near action that it not realizable, but realizable in restriction to some finite index subgroup. Indeed, some non-realizable action factoring through $\Z\times (\Z/2\Z)$ (as provided by Proposition \ref{hznore}) does the job; it is not $\star$-faithful but this is fixed by just taking the disjoint union with a simply transitive action.
\end{rem}

\begin{rem}
For $d\ge 2$, the above near action of $\Z^d$ on the disjoint union of $d+1$ copies of $\N$ extends to an action of the semidirect product $\Gamma=\Z^d\rtimes\mks_{d+1}$ (here we view $\Z^d$ as the $(d+1)$-tuples of sum zero in $\Z^{d+1}$. Then the near $\Gamma$-action is not stably realizable (since this holds in restriction to $\Z^d$), is one-ended (since $\sigma_{d+1}$ permutes transitively the $d+1$ ends of the original action), and is completable (by the same argument).
\end{rem}

%Second proof: Every Schreier graph of this almost action has linear growth. The Schreier graphs of $\Z^d$ with at most linear growth are either finite or 2-ended. If $d$ is odd, this already finishes the argument. Suppose that $d=2k$ is even. Then there are $k$ subgroups $L_1,\dots,L_k$ of $\Z^{2k}$ such that the given almost action (possibly after adding finitely many points) can be realized by the action on the disjoint union of $\Z^{2k}/L_i$, and $L_i$ has rank $2k-1$. Then the intersection of all $L_i$ has rank $k$. This contradicts the wide faithfulness.

%Third proof: we check that there is no almost subaction other than finite subsets and their complements. Indeed, since the set of ends is finite, the near sub-actions (i.e. the commensurated subsets) are given, up to finite perturbation, by subsets of the set of ends. But for each such finite subset other than the empty set and the whole set, the index is nonzero. 

%It follows that if the action is realizable, it is realizable by an action with a single infinite orbit. But such an orbit would have 1 or 2 ends (being the set of ends of a quotient group, since $\Z^d$ is abelian), and this is a contradiction.

%Let us now show that it is not stably realizable in restriction to any finite index subgroup.

\subsection{Non-completable near actions of $\Z^2$}\label{s_noncomp}

Possibly more surprisingly, $\Z^2$ admits a stably non-realizable $\star$-faithful one-ended balanced near action.

\begin{defn}
Consider the Schreier graph of $\Z^2$, with its standard generating subset, for its simply transitive action, drawn on the plane. View it as a square complex, with labeled oriented edges. Remove the interior of the square $[0,1]^2$; the resulting square complex has infinite cyclic fundamental group. For every $n\ge 1$, consider the connected $n$-cover of it, and retain the resulting connected graph; let $X_{n,0}$ be the corresponding balanced near action of $\Z^2$.
\end{defn}

Note that $X_{1,0}$ is the simply transitive action. See the introduction for a picture of $X_{2,0}$.

% For the most obvious lifts of the generators to permutations of $X_{m,0}$, the commutator is a $m$-cycle, and hence the Kapoudjian class is zero if and only $m$ is odd.

\begin{prop}\label{Z2nonsr1}
The near action $X_{n,0}$ of $\Z^2$ is not completable for any $n\ge 2$.
\end{prop}
\begin{proof} 
Write $J=X_{n,0}$. Suppose by contradiction that $J$ is completable. By Proposition \ref{streafg}(\ref{streafg2}), there exists a near $\Z^2$-set $L$ of finite type such that $J\sqcup K$ is realizable. We can write $L=L_1\sqcup L_2$, where (the Schreier graph of) $L_1$ has at most linear growth, and $L_2$ is a free action. Since $J$ is 1-ended, it is either near included in $L_1$ or $L_2$, and since it has quadratic growth, it is near included in $L_2$. Writing $L_2$ as union of its orbits, each of which is one-ended, we deduce that $J$ is near isomorphic to a simply transitive action. 

The near equivariant map $J\to \Z^2$ (the covering map) thus yields, by conjugation, a near equivariant map $\Z^2\to\Z^2$ whose fibers have $n$ elements, with finitely many exceptions. On the other hand, it follows from Proposition \ref{AE1e} that any near equivariant self-map of $\Z^2$ is a finite perturbation of a translation and hence its fibers are singletons with finitely many exceptions. This yields a contradiction.
\end{proof}

Let us also provide $\star$-faithful 1-ended near actions of $\Z^2$ with nonzero index. Fix $m\ge 1$. On $\Z^2$, define $u(x,y)=(x+1,y)$, define  $v(x,y)=(x,y+1)$ if $y\neq 0$ or $x\ge 1$, $v(x,0)=(x+m,1)$ if $x\le -m$, and leave $v(x,0)$ undefined when $x\in\{-m+1,\dots,0\}$. Then $u,v$ are near commuting near permutations, defining a near action of $\mathbf{Z^2}$. Denote it by $X_{1,(m,0)}$, or $K_m$ for short.

%One can check that the Kapoudjian class of $K_m$ is nonzero if and only if $m$ is odd.

\begin{prop}
The near action $K_m$ of $\Z^2$ is not completable.
\end{prop}
\begin{proof}
By contradiction, we deduce, using Proposition \ref{streafg}(\ref{streafg2}), that there exists a near $\Z^2$-set $L$ of finite type such that $K_m\sqcup L$ is realizable. Since $K_m$ is 1-ended, it is near included in a single orbit; since completable is a near isomorphism invariant, this implies that $K_m$ is completable to a transitive action, say on a set $X$. If $X$ is not faithful, it has linear growth and we have a contradiction. Otherwise, $X$ is a simply transitive action, and being 1-ended, we deduce that $K_m$ is near isomorphic to the simply transitive action. Since $K_m$ has nonzero index character, we reach a contradiction.
\end{proof}

In contrast, we show in \S\ref{s_ab} that for $d\ge 3$, every 1-ended, $\star$-faithful balanced near action of group $\Z^d$ is stably realizable.

\subsection{Non-completable near actions of other finitely presented groups}\label{noncomplother}

If $G$ is any group with a finitely generated and normal subgroup $N$ such that $G/N$ is isomorphic to $\Z^2$, from the examples in \S\ref{s_noncomp} and using Proposition \ref{completablek} we obtain a non-completable near-action of $G$. This applies in particular to all finitely generated abelian groups of $\Q$-rank $\ge 2$.

Let us sketch some other examples, still based on similar ideas. 

\begin{itemize}
\item Consider a finitely generated subgroup $\Gamma$ of $\GL_2(\Z)$, and consider the transitive ``affine" action of $\Z^2\rtimes\Gamma$ on $\Z^2$. Assume that the action of $\Gamma$ on the circle (quotient of $\R^2\smallsetminus\{0\}$ by positive homotheties) fixes a point. 

Then we can, similarly to the example $X_{m,0}$ above, perform a ``finite covering" of the Schreier graph. Indeed, by assumption there is a real half-line emanating from 0 that is $\Gamma$-invariant; we can cut along this half-line and perform a nontrivial finite covering. This applies whenever $\Gamma$ is a cyclic subgroup with real eigenvalues. The resulting near action is not completable, because it is not completable in restriction to $\Z^2$.

\item Let $\Gamma$ be the fundamental group of a closed oriented surface of genus $g\ge 2$. Its right Cayley graph can be described as a tiling of the hyperbolic plane by regular $4g$-gons, with all vertices of valency $4g$. Viewing it as a Schreier graph, we can remove one $4g$-gon and perform a similar covering. Such a construction can also be applied to some Fuchsian groups with torsion, including some Coxeter groups, notably the triangular Coxeter groups
\[\langle a,b,c:a^2=b^2=c^2=(ab)^\ell=(bc)^m=(ac)^n=1\rangle\]
when $1/\ell+1/m+1/n<1$. 
\end{itemize}
These near actions are non-completable. In the case of $\Z^2$, the ad-hoc above argument uses the fact that multi-ended transitive actions are extremely restricted, namely these factor through quotients of the form $F\times\Z$ with $F$ finite. Such an argument cannot be used in general since it is hopeless to classify all multi-ended transitive actions. Nevertheless, non-completability follows from the much more elaborate Corollary \ref{cor_noncomp}.

\subsection{Balanceably indecomposable $\infty$-ended near actions}\label{windec}

We give here an example of a balanceably indecomposable balanced near action (as defined in \S\ref{csends}). An independent interest of these examples is to provide typical balanced near actions of some finitely generated wreath products.

\begin{exe}[A balanceably indecomposable $\infty$-ended action]\label{zzin}
On the complex plane, define $f(z)=z+i$ and $g_n(z)=z+1$ when $\mathrm{Im}(z)\ge n$ and $g_n(z)=z$ otherwise. We see that $fg_nf^{-1}=g_{n+1}$ for all $n$, and then deduce that the group $\langle f,g_0\rangle$ is a wreath product $\Z\wr\Z$. One orbit is $\Z[i]$, and the corresponding Schreier graph is obtained from the standard Cayley graph of $\Z^2$ by removing horizontal edges in the half-plane $\{\mathrm{Im}<0\}$ and adding self-loops instead. In particular, the space of ends can be identified to the 1-point compactification of $\Z$, where each ray $n+(-i)\N$ corresponds to one isolated end. 

If $B$ is a clopen subset of the space of ends and does not contain the non-isolated end, then it is finite. It defines a near subaction, and the corresponding index character maps $g_0$ to $0$ and $f$ to $-|B|$. We deduce that this is nonzero unless $B$ is empty, and hence the near action is balanceably indecomposable.
\end{exe}

\begin{exe}
A variant of Example \ref{zzin} consists in replacing $g_n$ by the function $h_n$ defined by $h_n(z)=z+1$ if $\mathrm{Im}(z)=n$ and $h_n(z)=z$ otherwise. Then $fh_nf^{-1}=h_{n+1}$ and $\langle f,h_0\rangle$ is again a wreath product $\Z\wr\Z$. One orbit is $\Z[i]$, and the corresponding Schreier graph is now obtained from the standard Cayley graph of $\Z^2$ by removing horizontal edges outside the horizontal line $\{\mathrm{Im}=0\}$ and adding self-loops instead. The space of ends is a two-point compactification of a discrete set. This one is not balanceably indecomposable, since each vertical line $n+\Z i$ is then a commensurated subset with zero index.
\end{exe}

\begin{exe}
Another variant of Example \ref{zzin} consists in replacing $g_n$ by $k_n$ defined by $k_n(z)=-\bar{z}$ (reflection with respect to the vertical axis) if $\mathrm{Im}(z)\ge 0$, and $k_n(z)=z$ otherwise. Hence $fk_nf^{-1}=k_{n+1}$ and $\langle f,k_0\rangle$ is a wreath product $(\Z/2\Z)\wr\Z$. Then we restrict to the union of two discrete lines $X=\{\pm 1\}+i\Z$, which is a single orbit, on which the action is $\star$-faithful. The Schreier graph is obtained by keeping vertical edges, and add a horizontal edge between $-1+ni$ and $1+ni$ for $n\ge 0$, and self-loops at $\pm 1+ni$ when $n<0$. It has 3 ends (one at the top, two at the bottom). It is balanceably indecomposable for the same reason as in Example \ref{zzin}. This action is faithful, but as a near action, if factors through the near action of Example \ref{firstex}(\ref{exzz2}) of the quotient $\Z\times (\Z/2\Z)$ of $(\Z/2\Z)\wr\Z$.
\end{exe}

\subsection{Analogues of the Burnside ring: examples}\label{s_bure}
We now consider the notions introduced in \S\ref{s_bur}.
\begin{exe}
Let $G$ be a finite group. Then $\MN'(G)=\GN'(G)=0$, while the function $X\mapsto |X|$ induces (semi)group isomorphisms $\MN(G)\simeq\N$, $\GN(G)\simeq\Z$.
\end{exe}

This being mentioned, we now, as in \S\ref{redele}, concentrate on the groups $\Z$ and $D_{\infty}$.

\begin{exe}
In the case $G=\Z$, $\MN'(\Z)=\N^2$ and $\GN(\Z)=\Z^2$. Here, we use that $-\N$ and $\N$ (viewed as commensurated subsets of $\Z$) are the only 1-ended near $\Z$-sets up to isomorphism. The index character, in these coordinates, is given by $(a,b)\mapsto -a+b$ and the number of ends is given by $(a,b)\mapsto a+b$.

The monoid $\MN(\Z)$ is the quotient of $\N\oplus\N^2$ by identification of $(k,\ell,m)$ with $(0,\ell,m)$ whenever $(\ell,m)\neq (0,0)$ (here $k$ corresponds to the cardinal of the union of finite orbits). The associated group $\GN(\Z)$ is just $\Z^2$ and the map $\MN(\Z)\to\GN(\Z)$ is induced by the projection forgetting the first variable.

By Proposition \ref{gnb}(\ref{gnb3}), the index character induces an isomorphism $\GN^{\bowtie}(\Z)$ $\to\Z$.
\end{exe}

Now consider the infinite dihedral group $D=D_{\infty}$. It admits three infinite transitive $D$-sets, denoted $d,e,e'$, where $d$ corresponds to the simply transitive action, and $e,e"$ are 1-ended actions (corresponding to the two conjugacy classes of elements of order 2). Denote by $s$ the singleton with trivial action.

\begin{exe}
The monoid $\MN(D_\infty)$ has the abelian monoid presentation
\[\langle d,e,e',s\mid d=e+e',e+s=e_1,e'+s=e_1\rangle.\]
The function mapping a near $D_\infty$-set to its number of ends induces a homomorphism to $\N$, mapping $(d,e,e,s')$ to $(2,1,1,0)$. 

In particular, the associated group $\GN(D_\infty)$ is isomorphic to $(\Z/2\Z)\times\Z$, with generators $s=(1,0),e=(0,1)$, and identification of $d$ to $2e+s$ $e'$ to $e+s$. The monoid $\MN'(D_\infty)$ is isomorphic to $\N$, with generators $e=e'$ and $d=e$ (and $s=0$), and $\GN'(D_\infty)$ is infinite cyclic.

The group $\MN^{\bowtie}$ is reduced to zero, by Proposition \ref{gnb}(\ref{gnb3}).

Actually, the simple descriptions of $\GN(D)$, $\MN'(D)$ and $\GN'(D)$ are easier to check directly. Let us prove the description of $\MN(D)$. First the above presentation, removing the redundant generator $d$, we obtain the presentation $\langle e,e',s\mid e+s=e',e'+s=e\rangle$, and in turn removing the redundant generator $e'$, it has the abelian monoid presentation $\Gamma=\langle e,s\mid,e=e+2s\rangle$. To check that the canonical surjective monoid homomorphism $\Gamma\to\MN(D)$ is an isomorphism, we provide a set of representatives in $\Gamma$ and show they are mapped to pairwise distinct elements. By a straightforward argument, such a set is given by $\{ns:n\in\N\}\cup\{ne:n\in\N\}\cup\{ne+s:n\in\N_{>0}\}$. Since the function $X\mapsto |X|$ is an invariant, and maps $s^n$ to $n$ and others to $\aleph_0$, we are reduced to check the injectivity on $\{ne:n\in\N\}\cup\{ne+s:n\in\N_{>0}\}$. Since the function mapping $X$ to its number of ends is invariant as well, this reduces to check, for every $n$, the injectivity on $\{ne,ne+s\}$, that is, to show $ne$ and $ne+s$ are not balanceably near isomorphic. Indeed, this is equivalent to the statement that the near $D$-set corresponding to $ne$ has $\lst(ne)\neq 1$. Since $\lst(ne)=\lst(e)$, we have to check that $\lst(e)\neq 1$; this latter point was proved as part of Proposition \ref{lestdinfty}.
%, so we do not copy the argument here.
\end{exe}

%While we will not try to write the somewhat cumbersome description of $\MN^{\bowtie}(\Z)$, is not hard to check (using that almost actions are realizable, and the remark from \S\ref{burnside} that realizable near actions vanish $\GN^{\bowtie}) that $\GN^{\bowtie}(\Z)$ is reduced to $\Z$, where an identification is given by the index character map into $\Hom(\Z,\Z)=\Z$. 

%Note that the associated groups are then reduced to 0, because for every near 

\subsection{Near actions of amalgams of finite groups}\label{s_amalgam}

Denote by $C_n=\Z/n\Z$. When $m$ divides $n$, identify $C_m$ to a subgroup of $C_n$.

\begin{prop}\label{oddamal}
Let $F_1\supset F\subseteq F_2$ be finite groups. Suppose that every action of $F$ on any finite set extends to $F_1$ (e.g., $F$ is a retract of $F_1$). Then every near action of the amalgam $G=F_1\ast_F F_2$ is realizable.
In particular, every near action of $\mathrm{SL}_2(\Z)=C_6\ast_{C_2}C_4$ is realizable.
\end{prop}
\begin{proof}
Let $X$ be a near $G$-set. Choose separately realizations $\alpha_1$ and $\alpha_2$ of $F_1$ and $F_2$. Consider the finite subset $K\subseteq X$ of those $x$ such that $\alpha_1(g)x\neq\alpha_2(g)x$ for some $g\in F$. Let $Y$ be the smallest $\alpha_1(F_1)$-invariant subset including $K$. Since the $\alpha_2(F)$-action coincides with the $\alpha_1(F)$-action outside $Y$ and the latter stabilizes $Y$, we deduce that $\alpha_2(F)$ stabilizes $Y$. Hence, by the assumption, we can extend it to a new action of $F_1$ on $Y$. Then define $\alpha'_1$ as the new $F_1$-action on $Y$ and coinciding with $\alpha_1$ outside $Y$. Then $\alpha'_1$ and $\alpha_2$ match on $F$, and hence yield an action of the amalgam, which realizes the near action.

The last fact follows since $C_2$ is a direct factor in $C_6$.
\end{proof}

\begin{rem}
Let $F\subseteq F_1$ be an inclusion of finite groups. As we have noticed, the condition that every action of $F$ extends to $F_1$ holds when $F$ is a retract of $F$. It also holds for other pairs: for instance, if $p$ is prime, it is satisfied by the inclusion of a cyclic group of order $p$ into the symmetric group $\mks_p$. Moreover, it follows from the proof that it is enough to assume that there exists $n_0$ such that every action of $F$ on a set of cardinal $\ge n_0$ extends to $F_1$. This gives some further such pairs. For instance, if $F$ is a subgroup of order 3 in the alternating group $F_1=\mathrm{Alt}_5$, then every action of $F$ on a set of cardinal $\ge 10$ extends to $F_1$ (the reader can check this, as well as the fact that 10 is optimal).
\end{rem}

This result is not true for arbitrary amalgams of finite groups. Let us illustrate it in one case.

Fix a prime $p$. Let $n\ge 1$ be divisible by $p$; let $X$ be a set with a near action of the amalgam 
\[G=C_{pn}\ast_{C_p}C_{p^2}=\langle t,u:t^{pn}=u^{p^2}=1,t^n=u^p\rangle.\] 
Choose a realization of the $C_{pn}$-action and of the $C_{p^2}$-action, giving $t,u$ permutations of $X$ with $t^{pn}=u^{p^2}=\mathrm{id}$. Write $F(t,u)=\{x:t^nx\neq u^2x\}$, which is finite. Choose a finite $u$-invariant subset $Y$ including $F(t,u)$. Then $t^n$ coincides with $u^p$ outside $Y$, and hence $Y$ is $t^n$-invariant. 

\begin{prop}\label{parity}
The value of the number of $p$-cycles of $t^n|_Y$ only depends on the near action (and does not depend on the choice of $t,u,Y$). 

This is a near isomorphism invariant of near $G$-sets; it is additive under disjoint unions; it is 0 in $\Z/p\Z$ if and only if the near action on $X$ is realizable, if and only if it is stably realizable. Every near $G$-set is completable.
\end{prop}
\begin{proof}
Define this as a function $f(t,u,Y)$ (defined under the condition $F(t,u)\subseteq Y$ and $Y$ $u$-invariant).

First, we check that $t,u$ being given, it does not depend, modulo $p$, on the choice of $Y$. It is enough to check that it does not change when one passes from $Y$ to $Y\sqcup C$ with $C$ a single $u$-orbit. Since $F(t,u)\subseteq Y$, the permutations $t^n$ and $u^p$ coincide on $C$. If $C$ is a 1-cycle or $p$-cycle of $u$, then $u^p$ is the identity on $C$ and hence so is $t^n$, so $f(t,u,Y)=f(t,u,Y')$. If $C$ is a $p^2$-cycle of $u$, then $u^p$ is a disjoint product of $p$ $p$-cycles in $C$, and hence so is $t^n$, and then $f(t,u,Y')=f(t,u,Y')+p$.

So the function $f(t,u,Y)$ does not depend, modulo $p$, on $Y$. Write its value modulo $p$ as $g(t,u)$. To show that $g(t,u)$ does not depend on $t,u$ under near equality (respecting the condition $t^{pn}=u^p=\mathrm{id}$), it is enough to consider the cases when one ``adds" a cycle inside the same of fixed points. Namely:

\begin{itemize}
\item One passes from $u$ to $u'$ by $u'=uc$, $c$ one cycle with support disjoint of $u$. Choose $Y$ $u'$-invariant (so it is automatically $u$-invariant) and including $F(t,u)\cup F(t,u')$. Then $t^n$ is not affected on $Y$, and hence $f(t,u,Y)=(t,u',Y)$. Thus $g(t,u)=g(t,u')$.

\item One passes from $t$ to $t'$ by $t'=tc$, $c$ one cycle with support disjoint of $u$. Choose $Y$ large enough ($u$-invariant) and including $F(t,u)\cup F(t',u)$ as well as the support of $c$. Then $(c^{n/p})^{p^2}=\mathrm{id}$; it is made of cycles all of the same length. If it consists of $p$-cycles or 1-cycles, then $c^n$ is the identity and hence $f(t,u,Y)=f(t',u,Y)$. Otherwise, $c^{n/p}$ is made of $k$ $p^2$-cycles, and hence $c^n$ is made of $pk$ 2-cycles, so $f(t',u,Y')=f(t,u,Y)+pk$, and hence $g(t',u)=g(t,u)$.
\end{itemize}

Additivity is clear from the definition. If the action is realizable, we can choose $Y$ to be empty, so it is obvious that the invariant is even. Conversely, if the invariant is even, fix $Y$. Then $t^n$ has order $\le 2$, and has an even number of 2-cycles on $Y$. Let $U$ coincide with $u$ outside $Y$, and let it be equal to a square root of $t^n$ on $Y$. Then $t^n=u^2$ and $\langle t,u\rangle$ realizes the near action.

Clearly taking the disjoint union with a trivial near action does not affect the invariant, and thus realizable is equivalent to stably realizable. For an arbitrary near $G$-set, the disjoint union of $p$ copies of $X$ has a zero invariant and hence is realizable, so $X$ is completable.

Clearly this is a balanced isomorphism invariant. Since a finite set has a zero invariant (since it is realizable), this implies that it is also a near isomorphism invariant.
\end{proof}

\begin{rem}
The invariant is additive under disjoint unions. For instance, for $G=C_{2n}\ast_{C_2}C_4$, this has the following consequence: for every near $G$-set $X$ the near $G$-set $X\sqcup X$ is realizable. Also, for any two near $G$-sets $X,Y$, at least one of $X$, $Y$, $X\sqcup Y$ is realizable (and either a single one, of all three).
\end{rem}

The next proposition shows that this invariant is not trivial.

\begin{prop}\label{invamalgsu}
For $n\ge 2$ divisible by $p$, let $G=C_{pn}\ast_{C_p} C_{p^2}$ act on itself (by left translation); write the generators as $t,u$ with $t^{pn}=u^{p^2}=1$, $t^n=u^p$. Let $X$ be the subset of elements whose projection in the free product $C_n\ast C_p$ has a reduced form not finishing by a nontrivial power of $u$ (i.e., 1 or finishing by some nontrivial power of $t$). Then $X$ is commensurated, the corresponding near action is balanced and its invariant defined by Proposition \ref{parity} is equal to 1. In particular, this yields a non-realizable (and not stably realizable) near action of $G$.
\end{prop}
\begin{proof}
To check that it is commensurated, it is enough to work in the quotient $C_n\ast C_p$ by the central subgroup $C_p$. Let $\bar{X}$ be the set of reduced words not finishing by a nontrivial power of $u$ (nontrivial in the quotient!), and $Z$ the set of reduced words finishing by a nontrivial power of $t$ (so $Z\subseteq \bar{X}$ and $\bar{X}\smallsetminus Z=\{1\}$.
 If $w\in Z$, then $\langle u\rangle w\cup\langle t\rangle w\subseteq \bar{X}$. So $\bar{X}$ is commensurated. Hence $X$ is commensurated as well. Its index character is trivially zero since $\mathrm{Hom}(G,\mathbf{Z})=0$.
 
More precisely, we see that $\bar{X}$ is $\langle t\rangle$-invariant. We thus choose $t$ itself as lift of $t$, and define $u'$ as equal to (the left translation by) $u$ outside $\langle u^p\rangle$, an as the identity on $\langle u^p\rangle$. We now treat $t,u'$ as permutations of $X$. Then $F(t,u')=\langle u^p\rangle$, which is already $u'$-invariant, we choose $Y=\langle u^p\rangle=\langle t^n\rangle$. On $Y$, $t^n$ acts as a $p$-cycle. So $f(t,u',Y)=1$ and the invariant is indeed odd.
\end{proof}

\begin{cor}
Let $n$ be divisible by $p$ and $G=C_{pn}\ast_{C_p}C_{p^2}$. Then the map mapping a near $G$-set $X$ to the invariant of Proposition \ref{parity} induces a monoid isomorphisms $\MN^{\bowtie}(G)\to\GN^{\bowtie}(G)\to\Z/p\Z$ (where $\GN^{\bowtie}$ is defined in \S\ref{s_bur}).
\end{cor}
\begin{proof}
Since it is additive under disjoint unions by Proposition \ref{parity}, it induces a group homomorphism. It is surjective by Proposition \ref{invamalgsu}. For the injectivity, suppose that $X,Y$ are near $G$-sets such that the $X-Y$ is an element of the kernel. This means that there exists a near $G$-set $Z$ such that $X\sqcup Z$ and $Y\sqcup Z$ have the same invariant. Using surjectivity of the invariant, we can choose $Z$ so that $X\sqcup Z$ and $Y\sqcup Z$ both have a zero invariant. Hence, by Proposition \ref{parity}, $X\sqcup Z$ and $Y\sqcup Z$ are both (stably) realizable, and hence by Proposition \ref{gnb}(\ref{gnb1}), they both vanish in $\GN^{\bowtie}(G)$. Hence $X-Y$ also vanishes in $\GN^{\bowtie}(G)$.
\end{proof}

%By Proposition \ref{gnb}(\ref{gnb3}), the index character map induces an isomorphism $\GN^{\bowtie}(\Z)$ $\to\Z$.

\section{Near actions of locally finite groups}\label{s_nlofi}

\subsection{Countable locally finite groups}\label{s_lofi}

\begin{thm}\label{coulofi}
Let $G$ be a countable locally finite group and $X$ a near $G$-set.
\begin{enumerate}
\item\label{coulofi1} The near action on $X$ can be completed to an action on a set $X\sqcup Y$ such that the intersection of $X$ with every orbit is finite.
\item\label{coulofi2} Suppose in addition that we can write $G$ as ascending union of subgroups $\bigcup G_n$, such that $G_n$ is a retract of $G_{n+1}$. Then the near action on $X$ is realizable by an action with only finite orbits.
\end{enumerate}
\end{thm}

(Recall that a subgroup $H$ of a group $G$ is a retract if and only if there exists a normal subgroup $N$ of $G$ such that $G=N\rtimes H$.) 
\begin{rem}
Note that the condition that the set is countable cannot be dropped in the both assertions of Theorem \ref{coulofi}. Indeed, if we have an action of a countable group with uncountably many infinite orbits, then every finite perturbation of the action still has uncountably many infinite orbits. See however Corollary \ref{loficom}.
\end{rem}

In order to prove Theorem \ref{coulofi}, we emphasize the following lemma, whose proof is shorter than its very assertion.

\begin{lem}\label{extendaction}
For every finite group $F$ and subgroup $H$, every action of $F$ on a finite set can be extended to an $H$-action on a larger finite set, in a way that elements in different orbits remain in different orbits for the extended action.
\end{lem}
\begin{proof}
It is enough to consider the case of a transitive $F$-set $X=F/L$, and then $H/L$ does the job.
\end{proof}

\begin{proof}[Proof of Theorem \ref{coulofi}]
We prove both assertions with essentially the same argument. Fix an ascending sequence $(G_n)_{n\ge 0}$ of finite subgroups with union $G$ and $G_0=\{1\}$. Let $X=X_0$ be a countable near $G$-set, with trivial action of $G_0=\{1\}$. Write $X$ as ascending union of finite subsets $F_n$.

%. Define $X_0=X$ and realize the $G_0$-action on $X$.

 Define by induction a near $G_n$-set $(X_n,\rho_n)$ with, for $n\ge 1$, a $G_{n-1}$-equivariant inclusion $X_{n-1}\subseteq X_n$ with cofinite image, inducing on the cofinite subset $X_0$ the given near action of $G_n$.

To achieve this, fix a realization $\rho'_n$ of the near action of $G_n$ on $X_0$. Assuming that $X_{n-1}$ is constructed, extend it as a trivial action on $X_{n-1}\smallsetminus X_0$ (still denoted $\rho'_n$). There exists a cofinite subset $Y_n$ of $X_{n-1}$, invariant for both $\rho'_n(G_n)$ and $\rho_{n-1}(G_{n-1})$, on which both actions of $G_{n-1}$ coincide. Let $Z_n$ be its complement in $X_n$; we can in addition enlarge $Z_n$ to ensure that it includes $F_n\cup Z_{n-1}$. Then, by Lemma \ref{extendaction} the action $\rho_{n-1}$ of $G_{n-1}$ on $Z_n$ can be extended, on a larger finite set $Z'_n$, to an action $\rho_n$ of $G_n$, so that elements of distinct $G_{n-1}$-orbits remain in distinct $G_{n}$-orbits. Extend the action $\rho_n$ of $G_n$ as equal to $\rho'_n$ on $Y_n$, yielding an action $\rho_n$ on $X_n=Y_n\sqcup Z'_n$, whose restriction to $X_{n-1}$ extends the action $\rho_{n-1}$ of $G_{n-1}$. Hence, the resulting action $\rho$ of $G$ on the union $X'=\bigcup X_n$ commensurates $X=X_0$ and the induced near action is the original one.

% In addition, if $x\in X=X_0$, then $x\in F_k$ for some $k$; we claim that $\rho(G)x\cap X_0\subseteq Z_k$; namely we show by induction on $n\ge k-1$ that $\rho_n(G_n)x\cap X_0\subseteq Z_k$. Indeed, 
 
 The extension process ensured that $\rho(G_n)x\cap X_n\subseteq \rho(G_{n-1})x$ for all $n$ and $x\in Z_{n}$. Fixing $n$ and $x\in Z_n$, by induction on $m\ge n$, we claim that also have $\rho(G_m)x\cap X_n\subseteq \rho(G_{n-1})x$. So this holds for $m=n$; suppose that it holds for $m-1\ge n$. Then
\begin{align*}
\rho(G_m)x\cap X_n & = \rho(G_m)\cap X_m\cap X_n\\
& \subseteq \rho(G_{m-1})x\cap X_n\subseteq \rho(G_{n-1})x,
\end{align*}
where the last inclusion is by induction hypothesis. Hence $\rho(G)x\cap X_n\subseteq \rho(G_{n-1})x$ for all $x\in Z_{n}$. 

Given $x\in X$, there exists $n$ such that $x\in F_{n}$, so $x\in Z_n$, and so $\rho(G)x\subseteq X\subseteq \rho(G)x\cap X_n\subseteq \rho(G_{n-1})x$, which is finite. Thus every $G$-orbit has a finite intersection with $X$.

In the above construction, if $G_{n-1}$ is a retract of $G_n$, then the action of $G_{n-1}$ on $Z_n$ defines a homomorphism $G_{n-1}\to\mks(Z_n)$, which can be extended to a homomorphism $G_n\to\mks(Z_n)$ with the same image (and thus defining the same orbits on $Z_n$): in other words, we can choose $Z'_n=Z_n$ for all $n$, so that $X_n=X_{n-1}$. If this condition holds for all $n\ge 1$, then we can arrange $X_n=X_{n-1}$ for all $n\ge 1$, so we have $X'=X$ which is thus realizable with finite orbits.
\end{proof}

%If moreover $X$ is countable, write $X$ as ascending union of finite subsets $F_n$, then we can always enlarge $Z_n$ to ensure that $F_n\cup Z_{n-1}\subseteq Z_n$ for all $n$; moreover, in the setting of (\ref{lofi1}) extending the action from $G_{n-1}$ to $G_n$ on each orbit included in $Z_n$, we can assume that the orbits are the same. It follows that $Z_n$ is finite and $G$-invariant for ll $n$, and $\bigcup Z_n=X$.

%In the setting of (\ref{lofi1}), we can still extend the $G_{n-1}$ on each of its orbit in $Z_n$; this ensures that the orbital equivalence relation of $G_{n-1}$ and $G_n$, restricted to $Z_n$, are the same. In particular, for the orbital equivalence of $G$, each $Z_n\smallsetminus Z_{n-1}$ is saturated. Hence each orbit meets $X$ in a finite subset.

\begin{cor}\label{loficom}
Let $G$ be a countable locally finite group.
\begin{enumerate}
\item\label{lofi1}Every near action of $G$ is completable. 
%In addition, every near action on a countable set $X$ can be completed as an action on a set $X\sqcup Y$ such that the intersection of $X$ with every orbit is finite.
\item\label{lofi2} If moreover we can write $G$ as ascending union of subgroups $\bigcup G_n$, such that $G_n$ is a retract of $G_{n+1}$ for all $n$ (e.g., $G$ is a restricted direct product of a countable sequence of finite groups), then every near action of $G$ is realizable. 
%Moreover, if the set is countable, the realizing action can be chosen to have only finite orbits. In particular, this applies if $G$ is a restricted direct product $\bigoplus_nF_n$ of a sequence of finite groups.
\end{enumerate}
\end{cor}
\begin{proof}
By Proposition \ref{kapparea}, every near action of $G$ is disjoint union of a realizable near action and a near action on a countable set. We conclude by Theorem \ref{coulofi}.
\end{proof}

\begin{rem}
There exist (abelian) countable locally finite groups with near actions (on countable sets) that are not stably realizable, see \S\ref{s_qcy}.

There exist (abelian) uncountable locally finite groups, of every uncountable cardinal, with a non-completable near action (on a countable set), see \S\ref{unlofi}.
\end{rem}

\begin{cor}\label{endslocfin}
Let $G$ be a countable locally finite group and $X$ an infinite near $G$-set. Then $\mathcal{P}_G^\star(X)$ (see \S\ref{s_soe}) has cardinal $2^{|X|}$. Accordingly, in the space of ends $\mathcal{E}_G^\star(X)$ there is no nonempty clopen metrizable subset, and in particular there is no isolated point.
%\begin{itemize}
%\item every nonempty $\mathrm{G}_\delta$ subset (= countable intersection of open subsets) has nonempty interior;
%\item there is no nonempty clopen metrizable subset;
%\item there is no isolated point.
%\end{itemize}
\end{cor}
\begin{proof}
By Theorem \ref{coulofi}, there exists a $G$-action on $Z=X\sqcup Y$ commensurating $X$, inducing the original near action on $X$, such that every $G$-orbit has a nonempty finite intersection with $X$. Necessarily, $Z$, as well as the set $\Omega$ of orbits, has the same cardinal as $X$. Then the Boolean algebra homomorphism $\mathcal{P}(\Omega)\to \mathcal{P}_{(G)}(X)$, mapping $U$ to $(\cup U)\cap X$, induces an injective Boolean algebra homomorphism $\mathcal{P}^\star(\Omega)\to \mathcal{P}_G^\star(X)$, whence the first statement.

%For the consequences, each follows from the previous one, but we start with the simple argument for the second one: a

A metrizable clopen subset would mean a $G$-commensurated subset $X'$ in $X$ whose Boolean algebra of commensurated subsets is countable, and this is excluded by the first statement (applied to $X'$).
\end{proof}

%Let us prove the first consequence; we can suppose that $X$ is a $G$-set. We have to show that for every $\omega\in\mathcal{E}_G^\star(X)$ and every descending sequence $(U_n)_{n\ge 0}$ of clopen neighborhoods of $\omega$, the intersection $\bigcap U_n$ has nonempty interior. Indeed, $U_n$ is the boundary of some $G$-commensurated subset $X_n$ of $X$; then $X_{n+1}$ is near included in $X_n$ for every $n$. Write as an ascending union $G=\bigcup G_n$ of finite subgroups $G_n$. We define by induction, a finite $G_n$-invariant subsets $Y_n$, that is disjoint to $\bigcup_{j<n}(Y_j)$. Define $Y=\bigcup Y_n$; it is infinite, $G$-commensurated, and near included in $X_n$ for every $n$. Hence the nonempty clopen subset $U$ defined by $Y$ is included in $\bigcap U_n$.

\begin{cor}
Let $G$ be a group. The following are equivalent:
\begin{enumerate}
\item\label{notclf} $G$ is not countable locally finite;
\item\label{1engs} there exists a 1-ended near $G$-set;
\item\label{mengs} there exists a near $G$-set whose space of ends is non-empty and metrizable;
\item\label{2egs} there exists a 2-ended $G$-set.
\item\label{1egs} $G$ is 2-ended, or $G$ admits a 1-ended quotient.
\end{enumerate}
\end{cor}
\begin{proof}
(\ref{1egs})$\Rightarrow$(\ref{2egs})$\Rightarrow$(\ref{1engs})$\Rightarrow$(\ref{mengs}) is clear; (\ref{mengs}) implies (\ref{notclf}) by Corollary \ref{endslocfin}.

Finally, let us prove that (\ref{notclf}) implies (\ref{1egs}). If $G$ is uncountable and locally finite, then it is 1-ended (Holt's theorem \cite[Theorem 6.10]{DD}) and we are done. Now suppose that $G$ is not locally finite. If $G$ has $\ge 3$ ends, then again by \cite[Theorem 6.10]{DD}, it admits a minimal action on a tree with at least 3 ends, with finite edge stabilizers, and by \cite{DGO}, this implies that is has a free normal subgroup $F$ of infinite index (necessarily of infinite rank). Modding out by $[F,F]$ yields a group with an infinite normal abelian subgroup, which thus has to be a 1-ended group.
\end{proof} 

Actually, the last two statements of Corollary \ref{endslocfin} can be improved.

%be obtained directly. 

\begin{cor}\label{endslocfin2}
Let $G$ be a countable locally finite group and $X$ an infinite near $G$-set. Then every nonempty $\mathrm{G}_\delta$ subset (= countable intersection of open subsets) of $\mathcal{E}^\star_G(X)$ has an infinite interior.
\end{cor}
\begin{proof}
Since $X$ is completable to a countable $G$-set and the required property passes to clopen subsets, we can suppose that $X$ is a $G$-set.

We have to show that for every $\omega\in\mathcal{E}_G^\star(X)$ and every descending sequence $(U_n)_{n\ge 0}$ of clopen neighborhoods of $\omega$, the intersection $\bigcap U_n$ has nonempty interior. Indeed, $U_n$ is the boundary of some $G$-commensurated subset $X_n$ of $X$; then $X_{n+1}$ is near included in $X_n$ for every $n$. Write $G$ as an ascending union $G=\bigcup G_n$ of finite subgroups $G_n$. We define by induction, a finite $G_n$-invariant subsets $Y_n$, that is disjoint to $\bigcup_{j<n}(Y_j)$. Define $Y=\bigcup Y_n$; it is infinite, $G$-commensurated, and near included in $X_n$ for every $n$. Hence the nonempty clopen subset $U$ defined by $Y$ is included in $\bigcap U_n$. Since $\mathcal{E}^\star_G(X)$ has no isolated point by Corollary \ref{endslocfin}, $U$ is infinite.
\end{proof}

A Parovi\u{c}enko space\index{Parovi\u{c}enko space} is a Stone space whose Boolean algebra has cardinal $2^{\aleph_0}$ and in which every nonempty $\mathrm{G}_\delta$ subset has nonempty interior. It is unique up to homeomorphism if and only if the continuum hypothesis (CH) holds (see \cite[\S 1.2]{vM}), and one example is the Stone-\u{C}ech boundary of an infinite countable set. By the previous proposition, for $G$ a countable locally finite group, the space of ends of any infinite countable near $G$-set, is a Parovi\u{c}enko space.

\begin{cor}\label{c_ifch}
If CH holds, then for every countable locally finite group $G$, the space of ends of every infinite countable near $G$-set is homeomorphic to the Stone-\u{C}ech boundary of $\N$.
\end{cor}

Corollaries \ref{endslocfin}, \ref{endslocfin2} and \ref{c_ifch} were established by Protasov \cite{Pro11} in the case of $G$ itself, viewed as $G$-set under left multiplication. 

%The arguments used here in the case of a $G$-set, are of similar flavor. 

%. Replacing $X_n$ with $\bigcap_{i=0}^nX_i$, we can suppose that $(X_n)$ is a descending chain.
%, we can also suppose, replacing $X_n$ with its cofinite subset $\bigcap_{g\in G_n}X_n$, that $X_n$ is $G_n$-invariant. We define, by induction on $n$, a partition $X_

%Choose $Y\in\mathcal{P}_G^\star(X)$ such that $\omega(Y)\neq 0$. ??????????????

Theorem \ref{coulofi}(\ref{coulofi1}) states in particular that every near action of a countable locally finite group is completable. This has the following extension:

\begin{lem}\label{locfi_compl}
Let $G$ be a countable group with a finitely presented subgroup $H$, such that for every finite subset $F$, the subgroup $\langle H\cup F\rangle$ is an overgroup of finite index of $H$. Let $X$ be a near $G$-set which, in restriction to $H$, has only finite orbits. Then $X$ is a completable near $G$-set.
\end{lem}
\begin{proof}
Enumerating elements of $G$, one writes $G$ as ascending union of a sequence of subgroups $G_n$, with $G_0=H$ and $G_n$ having finite index in $G_{n+1}$ for every $n$. By Theorem \ref{nearactionfp} and Proposition \ref{comreafin}, the near $G_0$-action is realizable (necessarily, as an action with only finite orbits). The sequel of the proof is then an adaptation of that of Theorem \ref{coulofi}(\ref{coulofi1}) so we omit the details.
\end{proof}

\begin{lem}\label{lemqr1}
Let $G$ be a non-cyclic, torsion-free abelian group of $\Q$-rank 1, with an infinite cyclic subgroup $H$. Let $X$ be a near $G$-set which, in restriction to $H$, has only finitely many finite orbits. Then $X$ is a realizable near $G$-set.
%Let $G$ be a countable abelian group of $\Q$-rank 1, and $H$ an infinite cyclic subgroup of $G$. Let $X$ be a balanced near $G$-set that is near free in restriction to $H$. Then $X$ is a realizable near $G$-set.
\end{lem}
\begin{proof}
Since $G$ is non-cyclic, torsion-free abelian group of $\Q$-rank 1, we have $\Hom(G,\Z)=\{0\}$ and hence every near $G$-action is balanced. 

%Let $H=\langle h\rangle$ be an infinite cyclic subgroup of $G$.

%Let $X$ be a near free near $G$-set. 
Since the index is zero, the near action is realizable in restriction to $H$; fix a realization, namely a permutation $\sigma_0$ realizing the near action of some generator; since $\sigma_0$ has only finitely many finite cycles, we it has a finite perturbation $\sigma$ that consists only of infinite cycles. We claim that for every subgroup $L$ of $G$ including $H$ as a subgroup of infinite index, there exists a unique realization on $L$ extending the $H$-realization. Let $H$ have index $m$ in $L$: we have to show that there exists a unique $m$-root $\psi$ of $\sigma$ realizing the near action.

The uniqueness is easy: if $\psi$ and $\theta$ are such roots, then $\psi^{-1}\theta$ commutes with $\sigma$ and has finite support; since $\sigma$ has no finite orbit, it has to be the identity.

For the existence, for the same reason as we constructed $\sigma$, we can find $\theta$ realizing freely the $L$-action. Hence $\theta^m$ is another free realization for $H$. Hence, there exists a finitely supported permutation conjugating $\theta^m$ to $\sigma$. Conjugating $\theta$ instead, we obtain the desired $m$-root.
\end{proof}

\begin{thm}
Let $G$ be a torsion-free abelian group of $\Q$-rank 1. Then every near $G$-set is completable.
% , and $H$ a finitely generated subgroup such that $G/H$ is torsion. Then a near $G$-set is completable if and only if it is completable as near $H$-set. In particular, if $G$ has $\Q$-rank 1, then every near $G$-set is completable.Let $G$ be a countable abelian group of $\Q$-rank 1. 
\end{thm}
\begin{proof}
Let $H$ be an infinite cyclic subgroup of $G$. Let $X$ be a near $G$-set. Using a near Schreier graph for the near $H$-set $X$, write $X=Y\sqcup Z$, where $Z$ is the union of all finite components. Since this decomposition is canonical up to near equality and since $H$ is a normal subgroup of $G$, the subsets $Y,Z$ are $G$-commensurated. By Lemma \ref{locfi_compl}, $Z$ is a completable $G$-set; by Lemma \ref{lemqr1}, $Y$ is a realizable near $G$-set if $G$ is not cyclic (the result being clear when $G$ is cyclic).
\end{proof}

\subsection{Near free near actions of quasi-cyclic groups}\label{s_qcy}

\begin{prop}
Let $F$ be a finite group and $X$ a near free $F$-set. Let $\rho:F\to\mks(X)$ be a realization of the near action and let $\nu_\rho$ be the number of points with nontrivial stabilizer. Then the value of $\nu_\rho$, modulo $|F|$, does not depend on $\rho$ (and only on the near action). In particular, it is a balanced isomorphism invariant; call it its residual.\index{residual}
%(thus write it as $\nu_\alpha$ where $\alpha$ is the given near action).
\end{prop}
\begin{proof}
More generally fix a subgroup $H$ of $F$ and let us show the following: let $X$ be a near $F$-set such that, for some/any realization $\rho$, the number $\nu_\rho$ of $x\in X$ such that the stabilizer of $x$ is not conjugate to $H$, is finite. Then, modulo $|F|/|H|$, the number $\nu_\rho$ does not depend on $\rho$. 

Indeed, consider two realizations $\rho_1,\rho_2$. Let $Y$ be the set of $x$ such that $\rho_1(g)x=\rho_2(g)x$ for all $g\in F$, so $Y$ is cofinite. Then, for all $g,h\in F$ and $x\in Y$, we have
\[\rho_1(g)(\rho_1(h)x)
=\rho_1(gh)x=\rho_2(gh)x=\rho_2(g)(\rho_2(h)x)=\rho_2(g)(\rho_1(h)x),\]
so $Y$ is $\rho_1(F)$-invariant, and similarly $Y$ is $\rho_2(F)$-invariant, and the actions $\rho_1$ and $\rho_2$ coincide on $Y$. For $i=1,2$, let $Z_i$ be the (finite) set of points whose stabilizer for $\rho_i$ is not conjugate to $H$. Then $\nu_{\rho_i}=|Z_i\cap Y|+|Z_i\cap Y^c|$. Since $Z_1\cap Y=Z_2\cap Y$, and $|Z_i\cap Y^c|=|Y^c|\cap |Y^c\smallsetminus Z_i|$,
we deduce that 
\[\nu_{\rho_1}-\nu_{\rho_2}=|Z_1\cap Y^c|-|Z_2\cap Y^c|=|Y^c\cap Z_2|-|Y^c\cap Z_1|.\] 
Since $Y^c\smallsetminus Z_i$ consists of $\rho_i(F)$-orbits isomorphic to $F/H$, its cardinal is multiple of $|F|/|H|$, whence the result.
\end{proof}

%Let now $m=(m_n)$ be a sequence of positive integers, where $m_n$ divides $m_{n+1}$ for all $n$. Define $\underrightarrow{C_m}$ as the inductive limit of $\Z/m_n\Z$ (with the map $\Z/m_n\Z\to \Z/m_{n+1}\Z$ being induced by multiplication by $m_{n+1}/m_n$); through these maps we view $\Z/m_n\Z$ as a subgroup of $\underrightarrow{C_m}$. Also define the $m$-adic ring $\Z_m$ as the projective limit of $\Z/m_n\Z$, with maps induced by the identity (for $m=(k^n)_{n\ge 0}$ is corresponds to the classical notion of $k$-adic ring; for $m=(n!)_{n\ge 0}$ it corresponds to the profinite completion of $\Z$). Note that $\Z_m$ is uncountable as soon as $m$ is unbounded.

Given an integer $m\ge 2$, recall that $C_{m^\infty}$ is defined as the inductive limit of $\Z/m^n\Z$ (with the map $\Z/m^n\Z\to \Z/m^{n+1}\Z$ being induced by multiplication by $m$); through these maps we view $C_{m^n}\simeq \Z/m^n\Z$ as a subgroup of $C_{m^\infty}$. Note that $C_{m^\infty}\simeq\Z[1/m]/\Z$. Also recall that the $m$-adic ring $\Z_m$ as the projective limit of $\Z/m^n\Z$, with maps induced by the identity map of $\Z$. 

% (for $m=(k^n)_{n\ge 0}$ is corresponds to the classical notion of $k$-adic ring; for $m=(n!)_{n\ge 0}$ it corresponds to the profinite completion of $\Z$). Note that $\Z_m$ is uncountable as soon as $m$ is unbounded.

Given a near free near $\Z/m^{n+1}\Z$-set $(X,\alpha)$, it is also a near free near$\Z/m^n\Z$-set. The residual as $\Z/m^{n+1}\Z$-action reduces, modulo $m^n$, to the residual as $\Z/m^n\Z$-action. 

Hence, given a near free near $C_{m^\infty}$-set $(X,\alpha)$, the residual (of the restriction to the various subgroups $\Z/m^n\Z$) defines an $m$-adic number, which we call residual of the given near free near $C_{m^\infty}$. It is a balanced isomorphism invariant, denoted $\nu_\alpha$, or $\nu_X$. Note that is is additive under disjoint union: $\nu_{X\sqcup Y}=\nu_X+\nu_Y$, and that for $X$ finite of cardinal $n$, $\nu_X=n$. For any finite set $T$, we thus have $\nu_{X\sqcup T}=\nu_X+|T|$. In particular, the image of $\nu_X$ in $\Z_m/\Z$ is a isomorphism invariant of the near free near $C_{m^\infty}$-set.

\begin{prop}
Every element of $\Z_m$ can be achieved as residual of some near free, countable near $C_{m^\infty}$-set.

A near free near $C_{m^\infty}$-set $X$ is stably realizable if and only if it is finitely stably realizable, if and only if $\nu_X\in\Z$, and realizable if and only if $\nu_X\in\N$.
\end{prop}
\begin{proof}
%Suppose that $m$ is finitary: this means that the subgroup $W$ of $C_{m^\infty}$ generated by elements of finite order is finite (say of order $k$, a product of distinct primes).
%
Since every nontrivial subgroup of $C_{m^\infty}$ intersects $C_m$ non-trivially, a $C_{m^\infty}$-set is free if and only if it is $C_m$-free.

If $X$ is realizable, let $Y$ be the set of points with trivial stabilizer in restriction to $C_m$; then $Y$ is cofinite. Then in restriction to each $\Z/m_n\Z$ (for $n$ large enough so that $k$ divides $m_n$), the residual is equal to $|X\smallsetminus Y|$, and hence this is the residual.

If $X$ is stably realizable, then $X\sqcup Z$ is realizable for some $Z$ with trivial near action; fix a realization. Let $Y$ be the set of points with trivial stabilizer in restriction to $C_m$, so $Y$ is $C_{m^\infty}$-invariant (since $C_m$ is a normal subgroup). Since the near action of $C_m$ on $X$ is near free and on $Z$ is trivial, we have $Y\bigtriangleup X$ finite. So $X$ and $Y$ are isomorphic as near $C_{m^\infty}$-sets, and hence $X$ is finitely stably realizable.

If $X$ is stably realizable, then $X\sqcup Z$ is realizable for some finite set $Z$; fix a realization, and define $Y$ as above, so $Y$ is cofinite in $X\sqcup Z$. Then $\nu_Y=0$, and hence
\[\nu_X+|X\smallsetminus Z|=\nu_{X\sqcup Z}=\nu_Y+|Y\smallsetminus Z|=|Y\smallsetminus Z|,\]
so $\nu_X\in\Z$. 

Now for the existence result (for arbitrary $m$) fix $s\in\Z_m$, and let $s_n$ be the unique representative of $s$ in $[0,m^n\mathclose[$. Start from the free action of $C_{m^\infty}$ on the disjoint union $X$ of $\N$ copies of itself. Let $q=(q_k)$ be a non-decreasing sequence of positive integers tending to infinity. Define $Y_k$ as the copy of $\Z/m^{q_k}\Z$ in the $k$-th copy of $C_{m^\infty}$, and $Y$ as the (disjoint) union of all $Y_k$. The divisibility assumption implies that for each $n$, $Y_k$ is $(\Z/m_n\Z)$-invariant for large $k$, and hence $Y$ is a commensurated subset of $X$.

The residual of the $\Z/m_n\Z$-action is equal to $\sum_{k:q_k<n}m^{q_k}$. 
Given a sequence $(b_n)$ of non-negative integers with $\sum_nb_n=\infty$, arrange that $q_k$ equals $n$ for exactly $b_n$ values of $k$. Then 
\[\sum_{k:q_k<n}m^{q_k}=\sum_{\ell<n}\sum_{k:q_k=\ell}m^\ell=\sum_{\ell<n}b_nm^\ell\]

Now we construct $(b_n)$ by induction to ensure this to be equal to $s_n$ for all $n$. For $n=0$, the above equality is trivial $(0=0)$. Assuming that $n\ge 1$ and the induction holds for $n-1$, we have $\sum_{\ell<n-1}b_nm^\ell=s_{n-1}$. So the equality $\sum_{\ell<n}b_nm^\ell=s_n$ is equivalent to the requirement $b_nm^{n-1}=s_n-s_{n-1}$. Since indeed $s_n-s_{n-1}$ is divisible by $m^{n-1}$, this requirement can be fulfilled by setting $b_n=\frac{s_n-s_{n-1}}{m^{n-1}}$. If we assume that $s_n$ tends to infinity, we have $\sum b_n=\infty$ and we are done.

Otherwise, we have $s_n$ bounded, which precisely means that $s\in\N$. Then $s$ is simply achieved by considering the near action on a finite set with $s$ elements.
\end{proof}

\begin{cor}
Let $G$ be a group with a subgroup $H$ isomorphic to $C_{p^\infty}$ for some prime $p$. Then for every near $G$-set that is near free in restriction to $H$, the lest $\lst_{G}(X)$ (see \S\ref{lest}) is zero.
%For every $m\ge 2$ and every near free near $C_{m^\infty}$-set, the lest $\lst_{C_{m^\infty}}(X)$ (see \S\ref{lest}) is zero. The same holds for any group including $C_{p^\infty}$ as a subgroup for some prime $p$.
\end{cor}
\begin{proof}
Obviously (by Proposition \ref{basiclest}(\ref{bl2})), we can suppose that $G=C_{p^\infty}$.

If a near free $C_{p^\infty}$-set $X$ is balanceably near isomorphic to $X$ plus $k$ points, then $\nu_X=\nu_X+k$ in $\Z_p$, which forces $k=0$.
\end{proof}
% (of course we use that $\Z$ embeds into $\Z_m$, which would fail for $m=1$). The second statement obviously follows (by Proposition \ref{basiclest}(\ref{bl2})).

\begin{cor}\label{qcynr}
Let $X$ be a free $C_{m^\infty}$-set, and $S$ a nonempty finite subset of $X$. Then the near action of $C_{m^\infty}$ on $X\smallsetminus S$ is not realizable.
%In particular, the lest $\lst_{C_{m^\infty}}(X)$ (see \S\ref{lest}) is zero.
\end{cor}

It follows from Corollary \ref{qcynr} that any group $G$ including a quasi-cyclic subgroup $C_{p^\infty}$ as a subgroup for some prime $p$, for any free $G$-set $X$ and nonempty finite subset $K\subseteq X$, the near action of $G$ on $X\smallsetminus K$ is non-realizable. An application in a particular case is the following.
Recall that, for a prime $p$, Thompson's group $T_p$ is the group of piecewise affine self-homeomorphisms of the circle $\R/\Z$ whose breakpoints belong to $\Z[1/p]/\Z$ and whose slopes belong to $\{p^n:n\in\Z\}$. Given that it contains the subgroup $\Z[1/p]/\Z$ acting freely by translations, we deduce:

\begin{cor}\label{corthompt}
For every nonempty finite subset $K$ of the circle $\R/\Z$ (resp.\ of $\Z[1/p]/\Z$), the near action of $T_p$ on $(\R/\Z)\smallsetminus K$ (resp.\ on $(\Z[1/p]/\Z)\smallsetminus K$) is not realizable.\qed
\end{cor}

Note that this is not true when $T_p$ is replaced by $F_p$, the stabilizer of $0$. Indeed, the subset $\{p^{-n}:n\in\N\}$ is commensurated by the near $F_p$-action and has lest one, so the whole near action has lest one.

\subsection{Near actions of uncountable locally finite groups}\label{unlofi}

Let $T$ be the vertex set of a rooted binary tree (we can identify $T$ to the set of a free monoid on two letters $a,b$, with an oriented edges $(w,wa)$, and $(w,wb)$ for every $w\in T$). Let $\partial T$ be the boundary of $T$, which can naturally be identified with the set of infinite words (indexed by $\N$) in $\{a,b\}$. For $\omega\in\partial T$, write $(\omega_n)_{n\ge 0}$ for the geodesic ray from the root towards $\omega$; in other words, $\omega_n$ is the initial segment of length $n$ of $\omega$. For $\omega,\eta\in\partial T$, write $D(\omega,\eta)$ as the supremum of $n$ such that $\omega_n=\eta_n$.

% let $T_\omega$ be the image of the geodesic ray $(\omega_n from the root towards $\omega$: in other words, $T_\omega$ is the set of finite initial segments of $\omega$.
%Let $X$ be the countable set $T\times\{\pm 1\}$. For $\omega\in\partial T$, let $u_\omega$ be the permutation of $X$ of order $2$ mapping $(x,e)$ to $(x,e)$ if $x\notin T_\omega$, and $(x,e)$ to $(x,-e)$ if $x\in T_\omega$. 
%\subsection{Near actions of direct sums of finite groups}n_

For $\omega\in\partial T$, let $u_\omega$ be the permutation of $X$ of order $2$ exchanging $\omega_{2n}$ and $\omega_{2n+1}$ for all $n\ge 0$, and fixing all other vertices. 

For a set $Z$, let $G_Z$ be the free $\Z/2\Z$-module on $Z$, with basis $(\zeta_z)_{z\in Z}$. For all $\omega,\eta\in\partial T$, it is easy to see that they commute if $D(\omega,\eta)$ is odd or infinite, and otherwise, if $D(\omega,\eta)$ is even, the commutator $[\omega,\eta]$ is a 3-cycle. Hence the $u_\omega$ define a near action of $G_{\partial T}$ on $T$. This near action is $\star$-faithful: indeed, for any distinct $\omega[1],\dots,\omega[d]$, the closure in the compactification $T\sqcup\partial T$ of the support of the product $u_{\omega[1]}\dots u_{\omega[d]}$ intersects $\partial T$ in $\{\omega[1],\dots\omega[d]\}$; hence this support is infinite.

Let $K$ be the subset of $\partial T$ consisting of all infinite words $\omega$ such that $\omega_{2n}=\omega_{2n+1}$ for all $n$.

\begin{thm}
The near action of $G_{\partial T}$ on $T$ is not completable. More precisely, for any uncountable subset $L$ of $K$, the near action of $G_L$ on $T$ is not completable. 
\end{thm}
\begin{proof}
By contradiction, assume that for some uncountable subset $L\subseteq K$, the near action of $G_L$ is completable, namely as an action $\beta$ on a set $X\sqcup Y$. Let $T_n$ be the $n$-ball around the root in $T$ (so $T_n$ is $u_\omega$-invariant for $n$ odd). Then for each $\omega\in L$, there exists an odd number $n_\omega$ such that $\beta(\zeta_\omega)$ coincides with $u_\omega$ on $X\smallsetminus T_{n_\omega}$. 

Since $L$ is uncountable, there exists $n$ and an infinite subset $M\subseteq L$ such that $n_\omega=n$ for all $\omega\in L$. Then $M$ has an accumulation point in $\partial T$. It follows that there exist distinct $\omega,\eta\in M$ such that $D(\omega,\eta)>n$. Since $L\subseteq K$, $D(\omega,\eta)$ is even. Hence, in restriction to $X\smallsetminus T_n$, $u_\omega$ and $u_\eta$ do not commute, and hence $\beta(u_\omega)$ and $\beta(u_\eta)$ do not commute. This is a contradiction. 
\end{proof}

\begin{rem}
Since $L$ can be chosen of every cardinal in $[\aleph_1,2^{\aleph_0}]$, we obtain, for every such cardinal, a locally finite (abelian) group of this cardinal, with a non-completable, $\star$-faithful near action on a countable set. If we remove the $\star$-faithfulness restriction, the group can be chosen of any cardinal, since we can extend the above to a near action of $G_L\times H$, where $H$ is any group near acting trivially, with the prescribed cardinality.
\end{rem}

\begin{rem}\label{lofinocom}
Here is another example of a non-completable near action of an uncountable locally finite group (on a countable set).
Let $S$ be a finite nonabelian simple group and $I$ an infinite countable set, and the near power $G=S^{)I(}=S^I/S^{(I)}$. 

%The property of countable locally finite groups that every near action is completable, is not true for arbitrary locally finite groups, even on countable sets. Indeed, l

First, using the fact that $I$ includes $2^{\aleph_0}$ infinite subsets, pairwise with finite intersection, we see that $G$ includes a subgroup isomorphic to $S^{(2^{\aleph_0})}$. The latter group is known not to embed into $\mks(\aleph_0)$ \cite{McK}. Hence $G$ does not embed into $\mks(\aleph_0)$ either. 

However, the natural action of $S^I$ on $S\times I$, given by $s\cdot (t,i)=(s_it,i)$, with only finite orbits, induces a near action of $G$, with trivial kernel. We claim that this near action is not completable. Indeed, otherwise, using that $G$ has Property FW \cite{CorCA} along with Proposition \ref{FWcomp}, this near action is finitely stably realizable, and we deduce that $G$ embeds into $\mks(\aleph_0)$, a contradiction.
\end{rem}

\section{Rigidity results for 1-ended groups}\label{s_rig}

\subsection{Roughly equivariant maps}\label{s_equivariant1}
This subsection is preliminary for the next one.

\index{roughly equivariant}
Given a group $G$ and near $G$-sets $X,Y$, consider the set of near $\mkrst(X,Y)$ of near $G$-equivariant near maps $X\to Y$. Define $\mkr_{(G)}(X,Y)$ and $\mkrm_{(G)}(X,Y)$ as its preimages in $\mkr(X,Y)$ and $\mkrm(X,Y)$ respectively. These are called roughly $G$-equivariant rough mappings and roughly equivariant maps from $X$ to $Y$. 

\begin{prop}\label{equivaca}
Let $X,Y$ be near $G$-sets and let $f:X\to Y$ be a map.
\begin{enumerate}
\item\label{equivar} The map $f$ is roughly $G$-equivariant if for every $g\in G$ and some/any choice of lifts of $g$ in $\mkem(X)$ and $\mkem(Y)$, both denoted $\tilde{g}$, the image by $f$ of the set of $x\in X$ such that $f(\tilde{g}x)\neq \tilde{g}f(x)$ is finite. 

\item\label{equivacap} If $f$ is proper, this is equivalent to the condition that for all $g\in G$, the set of $x\in X$ such that $f(\tilde{g}x)\neq \tilde{g}f(x)$ is finite.
\end{enumerate}
\end{prop}
\begin{proof}
(\ref{equivar}) The forward implication is trivial. For the converse, fix lifts as above; write $U_g=\{x\in X:f(\tilde{g}x)\neq \tilde{g}f(x)\}$, $F_g=\{x\in X:\widetilde{g^{-1}}\tilde{g}x\neq x\}$ and $F'_g=\{y\in Y:\widetilde{g^{-1}}\tilde{g}y\neq y\}$. Then $F_g,F'_g$ are finite, and $f(U_g)$ is finite. Define $Y_g$ as the complement in $Y$ of $\tilde{g}f(U_g)\cup f(U_{g^{-1}})\cup f(F_g)\cup F'_g$.

%$K_g=\{y\in Y:\forall x\in f^{-1}(\{y\}):f(\tilde{g}x)\neq \tilde{g}f(x)\}$. By assumption, it is cofinite in $Y$, for every $g\in G$. Write $F_g=\{x\in X:\tilde{g^{-1}}\tilde{g}x\neq x\}$. Define $Y'_g=Y_g\cap \tilde{g}Y_g\smallsetminus f(F_g)$; this is cofinite in $Y$.

 Let us check that, viewing functions $X\to Y$ as subsets of $X\times Y$, we have $(\tilde{g}\circ f)\cap (X\times Y_g)=(f\times\tilde{g})\cap (X\times Y_g)$.

Fix $y\in Y'_g$. If, for some $x\in X$, we have $\tilde{g}(f(x))=y$, then since $y\notin\tilde{g}f(U_g)$, we have $x\notin U_g$, and hence $f(\tilde{g}x)=y$. This proves $(\tilde{g}\circ f)\cap (X\times Y_g)\subseteq (f\times\tilde{g})\cap (X\times Y_g)$. Conversely if, for some $x\in X$, we have $f(\tilde{g}x)=y$, then since $y\notin f(U_{g^{-1}})$, we have $\tilde{g}x\notin U_{g^{-1}}$. Hence $f(\widetilde{g^{-1}}\tilde{g}x)=\tilde{g^{-1}}f(\tilde{g}x)=\widetilde{g^{-1}}y$, and thus $\tilde{g}f(\widetilde{g^{-1}}\tilde{g}x)=\tilde{g}\widetilde{g^{-1}}y$. Since $y\notin f(F_g)$, we have $x\notin F_g$ and hence $\widetilde{g^{-1}}\tilde{g}x=x$; also using that $y\notin F'_g$, we deduce that 
$\tilde{g}f(x)=y$, proving the reverse inclusion.

(\ref{equivacap}) immediately follows.
\end{proof}

%When $X,Y$ are given as $G$-sets, this is equivalent to the condition that for all $g\in G$, the set of $x\in X$ such that $f(gx)\neq gf(x)$ is finite.

%\begin{defn}
%A map $f:X\to Y$ is near equivariant if for every $g\in G$, the set of $x\in X$ such that $f(gx)\neq gf(x)$ is finite. 
%\end{defn}
%This notion is invariant under finite perturbation.

\begin{defn}\label{cloemap}
A map between $G$-sets $X,Y$ is closely equivariant\index{closely equivariant} if for every $g\in G$, the set of $x\in X$ such that $f(gx)=gf(x)$ is cofinite in $X$.
\end{defn}

\begin{rem}
The notion of closely equivariant map does not behave well in the setting of non-proper maps. First, it cannot be extended to near $G$-sets: if we copy the characterization of Proposition \ref{equivaca}(\ref{equivacap}), then the resulting definition would depend on the choice of lifts. Indeed, for instance let $f:X\to Y$ be a $G$-equivariant map between $G$-sets, with $Y$ not a singleton, and suppose that some $y_0\in Y$ has $f^{-1}(\{y_0\})$ infinite. Choose the obvious lifts at the level of $X$, and on $Y$, for some $g$ prescribe $\tilde{g}y_0\neq g y_0$. Then for all $x\in f^{-1}(\{y_0\})$, we have $f(gx)=gf(x)=gy_0\neq \tilde{g}y_0=\tilde{g}f(x)$.

In addition, even in the setting of $G$-sets in which the notion of closely equivariant map is well-defined, it 
misbehaves under composition. Indeed, let $G$ be any nontrivial group. Let $X$ be an infinite $G$-set; let $\{v_0\}$ be a singleton (with trivial $G$-action) and $Y$ an arbitrary $G$-set with $y_0\in Y$ not fixed by some $g\in G$. Let $f$ be the unique map $X\to\{v_0\}$ and $u$ the map $\{v_0\}\to X$ mapping $v_0$ to $y_0$. These are clearly near $G$-equivariant. Then $v=u\circ f$ is the constant function $x\mapsto y_0$, and the set of $x\in X$ such that $v(gx)=gv(x)$ is empty. So $u\circ f$ is not near $G$-equivariant.

In general, it still holds that if we have closely $G$-equivariant maps $X\to Y\to Z$ where $X\to Y$ is proper, then the composite map is closely $G$-equivariant.
\end{rem}

%, since otherwise it misbehaves under composition.

%When we have $G$-sets, we can make a similar definition of equivariance without restricting to maps with finite fibers. Namely, write $F(X,Y)=Y^X$. Given $G$-sets $X,Y$ defined by homomorphisms $\alpha:G\to\mks(X)$, $\beta:G\to\mks(Y)$, we say that $f:X\to Y$ is near $G$-equivariant if for every $g\in G$, $f\circ\alpha(g)$ and $\beta(g)\circ f$ coincide on a cofinite subset. The set of near $G$-equivariant maps $X\to Y$ is saturated under finite perturbation. We denote by $\mk{F}^\star_G(X,Y)$ the corresponding quotient.

%Beware that this does not define a category, even when $G$ is the trivial group: composition of maps misbehaves with finite perturbation as soon as we allow maps with infinite fibers. So even if we define $\mk{F}^\star_G(X,Y)$ for convenience, it will not be our focus.

%Note that $\mk{F}^\star_G(X,Y)\cap \mk{E}^\star(X,Y)=\mk{E}^\star_G(X,Y)$.

We say that two actions $\alpha,\alpha':G\to\mks(X)$ are finite perturbations\index{finite perturbation} of each other if the induced near actions are equal, i.e. if $\pi\circ\alpha=\pi\circ\alpha'$, where $\pi$ is the canonical quotient map $\mks(X)\to\mkst(X)$. 

%Beware that this is does not mean finite perturbations in the sense of maps $G\to\mks(X)$; this should not cause ambiguity. 

The intuition of finite perturbation is natural when $G$ is finitely generated and endowed with a finite generating subset: indeed it then means that one passes from one Schreier graph to the other by moving finitely many edges.

%Clearly, if $\alpha'$ is a finite perturbation of $\alpha$ and $\beta'$ is a finite perturbation of $\alpha$, then $\mk{E}^\star_{\alpha,\beta}(X,Y)=\mk{E}^\star_{\alpha',\beta'}(X,Y)$. However, beware that $\mk{F}^\star_{\alpha,\beta}(X,Y)$ and $\mk{F}^\star_{\alpha',\beta'}(X,Y)$ can be distinct. More precisely, we always have $\mk{F}^\star_{\alpha,\beta}(X,Y)$ $=\mk{F}^\star_{\alpha',\beta}(X,Y)$, but $\mk{F}^\star_{\alpha,\beta}(X,Y)$ and $\mk{F}^\star_{\alpha,\beta'}(X,Y)$ can be distinct: for instance, suppose that $X$ is infinite and that $\beta(G)$ fixes a point $y_0\in Y$, but not by $\beta'(g)$ for some $g$. Then the constant function $y_0$ belongs to $\mk{F}^\star_{\alpha,\beta}(X,Y)$ but not to $\mk{F}^\star_{\alpha,\beta'}(X,Y)$.

\subsection{Near equivariant maps on 1-ended groups}\label{s_equivariant}

When $G$ is a group, recall that we write $[G]$ to mean $G$ viewed as a $G$-set, where the action is by left translation. For consistency of notation, the $G$-set $G/H$ will also sometimes be denoted $[G/H]$ (since we are often mainly motivated by the case $H=\{1\}$).

When $X$ is a $G$-set, the $G$-equivariant maps $[G]\to X$ are the maps $g\mapsto gx_0$ for $x_0\in X$. In particular, and being careful of the change of notation, the $G$-equivariant maps $[G]\to [G]$ are the right translations $x\mapsto xg_0$ for $g_0\in G$.

%Let $G$ be a group, $X,Y$ $G$-sets. An  $G$-equivariant map ($G$-AE for short) map $X\to Y$ is a map $f:X\to Y$ such that for all $g\in G$,
%the subset
%\[X_{(g)}=\{x\in X\mid f(gx)=gf(x)\}\]
%has finite complement in $X$.

%We start with the following lemma.

% For almost equivariant maps, see \S\ref{s_equivariant}.

%, which will be superseded by Proposition \ref{per1e}

\begin{lem}\label{AE1e}
Let $G$ be a group and $X$ a $G$-set. Consider a closely $G$-equivariant map $f:[G]\to X$ (Definition \ref{cloemap}). Consider the partition $(G_x)_{x\in X}$ of $G$ defined by $G_x=\{g\in G:f(g)=gx\}$, and for $Y\subseteq X$ define $G_Y=\bigcup_{x\in Y}G_x$ (so $G_{\{x\}}=G_x$). Then

%and $H$ a subgroup. Consider a near $G$-equivariant map $f:[G]\to [G/H]$. Define a partition $(X_y)_{y\in G/H}$ of $G$ by $X_y=\{x\in G:f(x)=xy\}$. Then
\begin{itemize}
\item for every $Y\subseteq X$, the subset $G_Y$ is a $G$-commensurated subset of $[G]$;
\item if $G$ is finitely generated, or more generally tamely-ended (Definition \ref{mended}), then $G_x$ is nonempty for only finitely many $x\in X$ (and hence, $G$ acts on $X$ with finite stabilizers, $f$ is a proper map).
\end{itemize} 

%Let $G$ be a group. Consider a near $G$-equivariant map $f:[G]\to [G]$. Define a partition $(X_g)_{g\in G}$ of $G$ by $X_g=\{x\in G:f(x)=xg\}$. Then
%\begin{itemize}
%\item each $X_g$ is a $G$-commensurated subset of $[G]$;
%\item if $G$ is finitely-ended, or finitely generated, or uncountable, then $X_g$ is nonempty for only finitely many $g$ (and hence $f$ is a proper map).
%\end{itemize} In particular, if $G$ is 1-ended (not necessarily finitely generated), then there exists a unique $y\in Y$ such that $f(g)=gy$ for all but finitely many $g\in G$.
\end{lem}
\begin{proof}
For every $g\in G$, define the cofinite subset
\[G^{(g)}=\{h\in G\mid f(gh)=gf(h)\}.\]
Write $u(g)=g^{-1}f(g)\in X$. For every $g\in G$ and $h\in G^{(g)}$, the above relation reads as $ghu(gh)=ghu(h)$, that is, $u(gh)=u(h)$. For $Y\subseteq X$, write $G_Y=u^{-1}(Y)$ (so $G_y=G_{\{y\}}$). Then for every $g\in G$ and $h\in G^{(g)}\cap G_Y$, we have $gh\in G_Y$; that is, $G_Y\smallsetminus g^{-1}G_Y$. Since this holds for all $g$, this proves that, $G_Y$ is a $G$-commensurated subset of $[G]$.

Let us prove the finiteness assertion, namely that the set $T=\{t\in X:G_t\neq\emptyset\}$ is finite when $G$ is tamely-ended. Otherwise, consider a partition $(T_n)_{n\in\N}$ of $T$ into infinitely many infinite subsets $T_n$. The map $\mathcal{P}(X)\to\mathcal{P}^*_G(G)$ mapping $Y$ to the class $G_Y=u^{-1}(Y)$ modulo finite subsets is a Boolean algebra homomorphism, and so is the map $\mathcal{P}(\N)\to\mathcal{P}(X)$ given by $I\mapsto\bigsqcup_{n\in I}T_n$. The composite Boolean algebra homomorphism $\mathcal{P}(\N)\to\mathcal{P}^\ast_G(G)$ is injective.
This defines an injection of the Boolean algebra $\mathcal{P}(\N)$ into $\mathcal{P}^\star_G(G)$. We thus contradict that $G$ is tamely-ended.
\end{proof}
%Then for every $I\subseteq \N$, $u^{-1}(T_n)$ is commensurated, and t

%REGLER L'AMBIGUITE SUR NEAR MAP: MAP OU CLASS?

%As a digression, here is a first corollary, before restricting the study to 1-ended groups.

\begin{rem}
Let $G$ be an $\infty$-ended finitely generated group, so action by right translation yields a faithful action of $G$ on the space of ends $\mathcal{E}_G([G])$ (for the left action); this action is minimal (Abels \cite{Ab77}).  Then the near automorphism group $\mkst_G([G])$ of $G$ as $G$-set under left translation, can naturally be identified, by Lemma \ref{AE1e}, with the topological-full group of the above action of $G$ on its space of ends $\mathcal{E}_G([G])$.
\end{rem}

\begin{cor}\label{ghx}
Let $G$ be a 1-ended group (not necessarily finitely generated) and $X$ a $G$-set. Then
\begin{enumerate}
\item for every closely $G$-equivariant map $f:[G]\to X$, there exists a unique $y\in X$ such that $f(g)=gy$ for all but finitely many $g\in G$;
\item more generally, for every finite subgroup $H$ of $G$ and every closely $G$-equivariant map $f:[G/H]\to X$, there exists a unique $y\in X$, stabilized by $H$ such that $f(gH)=gy$ for all but finitely many $g\in G$.
\end{enumerate}
\end{cor}
\begin{proof}
Define $G_y$ as in Lemma \ref{AE1e}; it is commensurated; if $G$ is 1-ended, then it follows that $G_x$ is finite or cofinite for all $x$, and cofinite for at most one $x\in X$. Such $x$ necessarily exists: otherwise $G_x$ is finite for all $x$ and hence is nonempty for infinitely many $x$, contradicting the finiteness statement of Lemma \ref{AE1e}.

Now consider the second assertion. Denote by $p:G\to G/H$ the projection; then $f\circ p$ is closely equivariant, by properness of $p$. By the first assertion, there exists $x\in X$ such that $f\circ p(g)=gx$ for all but finitely many $p$. We claim that $f_1:g\mapsto gx$ is constant on every fiber $\xi H$ of $p$: indeed, otherwise, there exists $h\in H$ such that $\xi x\neq \xi hx$. Hence $g\xi x\neq g\xi hx$ for all $g\in G$. That is, $f_1(g\xi )\neq f_1(g\xi h)$ for all $g\in G$. Since there exists $g\in G$ such that $f\circ p$ and $f_1$ coincide on both $g\xi$ and $g\xi h$, we get a contradiction. Hence $f_1$ factors as $p\circ f_2$, with $f_2(gH)= gx$ for all $gH\in G/H$. Then $x$ is fixed by $H$ and uniquely determined.
\end{proof}

\begin{cor}\label{ghgh}
Let $G$ be a 1-ended group, $H,H'$ subgroups with $H$ finite.
%\begin{itemize}
%\item every near $G$-equivariant map $f:[G]\to [G/H]$ is a finite perturbation of a right translation $g\mapsto gx$; in particular:
\begin{enumerate}
\item\label{ghgh1} $\mk{E}_G(G/H,G/H')=\mk{E}^{\mathrm{nsurj}}_G(G/H,G/H')$, that is, every proper near $G$-equiv\-ariant map $G/H\to G/H'$ has a cofinite image;
\item\label{ghgh2} if $H'$ is infinite, $\mk{E}_G(G/H,G/H')=\emptyset$; the same conclusion holds if $H'$ is finite and $H$ is not conjugate to any subgroup of $H'$;
\item\label{ghgh3} if $H,H'$ are not conjugate, $G/H$ and $G/H'$ are not isomorphic as near $G$-sets;
\item\label{ghgh4} if $H,H'$ are conjugate, then every proper near equivariant map $f:G/H\to G/H'$ is a near isomorphism between the near $G$-sets $G/H$ and $G/H'$, and has a representative of the form $f(gH)=g\xi H'$, where $\xi\in G$ is such that $\xi^{-1}H\xi=H'$, and $\xi$ is well-determined in $G/H'$. In particular, $\mkst_G(G/H,G/H)$ is isomorphic to $N_G(H)/H$, where $N_G(H)$ is the normalizer of $H$ in $G$. 
%\item\label{ghgh5} every proper near equivariant map $f:[G]\to [G]$ is an isomorphism 
\end{enumerate}
\end{cor}
\begin{proof}
By composition, every near $G$-equivariant map $G/H\to G/H'$ gives a near $G$-equivariant map $G\to G/H'$, which near equals the surjective map $g\mapsto gx$ for some $x\in G/H'$. So (\ref{ghgh1}) holds. If $H'$ is infinite, this map $g\mapsto gx$ is not proper, and hence since $H$ is finite, the original map $G/H\to G/H'$ is not proper either, showing the first part of (\ref{ghgh2}).

Now suppose that $H'$ is finite. Consider a representative $f:G/H\to G/H'$, and let $f_1:G\to G/H'$ be the composition. Then we can change $f_1$ on a finite subset, and hence $f$ as well to ensure that $f_1(g)=g\xi H'$ for some $\xi\in G$ and all $g\in G$ ($\xi$ is well-determined in $G/H'$). By construction, $f_1(gh)=f_1(g)$ for all $g\in G$ and $h\in H$; this implies that $h\xi H'=\xi H'$ for all $h\in H$, which means that $\xi^{-1}H\xi\subseteq H'$. This proves the second part of (\ref{ghgh2}). If moreover $f$ is a near permutation, the argument applied to the (near) inverse of $f$ implies that $H'$ is conjugate to a subgroup of $H$, and hence the finite subgroups $H$ and $H'$ are conjugate, so (\ref{ghgh3}) and (\ref{ghgh4}) are proved.
\end{proof}

\begin{prop}\label{nonrea_1end}
Let $G$ be a 1-ended group, and $H$ a finite subgroup. Then the realizability degree of $G/H$ is $0$. In particular, the lest of the (near) $G$-set $G/H$ is 0.
\end{prop}
\begin{proof}
Write $X=G/H$. Assuming otherwise, there exists $n>0$ such that the action on $X$ can be perturbed to an action (denoted $\ast$) that is trivial on a subset of cardinal $n$. Let $K$ be the union of finite orbits of the twisted action (so $|K|\ge n$). The identity map $\mathrm{id}_X$ is near equivariant and hence by Corollary \ref{ghx}, there exists $f:X\to X$, near equal to $\mathrm{id}_X$, that is equivariant from $(X,\cdot)$ to $(X,\ast)$. In particular, since $X$ is $G$-transitive, all fibers of $f$ have the same cardinal and hence, since $f$ is near equal to an injective map and $X$ is infinite, we deduce that $f$ is injective with cofinite image. Thus the index character of $\mkst(X)$ maps $f$ to $|X\smallsetminus f(X)|$. Since $f$ is near equal to the identity, the index of $f$ is zero. So $f$ is a permutation and thus $X$ is also $G$-transitive in the twisted action, and hence $K$ is empty, a contradiction.
\end{proof}

\subsection{Rigidity for free actions of 1-ended groups}\label{s_rifa}

\begin{thm}\label{per1e}
Let $G$ be a 1-ended group that is not locally finite; let $X$ be a $G$-set.

\begin{enumerate}
\item\label{pe1} Suppose that all point stabilizers are finite (e.g., $G$ acts freely). Then any finite perturbation of the $G$-action on $X$ is conjugate, by a unique finitely supported permutation of $X$, to the original action.

\item\label{pe2} More generally, if for all but finitely many $x$, the stabilizer of $x$ is finite, then any finite perturbation of the $G$-action on $X$ is conjugate, by a finitely supported permutation of $X$, to an action that is unchanged on all infinite orbits (i.e., outside the maximal finite invariant subset).
\end{enumerate}
%Let $G$ act on a set $X$, so that all but finitely many points have a finite stabilizer (e.g., $G$ acts freely). Then any finite perturbation of the action is conjugate, by a unique finitely supported permutation to the original action.
\end{thm}
\begin{proof}
First notice that in the setting of (\ref{pe1}), the uniqueness is immediate: indeed, if a finitely supported permutation conjugates an action to itself, its support is a finite invariant subset, and hence in this case has to be empty. Apart from this easy uniqueness observation, (\ref{pe2}) is a generalization of (\ref{pe1}).

We first reduce to the case when $X$ has finitely many $G$-orbits. Note that for an action of a finitely generated group, every finite perturbation is unchanged in restriction to all but finitely many orbits, and in particular the reduction is immediate when $G$ is finitely generated. We will use this fact, using that $G$ has an infinite, finitely generated subgroup $H$.

Write $(X,\cdot)$ for the original action and $(X,\ast)$ for the twisted action. Consider the identity embedding of any infinite orbit $Y$ into $X$. It is near equivariant. By Corollary \ref{ghx}, there exists a unique $G$-equivariant map $f_Y:(Y,\cdot)\to (X,\ast)$ that is a finite perturbation of $f$. Let $Z$ be the union of infinite orbits (in the original action), so $Z$ is cofinite in $X$; let $f:Z\to X$ be obtained by putting together all $f_Y$.  Then $f$ is $G$-equivariant from $(Z,\cdot)$ to $(X,\ast)$. 

Let $K$ be the set of non-fixed points of $f$. Then $K\cap Y$ is finite for every infinite (original) orbit $Y\subseteq X$. Since $f$ is near equivariant, $K$ is commensurated. Hence $K$ is $H$-commensurated; since $H$ is finitely generated, this implies that $K\cap Y$ is $H$-invariant for all but finitely many original orbits $Y\subseteq X$; since $K\cap Y$ is finite for every $Y$ and $H$ is infinite, this implies that $K\cap Y$ is empty for all but finitely many $Y$. Hence $K$ is finite. This means that $f$ is the identity on all but finitely many orbits $Y$, and hence since $f$ is equivariant, this means that the action is unchanged on all but finitely many original orbits $Y$.

Let $Z'=f(Z)$ be the union of infinite twisted orbits. Then $Z'$ is cofinite, and as a $G$-set for the twisted action, is isomorphic (through $f$) with the $G$-set $Z$.

%Let $X_1,\dots,X_n$ be the infinite original orbits on which $f$ differs from the identity. The corresponding twisted orbits are $X'_i=f(X_i)$. Write $F=X\smallsetminus\bigcup X_i$ and $F'=\smallsetminus\bigcup X'_i$, so $F$ and $F'$ are finite; $(X_i,\cdot)$ and $(X'_i,\ast)$ are isomorphic $G$-sets, 

By Propositions \ref{nonrea_1end} and \ref{lestunion}, the lest of the $G$-set $Z$ is 0. It follows from the definition of lest that the finite subsets $X\smallsetminus Z$ and $X\smallsetminus Z'$ have the same cardinal. Hence, after conjugating the twisted action by a finitely supported permutation, we can suppose that $Z'=Z$. So $f$ is a finitely supported permutation of $Z$, and hence conjugating the twisted action by $f$ (extended as the identity outside $Z$) yields an action that coincides with the original action on $Z$.
\end{proof}

\begin{exe}
The 1-ended assumption is essential: endow $\Z$ with its action on itself by translation. The permutation of $\Z$ mapping $n\mapsto n+1$ for $n\neq -1$, $0\mapsto 0$, $-1\mapsto 1$ defines a finite perturbation of the original action; however, it is not conjugate by a finitely supported permutation (and not at all conjugate) to the original action, since it is not transitive.

A variant is the following: consider the simply transitive action of the infinite dihedral group $\langle u_{\mathsf{eo}},u_{\mathsf{oe}}\rangle$ on $\Z$, where the involution $u_{\mathsf{eo}}$ swaps every even number with the next odd number, while the involution $u_{\mathsf{oe}}$ swaps every odd number with the next even number; this is a free action with a single orbit. Perturb the action by declaring 0 and 1 to be fixed by $u_{\mathsf{eo}}$. Then the new action has two 1-ended orbits, so is not transitive.
\end{exe}

%\begin{cor}\label{freefix}Let $G$ be a 1-ended finitely generated group. Let $X$ be a $G$-set, with $X=Y\sqcup X^G$ and $G$ acts freely on $Y$. Then any finite perturbation of the action is conjugate, by a finitely supported permutation, to an action that is unchanged on $Y$.\end{cor}\begin{proof}
%Let $Y'$ be the union of its infinite orbits of the twisted action.

%First suppose that the twisted action has no non-singleton finite orbit. Then after conjugating by a finitely supported permutation, we can suppose that either $Y'\subseteq Y$ or $Y\subseteq Y'$. In the first case, we can apply Theorem \ref{per1e} within $Y'$ to conclude. In the second case, we can also do the same (exchanging the roles of the old and new actions). 

%In general, we have a third action, replacing the action on the finite orbits by the trivial action. So we see, applying the previous case, that there exists a finitely supported permutation conjugating $Y'$ to $Y$; hence we can suppose $Y'=Y$. Then we can apply Theorem \ref{per1e} within $Y$ and we obtain the result.\end{proof}

%\begin{prop}\label{nonrea_1end}For any group $G$ containing a finitely generated 1-ended subgroup and any free $G$-set $X$ and non-empty subset $T\subseteq X$, the almost action of $G$ on $X\smallsetminus T$ is non-realizable. In particular, its lest (see \S\ref{lest}) is zero.\end{prop}
%\begin{proof}If it were realizable, then the action of $G$ on $X$ would be realizable by an action with $|F|$ fixed points. This is in contradiction with Theorem \ref{per1e}.\end{proof}

%Theorem \ref{per1e} and Corollary \ref{freefix} 
Theorem \ref{per1e} is extensively used in non-realizability results for the group of interval exchanges with flips \cite{CIET}. It is also used in the proof of Theorem \ref{3more} on the classification of near actions of $\Z^d$ for $d\ge 3$. Lemma \ref{AE1e} is used in \S\ref{morecen}.

%Let us provide some more corollaries.

\subsection{Application to non-completability}

The above results can also be used to prove non-completability results.

\begin{prop}
Let $G$ be a 1-ended group.
Let $X$ be a completable near $G$-set and $M$ a simply transitive $G$-set. Let $f:X\to M$ be a near $G$-equivariant proper map. Then the number $n$ of ends of $X$ is finite, and for the decomposition $X=X_1\sqcup \dots \sqcup X_n$ into 1-ended commensurated subsets, the restriction of $f$ to each $X_i$ is a near isomorphism.
% and there exists a near $G$-isomorphism $j:n[G]\to X$ (where $n[G]$ is the free action of $G$ on $n$ copies of $[G]$) such that $f\circ j:n[G]\to [G]$ is the identity on each copy of $[G]$. 
\end{prop}
\begin{proof}
Fix a function $f_1:X\to M$ realizing $f$. Write $X\sqcup Y=Z$ with $Z$ realizable, and fix a realization. 

For $g\in G$, write $X_{(g)}=\{x\in X\cap g^{-1}X:f_1(gx)=gf_1(x)\}$; this is a cofinite subset of $X$, and its finite complement $X_{(g)}$. Also define $M'_{(g)}\subseteq M$ as the finite subset
$f_1(X'_{(g)})\cup g^{-1}f_1(X'_{(g^{-1})})$. For $m\in M$, define $F_m=f_1^{-1}(\{m\})$. For $m\notin M'_{(g)}$, we claim that $F_{gm}=gF_m$.

%\{g^{-1}f_1(gx):x\in X'_{(g)}\cap g^{-1}X\}

%\[f_1(X'_{(g)})\cup\{g^{-1}f_1(gx):x\in X'_{(g)}\cap g^{-1}X\}\cup g^{-1}f_1(X'_{(g^{-1})}).\]

Indeed, suppose that $x'\in gF_m$. So $f_1(x)=m$, for $x=g^{-1}x'$; then $x\in X$, so $x'\in gX$. Since $m\notin f_1(X'_{(g)})$, we have $x\in X_{(g)}$, whence $gx\in X$ and $f_1(gx)=gf_1(x)=gm$. That is, $x'\in X$ and $f_1(x')=gm$. Thus $x'\in F_{gm}$.  

%and write $x=g^{-1}x'$

Conversely, suppose that $x'\in F_{gm}$. So $f_1(x')=gm$. Since $gm\notin f_1(X'_{(g^{-1})})$, we have $x'\in X_{(g^{-1})}$. So $g^{-1}x'\in X$ and $f_1(g^{-1}x')=g^{-1}f_1(x')=m$. Thus $x'\in gF_m$.

% we have $x\in X$. Since $m\notin\{g^{-1}f_1(gx):x\in X'_{(g)}\}$, we have 
%we have $x\notin X_{(g)}$. So $f_1(gx)=gf_1(x)$, that is, $f_1(x')=gf_1(g^{-1}x')$. Thus $m=g^{-1}f_1(x')=f_1(g^{-1}x)$. That is, $g^{-1}x\in F_{m}$, so $x\in gF_m$.  
%For $g\in G$, write $X'_{(g)}$ for the set of $x\in X$ such that $gx\notin X$ or $f_1(gx)\neq gf_1(x)$. This is a finite subset. Define $M'_{(g)}\subseteq M$ as the finite subset
%\[\{f_1(x):x\in X'_{(g)}\}\cup\{g^{-1}f_1(gx):x\in X'_{(g)}\}\cup \{g^{-1}f_1(x):x\in X'_{(g^{-1})}\}\]
%For $m\in M$, define $F_m=f_1^{-1}(\{m\})$. For $m\notin M'_{(g)}$, we claim that $F_{gm}=gF_m$. 
%Indeed, suppose that $x'\in gF_m$. So $f_1(x)=m$, for $x=g^{-1}x'$. Since $m\notin\{f_1(x):x\in X'_{(g)}\}$, we have $x\notin X'_{(g)}$, and hence $f_1(gx)=gf_1(x)$. That is, $f_1(x')=gf_1(x)=gm$. Thus $x'\in F_{gm}$.  
%Conversely, suppose that $x'\in F_{gm}$, and write $x=g^{-1}x'$. So $f(x')=gm$. Since $m\notin\{g^{-1}f(x):x\in X'_{(g^{-1})}\}$, we have $x\in X$. Since $m\notin\{g^{-1}f_1(gx):x\in X'_{(g)}\}$, we have $x\notin X_{(g)}$. So $f_1(gx)=gf_1(x)$, that is, $f_1(x')=gf_1(g^{-1}x')$. Thus $m=g^{-1}f_1(x')=f_1(g^{-1}x)$. That is, $g^{-1}x\in F_{m}$, so $x\in gF_m$.  

Now fix a point in $M$ so as to identify $M$ with $[G]$. Thus $g\mapsto f_1^{-1}(\{g\})$ is a near equivariant map from $[G]$ to the $G$-set of finite subsets of $Z$. By Corollary \ref{ghgh}, there exists a finite subset $F$ of $Z$ such $f_1^{-1}(\{g\})=gF$ for all but finitely many $g\in G$, say all $g\in K$, with $K$ cofinite subset of $G$. 

We claim that the subsets $gF$, for $g\in G$, are pairwise disjoint. Indeed, for all $g,g'$, there exists $g''\in G$ such that both $g''g$ and $g''g'$ belong to $K$, and hence $g''gF\cap g''g'F=f^{-1}(\{g''g\})\cap f^{-1}(g''g')=\emptyset$, so $gF\cap g'F$ is empty. 

%Define $X_1=\bigcup_{g\in G}gF$; it is $G$-invariant, and $\bigcup_{g\in K}gF=\bigcup_{g\in K}f_1^{-1}(\{g\})=f_1^{-1}(K)$ is cofinite in both $X_1$ and $X$ (in $X$, by properness of $f_1$). 

Let $H$ be the stabilizer of $x\in F$. If $g,h$ are distinct elements of $K$ and $x\in F$, we have $f(gx)=g\neq h=f(hx)$, and hence $gx\neq hx$. Thus $K$ is mapped injectively in $G/H$. Since $K$ is a cofinite subset of $G$, this implies that $H=1$. So $G$ acts freely on $X_1$. By definition, each $G$-orbit in $X_1$ meets $F$, so there are finitely many orbits. On each orbit $U$, $f_1$ restricts to a near $G$-equivariant map, which after identification of $U$ with $[G]$ is, by Corollary \ref{ghx} near equal to a right translation, and hence is a near isomorphism.
\end{proof}

%Define $f_2$ as the partially defined function on $Z$ that takes the value $g$ on $gF$, and undefined outside $X_1=\bigcup_{g\in G}gF$. This is a finite perturbation of $f_1$, and satisfies $f_2^{-1}(\{g\})=gF$ for all $g\in G$. Also, $X_1$ has finite symmetric difference with $X$ and is $G$-invariant.
%If $x\in X_1$, say $x\in hF$, and $g\in G$, then $gx\in ghF$ and hence $f_2(gx)=gh=gf_2(x)$. Hence $f_2:X_1\to G$ is $G$-equivariant. Since there is no $G$-equivariant map $G/H\to G$ for any nontrivial subgroup $H$, we deduce that $G$ acts freely on $X_1$, say with $k$ orbits. 

%Then fibers of $f_2$ all have cardinal $k$; since $f$ is proper, $k$ is finite. The number of ends of $X_1$, and hence of $X$ is $k$. If $X$ is 1-ended, then $f_2$ is a bijection and hence $f$ is a near $G$-isomorphism.
\begin{cor}\label{cor_noncomp}
Let $G$ be a 1-ended group. Let $X$ be a near $G$-set with a near equivariant proper map $f:X\to [G]$. Suppose that one of the following holds:
\begin{itemize}
\item $X$ is infinitely-ended; 
\item $X$ is not stably realizable;
\item $X$ is 1-ended and $f$ is not a near isomorphism;
\end{itemize}
Then $X$ is not a completable near $G$-set.\qed
\end{cor}

%\subsection{A stably realizable but not finitely stably realizable near action}\label{s_nofista}

Let us finish with another application of Proposition \ref{nonrea_1end}, contrasting with Proposition \ref{abstre0}. 

\begin{prop}\label{nonfistre}
There exists a near action (of a countable group, on a countable set) that is stably realizable but not finitely stably realizable.
\end{prop}
\begin{proof}
Let $X$ be an infinite countable set. Consider $\N=\{0,1,2,\dots\}$, and $X_n=X\sqcup\{0,\dots,n-1\}\subseteq X\sqcup\N$. Let $u_n$ be an action of $\Z^2$ on $X\sqcup\N$ that is free on $X_n$ and trivial on $\N_{\ge n}$. Together this defines an action of the infinite free power $G=(\Z^2)^{\ast \N}$, which commensurates $X$ and is a trivial near action on $\N$. Thus this defines a stably realizable near action on $X$.

We claim that this near action is not finitely stably realizable. Indeed, suppose by contradiction that it is the case, so for some $k$ the near action on $X_k$ is realizable. Extend such an action as a trivial action on $\N_{\ge k}$. 

Let $(\Z^2)_n$ be the $n$th copy of $\Z^2$ in $G$. Then for $n>k$, the inclusion $X_k\subseteq X_n$ shows that the near $(\Z^2)_n$-set $X_n$ has realizability degree $\ge n-k$. Since on the other hand it is realizable by a free action, it has realizability degree $0$ by Proposition \ref{nonrea_1end}. This is a contradiction.
\end{proof}

\section{More topics on near actions}\label{s_endsri}

\subsection{Amenability of near actions}\label{s_amen}

Recall that a mean (often called measure) on a Boolean algebra $A$ is a function $A\to[0,1]$ that is additive on disjoint pairs, maps $0_A$ to $0$ and maps $1_A$ to~$1$. See the book \cite{Fre} for a systematic treatment.

For a group $G$, we recall that a $G$-set $X$ is amenable\index{amenable (group, action)} if there is a $G$-invariant mean on the Boolean algebra $\mathcal{P}(X)$; in particular the group $G$ is amenable if there is a $G$-invariant mean on the Boolean algebra $\mathcal{P}(G)$, for the action induced by the left action.

\begin{defn}\index{amenable ($\star$-)}
Let $G$ be a group and $X$ a near $G$-set. We say that $X$ is a $\star$-amenable, or amenable near $G$-set, if there is a $G$-invariant mean on $\mathcal{P}^\star(X)$. Otherwise, it is called $\star$-non-amenable, or near non-amenable.
\end{defn}

It satisfies the following immediate properties.

\begin{prop}
Let $G$ be a group.
\begin{enumerate}
\item\label{amgroup} The group $G$ is amenable if and only if every infinite near $G$-set $X$ is $\star$-amenable.
\item\label{ampro} If $X,Y$ are near $G$-sets, $X$ is $\star$-amenable and $\mkpst_G(X,Y)$ is non-empty, then $Y$ is $\star$-amenable.
\item\label{amdsj} If $X,Y$ are near $G$-sets, then the disjoint union $X\sqcup Y$ is non-$\star$-amenable if and only if both $X$ and $Y$ are non-$\star$-amenable.
\item\label{amhom} If $H\to G$ is a homomorphism and $X$ is a $\star$-amenable near $G$-set, then it is a $\star$-amenable near $H$-set.
\end{enumerate}
\end{prop}
\begin{proof}
(\ref{amgroup}). If $G$ is non-amenable, the $G$ is infinite and is a non-$\star$-amenable $G$-set. Conversely, if $G$ is amenable, the set of means on $\mathcal{P}^\star(X)$ is convex and compact (for the pointwise convergence topology on functions from $\mathcal{P}^\star(X)$ to $[0,1]$) and is nonempty (because $X$ is infinite), so the fixed point criterion for amenability of groups applies.

(\ref{ampro}) Indeed, any element of $\mkp(X,Y)$ induces a map $\mathcal{P}^\star(Y)$ to $\mathcal{P}^\star(X)$ in a functorial way (see \S\ref{nepose}) and hence we can push forward any $G$-invariant mean from $\mathcal{P}^\star(X)$ to $\mathcal{P}^\star(Y)$. 

(\ref{amdsj}) The reverse implication follows from (\ref{ampro}). Conversely, if $X\sqcup Y$ is $\star$-amenable, then any invariant mean yields positive weight to either $X$ or $Y$, and hence either $X$ or $Y$ is $\star$-amenable.

(\ref{amhom}) is trivial.
\end{proof}

%A first remark is that for an amenable group, near actions on infinite sets are always $\star$-amenable: indeed, tAnother remark is that being near non-amenable is stable under disjoint unions of near $G$-sets.

The $\star$-sign in ``$\star$-amenable" is especially useful when we consider $G$-sets, so as to distinguish with bare amenability of a $G$-set $X$, which means that there is a $G$-invariant mean on the Boolean algebra $\mathcal{P}(X)$ (see \cite{Gre}). This is also why we avoid calling it ``near amenable", since ``near" suggests a weakening. The following proposition precisely states the little gap between these notions.

\begin{prop}\label{gsetngsetam}
Let $X$ be a $G$-set. If $X$ is a $\star$-amenable near $G$-set, then it is an amenable $G$-set. Suppose conversely that $X$ is an amenable $G$-set. Then $X$ is not a $\star$-amenable near $G$-set if and only if there is a nonempty $G$-invariant finite subset $F$ such that $X\smallsetminus F$ is not an amenable $G$-set;
\end{prop}
\begin{proof}
The first assertion is trivial, since any $G$-invariant mean on $\mathcal{P}^\star(X)$ yields a $G$-invariant mean on $\mathcal{P}(X)$ by composition. Now  assume that $X$ is an amenable $G$-set. The reverse implication is clear. Suppose that the second condition does not hold. Let $F$ be the union of all finite $G$-orbits. If $F$ is infinite, using compactness of the set of means on $A$, we see that $X$ is $\star$-amenable. If $F$ is finite, then by the failure of the second condition, $X\smallsetminus F$ is an amenable $G$-set; consider an invariant mean $\mu$. If $\mu$ is non-atomic, it passes to the quotient and hence $X$ is $\star$-amenable. Otherwise, the orbit of any singleton with positive measure is finite and we contradict the definition of $F$.\end{proof}

\begin{rem}\label{ultramenable}\index{ultra-amenable}
Replacing $[0,1]$ with $\{0,1\}$, we could strengthen the definitions.
Namely, we could define an action of $G$ on $X$, resp.\ a near action of $G$ on $X$ to be ultra-amenable (resp.\ $\star$-ultra-amenable) it preserves an ultrafilter, resp.\ if it preserves a non-principal ultrafilter. Then it is easy to check that an action is ultra-amenable if and only if every finitely generated subgroup has a fixed point, and that a near action is ultra-amenable if and only if every finitely generated subgroup has infinitely many fixed points (this property appears in Lemma \ref{strfix}).
\end{rem}

There is a well-known interpretation of means in terms of linear forms. For convenience, we state it here, since it often incorporated in proofs rather than considered as a statement for its own sake. Define a signed mean as a bounded function $\mathcal{P}(X)\to\R$ that is additive on pairs of disjoints subsets. So means are signed means that are valued in $[0,1]$ and mapping the function 1 to $1$. Identifying $\mathcal{P}(X)$ with $\{0,1\}^X$, we can embed it into $\ell^\infty(X)$ (real-valued bounded functions); we view this as an inclusion.

Recall that $\mkst(X)$ has a natural isometric linear action on the space $\ell^{\infty}(X)/\co(X)$ of bounded real-valued functions modulo those tending to zero at infinity, see \S\ref{boco}.

\begin{lem}\label{means}
Let $\Psi$ be the mapping from $\ell^\infty(X)^*$, endowed with the weak-star topology mapping any continuous linear form $u$ on $\ell^\infty(X)$ to its restriction to $\mathcal{P}(X)$. Then $\Psi$ is a linear continuous $\mks(X)$-equivariant bijection onto the set of signed means on $X$ (endowed with pointwise convergence, for evaluation on subsets on $X$), and it restricts to an affine homeomorphism between the compact convex subset of positive normalized linear forms and the compact convex subset of means.

Moreover, $\Psi$ restricts to an affine, $\mkst(X)$-equivariant homeomorphism between the compact convex subset of positive normalized linear forms on $\ell^\infty(X)/\co(X)$ (that is, such that the linear form given by composition on $\ell^\infty(X)$ is positive and normalized) and the compact convex subset of means on $X$ vanishing on singletons.
\end{lem}
\begin{proof}
The equivariance is trivial. The map $\Psi$ is injective: indeed if $\Psi u=0$, then $u$ vanishes on all functions valued in $\{0,1\}$, and hence on all bounded $\Z$-valued functions, and in turn on all bounded $\Q$-valued functions with bounded denominator, and hence vanishes by density. Also, $\Psi$ is surjective. Indeed, given a signed mean $m$ with supremum $N$, we first have to extend it on bounded $\Z$-valued functions. The way to do this is to map a finite sum $f=\sum n_i1_{A_i}$ to $s(f)=\sum n_im(A_i)$. To check that this definition is valid, we have to check ($\ast$) that $\sum n_i1_{A_i}=0$ to $\sum n_im(A_i)=0$. This is true when the $A_i$ are pairwise disjoint (since this forces all $n_i$ to vanish), and in turn holds when the distinct $A_i$ are pairwise disjoint (i.e., for every $i,j$ either $A_i=A_j$ or $A_i\cap A_j=\emptyset$), and finally in general, by decomposing each of the $A_i$ in the minimal nonempty subsets of the Boolean algebra they generate. Next, we can extend the mean to all $\Q$-valued bounded functions with bounded denominator. Let us check that for every such function $f$, we have $|s(f)|\le \|f\|_\infty N$. Every bounded function valued in $\{-n,\dots,n-1,n\}$ is a sum of $2n$ functions valued in $\{0,1\}$ or $\{0,-1\}$, and we deduce that $|s(f)|\le nN\le \|f\|_\infty N$. Now if $f$ has bounded denominators, say $qf$ is valued in $\Z$ with $q\in\N_{>0}$. Then $|s(qf)|\le \|qf\|_\infty N$. Hence $|s(f)|=q^{-1}|s(qf)|\le  q^{-1}\|qf\|_\infty N= \|f\|_\infty N$. Therefore $s$ extends by continuity to all bounded functions.

In addition, $\Psi$ maps positive normalized linear forms to means, and from the above construction of $s$ it is clear that $\Psi^{-1}$ maps means to positive normalized linear forms. Continuity of $\Psi$ is clear; in restriction to positive normalized linear forms, it induces a continuous bijection between compact Hausdorff spaces and hence is a homeomorphism.

Obviously $u$ vanishes on singletons if and only if $\Psi u$ vanishes on singletons, and this is equivalent to $u$ vanishing on finite subsets, and also, by density, on $\co(X)$. This yields the last bijection. The $\mkst(X)$-equivariance is immediate from the definitions (it was already observed in \S\ref{boco}).
\end{proof}

%restricts to a signed mean $Tu$ on $\mathcal{P}(X)$, that is a bounded function $\mathcal{P}(X)\to\R$ that is additive on pairs of disjoints subsets. Then $T:u\mapsto Tu$ is linear

This yields the following:

\begin{prop}
A near $G$-set $X$ is $\star$-amenable if and only if there is a $G$-invariant linear form on $\ell^{\infty}(X)/\co(X)$, that is positive and normalized (in the sense that the composite map $\ell^\infty(X)\to \R$ maps non-negative functions to non-negative numbers and maps $1_X$ to $1$).\qed
%mapping non-negative functions into $[0,1]$, and vanishing on finitely supported functions.
\end{prop}
\begin{proof}
By Lemma \ref{means}, the restriction map $\Psi$ considered in it induces a $\mkst(X)$-equivariant bijection between the set of positive normalized linear forms on $\ell^{\infty}(X)/\co(X)$ and the set of means on $\mathcal{P}^\star(X)$. In particular, it yields a bijection between $G$-fixed points in each of these subsets, so one is empty if and only if the other is empty.
\end{proof}

The notion of $\star$-amenability has a F\o lner-type criterion. Denote by $\mathcal{P}_{\mathrm{fin}}(X)$ and $\mathcal{P}_{\mathrm{fin}*}(X)$ the set of finite subsets and nonempty finite subsets of $X$. Let $\pi$ denote the quotient monoid homomorphism $\mkem(X)\to\mkst(X)$. For a subset $S\in\mkem(X)$ and $F\subseteq X$, denote $\partial_S(F)=\{x\in F:\exists s\in S:sx\notin F\}$. For $s\in\mkem(X)$ and $f$ a function on $X$, we write $(s^{-1}f)(x)=f(s(x))$.

\begin{prop}\label{amefol}
Let $\alpha:G\to\mkst(X)$ be a near action. Equivalent statements: 
\begin{enumerate}
\item\label{xsangs} $X$ is a $\star$-amenable near $G$-set;
\item\label{nfol} for every finite subset $S$ of $\pi^{-1}(\alpha(G))$, for every $\varepsilon>0$ and every finite subset $F_0$ of $X$,
\[\exists F\in\mathcal{P}_{\mathrm{fin}*}(X):\;F\cap F_0=\emptyset,\;\frac{|\partial_SF|}{|F|}\le\varepsilon;\]
\item\label{nlone} for every finite subset $S$ of $\pi^{-1}(\alpha(G))$, for every $\varepsilon>0$ and every finite subset $F_0$ of $X$,
\[\exists u\in\ell^1(X\smallsetminus F_0):\;u\ge 0,\;\|u\|>0,\;\sup_{s\in S}\frac{\|s^{-1}u-u\|}{\|u\|}\le\varepsilon.\]
\end{enumerate}
%Then $X$ is a $\star$-amenable near $G$-set if and only if for every finite subset $S$ of $\pi^{-1}(\alpha(G))$, for every $\varepsilon>0$ and every finite subset $F_0$ of $X$,
%\[\exists F\in\mathcal{P}_{\mathrm{fin}*}(X):\;F\cap F_0=\emptyset,\;\frac{|\partial_SF|}{|F|}\le\varepsilon.\]
\end{prop}
%\forall S\in\mathcal{P}_{\mathrm{fin}}(\pi^{-1}(G)),\forall\eps>0,\forall F_0\in\mathcal{P}_{\mathrm{fin}}(X)

We start with a lemma. For $s\in \mkem(X)$ and $L\in\ell^\infty(X)^\star$, define $sL(u)=L(s^{-1}u)$, for $u\in\ell^\infty(X)$. For $f\in\ell^1(X)$, let $L_f$ be the corresponding element of $L\in\ell^\infty(X)^\star$. Let $A_s$ be the finite set of $y\in X$ such that $s^{-1}(\{y\})$ is not a singleton, and $N_s=\sum_{y\in X}\big||s^{-1}(\{y\})|-1\big|$.

\begin{lem}\label{nlonely}
For every $u\in\ell^1(X)$ and $s\in\mkem(X)$, we have 
\[\|L_u-sL_u\|\le \|u-s^{-1}u\|_1+N_s\|u\|_\infty^{A_s},\quad (\|u\|_\infty^A=\sup_{x\in A}|u(x)|).\]
\end{lem}
\begin{proof}
For $f\in\ell^\infty(X)$, we have
\[sL_u(f)=\sum_{x\in X}u(x)f(sx)=\sum_{x\in X}u(sx)f(sx)+\sum_{x\in X}(u(x)-u(sx))f(sx),\]
and, writing $\nu(y)=|s^{-1}(\{y\})|-1$,
\[\sum_{x\in X}u(sx)f(sx)=\sum_{x\in X}u(x)f(x)+\sum_{x\in A_s}\nu(x)u(x)f(x),\]
so
\[sL_u(f)-L_u(f)=\sum_{x\in A_s}\nu(x)u(x)f(x)+\sum_{x\in X}(u(x)-u(sx))f(sx),\]
whence
\[|sL_u(f)-L_u(f)|\le N_s\|u\|_\infty^{A_s}\|f\|_\infty+\|u-s^{-1}u\|_1\|f\|_\infty.\qedhere\]
\end{proof}

\begin{proof}[Proof of Proposition \ref{amefol}]
Suppose (\ref{nlone}). Let $K_{S,n,F_0}$ be the set of positive normalized linear forms $v$ on $\ell^\infty(X)$ that vanish on functions supported by $F_0$, such that $\|sv-v\|\le 2^{-n}$ for all $s\in S$. By (\ref{nlone}) combined with Lemma \ref{nlonely}, we have $K_{s,n,F_0}$ nonempty for all $s,n,F_0$. Since this is a filtering decreasing family of nonempty compact subsets of $\ell^\infty(X)$ with the weak-star topology, its intersection is nonempty. Then any element in the intersection is an invariant mean vanishing on singletons, so (\ref{xsangs}) holds.

Conversely, suppose (\ref{xsangs}). Let $\mu$ be an invariant mean on $\ell^\infty(X)/\co(X)$. 
Then there is a net $(u_i)_{i\in I}$ of non-negative finitely supported $\ell^1$-functions of norm 1 such that $u_i$ tends to $\mu$ in the weak-star topology. 

%Let $K'_{S,n,F_0}$ be the set of positive normalized linear forms $v$ on $\ell^\infty(X)$ that are $\le 2^{-n}$ on functions valued in $[-1,1]$ supported by $F_0$, such that $\|sv-v\|\le 2^{-n}$ for all $s\in S$. These are closed convex subsets whose interior contain the closed convex subset $K$ of $G$-invariant means on $\ell^\infty(X)$ vanishing on finitely supported functions.

%For a finite subset $S\subseteq V$ and $n\in\N$, write $V_{V,S,n}\subseteq \ell^1(X)^V$ as the set of $(f_s)_{s\in V}$ in $\ell^1(X)^{V}$ such that $\sup_{s\in S}\|f_s\|_1\le 2^{-n}$. Letting $S,n$ vary, this is a basis of neighborhoods of 0 in $\ell^1(X)^V$.

Write $W=\pi^{-1}(\alpha(G))$. Consider the power $\ell^1(X)^W$ in its strong topology (product of the norm topologies).
For $V\subseteq W$, let $\Phi_{V}$ be the continuous linear map from $\ell^1(X)$ to $\ell^1(X)^{V}$ mapping $f$ to $(f-s^{-1}f)_{s\in V}$. 
The map $\Phi_V$ is also continuous from the weak topology to the topology $\mathcal{T}$ on $\ell^1(X)^V$ defined as the product of weak topologies. Then $(\Phi(u_i))$ tends to 0 in $(\ell^1(X)^V,\mathcal{T})$. Hence for every $i$, the point 0 belongs to the closed convex hull of $\{\Phi_V(u_j):j\ge i\}$ in the weak topology. If $S$ is finite, closed convex subsets of $\ell^1(X)^S$ are closed in the weak topology. Hence for $n$ and a finite subset $S\subseteq W$ given, this implies that for every $i$, the point 0 belongs to the closed convex hull of $\{\Phi_S(u_j):j\ge i\}$, for the product of norm topologies. Since this holds for all finite $S$ and just using the definition of product topology, it formally implies that for every $i$, the point 0 belongs to the closed convex hull of $\{\Phi_W(u_j):j\ge i\}$. Hence there exists a net (possibly indexed by a larger poset, which can be chosen as $I\times\mathcal{P}_{\mathrm{fin}}(W)\times\N$) $(u'_j)$ of non-negative finitely supported $\ell^1$-functions of norm 1 such that $s^{-1}u'_i-u'_i$ tends to zero in the norm topology, for each given $s\in W$. This precisely yields (\ref{nlone}).  

Now let us prove the equivalence between (\ref{nfol}) and (\ref{nlone}). We start with the implication (\ref{nfol})$\Rightarrow$(\ref{nlone}). In the context of actions, the analogous implication is trivial, and here there is a minor addendum due to the change of context. Namely, consider $S$ ,$\varepsilon$ and $F_0$. We first enlarge $S$ to have a symmetric image in $\mkst(X)$, and for every $s\in S$ define $s'\in S$ as an element such that $\pi(s')\pi(s)=1$. Then we enlarge $F_0$ as follows: we define $F_1$ as the union of $F_0$ and the set of $x\in X$ such that for some $s\in S$ we have $(ss')^{-1}x\neq x$ or $(s's)^{-1}x\neq x$. Then, for every finite subset $F$ of $X\smallsetminus F_0$, we have $\|s^{-1}1_F-1_F\|_1\le 2|\partial_SF|$. Given this, we immediately obtain the desired implication.

The less trivial implication (\ref{nlone})$\Rightarrow$(\ref{nfol}) also starts with the same process of enlarging $F_0$ to avoid a few bad points. Then the remainder of the proof follows the same line as the analogous proof for group actions, due to Greenleaf \cite[Theorem 4.1]{Gre}.
\end{proof}

% By continuity, for all $S,n,F_0$, the point 0 belongs to the closure of the convex subset $\Phi(K'_{S,n,F_0})$. 
% in its strong topology (product of the norm topologies).

Recall that a graph $X$ (identified with its set of vertices) of bounded valency (possibly not connected) is amenable\index{amenable (graph)} if for every $\varepsilon>0$ there exists a nonempty finite subset $F$ of $X$ such that $\frac{|\partial F|}{|F|}\le\varepsilon$, where $\partial F$ is the set of elements of $F$ that are adjacent to some element not in $F$. 

\begin{cor}
Let $\Gamma$ be a finitely generated group and $X$ a near $\Gamma$-set. Consider a near Schreier graph $\mathcal{G}$ for $X$. Then $X$ is a $\star$-amenable near $\Gamma$-set if and only if one the following holds:
\begin{enumerate}
\item $\mathcal{G}$ has infinitely many finite components;
\item the union of all infinite components of $\mathcal{G}$ is an amenable graph.
\end{enumerate}
In particular, if $X$ is a near $\Gamma$-set of finite type, then $X$ is a $\star$-amenable near $\Gamma$-set if and only at least one infinite component of $\mathcal{G}$ is an amenable graph.
\end{cor}
\begin{proof}

Suppose that the near Schreier graph is constructed from $S_0$ a finite symmetric generating subset of $\Gamma$. Assuming it symmetric does not affect the criterion, because it affects the adjacency relation for only finitely many points.

Using that $S_0$ is symmetric, observe that in the criterion (\ref{nfol}) of Proposition \ref{amefol}, it is enough to check it for some lift $S$ of $S_0$ in $\mkst(X)$. For every $s\in S$, choose $s'\in S$ such that $\pi(s)\pi(s')=1$. Define $F_1=\{x:ss'x\neq x$ or $s'sx\neq x\}$.

Given this remark (and focussing on $F_0$ including $F_1$), this criterion translates into: the near $\Gamma$-set $X$, with given near Schreier graph, is amenable if and only if for every $\varepsilon>0$ and every finite subset $F_0$ of $X$,
\[\exists F\in\mathcal{P}_{\mathrm{fin}*}(X):\;F\cap F_0=\emptyset,\;\frac{|\partial F|}{|F|}\le\varepsilon,\]
where $\partial F$ is the set of elements of $F$ adjacent to a vertex outside $F$. From this, the desired characterization immediately follows.
\end{proof}

\begin{cor}\label{folseq}
Let $G$ be a countable group and $X$ a near $G$-set. Then $X$ is a 
$\star$-amenable near $G$-set if and only if there exists a sequence $(F_n)$ of disjoint nonempty finite subsets of $X$ such that for every finite subset $S$ of $\pi^{-1}(\alpha(G))$ we have $\lim_{n\to\infty}\frac{|\partial_SF_n|}{|F_n|}=0$.
\end{cor}
\begin{proof}
Obviously the condition implies (\ref{nfol}) of Proposition \ref{amefol}. Conversely, assume that (\ref{nfol}) of Proposition \ref{amefol} holds; in particular $X$ is nonempty. Let $T$ be a countable subset of $\pi^{-1}(\alpha(G))$ such that $\pi(T)=\alpha(G)$.
Let $(S_n)_{n\ge 1}$ be an increasing sequence of finite subsets of $T$ covering $T$. Define $F_0$ as an arbitrary singleton of $X$. By (\ref{nfol}) of Proposition \ref{amefol} and by induction, there exists for $n\ge 1$ a nonempty finite subset $F_n$ of $X$, disjoint of $\bigsqcup_{i=0}^{n-1}F_i$, such that $\frac{|\partial_{S_n}F_n|}{|F_n|}\le 2^{-n}$. Hence for every finite subset $S$ of $T$, we have $\lim_{n\to\infty}\frac{|\partial_{S}F_n|}{|F_n|}\to 0$. Now let $S$ be an arbitrary finite subset of $\pi^{-1}(\alpha(G))$. Then there exists a finite subset $S'$ of $T$ such that $\pi(S)=\pi(S')$. Then since the $(F_n)$ are pairwise disjoint, for $n$ large enough we have $\partial_{S}F_n=\partial_{S'}F_n$. Hence $\lim_{n\to\infty}\frac{|\partial_{S}F_n|}{|F_n|}\to 0$.
\end{proof}

It was proved by Granirer \cite[Theorem B]{Gra} that the convex set of left-invariant means on any infinite amenable group has infinite dimension, and Chou \cite{Chou} proved it has cardinal (and hence dimension) $\ge 2^{2^{\aleph_0}}$. Rosenblatt and Talagrand \cite{RoT} proved the latter result for an arbitrary action of an amenable group on any infinite set. There is just a sample of a large literature on counting the number of invariant means on a group or more generally for actions of amenable groups; however I am not aware of any such study encompassing amenable actions of non-amenable groups. Here we provide the following simple result, for near actions of countable groups (for actions of uncountable groups, see Proposition \ref{stabmeandist}).
% nouveau pour action de G non moyennable?

\begin{cor}\label{manymeans}
Let $G$ be a countable group and $X$ a near $G$-set. Then the convex subset of $G$-invariant means on $\mathcal{P}^\star(X)$ is either empty, or has at cardinal $\ge 2^{2^{\aleph_0}}$. 
\end{cor}
\begin{proof}
Assume that it is not empty, i.e., $X$ is $\star$-amenable.
Consider a sequence $(F_n)$ as in Corollary \ref{folseq}, and $f_n=\frac{1}{|F_n|}1_{F_n}$. For every ultrafilter $\omega$ on $\N$, define $f_\omega=\lim_{n\in\omega}f_n$, where the limit is in the unit ball of the $\ell^\infty(X)^*$ endowed with the weak-star topology (or equivalently in $[0,1]$-valued functions on $\mathcal{P}^\star(X)$ with pointwise convergence).

% Fix a nonprincipal ultrafilter $\omega$ on $\N$ and define $f_t=\lim_{n\to\omega}f_{u_n^t}$, where the limit is in $\ell^\infty(X)$ endowed with the weak-star topology. 

For every subset $I$ of $\N$, define $F_I=\bigcup_{n\in I}F_n$. Then for every ultrafilter $\omega$ and $I\in\omega$, we have $\langle 1_{F_I},f_\omega\rangle=1$. Moreover, using that the $F_n$ are pairwise disjoint, for every $I\notin\omega$, we have $\langle 1_{F_I},f_\omega\rangle=0$. It immediately follows that the $(f_\omega)$ form a linearly independent family, and in particular are pairwise distinct, when $\omega$ ranges over ultrafilters.

By the F\o lner condition (of Corollary \ref{folseq}), for every non-principal ultrafilter $\omega$ on $\N$, the mean $f_\omega$ on $\mathcal{P}^\star(X)$ is $G$-invariant.
\end{proof}
 
%If $\omega\neq\eta$ are ultrafilters, then there exists a subset $I$ of $\N$ such that $\omega(I)\neq 0=\eta(I)$. Define $F_I=\bigcup_{\n\in I}F_n$. Then $\langle f_\omega$, using that the $F_n$ are pairwise disjoint, we have Then $f_t$ is a $G$-invariant mean on $\mathcal{P}^\star(X)$. In addition, we have $\langle f_t,1_{F_t}\rangle=1$ while $\langle f_s,1_{F_t}\rangle=0$ for all $t\neq s$. Hence the $f_t$ are form a linearly free family. 

%Consider a family $(I_t)_{t\in T}$ of infinite subsets $\N$, indexed by a set $T$ of continuum cardinal and with $I_t\cap I_s$ finite for all $t\neq s$ in $T$. Write $I_t$ as image of an injective sequence $(u_n^t)_n$. 

In view of Proposition \ref{gsetngsetam}, we deduce:

\begin{cor}
Let $G$ be a countable group and $X$ a near $G$-set. Then $X$ is a $\star$-amenable $G$-set if and only if the set of invariant means is infinite.\qed
\end{cor}

We now proceed to show the failure of Corollary \ref{manymeans} for suitable actions of uncountable group, including the case when $X$ is infinite countable. 

\begin{lem}\label{combiab}
Let $X$ be an infinite set. Let $\omega$ be an ultrafilter on $X$ and $m\neq\omega$ a mean on $X$. Then there exist two disjoint moieties $A,B$ on $X$ such that $\omega(A\cup B)=0$ and $m(A)\neq m(B)$.
\end{lem}
\begin{proof}
We first claim that there exists a moiety $C$ such that $m(C)\neq 0=\omega(C)$. By assumption, there exists a subset $D$ such that $m(D)\neq 0=\omega(D)$. 
If $D$ contains a moiety, then it is the union of two moieties and one of them has positive $m$-weight. If $D^c$ contains a moiety, the it is the union of two moieties and one of them has full $\omega$-weight. In both case we find $C$.

Now partition $C$ into two moieties $A_0,B_0$. If $m(A_0)\neq m(B_0)$, we are done with $(A,B)=(A_0,B_0)$. Otherwise, partition $B$ into two moieties $B_1\sqcup B_2$ with $m(B_1)\ge m(B_2)$, and then set $A=A_0\cup B_1$, $B=B_2$. 
\end{proof}

\begin{prop}\label{stabmeandist}
Let $\omega$ be an ultrafilter on a set $X$. Then $\omega$ is the unique mean on $\mathcal{P}(X)$ preserved by the stabilizer $\mks(X)_\omega$ of $\omega$. If $\omega$ is non-principal, then $\omega$ is the unique mean on $\mathcal{P}^\star(X)$ preserved by the stabilizer $\mks^\star(X)_\omega$ of $\omega$.
\end{prop}
\begin{proof}
Let $m\neq\omega$ be another mean. Let $A,B$ be as in Lemma \ref{combiab}. Let $\sigma$ be any permutation exchanging $A$ and $B$, and being the identity outside $A\cup B$. Then $\sigma$ preserves $\omega$ but not $m$.

This proves the first statement, and the second (which assumes that $m$ and $\omega$ vanish on finite subsets) immediately follows (since this implies that $\mks^\star_0(X)_\omega$ does not preserve $m$).
\end{proof}

Rosenblatt and Talagrand \cite{RoT} asked whether there exists an action of some amenable group $G$ on an infinite set with a unique invariant mean. The question can be specified to locally finite groups acting on a infinite countable set. Then this is undecidable in ZFC: Yang \cite{Yang} proved that under CH (continuum hypothesis) it has a positive answer. Then Foreman \cite{For} proved the same relaxing CH to Martin's axiom, and conversely proved the consistency of ZFC + ``every locally finite group of permutations of $\N$ has at least two invariant means". It seems that the consistency in ZFC of the same for arbitrary amenable group actions is an open question. Note that the stabilizer $\mks(X)_\omega$ in Proposition \ref{stabmeandist} contains an isomorphic copy of $\mks(Y)$ for every subset $Y$ with $\omega(Y)=0$, and hence is non-amenable.

\subsection{Growth of near Schreier graphs}\label{growthnsch}

Given two functions $f,g\N\to\R_{\ge 0}$, we say that $f\preceq g$, and say that $f$ is asymptotically bounded above by $g$, if there exist a constant $C\ge 1$ such that $f(n)\le Cg(Cn)+C$ for all $n$, and $f\simeq g$ if $f\preceq g\preceq f$. This is an equivalence relation, and the class of $f$ modulo this equivalence relation is called the asymptotic growth of $f$.

Given a connected graph $\mathcal{G}$ of bounded valency and a choice $x_0$ of vertex, the growth function of $(\mathcal{G},x_0)$ is the function mapping $n\in\N$ to the cardinal of the $n$-ball around $x_0$. Its asymptotic growth is a quasi-isometry invariant of $\mathcal{G}$, see for instance \cite[\S 3.D]{CH}.

Given a finitely generated group $\Gamma$ and a near $\Gamma$-set $X$ of finite type, consider a key fob completion (see \S\ref{nfg}) of its near Schreier graph $\mathcal{G}$ and choose a base vertex $x_0$. 
Since the quasi-isometry type of $\mathcal{G}$ on the near $G$-set $X$ and not on the various choices, Lemma \ref{kfqi} implies that the asymptotic growth of this graph only depends on the near $G$-set $X$, and not on the auxiliary choices. We call it the growth of the near $G$-set $X$. When $X$ is $G$ is $G$ endowed with the left action, this is well-known as the growth of $G$.

We ignore the answer to the following question:

%tart with a question to which we ignore the answer. 

%ADAPTER 5C

\begin{que}\label{growthque}
Consider a finitely generated group $\Gamma$ and a near $G$-set $X$ of finite type. Is the growth of the near $G$-set $X$ asymptotically bounded above by the growth of $\Gamma$?
\end{que}

This is true for actions, and hence the answer is positive for completable near actions. The answer is also trivially true when $\Gamma$ has exponential growth, since the growth of any graph with bounded valency growth at most exponentially. In \S\ref{s_finends}, we obtain a positive answer when $\Gamma$ is virtually abelian; let us insist that even the case of $\Z^2$ is nontrivial. One open case is the case of the integral Heisenberg group, for which we only have the superpolynomial bound of Proposition \ref{t_ineqgr} below.

In Lemma \ref{sumineq} below, we establish an inequality that leads the following partial answer:

\begin{prop}\label{t_ineqgr}
If $\Gamma$ has polynomial growth, then there exists $C$ such that the growth of every near Schreier graph of $\Gamma$ is $\preceq n^{C\log(n)}$.

If $\Gamma$ has growth $\preceq\exp(n^\alpha)$ for some $0<\alpha\le 1$, so does every near Schreier graph of $\Gamma$.

If $\Gamma$ has subexponential growth, so does every near Schreier graph of $\Gamma$.
\end{prop}

Fix a finite generating subset of $G$ and let $b_0(r)$ be the cardinal of the ball of radius $r$. Fix a connected near Schreier graph for a near $G$-set $X$ and fix a vertex $v_0$. The first immediate but crucial lemma is the following:

\begin{lem}\label{ballpascen}
There exists $k_0$ such that for every vertex $v$ in the near Schreier graph and every $r\le d(v,v_0)-k_0$, the cardinal of the $r$-ball centered at $v$ is $\le b_0(r)$.\qed
\end{lem}

Let now $b(r)$ be the cardinal of the ball of radius $r$ centered at $v_0$ in the near Schreier graph.

\begin{lem}\label{sumineq}
We have $b(3r/2)\le b(r)(1+b_0(2r/3))$ for all large enough $r$. In particular, writing $f(r)=\log(b((3/2)^{r}))$ and $f_0(r)=\log(1+b_0((3/2)^{r+1}))$, we have, for all $r$ large enough, the inequality $f(r+1)\le f(r)+f_0(r)$. In particular, writing $F_0(r)=\sum_{1\le n\le r-1}f_0(n)$, there exists a constant $C$ such that $f(r)\le C+F_0(r)$ for all $r$.
\end{lem}
\begin{proof}
Indeed, the $3r/2$-ball is covered by the $\lfloor r\rfloor$-ball and the union over the $\lfloor r\rfloor$-sphere $S$ of all balls of radius $r/2$. For $r$ large enough, these balls have cardinal $\le b_0(1+r/2)\le b_0(2r/3)$ by Lemma \ref{ballpascen}. Hence, we have the inequality
\[b(3r/2)\le b(r)+|S|b_0(2r/3).\]
Since $|S|\le b(r)$, the first inequality follows. 

From the first inequality applied to $(3/2)^r$ and composing with log, we deduce the second inequality. The last inequality follows immediately.
%\[f(r+1)\le f(r)+f_0(r).\]
\end{proof}

\begin{proof}[Proof of Proposition \ref{t_ineqgr}]
These are now simple analytic considerations based on the last inequality of Lemma \ref{sumineq}.

First suppose that $b_0(r)\preceq r^d$. So $1+b_0(3r/2)\le cr^d$ for some $C$ and all $r$. Hence $\exp(f_0(r))\le c(3/2)^{rd}$ for all large $r$, that is, $f_0(r)\le c+rd\log(3/2)$ for all large $r$. Hence $F_0(r)\le c'+dr^2\log(3/2)/2$ for some constant $c'$ and all large $r$. By the last inequality of Lemma \ref{sumineq}, for some constant $c''$, we have $f(r)\le c''+dr^2\log(3/2)/2$ for all large $r$. Hence $\exp(f(r))\le e^{c''}\exp(dr^2\log(3/2)/2)$ for all large $r$. Applying this to $\log_{3/2}r$, we obtain $b(r)\le e^{c''}\exp(d\log(r)^2/(2\log(3/2)))$, or equivalently 
$b(r)\le e^{c''}r^{d\log(r)/(2\log(3/2)))}$. (The constant $1/(2\log(3/2))$ is not optimal since the choice of $3/2$ is an artifact of the proof.)

Now consider the case when $b_0$ grows subexponentially. This means that $\log(1+b_0(r))=o(r)$. We write this as: $\log(1+b_0(r))=r\eps(\log_{3/2}(r))$, where $\eps(s)=o(1)$ when $s\to\infty$. So $f_0(r)=(3/2)^r\eps(r)$. Hence, for every $\eta>0$ there exists $r$ such that $f_0(r)\le \eta(3/2)^r$ for $r$ large enough. Summing, we deduce that $\limsup F_0(r)(2/3)^r\le 2\eta$. Since this holds for all $\eta$, we deduce that $F_0(r)=o((3/2)^r)$. By Lemma \ref{sumineq}, there exists a constant $c$ such that $f(r)\le c+F_0(r)$ for all $r$. So we can write $c+F_0(r)=(3/2)^r\eps'((3/2)^r)$, with $\eps'(s)=o(1)$. Hence $f(\log_{3/2}(r))\le r\eps'(r)$ for all $r$, and passing to the exponential, we obtain $b(r)\le \exp(r\eps'(r))$ for all $r$, which means that $b$ growth subexponentially.

Finally consider the case $b_0(r)\preceq\exp(r^{\alpha}))$. This means that there exists a constant $c$ such that $1+b_0(3r/2)\le\exp(cr^\alpha)$ for all $r$. Hence $f_0(r)\le c(3r/2)^\alpha=c((3/2)^\alpha)^r$ for all $r$. By Lemma \ref{sumineq}, there exists a constant $c'$ such that $f(r)\le c+F_0(r)$ Hence, for some constant $c''$, we have $c'+F_0(r)\le c''((3/2)^{\alpha})^r$. Hence this is an upper bound for $f(r)$. Hence $f(\log_{3/2}r)\le c''r^\alpha$ for all $r$. Passing to the exponential, we deduce $b(r)\le\exp(c''r^\alpha)$ for all $r$.
\end{proof}

\subsection{The Kapoudjian class}\label{s_omega}
The idea of using the quotient by the alternating group to construct 2-cocycles in $\Z/2\Z$ originates, as far I as know, in work of Kapoudjian and Kapoudjian-Sergescu \cite{Ka,KaS}, in the context of groups of near automorphisms of trees. Here we formulate this simple cohomological construction in the general setting of balanced near actions. Then, with significantly more work, we extend it to arbitrary near actions.

Let $X$ be an infinite set; let $\mka(X)$ be the alternating group of $X$, which is the unique subgroup of index 2 in $\mks_{<\aleph_0}(X)$. Consider the quotient $\tmksh(X)=\mks(X)/\mka(X)$. Thus, we have a central extension 
\[0\to\Z/2\Z\to \tmksh(X)\to\mksh(X)\to 1.\] 

Let $\omega^X\in H^2(\mksh(X),\Z/2\Z)$ be the cohomology class of this extension. We call it the canonical class\index{Kapoudjian class} of the set $X$ (if $X$ is finite we set $\omega^X=0$). I asked the question whether this is the only nontrivial element in $H^2(\mksh(X),\Z/2\Z)$, and V.\ Sergiescu gave me a positive answer.

\begin{prop}[Sergiescu]\label{sergi}
For every infinite set $X$, the group $H_2(\mksh(X))$ is reduced to $\Z/2\Z$. In particular, for any abelian group $A$, the group $H^2(\mksh(X),A)$ is naturally isomorphic to $\Hom(\Z/2\Z,A)$ (the kernel of multiplication by 2 in $A$).
\end{prop}
\begin{proof}
This is based on the result by P.\ de la Harpe and D. McDuff \cite{HMD} that the symmetric group $\mks(X)$ is acyclic; more precisely we only need that $H_2(\mks(X))$ and $H_1(\mks(X))$ are zero (the latter being of course known much before). Namely, we use that for any group $G$ and normal subgroup $N$, there is an exact sequence $H_2(G)\to H_2(G/N)\to H_1(N)_G\to H_1(G)$, where $H_1(N)_G=N/[G,N]$ is the abelian group of coinvariants of $H_1(N)$ with respect to the $G$-action. See \cite[Exercise 6 p.47 or Corollary VII.6.4]{Bro}. Apply this to $G=\mks(X)$ and $G/N=\mksh(X)$, we obtain the result, since $H_1(N)=\Z/2\Z$ and hence $H_1(N)_G=\Z/2\Z$. 

Since $\mksh(X)$ is a perfect group, the cohomology result follows, e.g., by \cite[Exercice 7(a) p.96]{Bro}.
\end{proof}

%\cite[Exercice 3 p.60]{Bro} HOMOLOGY -> COHOMOLOGY

%\cite[Corollary 8.2]{HiSt} or 

%\bibitem[HiSt]{HiSt}

Let $G$ be a group with a balanced near action given by a homomorphism $\alpha:G\to\mksh(X)$. 

\begin{defn}\label{omegaalmost}
We call the cohomology class $\alpha^*\omega^X\in H^2(G,\Z/2\Z)$ the Kapoudjian class of the balanced near action; we denote it $\omega_\alpha$ (or $\omega_{G,X}$, or $\omega_X$ when the context permits).
\end{defn}

\begin{prop}\label{propomegaa}
Fix a group $G$.
\begin{enumerate}
\item\label{omega1} Let $f:H\to G$ be a homomorphism and $X$ a balanced near $G$-set. Then $\omega_{H,X}=f^*\omega_{G,X}$.
\item\label{omega2} Let $X,Y$ be balanced near $G$-sets. Then $\omega_{X\sqcup Y}=\omega_X+\omega_Y$.
\item\label{omega3} If $X$ is a stably realizable $G$-set, then $\omega_X=0$.
\item\label{omega4} $\omega_X$ is a near isomorphism invariant of the balanced near $G$-set $X$. 
\end{enumerate}
\end{prop}
\begin{proof}
(\ref{omega1}) is immediate from the definition.

(\ref{omega2}) For each $g\in G$, choose a representative $\tilde{g}\in\mks(X)$. Then a 2-cocycle representing $\omega_X$ is given by $(g,h)\mapsto \widetilde{gh}(\tilde{g}\tilde{h})^{-1}$. The result immediately follows.

(\ref{omega3}) If $Y$ is realizable (e.g., with trivial near), clearly $\omega_Y=0$. Then the result follows from (\ref{omega2}). 

(\ref{omega4}) Clearly, this is a balanced near isomorphism invariant. Moreover, it is the same for $X$ and $X\smallsetminus\{x_0\}$. Indeed, in the above description of the cocycle, we can choose $\tilde{g}$ to fix $x_0$ for all $g$. Hence, it is a near isomorphism invariant.
\end{proof}

The group $H^2(\Z^2,\Z/2\Z)$ is cyclic of order 2; its nontrivial element is represented by the 2-cocycle $((m_1,m_2),(n_1,n_2))\mapsto m_1n_2$ ($\mathrm{mod}\,2$).

\begin{prop}\label{omegaz2}
Let $G$ be a free abelian group of rank 2, with generators $u,v$. Consider a balanced near action of $G$ on a set $X$, and let $\tilde{u},\tilde{v}$ be permutations lifting $u,v$. Then the Kapoudjian class $\omega_X$ is the nontrivial element of $H^2(G,\Z/2\Z)$ if and only if the commutator $[\tilde{u},\tilde{v}]$, which is finitely supported, is an odd permutation.
\end{prop}
\begin{proof}
A lift for the whole group is given by $u^mv^n\mapsto\tilde{u}^n\tilde{v}^m$.  
Let $b:(g,h)\mapsto\widetilde{gh}\tilde{h}^{-1}\tilde{g}^{-1}$ ($\mathrm{mod}$ $\mka(X)$). Since $\widetilde{ug}=\tilde{u}\tilde{g}$ and $\widetilde{gv}=\tilde{g}\tilde{v}$, we see that $b(ug,hv)=b(g,h)$ for all $g,h$. Hence
\[b(u^av^b,u^cv^d)=b(v^b,u^c)=\widetilde{v^bu^c}\widetilde{u^c}^{-1}\widetilde{v^b}^{-1}=\tilde{u}^c\tilde{v}^b\tilde{u}^{-c}\tilde{v}^{-b};\;(\mathrm{mod}\;\mka(X))\]
using that this lies in an abelian group, the commutator is bilinear, and hence this is equal to $[\tilde{u},\tilde{v}]^{bc}$ modulo $\mka(X)$. Hence if this commutator is even, the cocycle is zero; if the commutator is odd, we recognize the nontrivial cohomology class. 
\end{proof}

\begin{exe}\label{e_hough}
We provide here a fundamental example. Consider the ``Hough\-ton" balanced near action of $\Z^2$\index{Houghton near action}. Namely, consider $X=\N\times\{1,2,3\}$. Define, for  $\{i,j\}=\{1,2\}$, a permutation $f_i$ of $X$ by
\[f_i(n,i)=n+1,\;f_i(n,j)=(n,j),\;f_i(n+1,3)=n,\;f_i(0,3)=(0,i),\;n\in\N.\]
The group $\langle f_1,f_2\rangle\subseteq \mks(X)$ is known as Houghton's group. 

The commutator $f_1f_2f_1^{-1}f_2^{-1}$ is equal to the transposition of $(0,1)$ and $(0,2)$. 
Then $f_1$ and $f_2$ commute modulo finitely supported permutations, and hence this defines a balanced near action of $\Z^2$ on $X$, whose Kapoudjian class is nonzero by Proposition \ref{omegaz2}. See \S\ref{s_hou} for more on Houghton near actions.

An analogous example is the balanced near action of $\Z\times (\Z/2\Z)$ described in Example \ref{firstex}(\ref{exzz2}), in which the commutator $a^{-1}b^{-1}ab$ is the transposition $(0,1)\leftrightarrow (0,-1)$, so the Kapoudjian class is the unique nonzero element of $H^2(\Z\times (\Z/2\Z),\Z/2\Z)$. 
\end{exe}

%Hence $\omega_{\Z^2,V}\neq 0$. On $\Z^2$, this cohomology class is represented by the 2-cocycle $((m_1,m_2),(n_1,n_2))\mapsto m_1n_2$ ($\mathrm{mod}\,2$).

%01  01 03 02 02
%02  03 13 03 01

%\N\times\{1,2,3\}$ where for $\{i,j\}=\{1,2\}$, the generator $e_i$ of $\Z^2$ near acts on $\N\times\{i\}$ by $n\mapsto n+1$, near acts trivially on $\N\times\{j\}$, and near acts on $\N\times\{3\}$ by $n\mapsto n-1$. Define a 

We wish now to extend the Kapoudjian class to arbitrary near actions, using the Hilbert hotel join (Remark \ref{hilbho}). Let us reintroduce the latter in detail. It consists in the construction, for every set $X$, of a canonical injective homomorphism $j_X$ of $\mkst(X)$ into $\mksh(X\sqcup\N)$. Namely, we consider the near $\mkst(X)$-set $\N$, with near action given by $g\cdot n=n-\phi_X(g)$, where $\phi_X:\mkst(X)\to\Z$ is the index character (its image is the group $A_X$ of \S\ref{i_nsgi}). Then the index character of $\N$ is equal to $-\phi_X$, and thus the index character of the disjoint union $X\sqcup\N$ vanishes, so $j_X$ is well-defined. We write $\omega^X_\star=j_X^*\omega^X\in H^2(\mkst(X),\Z/2\Z)$. We call $\omega^X_\star$ the second canonical class of the set $X$.

Let $G$ be a group with a near action given by a homomorphism $\alpha:G\to\mksh(X)$. 

\begin{defn}\label{omeganear}
We call the cohomology class $\alpha^*\omega^X_\star\in H^2(G,\Z/2\Z)$ the Kapoudjian class of the near action;\index{Kapoudjian class} we denote it $\omega_\alpha$ (or $\omega_{G,X}$ or $\omega_X$ when the context permits). In other words, $\omega_\alpha=(j_X\circ\alpha)^*\omega^{X\sqcup\N}$.
\end{defn}

Thanks to the following lemma, this definition is compatible with Definition \ref{omegaalmost} and thus extends it.

\begin{lem}\label{kapcom}
Provisionally denote by $\omega^\star_\alpha$ the cohomology class defined in Definition \ref{omeganear}. Then for every balanced near action $\alpha$, we have $\omega_\alpha=\omega^\star_\alpha$.
\end{lem}
\begin{proof}
Let $\beta$ be the balanced near $G$-action on $X\sqcup\N$ (with near action $\alpha$ on $X$ and trivial near action on $\N$). Since $\alpha$ is a balanced near action, we have $j_X\circ\alpha=\beta$. By definition, $\omega^\star_\alpha=(j_X\circ\alpha)^*\omega^{X\sqcup\N}=\beta^*\omega^{X\sqcup\N}=\omega_{X\sqcup\N}$. By Proposition \ref{propomegaa}(\ref{omega2}), this is equal to $\omega_X+\omega_{\N}=\omega_X=\omega_\alpha$. 
\end{proof}

The following is now immediate.

\begin{prop}
Let $f:H\to G$ be a homomorphism and $X$ a near $G$-set. Then $\omega_{H,X}=f^*\omega_{G,X}$. The class $\omega_X\in H^2(G,\Z/2\Z)$ is a near isomorphism invariant of the near $G$-set $X$.\qed 
\end{prop}

\begin{lem}\label{restrcoh}
For any abelian group $A$, any perfect group $N$ and any automorphism of $N$ defining a semidirect product $N\rtimes\Z$, the restriction homomorphism $H^2(N\rtimes\Z,A)\to H^2(N,A)$ is injective, and the canonical homomorphism $H_2(N)\to H_2(N\rtimes\Z)$ is surjective, and we have an isomorphism $H^2(G,A)\to \Hom(H_2(G),A)$.
\end{lem}
\begin{proof}
Indeed, if we have a group $G$ with a central embedding of $A$ and $G/A\simeq N\rtimes\Z$ such that the cohomology class is trivial on $N$, we can lift $N$ to a subgroup $N'$ of $G$; we can also lift $\Z$ to a subgroup $\langle t\rangle$. If $g,h$ belong to $N'$, then there exist $g',h'\in N$ and $z,z'$ in $A$ such that $tgt^{-1}=g'z$ and $tht^{-1}=h'z'$. Then 
\[t[g,h]t^{-1}=[tgt^{-1},tht^{-1}]=[g'z,h'z']=[g',h']\in N\]
% belongs to the lifted $N$; s
Since $N$ is perfect we deduce that $tNt^{-1}\subseteq N$ and similarly $t^{-1}Nt\subseteq N$. Hence $\langle t\rangle$ normalizes $N$ and thus $A$ has a direct summand.

Now write $G=N\rtimes\Z$. By \cite[Exercise 3 p.60]{Bro}, the homomorphism $H^2(G,A)\to \Hom(H_2(G),A)$ is surjective with kernel $\mathrm{Ext}^1_\Z(H_{1}(G),A)$. Since $H_1(G)=\Z$ is a projective $\Z$-module, this Ext-group vanishes and the above homomorphism is an isomorphism. The same also holds for $N$ instead of $G$ (since $H_1(N)=0$). Thus the homomorphism $H_2(N)\to H_2(G)$ has the property that $\Hom(H_2(G),A)\to\Hom(H_2(N),A)$ is injective for any abelian group $A$. This means that $H_2(N)\to H_2(A)$ is an epimorphism in the category of abelian groups, and hence is surjective.
\end{proof}

\begin{prop}
For every infinite set $X$, the group $H_2(\mkst(X))$ is isomorphic to $\Z/2\Z$, and the canonical homomorphism $H_2(\mksh(X))\to H_2(\mkst(X))$ is an isomorphism. For any abelian group $A$, the restriction map $H^2(\mkst(X),A)\to H^2(\mksh(X),A)$ is an isomorphism, and these are isomorphic to $\Hom(\Z/2\Z,A)$. In particular, $H^2(\mkst(X),\Z/2\Z)$ is reduced to $\{0,\omega^X\}$. 
\end{prop}
\begin{proof}
By Proposition \ref{sergi}, we have $H_2(N)=\Z/2\Z$. By Lemma \ref{restrcoh}, $H_2(N)\to H_2(G)$ is surjective, and $H^2(G,A)$ is isomorphic to $\Hom(H_2(G),A)$ for every abelian group $A$. By Lemma \ref{kapcom}, $H^2(G,\Z/2\Z)\neq 0$. This excludes the possibility $H_2(G)=0$, and hence $H_2(G)=\Z/2\Z$. 
\end{proof}

%\begin{rem}
%It is a general easy fact that for any abelian group $A$, any perfect group $N$ and any automorphism of $N$ defining a semidirect product $N\rtimes\Z$, the restriction homomorphism $H^2(N\rtimes\Z,A)\to H^2(N,A)$ is injective. Indeed, if we have a group $G$ with a central embedding of $A$ and $G/A\simeq N\rtimes\Z$ such that the cohomology class is trivial on $N$, we can lift $N$; we can also lift $\Z$ to a subgroup $\langle t\rangle$. If $g,h$ belong to the lifted $N$, then $t[g,h]t^{-1}=[tgt^{-1},tht^{-1}]$ belongs to the lifted $N$; since $N$ is perfect we deduce that $t$ normalizes $N$ and hence $Z$ has a direct summand.

%It follows in particular that for every infinite set $X$, the restriction homomorphism $H^2(\mkst(X),\Z/2\Z)\to H^2(\mksh(X),\Z/2\Z)$ is injective. What precedes shows that the canonical class $\omega^X$ belongs to its image, namely is the image of the second canonical class $\omega^X_\star$. The above injectivity shows that $\omega^X_\star$ is the unique pre-image, and in particular shows that it is a canonically defined object (while the way it was defined can appear as artificial). 

%I do not know whether $H^2(\mksh(X),\Z/2\Z)$ is reduced to $\{0,\omega^X\}$. More generally, I do not know whether the homology group $H_2(\mksh(X),\Z)$ is reduced to a cyclic group on 2 elements.
%\end{rem}

\begin{rem}
By definition, for every near $G$-set, the Kapoudjian class $\omega_X$ equals the Kapoudjian class of the balanced near $G$-set $X\sqcup\N$. Hence, by Proposition \ref{propomegaa}(\ref{omega2}) the set of Kapoudjian classes of all near $G$-sets forms a subgroup of $H^2(G,\Z/2\Z)$.

Since any finite group action is realizable, for every group $G$ and any near $G$-set $X$, the Kapoudjian class $\omega_X$ belongs to the intersection of kernels of all restriction maps $H^2(G,\Z/2\Z)\to H^2(F,\Z/2\Z)$, when $F$ ranges over all finite subgroups of $G$.
\end{rem}

While the other assertions of Proposition \ref{propomegaa} hold, its assertion (\ref{omega2}) about disjoint unions takes a new form. 

%For a near $G$-set $X$, composing the index character $\phi_X$ by reduction modulo 2, we obtain a homomorphism $\phi_X^{\mathrm{mod}\,2}$. 

%Taking the cup-product with itself, we obtain a cohomology class $\phi_X^{\mathrm{mod}\,2}\smsm\phi_X^{\mathrm{mod}\,2}$, represented by the 2-cocycle $(g,h)\mapsto \phi_X^{\mathrm{mod}\,2}(g)\phi_X^{\mathrm{mod}\,2}(h)$.

%The definition of the omega class for near actions can look artificial, or, at least, . It was produced by adding one copy of $\N$ with a 

\begin{thm}\label{t_union}
Let $G$ be a group and let $X,Y$ be near $G$-sets with index characters $\phi_X$ and $\phi_Y$. Then 
\[\omega_{X\sqcup Y}=\omega_X+\omega_Y+q_{X,Y},\]
where $q_{X,Y}\in H^2(G,\Z/2\Z)$ is the cohomology class defined indifferently  by one of the following:
\begin{itemize}
\item $q_{X,Y}$ is represented by the 2-cocycle $(g,h)\mapsto\phi_X(g)\phi_Y(h)$ $(\mathrm{mod}\,2)$;
\item $q_{X,Y}$ the pull-back of the unique nonzero cohomology class in $H^2(\Z^2,\Z/2\Z)$ by the homomorphism $(\phi_X,\phi_Y):G\to\Z^2$.
\item $q_{X,Y}$ is the reduction modulo 2 of the cup-product $\phi_X\smsm\phi_Y$.
\end{itemize}

%represented by the 2-cocycle $(g,h)\mapsto\phi_X(g)\phi_Y(h)$ $(\mathrm{mod}\,2)$, or, in other words, the pull-back of the unique nonzero cohomology class in $H^2(\Z^2,\Z/2\Z)$ by the homomorphism $(\phi_X,\phi_Y)$. 
%(\phi_{X\sqcup Y}^{\mathrm{mod}\,2}\smsm\phi_{X\sqcup Y}^{\mathrm{mod}\,2}).\]
\end{thm}
\begin{proof}
The three cohomology classes defining $q_{X,Y}$ are clearly equal.

Write $Z=X\sqcup Y$. Consider the set $E=X\sqcup Y\sqcup\N_{X}\sqcup\N_{Y}\sqcup\N_Z$, where each of $\N_X,\N_Y,\N_Z$ is a copy of $\N$. Let $\alpha_X$ and $\alpha_Y$ be the original near actions on $X$ and $Y$, extended to $E$ by being the identity elsewhere. Write $\alpha_Z(g)=\alpha_X(g)\alpha_Y(g)$; this is the union action on $X\sqcup Y=Z$. 

For $T\in\{X,Y,Z\}$, view $j_T$ as defined above, as the inclusion $\mkst(T)\subseteq \mksh(X\sqcup\N_T)$. Define $\beta_T=j_T\circ\alpha_T$. We can write $\beta_T(g)=\alpha_T(x)\eta_T(g)$; then $\eta_T$ is a near action on $\N_T$ by translation; by construction, $\beta_T$ is a balanced near action on $T\sqcup\N_T$ (and hence on $E$ by extending as the identity elsewhere). In addition, for every $g$, the element $q(g)=\eta_{X\sqcup Y}(g)\eta_X(g)^{-1}\eta_Y(g)^{-1}$ is a balanced near permutation of $\N_X\sqcup\N_Y\sqcup\N_Z$. We have, by construction, $\beta_Z(g)=\beta_X(g)\beta_Y(g)q(g)$.

By definition, $\omega_{X\sqcup Y}=\omega_{\beta_Z}$. To compute it, let us choose, for every $g\in G$ and $T\in\{X,Y\}$, a representative $\widetilde{\beta_T}(g)\in\mks(T\sqcup\N_T)$. Also lift $q(g)$ to a permutation $\tq(g)$ of $\N_X\sqcup\N_Y\sqcup\N_Z$ (extended to the identity elsewhere). We then choose the representative $\bz(g)=\bx(g)\by(g)\tq(g)$ for $\beta_Z(g)$.

Define $b^T_{g,h}=\bt(gh)\bt(h)^{-1}\bt(g)^{-1}$. we now compute $b^Z$ in the quotient group $\tmksh(E)=\mks(E)/\mka(E)$. Then 
\[b^Z_{g,h}=\bx(gh)\by(gh)\tq(gh)\tq(h)^{-1}c_{g,h},\]
where \[c_{g,h}=\by(h)^{-1}\bx(h)^{-1}\tq(g)^{-1}\by(g)^{-1}\bx(g)^{-1}.\]
(We just use $c_{g,h}$ as a shorthand since we well not touch it at once; it has no particular relevance.) Write $q_{g,h}=\tq(gh)\tq(h)^{-1}\tq(g)^{-1}$; this is a finitely supported permutation, and hence, since we work modulo $\mka(E)$, this is a central element. Therefore we have 
\[b^Z_{g,h}=q_{g,h}\bx(gh)\by(gh)\tq(g)c_{g,h}\mod\mka(E).\]
Similarly, use that $b^X_{g,h}$ and $b^Y_{g,h}$ are central modulo $\mka(E)$, and also use that $\bx(h)$ commutes with $\by(g)$ to obtain
\begin{align*}
b^Z_{g,h}=& q_{g,h}b^X_{g,h}\bx(g)\bx(h)\by(gh)\tq(g)c_{g,h} \mod\mka(E),\\
= & q_{g,h}b^X_{g,h}b^Y_{g,h}\bx(g)\bx(h)\by(g)\by(h)\tq(g)c_{g,h}\mod\mka(E)\\
= & q_{g,h}b^X_{g,h}b^Y_{g,h}\bx(g)\by(g)\bx(h)\by(h)\tq(g)c_{g,h}\mod\mka(E)\\
=& q_{g,h}b^X_{g,h}b^Y_{g,h}\bx(g)\by(g)d_{g,h}\bx(g)^{-1}\by(g)^{-1},\end{align*}
where we write $d_{g,h}=\bx(h)\by(h)\tq(g)c_{g,h}\bx(g)\by(g)$. Then we have 
\begin{align*}d_{g,h}=&\bx(h)\by(h)\tq(g)\by(h)^{-1}\bx(h)^{-1}\tq(g)^{-1}\\
=& [\bx(h)\by(h),\tq(g)].
\end{align*}
Note that $q(g)$ centralizes $\beta_Y(h)$ and $\beta_X(h)$, and hence this commutator is also central modulo $\mka(E)$. Hence 
\begin{equation}b^Z_{g,h}=b^X_{g,h}b^Y_{g,h}q_{g,h}d_{g,h}\mod\mka(E)\label{bzxy}\end{equation}

It remains to understand the terms $q_{g,h}$ and $d_{g,h}$. Write them in capitals when working in $\Z/2\Z$, which will be useful since we will both use addition and multiplication of $\Z/2\Z$. Write $V=\N_X\sqcup\N_Y\sqcup\N_Z$. 

First, it follows from the definition that $(g,h)\mapsto q_{g,h}=\omega_{G,V}$, the Kapoudjian class of the balanced near $G$-set $V$. Note that we can view $V$ as a balanced near $\Z^2$-set, where the first generator $e_1$ near acts on $(\N_X,\N_Y,\N_Z)$ by $(+1,+0,-1)$ and the second generator $e_2$ near acts by $(+0,+1,-1)$. The Kapoudjian class of this near $\Z^2$-set is the nonzero element of $H^2(\Z^2,\Z/2\Z)$, by Example \ref{e_hough}. 

%Hence $\omega_{\Z^2,V}\neq 0$. On $\Z^2$, this cohomology class is represented by the 2-cocycle $((m_1,m_2),(n_1,n_2))\mapsto m_1n_2$ ($\mathrm{mod}\,2$).

By Proposition \ref{propomegaa}, $\omega_{G,V}$ is the pull-back of this nontrivial class; therefore it is the cohomology class determined by the 2-cocycle
\[(g,h)\mapsto \phi_X(g)\phi_Y(h)\;\;(\mathrm{mod}\;2).\]

To compute $d_{g,h}$, define $P_X$ as the set of permutations of $X\sqcup\N_X$ that eventually act as a translation on $\N_X$, and $L$ as the set of permutations of $V$ that eventually act as a translation of each of $\N_X$, $\N_Y$, $\N_Z$. Note that $P_X$ and $L$ centralize each other modulo finitely supported permutations. It follows that the commutator map induces a bilinear map $\gamma:P_X\times L\to\Z/2\Z$, where $\Z/2\Z$ is identified with $\mks_{<\aleph_0}(E)/\mka(E)$. In particular, it factors through the abelianization. Note that $P_X$ is isomorphic to $\mkst(X)$ and hence its abelianization is $\Z$, given by the index character; choose an element $f\in P_X$ with $\phi_X(f)=1$. In view of Example \ref{e_hough}, we can identify $L$ as the Houghton group, whose abelianization is precisely $\Z^2$. Choose cycles $f_1,f_2$ as in Example \ref{e_hough}. In addition, we can choose $h\in P_X$ that is an infinite cycle, acting as $+1$ on $\N_X$; it generates $P_X$ modulo its derived subgroup. Then $P_X$ and $L$ have disjoint support and hence $\gamma(h,f_2)=0$. In addition, we observe that $\langle h,f_1\rangle$ is also a Houghton group, and hence, again by Example \ref{e_hough}, we have $\gamma(h,f_1)=0$.

%almost action on $\N_X\sqcup\N_Y\sqcup\N_Y$ is precisely the Houghton almost action (Example \ref{e_hough}), and $L$ is the Houghton group. 

%he group $L$ maps onto $\Z^2$ (for the triple of translations length with sum zero) and the kernel consists of finitely supported permutations. Since some transposition is a commutator, it follows that the abelianization is exactly $\Z^2$ ($L$ is known as a Houghton group). We can choose again lifts $\tilde{e}_1$ and $\tilde{e}_2$; we can choose $\tilde{e}_2$ supported by $\N_Y\sqcup\N_Z$ and hence $[f,\tilde{e}_2]=1$, so $\gamma(f,\tilde{e}_2)=0$. Also, the same computation yielding that $[\tilde{e_1},\tilde{e}_2]$ is a transposition yields the same conclusion for $[f,\tilde{e}_1]$. Therefore, $\gamma(f,\tilde{e}_1)$. 

From the above remarks, $\gamma$ is determined by its values on $(h,f_1)$ and $(h,f_1)$. Hence we deduce that $\gamma(u,v)=\phi_X(u)\phi_X(v)$ for all $u\in P_X$ and $v\in L$. 
Hence, the commutator 
$[\bx(h),\tq(g)]$, viewed as an element of $\Z/2\Z$, is equal to $\phi_X(h)\phi_X(g)$ modulo 2. As a function of $(g,h)$, this is the pull-back by $\phi_X$ of the 2-cocycle $(m,n)\mapsto mn\in Z^2(\Z,\Z/2\Z)$. But since $\Z$ is a free group, $H^2(\Z,\Z/2\Z)$ and this 2-cocycle is a 2-coboundary (namely of the function mapping $n$ to 1 if and only $n\mod 4\in\{2,3\}$). Hence the pull-back is a 2-coboundary. For the same reason, the function mapping $(g,h)$ to $[\bx(h),\tq(g)]$ is a 2-coboundary. Since the commutator takes central values modulo $\mka(E)$, it is bilinear and hence we deduce that $(g,h)\mapsto d_{g,h}$ is a 2-coboundary. Hence (\ref{bzxy}) yields the desired formula.
\end{proof}

By a straightforward induction, we deduce:

\begin{cor}
For any $k$ and near $G$-sets $X_1,\dots,X_k$, we have $\omega_{\bigsqcup X_i}=\sum_{i=1}^k\omega_{X_i}+\sum_{1\le i<j\le k}[\phi_{X_i}\smsm\phi_{X_j}]_2$, where $[\cdot]_2$ means reduction modulo 2.\qed
\end{cor}

One motivation for Theorem \ref{t_union} is to compute the Kapoudjian class of a balanced near action that is split into near subactions, in which case it is convenient to formulate the following corollary:

\begin{cor}\label{omega_nunion}
For any $k$ and near $G$-sets $X_0,\dots,X_k$ such that $\sum_{i=0}^k\phi_{X_i}=0$ mod 2 (for instance $\bigsqcup_{i=0}^kX_i$ is a balanced near $G$-set), we have \[\omega_{\bigsqcup X_i}=\sum_{i=1}^k\omega_{X_i}+\sum_{1\le i<j\le k}[\phi_{X_i}\smsm\phi_{X_j}]_2.\]
\end{cor}
\begin{proof}
We have to check that $\sum_{1\le i<j\le k}[\phi_{X_i}\smsm\phi_{X_j}]_2=\sum_{0\le i<j\le k}[\phi_{X_i}\smsm\phi_{X_j}]_2$. 
Let us omit $[\cdot]_2$ in the notation. This amounts to showing that the $i=0$ contribution is zero. This contribution is 
$\sum_{j=1}^k\phi_{X_0}\smsm\phi_{X_j}$, which is equal to $\sum_{i,j=1}^k\phi_{X_i}\smsm\phi_{X_j}$.
Since $\phi\smsm\phi$ vanishes in $H^2(G,\Z/2\Z)$ for all $\phi\in\Hom(G,\Z)$, both $\phi_{X_i}\!\smsm\!\phi_{X_i}$ and $\phi_{X_i}\smsm\phi_{X_j}+\phi_{X_j}\smsm\phi_{X_i}$ vanish and hence this sum is zero.
\end{proof}

\begin{exe}
Consider the Houghton near action of $\Z^d$ (\S\ref{s_hou}). By Corollary \ref{omega_nunion}, its Kapoudjian class is the ``canonical" cohomology class $\sum_{i<j}\pi_i\smsm \pi_j$ of $\Z^d$, where $\pi_i$ is the canonical projection (this element is invariant under the action of $\GL_d(\Z)$ and hence does not depend on the choice of basis. In particular, it is nonzero, which is another way of seeing that this near action is not stably realizable (see Proposition \ref{Hou_not}).
\end{exe}

\begin{exe}
Consider the near $\Z^2$-set $X_{m,0}$ of \S\ref{s_noncomp}. For the most obvious lifts of the generators to permutations of $X_{m,0}$, the commutator is a $m$-cycle, and hence the Kapoudjian class is zero if and only $m$ is odd.

As regards the near $\Z^2$-set $K_m$ of \S\ref{s_noncomp}, its Kapoudjian class of $K_m$ is nonzero if and only if $m$ is odd.
\end{exe}

\section{Near actions of finitely generated abelian groups}\label{s_ab}

This section relies on the notions introduced from \S\ref{nearaction} to \S\ref{s_rnfg} (also described in the introduction), and will also use results from \S\ref{s_equivariant}.

\subsection{Two preliminary results}

For a near action on an infinite set, near free implies $\star$-faithfulness. Here is a setting where the converse holds:

\begin{prop}
Let $G$ be an abelian group. Let $X$ be a 1-ended near $G$-set, and let $H$ be the kernel of the near action. Then the near action of $G/H$ is near free (that is, any $g\notin H$ fixes only finitely many points). In other words, for a 1-ended near $G$-set, $\star$-faithful $\Leftrightarrow$ near free.
\end{prop}
\begin{proof}
Since $G$ is abelian, the set of fixed points of $g$ is commensurated; since the action is 1-ended, it is finite or cofinite. If cofinite, $g$ belongs to the kernel of the near action.
\end{proof}

\begin{lem}\label{indexlest}
Let $G$ be an abelian group. Let $\alpha:G\to\mkst(X)$ be a near action. Then the lest $\lst_G(X)$ divides the index number $\iota_G(X)$.
\end{lem}
\begin{proof}
This follows from Proposition \ref{basiclest}: since $G$ is abelian, $\alpha$ maps $G$ into $\mkst_G(X)$, and hence the result follows from (\ref{bl2}) and (\ref{bl4}) of this proposition.
\end{proof}

%Choose $g\in G$ such that $g$ near acts as near permutation of index $n=\iota_G(X)$. Since $G$ is abelian, $g$ is central and hence belongs to the group of near automorphisms of the near $G$-set $X$. Choose a bijection from $X$ to a cofinite subset $X-n$ representing $g$. So $g$ yields a balanced near isomorphism from $X$ to $X-n$, and hence $\lst_G(X)$ divides $n$.

%\section{Specific results on near actions}\label{s_specific}

%We now provide results for some particular classes of groups, in this section, as well as \S\ref{s_ab}, which focuses on near actions of finitely generated abelian groups; \S\ref{s_ab} partly relies on \S\ref{s_equivariant} and \S\ref{difi}.
% of this section.

\subsection{Direct products with finite groups}\label{difi}

\begin{lem}\label{Cpfreecom}
Let $X$ be a set, let $F$ be finite group with a free action on $X$. Let $f$ be a permutation of $X$ that commutes with $s$ modulo finitely supported permutations. Then $f$ near equals a permutation that commutes with $s$.

In other words, every balanced near action of $F\times\Z$ that is realizable as a free action of $F$, is realizable. 
\end{lem}
\begin{proof}
Let $T$ be the set of $x\in X$ such that there exist $s\in S$ and $\eps\in\{\pm 1\}$ such that $f^\eps sx\neq sf^\eps x$. Then $T$ is finite.

Initial claim: that every component $Y$ of $X$ with respect to $\langle F,f\rangle$, with no infinite $f$-cycle, is finite. 

This is clear if $Y\cap T=\emptyset$, since then $f$ and $F$ commute on $Y$. Otherwise, there exists a sequence $(x_n)_{n\ge 0}$ in $Y$ with $x_0\in T$, and $d(x_n,T)=n$ for all $n$; necessarily this is a geodesic ray. Write $x_n=t_nx_{n-1}$, with $t_n\in F\cup\{f^{\pm 1}\}$. We claim that $t_n\in F$ for at most one $n\ge 2$. If not, otherwise, assume $t_k\in S$ and $t_\ell\in S$, with $2\le k<\ell$ and $\ell-k$ minimal. Since this is a geodesic ray, we have $\ell-k\ge 2$.
We have $x_{k+1}=t_{k+1}t_kx_{k-1}=t_kt_{k+1}x_{k-1}$, because $x_{k-1}\notin T$. Then we can replace $x_k$ with $t_{k+1}x_{k-1}$. This yields a new sequence where the minimal value of $\ell-k$ is smaller. We eventually get a contradiction. Hence, for some $\eps\in\{\pm 1\}$, we have $t_n=f^\eps$ for all $n$ large enough. This contradicts that there is no infinite $f$-cycle. So the ``initial claim" is proved.

Let $Y=(y_n)$ be an infinite $f$-cycle with empty intersection with $T$, with $f(y_n)=y_{n+1}$. Then $s(y_{n+1})=sf(y_n)=fs(y_n)$ for all $s\in F$ and $n\in\Z$, and we deduce that the component of $Y$ is $\bigcup_{s\in F}Y$, on which $F$ and $f$ commute.

Consider a component $Z$ of the near Schreier graph (i.e., a $(F\cup\langle f\rangle)$-orbit) with respect to the generating subset $F\cup\{\pm 1\}$, on which $F$ and $f$ do not commute. Then every infinite $f$-cycle in $Z$ has nonempty intersection with $T$; in particular, the set $\Xi$ of infinite $f$-cycles in $Z$ is finite.

There are two natural actions of $F$ on $\Xi$. If $C\in\Xi$ and $s\in F$, we define $s\cdot_+ C=C'$, where $C$ and $C'$ coincide ``at infinity", and $s\cdot_- C=C''$, where $C$ and $C''$ coincide ``at minus infinity".

Each of these actions of $F$ is free. Indeed, if $s\in F$ and $s\cdot_+C=C$, then, writing $C=(x_n)$, it means there exists $m\in\Z$ such that $sx_n=x_{n+m}$ for all large $n$. If $s^k=1$, then we deduce $sx_n=x_{n+km}$ for all large $n$. If $m\neq 0$, we deduce that $(x_n)$ is eventually periodic, a contradiction. So $sx_n=x_n$ for all large $n$, and this contradicts freeness of the $F$-action on $X$.

Choose representatives $C^+_1,\dots,C^+_q$, respectively $C^-_1,\dots,C^-_q$ for each of these $F$-actions on $\Xi$ (the number of orbits $q=|\Xi|/|F|$ is the same for both actions).

Write $C^\pm_i=\{x^\pm_{i,n}:n\in\Z\}$, with $f(x^\pm_{i,n})=x^\pm_{i,n+1}$ for all $i$. Reindexing if necessary, we can suppose that $x^+_{i,n}\notin T$ for all $i$ and all $n\ge 1$, and $x^-_{i,n}\notin T$ for all $i$ and all $n\ge 0$, and that $x^-_{i,n}\neq sx^+_{j,m}$ for all $i,j$, $n\le 0$ and $m\ge 1$.

Define $g$ as the following perturbation of $f$ on $Z$, by setting
\begin{itemize}
\item $g(x)=f(x)$ for each $x$ of the form $sx^-_{i,n}$ for $n<0$ or $sx^+_{i,n}$ for $n>0$ and $s\in F$, or any $x$ not in any infinite $f$-cycle;
\item $g(sx^-_{i,0})=sx^+_{i,1}$ for all $i$ and $s\in F$;
\item $g(x)=x$ for all remaining elements of infinite $f$-cycles.
\end{itemize}

The last requirement involves only finitely many $x$, since there are finitely many infinite $f$-cycles, and for every infinite $f$-cycle, only finitely many elements are not taken in charge by the first requirement.

Let $Z'$ be the union for all $i$ and all $s\in F$ of $x^-_{i,n}$ for $n\le 0$ and $x^+_{i,n}$. Then $Z'$ is both $f$-invariant and $F$-invariant, and $g$ and $F$ commute on $Z'$.

On $Z\smallsetminus Z'$, $g$ acts with no infinite $g$-cycle, since germs at infinity of $g$-cycles are germs at infinity of $f$-cycles, and are included in $Z'$. Since $Z'$ is a connected in the original near Schreier graph, it has finitely components for $\langle F,g\rangle$, so to prove that $Z\smallsetminus Z'$ is finite, it is enough to prove that each component of 
$Z\smallsetminus Z'$ with respect to $\langle F,g\rangle$ is finite. This follows from the initial claim.
\end{proof}

If $F$ is a finite abelian group and $X$ is an $F$-set and $H$ a subgroup of $F$, write $X^H$ for the set of fixed points of $H$, and $X^{[H]}=X^H\smallsetminus\bigcup_{L<H}X^L$; thus $X^{[H]}$ is the union of all orbits on which the action factors through a free action of $F/H$.

\begin{prop}\label{finitimesfree}
Let $L$ be a free group (for instance, $\Z$) and $F$ a finite abelian group. For $G=F\times L$, consider a near $G$-set $X$ and fix an $F$-realization. The following are equivalent
\begin{enumerate}
\item\label{cpl1} $X$ is realizable;
\item\label{cpl2} $X$ is stably realizable;
\item\label{cpl3} for every subgroup $H$ of $F$, 
the induced near $G$-action on $X^H$ has zero index (i.e., is balanced).
\item\label{cpl4} for every subgroup $H$ of $F$, the induced near $G$-action on $X^{[H]}$ has zero index. 
\end{enumerate}
Moreover, any near $G$-set is completable, and if $L=\Z$, every near $G$-set of finite type is finitely-ended.
\end{prop}

\begin{lem}\label{boolunion}
Let $f$ be a fixed near permutation of a set $X$. The index character, viewed as a map $Y\mapsto\phi_Y(f)$ from the set $\mathcal{P}^\star_{\langle f\rangle}(X)$ of $\langle f\rangle$-commensurated subsets of $X$ to $\Z$, extends to an additive map from the set of finitely supported functions from $\mathcal{P}^\star_{\langle f\rangle}(X)$ to $\Z$.

In particular, if $X_1,\dots,X_k$ are $\langle f\rangle$-commensurated subsets of $X$ and $\phi_{X_J}(f)=0$ for all subsets $J$ of $\{1,\dots,k\}$, where $X_J=\bigcap_{j\in J}X_j$, then $\phi_{\bigcup_{j=1}^k X_j}(f)=0$.
\end{lem}
\begin{proof}
The first fact is immediate from the definition. The second fact follows because
\[1_{\bigcup_{j=1}^k X_j}=\sum_{J\subseteq \{1,\dots,k\}}(-1)^{|J|-1}1_{X_J}.\qedhere\]
\end{proof}

\begin{proof}[Proof of Proposition \ref{finitimesfree}]
First observe that the set of points fixed by $H$, for some choice of $C_p$-realization, is unique up to near equality, and in particular is a well-defined near action. In particular, whether it has zero index does not depend on the choice of $F$-realization.

Clearly (\ref{cpl1}) implies (\ref{cpl2}). Suppose (\ref{cpl2}). 
Fix a realization (on $Z=X\sqcup W$, with trivial near action on $W$).
Then the set $Z^H$ of fixed points is $G$-invariant; in particular its index is zero. Now $Z^{H}\cap X$ has finite symmetric difference with the set of points fixed by $H$ for the initial $H$-realization on $X$, so $X^{H}$ (which is defined up to near equality) has zero index. This proves (\ref{cpl3}).

Suppose (\ref{cpl3}). Observe that $X^L\cap X^{L'}=X^{L+L'}$ for all subgroups $L,L'$ of $F$, and hence the set of subsets of the form $X^H$ is closed under finite intersections. Since $X^L$ has zero index for all $L$, by Lemma \ref{boolunion}, we deduce that $\bigcup_{H'>H}X^{H'}$ has zero index; since $X^H$ also has zero index, we deduce that $X^{[H]}$ has zero index, so (\ref{cpl4}) holds.

Suppose (\ref{cpl4}). By Lemma \ref{Cpfreecom}, $X^{[H]}$ is realizable as near $((F/H)\times L)$-set, and hence as near $G$-set. Since the $X^{[H]}$ form a finite partition of $X$, we deduce that $X$ is realizable as near $G$-set, that is, (\ref{cpl1}) holds.

Now let $X$ be an arbitrary near $G$-set, defined by a homomorphism $\alpha:G\to\mkst(X)$. Choose a free generating subset of $L$. Define a new near action on $X$ by a homomorphism $\beta:G\to\mkst(X)$ defined by: $\beta(s)=\alpha(s)$ for all $s\in C_p$, and $\beta(t)=\alpha(t)^{-1}$ for all $t$ in the free generating subset of $L$. Then $X^{H}$ is the same for both near actions (since they have the same images); the index character on $X^{H}$ is, by construction, opposite for $\alpha$ and $\beta$. Hence the near action $\alpha\sqcup\beta$ on $X\sqcup X$ has zero index on both the $X^H$-fixed points for every $H\le F$. Hence this near $G$-action on $X\sqcup X$ is realizable by the given criterion.

If $F=\Z$, then every transitive $G$-set is finitely-ended, and hence the same conclusion passes to completable near actions of finite type, and hence to all near actions of finite type, by the completability result.
\end{proof}

\subsection{Classification of completable near actions}\label{ccomple}
Let $G$ be a finitely generated abelian group. 
We start with the classification of completable near actions. It is much easier than the general case, since it boils down to a large extent in understanding genuine actions; nevertheless the formulation requires some discussion.

\index{corank}The $\Q$-corank of a subgroup $H$ of $G$ is by definition the $\Q$-rank of $G/H$, that is, the dimension of $(G/H)\otimes_\Z\Q$ over $\Q$.

Let $G$ be a finitely generated abelian group. If $H$ is a subgroup of $\Q$-corank 1, then $G/H$ is 2-ended; it admits exactly 2 (opposite) homomorphisms onto $\Z$. If $\phi$ is such a homomorphism, then $\phi^{-1}(\N)$ is a 1-ended near $G$-set with index character $\phi\circ\pi$, where $\pi$ is the projection $G\to G/H$. 

\begin{thm}\label{classnear1}
The 1-ended completable near $G$-sets are, up to near $G$-isomorphism:
\begin{itemize}
\item the $G$-set $G/H$, where $H$ is ranges over subgroups of $\Q$-corank $\ge 2$ of $G$. Its lest is zero.
\item the near $G$-set $\phi^{-1}(\N)(\subseteq G/H)$, where $H$ ranges over subgroups of $\Q$-corank 1 of $G$ and $\phi$ ranges over the two surjective  homomorphisms of $G/H$ onto $\Z$. The lest of such a near $G$-set is equal to $|F|$, where $F$ is the torsion subgroup of $G/H$.
\end{itemize}

The completable near $G$-sets of finite type are, up to near $G$-isomorphism, the finite disjoint unions of such 1-ended near $G$-sets.
\end{thm}
\begin{proof}
Let $X$ be completable near $G$-set of finite type. By Proposition \ref{streafg}(\ref{streafg2}), there exists another near $G$-set of finite type $Y$ such that $Z=X\sqcup Y$ is realizable.

First suppose that $X$ is 1-ended. We can decompose $Y$ into orbits; then some cofinite subset of $X$ is included in a single orbit. Hence we can suppose that $Z$ is a transitive $G$-set, namely isomorphic to $G/H$ as $G$-set, where $H$ is a subgroup of infinite index. 

First suppose that $G/H$ has $\Q$-rank $\ge 2$. So $Z$ is 1-ended, and $X$ being commensurated and 1-ended, it follows that $X$ is cofinite in $Z$. Hence $X$ is isomorphic to $G/H$ as near $G$-set. Its lest is zero, since $\mkst_G(G/H)$ is reduced to $G/H$ by Corollary \ref{ghx}.

Second, suppose that $G/H$ has $\Q$-rank 1. So $Z$ is 2-ended. Fix a homomorphism $\phi$ of $G/H$ onto $\Z$. Since $X$ is 1-ended, it is has finite symmetric difference with either $\phi^{-1}(\N)$ or $\phi^{-1}(-\N)$, and hence is near isomorphic with it. 

Finally, we have to check that these are pairwise non-isomorphic as near $G$-sets. First, the kernel of the near action on a near $G$-set as above is $H$. So it remains to distinguish, for $G/H$ of $\Q$-rank 1, between $\phi^{-1}(\N)$ and $\phi^{-1}(-\N)$: indeed, they have nonzero opposite indices. In restriction to $F$, this is near isomorphic to an infinite free action, so for such $X$ we have $\lst_F(X)=|F|$. So $|F|$ divides $\lst_G(X)$. Since the index character has image $|F|\Z$, we also see, using Lemma \ref{indexlest}, that $\lst_G(X)$ divides $|F|$. Hence the lest is equal to $|F|$.

Now let more generally suppose that $X$ is of finite type. Since any quotient of $G$ has finitely many ends, $Z$ has finitely many ends. We conclude using Proposition \ref{fgfended}.
\end{proof}

Let $X$ be a near $G$-set. By Theorem \ref{nearactionfp}, we can write it as $Y\sqcup Z$, where $Y$ has finite type and $Z$ is realizable. Let $H$ be a subgroup of $G$ of $\Q$-corank 1, and $\phi$ a homomorphism of $G/H$ onto $\Z$. Let $X^H$ be the subset (defined up to near equality) of points fixed by $H$; this is valid since $H$ is finitely generated. As a near $(G/H)$-set, $X^H$ is completable by Proposition \ref{finitimesfree}. One can write the index character of the near $G/H$-set $X^{[H]}=X^H\smallsetminus\bigcup_{H'>H}X^{H'}$ as $n\phi$ for some $n\in\Z$. We write $n=n_{H,\phi}=n_{X,H,\phi}$.

It is additive on $X$, that is, $n_{X\sqcup Y,H,\phi}=n_{X,H,\phi}+n_{Y,H,\phi}$. Since $n_{H,-\phi}=-n_{H,\phi}$ and the pair $\{\phi,-\phi\}$ is determined by $H$, the number $|n_H|=|n_{H,\phi}|$ is only determined by $H$. By convention, let us write $|n_H|=0$ if the $\Q$-rank of $G/H$ is not equal to 1.

\begin{thm}\label{complab}
Let $G$ be a finitely generated abelian group. Let $X$ be a completable near $G$-set. Then $|n_H|\neq 0$ for only finitely many subgroups $H$. The following are equivalent:
\begin{enumerate}
\item\label{zsr1} $X$ is stably realizable;
\item\label{zsr2} $|n_H|=0$ for all subgroups $H$;
\item\label{zsr3} the index of $X^H$ is zero for all subgroups $H$.
\item\label{zsr4} the index of $X^H$ is zero for all subgroups $H$ of $\Q$-corank 1.
\end{enumerate}
Moreover, in case $X$ is stably realizable and not realizable, there exists a nonempty finite subset $F$ such that the near action on $X\sqcup F$ is realizable as a free action such that each stabilizer $H$ has $\Q$-corank $\ge 2$.
\end{thm}
\begin{proof}
For a realizable near action, the union of free $G/H$-orbits is invariant and hence has zero index, so $|n_H|=0$. By additivity on $X$, we deduce the vanishing for all stably realizable near actions. This proves (\ref{zsr1})$\Rightarrow$(\ref{zsr2}) (although we will not need it as we will prove another chain of implication, we need this easy fact to prove the first assertion.)

The finiteness result is immediate for completable near $G$-sets of finite type, as a consequence of Theorem \ref{classnear1}. Given an arbitrary completable near $G$-set $X$, by Proposition \ref{streafg}(\ref{streafg2}) there exists a near $G$-set $Y$ of finite type such that $X\sqcup Y$ is realizable. Hence $n_{X,H,\phi}=-n_{Y,H,\phi}$ for all $H$, whence the finiteness result in general.

(\ref{zsr2})$\Rightarrow$(\ref{zsr1}) Assume that $X$ is completable and $|n_H|=0$ for all subgroups $H$ of $\Q$-corank 1 and let show that $X$ is stably realizable. By Theorem \ref{nearactionfp}, we can reduce to the case when $X$ has finite type. By Proposition \ref{streafg}(\ref{streafg2}), write $X\sqcup Y=Z$ with $Z$ realizable and $Y$ of finite type; fix a realization on $Z$. If $H$ is a subgroup of $G$, we write $Z^{[H]}=X^{[H]}\sqcup Y^{[H]}$. Since $Z^{[H]}$ is empty for all but finitely many subgroups $H$, it is enough to prove that $X^{[H]}$ is stably realizable. By assumption, it has zero index. If $G/H$ has $\Q$-rank $\ge 2$, then it follows that some subset near equal to $X$ is invariant ($X$ is ``transfixed" for the action), and it follows that $X^{[H]}$ is stably realizable.

If $G/H$ has $\Q$-rank 1, then it follows from Theorem \ref{classnear1} that $X^{[H]}$ is isomorphic, as near $G$-set, to a union of $k$ copies of $\phi^{-1}(\N)$ and $\ell$ copies of $\phi^{-1}(-\N)$, and the vanishing of $|n_H|$ implies that $k=\ell$. Hence $X^{[H]}$ has zero index. By Proposition \ref{finitimesfree} (applied to $X^{[H]}$), we deduce that $X^{[H]}$ is realizable.

(\ref{zsr3})$\Rightarrow$(\ref{zsr4}) is trivial.

(\ref{zsr4})$\Rightarrow$(\ref{zsr2}) Suppose $X^H$ has zero index for every $H$. We have to prove that $|n_H|=0$ for all $H$. We can suppose that $G/H$ has $\Q$-rank 1. By a noetherian induction, it is enough to prove $|n_H|=0$ assuming that $|n_L|=0$ for every subgroup $L$ properly including $H$. Let $k_L$ be the order of the maximal finite subgroup of $L/H$. Let $w_L$ be the index of $X_L$. Then the index of $X^H$ is the sum of $w_L$ where $L$ ranges over subgroups $L$ including $H$ as a finite index subgroup. We have $w_L=k_Ln_L$. Since $n_L=0$ for all $L\neq H$, we deduce that the index of $X^H$ is equal to $k_Hn_H$. Since the former is zero, we deduce that $n_H=0$.

(\ref{zsr1})$\Rightarrow$(\ref{zsr3}) is clear. This finishes the proof of equivalences.

Continue with this vanishing assumption, so that $X$ is stably realizable. Assume that $X$ is not purely of $\Q$-rank $\ge 2$: that is, for some realization of $X\sqcup F$, $F$ finite, there is an orbit $G/H$ with $H$ of $\Q$-corank 1, or that there are infinitely many finite orbits. In the first case, can use this orbit to perturb the action on this orbit to get more that $|F|$ fixed points. In the second case, we can perturb the action on finitely many finite orbits to make it the identity on at least a subset of cardinal $|F|$. Then, in both cases, we can conjugate by a finitely supported permutation and obtain a realization that is identity on $F$. This proves realizability.
\end{proof}

\begin{rem}
The non-completable near actions of \S\ref{s_noncomp} satisfy the vanishing conditions, so completability cannot be dropped for these equivalences. On the other hand, the vanishing of $|n_H|$ for all but finitely many $H$ holds in general, see Corollary \ref{nhfini}. 
\end{rem}

For an abelian group $G$ and a near action $\alpha:G\to\mkst(X)$, define $\check{\alpha}(g)=\alpha(g^{-1})$; this is a near action as well.

\begin{cor}
Let $G$ be a finitely generated abelian group and $\alpha$ a completable near action. The following are equivalent:
\begin{enumerate}
\item $\alpha$ is stably realizable;
\item $\alpha$ is near isomorphic to $\check{\alpha}$
\item $\alpha$ is balanceably near isomorphic to $\check{\alpha}$.
\end{enumerate}
In particular, $\alpha\sqcup\check{\alpha}$ is stably realizable.
\end{cor}
\begin{proof}
If $\beta:G\to\mks(X)$ is an action and $\check{\beta}(g)=\beta(g^{-1})$, then it is straightforward (by reduction to the transitive case) that $\beta$ and $\check{\beta}$ are isomorphic as $G$-actions. It follows (for an arbitrary abelian group $G$) that if $\alpha:G\to\mkst(X)$ is realizable, then $\alpha$ is balanceably near isomorphic to $\check{\alpha}$.

For an arbitrary countable abelian group, if $\alpha$ is stably realizable, then there exists a finite set $F$ such that $\alpha\sqcup F$ is realizable (Corollary \ref{abstre}), and hence balanceably near isomorphic to $\check{\alpha}\sqcup F$. Hence $\alpha$ is balanceably near isomorphic to $\check{\alpha}$.

For a finitely generated abelian group, it remains to prove that is $\alpha$ is not stably realizable, then it is not near isomorphic to $\check{\alpha}$. Indeed, by Theorem \ref{complab}, there exists a subgroup $H$ such that $X^H$ has index $f\neq 0$. Hence, for $\check{\alpha}$, the index of $X^H$ is $-f\neq f$, so $\check{\alpha}$ is not near isomorphic to $\alpha$.
\end{proof}

\subsection{Finiteness of the number of ends}\label{s_finends}

The classification of near actions of finitely generated abelian groups, beyond the completable case, requires considerably more work. We start with the following theorem.

\begin{thm}\label{abends}
Let $G$ be a finitely generated, virtually abelian group. Then every near $G$-set of finite type has finitely many ends, and grows asymptotically no faster than $G$.
\end{thm}

Using Proposition \ref{finitetypeind} reduces this fact to finitely generated abelian groups, so the proof will not make any incursion outside the abelian setting.

We do not know if this result about ends holds for virtually polycyclic groups, for which Houghton \cite{Ho} proved that transitive Schreier graphs are at most 2-ended (and thus completable near $G$-sets of finite type are finitely-ended).

As regards the result about growth, it partially answers Question \ref{growthque}; see \S\ref{growthnsch} for more context.

%it trivially holds for groups of exponential growth, but we do not know whether it holds for groups of subexponential growth; we expect a positive result in the case of polynomial growth, which reduces to finitely generated nilpotent groups.

To prove the theorem, we will use geometric arguments in the near Schreier graph, which will also be used to obtain the precise classification. The main interest of Theorem \ref{abends} is that is reduces the study to 1-ended near $G$-sets.

Let us start with a general graph-theoretic lemma. Let $X$ be a connected graph of finite valency; here $X$ is identified with the set of vertices; since we only consider the distance on $X$, self-loops or multiple edges do not matter. Let $\Sigma$ be a nonempty finite subset of $X$. Let $\Sigma_1$ be the set of points at distance $\le 1$ of $\Sigma$; it is also finite.

For $x\in X$ and $\sigma\in\Sigma_1$, define $\delta_x(\sigma)=d(x,\sigma)-d(x,\Sigma)$. Then $\delta_x\in\Sigma_1^{\{-1,0,\dots,m\}}$, where $m$ is the diameter of $\Sigma$.
Fix a fiber $Z_x=\{y\in X:\delta_x=\delta_y\}$; note that there are finitely many fibers.

If $x\in\Sigma$, then $\delta_x\ge 0$ and vanishes only at $x$, so $Z_x=\{x\}$. Otherwise, $\min(\delta_x)=-1$. Now fix $x\notin\Sigma$ and write $Z=Z_x$.

\begin{lem}\label{nomeet}
No geodesic segment between two points of the fiber $Z$ meets $\Sigma$.
\end{lem}
\begin{proof}
Indeed, let $x_0,\dots,x_n$ be such a geodesic segment; by assumption we have $\delta_{x_0}=\delta_{x_n}$. Suppose by contradiction that it meets $\Sigma$. Since $x_0,x_n\notin\Sigma$, we have $n\ge 2$ and $x_i\in\Sigma$ for some $i$ with $0<i<n$.
For all $\sigma,\sigma'\in\Sigma_1$, and $k\in\{0,n\}$, we have  
\[\delta_{x_k}(\sigma')-\delta_{x_0}(\sigma)=d(x_k,\sigma')-d(x_k,\sigma),\]
Then we have
\[\delta_{x_0}(x_{i+1})-\delta_{x_0}(x_i)=\delta_{x_n}(x_{i+1})-\delta_{x_n}(x_i),\]
which yields
\[1=d(x_0,x_{i+1})-d(x_0,x_i)=d(x_n,x_{i+1})-d(x_n,x_i)=-1,\]
a contradiction.
\end{proof}

\begin{sett}\label{setting}
Consider a finitely generated free abelian group $G$, with basis $(s_1,\dots,s_d)$, written multiplicatively. Write $s'_i=s_i^{-1}$.

Let $X$ be a near $G$-set. Lift each $s_i,s'_i$ to a self-map $S_i,S'_i$ of $X$. This defines a near Schreier graph structure on $X$.
\end{sett}

\begin{defn}\label{d_regular}
Consider the set of regular points, namely
the cofinite subset $R=R(S_1,\dots,S'_d)$ of $X$ consisting of those $x\in X$ such that one of the following holds
\begin{itemize}
\item for every $i$, $S_iS'_ix=S'_iS_ix=x$
\item for all $i,j$, $S_iS_jx=S_jS_ix$, $S_iS'_jx=S'_jS_ix$, $S'_iS_jx=S_jS'_ix$, $S'_iS'_jx=S'_jS'_ix$.
\end{itemize}
\end{defn}

Let $\Sigma$ be the finite complement of $R$. If $\Sigma$ is empty, we have a genuine action and hence let us assume that $\Sigma$ is non-empty (or declare $\Sigma$ to be a singleton in this case). Define $\Sigma_1$ as the 1-neighborhood of $\Sigma$ as above, and define $\delta_x$ as above. Then the function $x\mapsto\delta_x$, for $x\in X$, takes values in a finite set. For $x\in X\smallsetminus\Sigma$, the fiber $Z=Z_x=\{y\in X:\delta_x=\delta_y\}$ satisfies the conclusion of Lemma \ref{nomeet}.

\begin{lem}\label{simpconex}
Let $t_1,\dots,t_n$ belong to $\{S_1,S'_1,\dots,S'_d\}$. For $x\in R$, suppose that $(x,t_1x,t_2t_1x,\dots,t_n\dots t_1x)$ is a geodesic segment between two points of $Z$. Then for any permutation $(t'_1,\dots,t'_n)$ of $(t_1,\dots,t_n)$, we have $t_n\dots t_1x=t'_n\dots t'_1x$. In particular, $(x,t'_1x,t'_2t'_1x,\dots,$ $t'_n\dots t'_1x)$ is a geodesic segment and for each $i\in\{1,\dots,d\}$, $S_i$ and $S'_i$ cannot simultaneously appear among the $t_i$.
\end{lem}
\begin{proof}
It is enough to consider a permutation of two consecutive terms, namely, for some $j$, $t'_{j+1}=t_j$, $t'_j=t_{j+1}$, and $t'_k=t_k$ for other values of $k$. By Lemma \ref{nomeet}, $y=t_{j-1}\dots t_1x\in R$. So $t_{j+1}t_jy=t_jt_{j+1}y$. This means that $t'_{j+1}\dots t'_1x=t_{j+1}\dots t_1x$, whence $t'_n\dots t'_1x=t_n\dots t_1x$. 

The first consequence immediately follows (since it is a combinatorial path and achieves the distance for endpoints). If $S_i$ and $S'_i$ both appear, then after permutation they appear consecutively, and this is a contradiction with being a geodesic segment. 
\end{proof}

For $x\in F$ and $t\in\{S_1,\dots,S'_d\}$, let $Z_x(t)$ be the set of elements of $Z_x$ that can be written as $y=t_n\dots t_1x$ with $d(x,y)=n$, such that $t$ appears at least $n/d$ times among the $t_i$. In view of Lemma \ref{simpconex}, we have $Z_x=\bigcup_tZ_x(t)$.

\begin{lem}
For each $x,t$ as above, $Z_x(t)$ accumulates at most one end of $X$. In other words, every $G$-commensurated subset of $X$ with infinite intersection with $Z_x(t)$ includes a cofinite subset of $Z_x(t)$.
\end{lem}
\begin{proof}
Given $k$, consider $y,z\in Z_x(t)$ with $d(z,x),d(y,x)\ge kd$. Then, by Lemma \ref{simpconex} there are geodesic paths $(x,t_1x,\dots,t_n\dots t_1x=y)$, $(x,t'_1x,\dots,t'_m\dots t'_1x=z)$ with $n,m\ge k$, with at least $k$ occurrences of $t$, which we choose to appear first, so that $t_j=t'_j=t$ for all $i\le k$. Hence the concatenation of the paths $(t_n\dots t_{k+1}t^kx,\dots, t^kx)$ and $(t^kx,\dots t'_n\dots t'_{k+1}t^kx)$ joins $y$ to $z$ and stays at distance $\ge k$ to $x$.
\end{proof}

\begin{proof}[Proof of Theorem \ref{abends}]
By the above, the number of ends is no more that $2d$ times the number of infinite fibers, and each fiber grows at most as $G$.
\end{proof}

\begin{cor}\label{nhfini}
Let $X$ be a near $G$-set. Then $X^{[H]}$ has nonzero index (in other words, $|n_H|\neq 0$, in the language used in \S\ref{ccomple}) for only finitely many subgroups $H$.
\end{cor}
\begin{proof}
By Theorem \ref{nearactionfp} and the vanishing result in realizable case (Theorem \ref{complab}), we can suppose that $X$ is of finite type. Then by Theorem \ref{abends}, we can suppose that $X$ is 1-ended. So $X^{[H]}$, for each subgroup $H$, is either finite or cofinite, and is cofinite for at most one $H$. Hence in the 1-ended case, $|n_H|\neq 0$ for at most one subgroup~$H$.
\end{proof}

\subsection{Rectangles and corners in $\Z^d$}

In $\Z^d$, call rectangle any subset $P$ that is $\ell^1$-convex.
Thus, a rectangle is a product $\prod_{i=1}^dJ_i$ where each $J_i$ is an ``interval" in $\Z$. If no $J_i$ equals $\Z$, we say that the rectangle is strict. Clearly, any rectangle is a disjoint union of finitely (at most $2^d$) many strict rectangles. \index{rectangle}\index{corner}

\begin{lem}\label{rectcomp}
In $\Z^d$, any union of rectangles is finite union of rectangles.
\end{lem}
\begin{proof}
It is enough to prove that in $\N^d$, any union of rectangles containing 0 is a finite sub-union. Indeed, we can translate to assume that the union contains 0 and then work in each of the $2^d$ corners.

To prove the result in $\N^d$, write $\overline{\N}=\N\cup\{+\infty\}$. For $n\in\overline{\N}^d$, define $P(n)=\{m\in\N^d:0\le m\le n\}$ where $\le$ is the product (coordinate-wise) ordering. Then $P(n)$ is a rectangle; $n$ is determined by $P(n)$, and every rectangle in $\N^d$ containing 0 arises this way. 

Now consider an union $M$ of rectangles in $\N^d$, all containing 0; then it is the union of the maximal rectangles included in $M$. So we have to show that there are finitely many maximal rectangles in $M$. Otherwise, by contradiction, there is an injective sequence $(P(n(j)))$ of maximal rectangles, with $n(j)\in\overline{\N}^d$ for each $j$. After extraction, we can suppose that for each $i$, the sequence $(n(j)_i)_{j\ge 1}$ is monotonous. Since it is valued in $\overline{\N}$, this means that for $j$ large enough, we have $n(j)_i\le n(j+1)_i$ for all $i$. Hence $P_{n(j)}\subseteq P_{n(j+1)}$, contradicting the maximality.
\end{proof}

\begin{lem}\label{partrect}
For any rectangles $R_1,\dots,R_k$ in $\Z^d$, $R_k\smallsetminus\bigcup_{i<k}R_i$ is a finite disjoint union of rectangles.
\end{lem}
\begin{proof}
We argue by induction on $d$. We can suppose that $R_i\subseteq \N^d$ for all $i$. Write $Q_j(n)=\{x\in\N^d:x_j\le n\}$. Choose $n$ strictly larger than any finite coordinate occurring in the $R_i's$. Write $Q'_j(n)=\{x\in\N^d:x_j\le n,\forall k<j,x_k>n$, and $Q=\{x\in\N^d:\forall j:x_j>n\}$. Then the $Q'_j(n),1\le j\le n$ and $Q$ form a partition of $\N^d$ by rectangles, and it is enough to prove the result for the intersection with each component of the partition. For $Q'_j(n)$, we argue by induction (partitioning it again into slices according to the value of $x_j$). For $Q$, it follows from the definition of $n$ that each $R_i$ either is disjoint with $Q$, or includes $Q$, and the result is clear too.
\end{proof}

In $\Z^2$, we have essentially 3 types of strict rectangles: corners (unbounded in both directions), strips (bounded in a single direction), and bounded rectangles. There are 4 types of corners: upper-right, upper-left, lower-left, lower-right, and 4 types of strips: up, left, down, right; each strip has a given width.

The boundary of a subset $E$ of $\Z^d$ is the set of elements of $E$ that have at least one neighbor (for the standard Cayley graph structure of $\Z^d$) that does not belong to $E$.

In $\Z^2$, the boundary of an infinite near strict rectangle consists of two rays (equal in the case of a strip of width one). We refer to them as left boundary and right boundary according to the following convention, which amounts to ``looking from the center": 
\begin{itemize}
\item for the corner $\N^2$, the left boundary is the ray $\{0\}\times\N$ and the right boundary is the ray $\N\times\{0\}$. For an arbitrary corner, we choose a direct $\ell^1$-isometry and define accordingly. For instance, for the corner $(-\N+a)\times (-\N+b)$, the left boundary is $\{a\}\times(-\N+b)$; for the corner $(\N+a)\times (-\N+b)$, the left boundary is $(\N+a)\times\{b\}$.
\item for a strip $\{a,\dots,b\}\times\N$, the left boundary is the ray $\{a\}\times\N$ and the right boundary is the ray $\{b\}\times\N$, etc.
\end{itemize}

We call a near strict rectangle in $\Z^d$ a subset with finite symmetric difference with a rectangle, modulo finite symmetric difference. In $\Z^2$, there are again 3 types of near strict rectangles: near corners, near strips, and the bounded ones, which are all identified to the empty set.

In $\Z^2$, the boundary of an infinite near strict rectangle $P$ is well-defined up to near equality, and consists of two near rays (equal in the case of a strip of width one). We refer to them as left boundary and right boundary according to the same convention; we denote them by $\pl P$ and $\pr P$.

\subsection{Rectangles in near $\Z^d$-sets}

Let $X$ be a near $\Z^d$-set. Given $S_1,\dots,S'_d$ as in Setting \ref{setting}, define a graded subset in $X$ as a pair $(P,f)$ where $P$ is a subset of $X$ and $f:P\to\Z^2$ is a map such that
\begin{itemize}
\item $f$ is injective
\item for all $x\in X$ and all $i$ such that $f(x)+s_i\in f(P)$, we have $f(S_ix)=f(x)+s_i$;
for all $x\in X$ and all $i$ such that $f(x)-s_i\in f(P)$, we have $f(S'_ix)=f(x)-s_i$;
\end{itemize}

If $f(P)$ is a rectangle, call it a graded rectangle. If we forget $f$, we call it an embedded rectangle. We have an easy compactness lemma:

\begin{lem}\label{erectcomp2}
In $X$, any union of embedded rectangles containing a given $x\in X$ is a finite union of embedded rectangles containing $x$.
\end{lem}
\begin{proof}
It is enough to show that there are finitely many maximal embedded rectangles containing $x$. Each embedded rectangle $L$ in $X$ containing $x$ corresponds to a rectangle $P_L$ in $\Z^d$ containing 0. We have $L\subseteq L'$ if and only if $P_L\subseteq P_{L'}$. Hence the result follows from Lemma \ref{rectcomp}.
\end{proof}

\begin{prop}
The set $X$ (of finite type) has a disjoint covering by finitely many embedded rectangles. 
\end{prop}
\begin{proof}
Let us first show that there is a finite covering. We can work in a single component of the given near Schreier graph. There are finitely many fibers. For each fiber $Z$, fix a point $x$; then for every $y\in Z$, there is an embedded rectangle containing both $x$ and $y$, by Lemma \ref{simpconex}. By Lemma \ref{erectcomp2}, $Z$ is included in a finite union of rectangles. Since there are finitely many fibers, this proves the result.

Finally, to obtain a disjoint covering, we first write a disjoint covering $X=R_1\cup\dots \cup R_k$. Then by Lemma \ref{partrect}, we can write $R_j\smallsetminus\bigcup_{i<j}R_i$ as a finite disjoint of rectangles, and thus we obtain a disjoint covering. 
\end{proof}

\subsection{Corners in near $\Z^2$-sets}

We write ``near included", ``near disjoint", to mean the inclusion or disjointness up near to near equality, etc.

Let $X$ be a near $\Z^2$-set. In $X$, we call graded near (strict) rectangle a graded subset $(P,f)$ such that $f(P)$ is a near (strict) rectangle, modulo finite symmetric difference. If we forget $f$ and only retain $P$, we call it a near (strict) embedded rectangle. 

We say that an ordered pair $(P,Q)$ of near disjoint near strict rectangles in $X$ is adjacent if $\pl P\cup \pr Q$ is a near strip.

For instance, if $P$ is an upper right corner, this means that $Q$ is either an upper strip or an upper left corner, and that, for some near indexations $(v_n)$ $(w_n)$ of $\pl P$ and $\pr Q$, we have $S_1w_n=v_n$ and $S'_1v_n=w_n$ for all $n$.

Consider a finite near covering of $X$ by near disjoint near strict rectangles. Consider the oriented graph $\mathcal{G}_X$ whose vertices are infinite near triangles in the covering, with an oriented edge from $P$ to $Q$ whenever $P,Q$ are adjacent. 

For instance in case we have $\Z^2$ cut into 4 rectangles in an obvious way, this yield a 4-cycle.

\begin{lem}The graph $\mathcal{G}_X$ is connected if and only if $X$ is a 1-ended near $G$-set.
\end{lem}
\begin{proof}
Let $Y$ be a commensurated subset. Since every infinite near strict rectangle is 1-ended, $Y$ has finite symmetric difference with a union of some subset of near rectangles of the covering. Moreover, if $(P,Q)$ is adjacent, there are infinitely many edges between $P$ and $Q$ and hence $Y$ near includes $P$ if and only $Y$ includes $Q$. Hence if $\mathcal{G}_X$ is connected, then $Y$ should be finite or cofinite. Conversely, if it is not connected, then the union of near rectangles in one connected components of $\mathcal{G}_X$ is commensurated, infinite and not cofinite.
\end{proof}

\begin{center}
{\includegraphics[keepaspectratio=true,
width=9cm,clip=true,trim= 0cm 0.7cm 0cm 0.8cm]{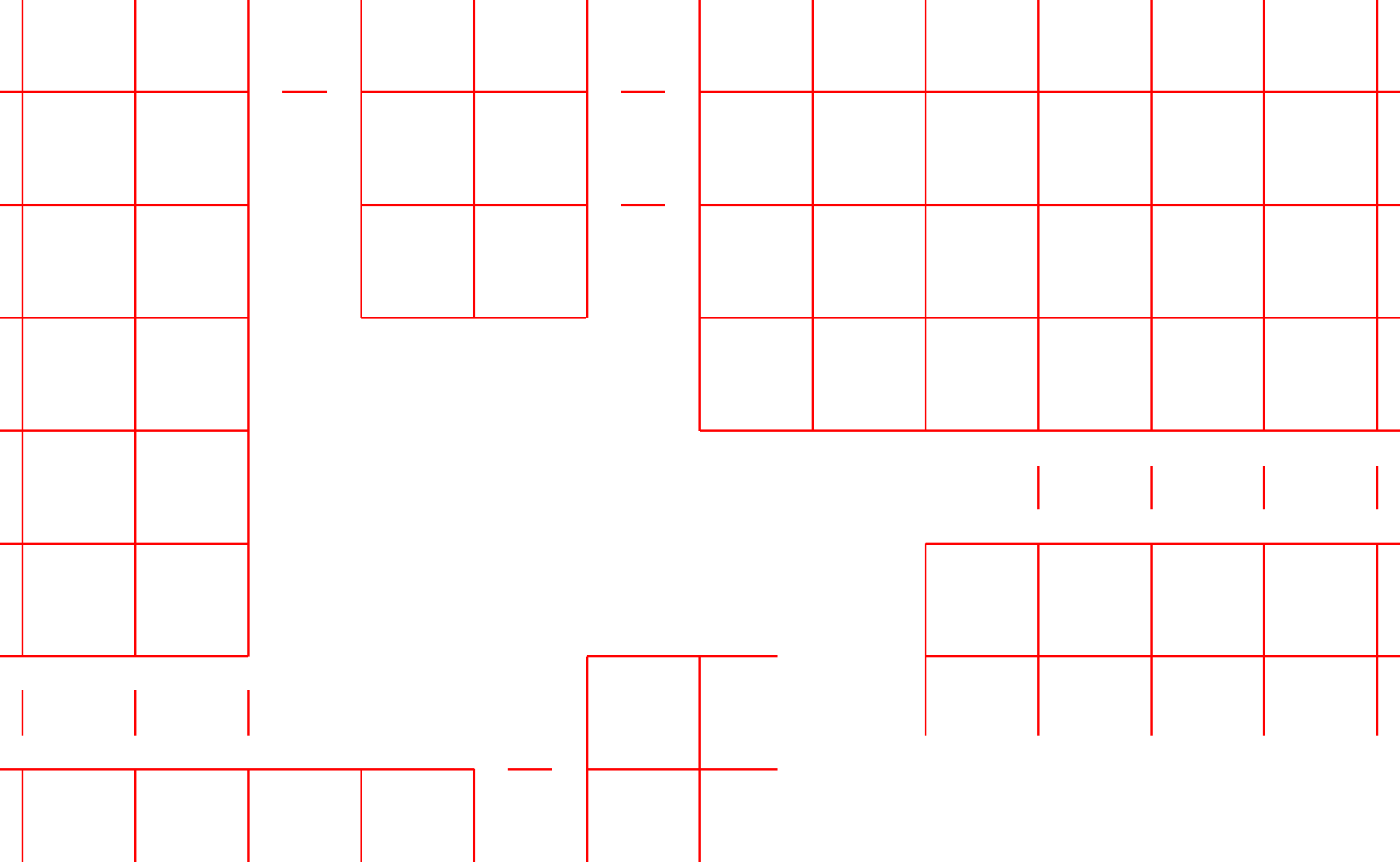}}
\end{center}

\begin{lem}\label{graphgx}
Assume that $X$ is a near free near $\Z^2$-set. The finite graph $\mathcal{G}_X$ is of valency 2. More precisely, for every $P\in\mathcal{G}_X$ there exists a unique $Q$ such that $(P,Q)$ is adjacent and a unique $Q'$ such that $(Q',P)$ is adjacent; we have $Q'\neq P\neq Q$ (that is, there are no self-loops).
\end{lem}
\begin{proof}
Let $P$ be either an upper-right corner or an upper strip (the other cases are similar). Let $(v_n)$ be a near indexation of the left boundary of $P$. Define $w_n=S'_1v_n$. Then for large $n$, we have $S_2w_n=w_{n+1}$, $S'_2w_n=w_{n-1}$, and $S_1w_n=v_n$. 

We claim that there is some infinite rectangle $Q$ in the given covering such that $w_n\in R$ for all large $n$.
Otherwise, there is some infinite rectangle $R$ visited and quitted infinitely many times by $(w_n)$. Since $(w_n)$ enters into $R$ infinitely many times, the second coordinate of $R$ has a lower bound. Since $(w_n)$ leaves $R$ infinitely many times, the first coordinate of $R$ has an upper bound. So $R$ is a horizontal strip. Hence, for $n$ large enough, $v_n=S_1w_n$ belongs to $R$ for infinitely many $n$. This is a contradiction.

We have $Q\neq P$: indeed, otherwise there would exist $k,\ell$ such that $w_n=S_1^kS_2^\ell w_n$ for all large $n$, contradicting that the action is near free.

It is immediate that the union of the two rays form a near strip.
\end{proof}

\begin{lem}\label{gluestrip}
Assume that $X$ is a near free near $\Z^2$-set.
Let $(P,Q)$ be adjacent. Suppose that at least one of $P,Q$ is a near strip (i.e., no both are corners). Then $P\cup Q$ is a strict near rectangle.\qed
\end{lem}
\begin{proof}Up to rotate, we can suppose $\pl P\cup \pr Q$ is an upper near strip, and up to a reflection, we can also suppose that $Q$ is a near strip. We can choose the grading on $P$ so that it maps onto a subset near $\N_{<k+1}\times\N$, for some $k\in\N\cup\{\infty\}$. We can choose the grading on $\pl P\cup \pr Q$ to map onto a subset near $\{-1,0\}\times\N$, and in turn, using that $P\neq Q$ (by Lemma \ref{graphgx}), we can choose the grading on $Q$ to map onto a subset near $\{-k',\dots,-1\}\times\N$ (for some $k'\le -1$). Up to shift the vertical coordinate of the grading of $\pl P\cup \pr Q$, we can suppose that the gradings of $P$ and $\pl P\cup \pr Q$ near coincide on the intersection $\pl P$, and similarly, up to shift the vertical coordinate of the grading of $Q$, we can suppose that the gradings of $Q$ and $\pl P\cup \pr Q$ near coincide on the intersection $\pl Q$. Hence, removing finitely many points if necessary, we have a well-defined grading map, mapping onto a subset near $\Z_{-k'\le\cdot<k+1}\times\N$. The axioms of grading subsets are immediately satisfied, and hence $P\cup Q$ is near rectangle.
\end{proof}

(When $P,Q$ are both corners, the union can fail to be a near rectangle, and is never a strict near rectangle).

\subsection{Winding number and additive holonomy}\label{s_windholo}

Continue with a near free near $\Z^2$-set.
Now consider the graph $\mathcal{G}_X$, which we can suppose to be connected. It is thus a loop. We can suppose, gluing strips using Lemma \ref{gluestrip} as many times as possible, that $\mathcal{G}_X$ is only made of near corners. It is of ``period 4" since it follows the pattern (upper right) $\to$ (upper left) $\to$ (lower left) $\to$ (lower right) $\to$ (upper right). 

So its length is $4m$ for some $m$. Actually, $m$ is uniquely defined. For instance, $m$ is the number of upper right corners in any finite disjoint covering by near strict rectangles. It is the number of upper right corners up to the equivalence relation identifying two upper right corners whenever their intersection includes an upper right corner. 

\begin{defn}
For $X$ a near free 1-ended near $\Z^2$-set, we call $m=m_X\in\N_{\ge 1}$ the winding number of $X$.\index{winding number}
\end{defn}

It depends on the choice of basis a priori, actually not. Indeed, if one passes from the basis $(u,v)$ to $(v,u)$, it clearly does not change. Also, if we pass from the basis $(u,v)$ to $(uv,v)$, then it does not change. Indeed, every upper-right corner for $(u,v)$ includes an upper-right corner for $(u,uv)$, while every upper-left corner for $(u,uv)$ includes an upper-left corner for $(u,v)$. Since $(u,v)\mapsto (v,u)$ and $(u,uv)$ generate $\GL_2(\Z)$, this proves the invariance of the given invariant.)

There is a second invariant, which we call additive holonomy. Consider a covering by $4m$ near corners as above. Start with a corner $P$ and a grading $f:P\to\Z^2$ on $P$ (it is unique up to translation). Then for $Q$ such that $(P,Q)$ is adjacent, there is a unique grading on $Q$ such that we have a grading on $P\cup Q$. Continuing this way until we make one turn, we obtain a new grading on $P$, thus equal to $f+s$ for some $s\in\Z^2$. It is straightforward that $s$ does not depend on the choice of $P$.

\begin{defn}
For $X$ a near free 1-ended near $\Z^2$-set, we call $s=s_X\in\Z^2$ the additive holonomy of $X$.\index{additive holonomy}
\end{defn}

Concerning dependence on the choice of basis, for $(u,u^{\pm 1}v)$ or $(uv^{\pm 1},v)$ we also obtain $s$, while for $(v,u)$ we obtain $-s$. It follows that changing the basis by some $A\in\GL_2(\Z)$ changes $s$ into $\det(A)s\in\{s,-s\}$.

\begin{thm}\label{mscomplete}
The pair $(m,s)\in \N_{>0}\times\Z^2$ is a complete invariant of the classification up to near isomorphism of near free 1-ended near actions of $\Z^2$.
\end{thm}
\begin{proof}
Start with an near upper-right corner $P$ and a chain $P_0,P_1,\dots,P_{4m}$ of near upper-right corners with $(P_i,P_{i+1})$ adjacent for all $i$.

Choose a grading $f_0$ on $P$ whose image is near $\N^2$. Let $\M=\Z\smallsetminus\N$ denote the set of negative numbers. This defines a grading $f_1$ on $P_1$. The image of this grading is near $\M\times\Z_{\ge i}\times\M$ for some $i$. If $i<0$, we can remove a left strip to $P_1$ and add it to $P_2$; if $i>0$ we can do the contrary; in both cases, we ensure that the grading $f_1$ on $P_1$ has image near $\M\times\N$. In turn, possibly exchanging a down strip between $P_2$ and $P_3$, we can ensure that the grading on $P_2$ has image near $\M^2$, and so on, until we ensure that for all $j\le 4m-2$, the grading on $P_j$ has image near $C_j$, where $C_j$ is equal to $\N^2$, $\M\times\N$, $\M^2$, $\N\times\M$ according to whether $j$ equals 0,1,2,3 modulo 4. Next, the grading on $P_{4m-1}$ has image near $\N\times\Z_{<i}$ for some $i$. Hence, the resulting grading $f_{4m}$ on $P_{4m}=P$ has image near $\Z_{\ge j}\times\Z_{\ge i}$ for some $j$, and hence $s=(j,i)$. This procedure characterizes the near action in terms of $(m,s)$ only.
\end{proof}

Denoting by $X_{m,s}$ the above near $\Z^2$-set, it is not hard to check that its index character is given by $u^av^b\mapsto s_2a-s_1b$. It particular, it is surjective if and only if $s$ is primitive, and is zero if and only if $s=0$. If $s\neq 0$, can write $s=k_ss'$ for some unique positive integer $k_s\ge 1$ and primitive element $s'$; for $s=0$ write $k_v=0$.

\begin{prop}\label{nearaz2}
The group $A_{m,s}=\mkst_{\Z^2}(X_{m,s})$ of near automorphisms of the near $\Z^2$-set $X_{m,s}$ is abelian and includes the given $\Z^2$ as a subgroup of index $m$. More precisely, it has the presentation as abelian group $\langle u,v,z:z^m=s(u,v)\rangle$. In particular, it is isomorphic to $\Z^2\times (\Z/\gcd(k_v,m)\Z)$. Thus it is torsion-free if and only if no prime divisor of $m$ divides $s$, and is isomorphic to $\Z^2\times (\Z/m\Z)$ if and only if $m$ divides $s$.
\end{prop}
\begin{proof}
We choose a model as in the proof of Theorem \ref{mscomplete}, so $P_j=C_j$ for $j\le 4m-2$, $P_{4m-1}=\N\times\Z_{<i}$ and is glued to $P_0=\N^2$ by attaching $(n,i-1)\in P_{4m-1}$ to $(n-j,0)\in P_0$.

Suppose that $m\ge 2$. Let $w$ map $P_0$ to $P_4$ by the identity map. Extend it to the identity map $P_1\to P_5$, and so on, until we have the identity map $P_{4m-6}\to P_{4m-2}$. If $i\ge 0$, we can also extend it to the identity embedding $P_{4m-5}\to P_{4m-1}$. Then it extends to the lower strip $\N\times\{0,\dots,i-1\}\subseteq P_{4m}$, mapping into $P_{4m-1}$, again by the identity map. Then it extends to $\Z\times\Z_{\ge i}\subseteq P_{4m-4}\cup P_{4m-3}$ to $P_0\cup P_1$, mapping $(n,i)$ to $(n-j,0)$ for $n\ge 0$, and more generally mapping $(n,q)$ to $(n-j,q-i)$ for all $n\in\Z,q\ge i$. Then, it extends to $\Z_{\le\min(0,j)}\times\Z\subseteq P_{4m-3}\cup P_{4m-2}$ to $P_1\cup P_2$, again by the same formula $(n,q)\mapsto(n-j,q-i)$, if $j\ge 0$. Next, it extends to $\Z\times\M\subseteq P_{4m-2}\cup P_{4m-1}$ to $P_2\cup P_3$, again by the same formula $(n,q)\mapsto(n-j,q-i)$. Next, it extends to $\Z_{\ge\max(0,j)}\times\Z_{<i}\subseteq P_{4m-1}$ to $P_3\cup P_4$, again by the same formula $(n,q)\mapsto(n-j,q-i)$; in particular this maps $(n,i-1)$ to $(n-j,-1)$, and then this precisely extends as the identity map $P_0\to P_4$. The case $i<0$ is similar.

It follows from the definition that $w^{-m}=s$.

Next consider a near automorphism $\phi$ and suppose that $P_0\Delta\phi(P_0)$ is infinite. Then we claim that $\phi$ is given by the near action of some element of $\Z^2$. Choose a representative $\tilde{\phi}$. It easily follows that some upper corner included in $P_0$ is mapped by $\tilde{\phi}$ into an upper-corner included in $P_0$. Then, composing by some element of $\Z^2$ (which acts as graph near automorphisms -- this is specific to $\Z^2$ being abelian), we can suppose that $\hat{\phi}$ acts as the identity on some upper corner included in $P_0$. The set of fixed points of $\hat{\phi}$ has finite boundary. It follows that is is cofinite.

Next, if $\phi$ is an arbitrary near automorphism, there exists $k\in\{0,\dots,m-1\}$ such that $P_0\cap w^k\phi(P_0)$ is infinite, and hence, as we have just checked, belongs to the action of $\Z^2$. 

By definition, the action of $\Z^2$ is central and the quotient is cyclic. So the near automorphism group is abelian, and its description is therefore as predicted. 
\end{proof}

We have seen that the near automorphism group of the near $\Z^2$-set $X_{m,s}$ is isomorphic to $\Z^2\times(\Z/\ell\Z)$, where $\ell=\gcd(m,k_s)$ and $s=k_ss'$ with $k_s\ge 1$, $s'$ primitive (with $k_0=0$). Denote by $A=A_{m,s}$ the near automorphism group of $X_{m,s}$, which is abelian and whose torsion subgroup is cyclic of order $\ell$. Consider a subgroup $H$ of $F$ and a homomorphism $u:F\to A$ such that the diagonal homomorphism $F\to A\times (F/H$) is injective (that is, $H\cap\Ker(u)=\{0\}$, or still equivalently, $u|_H$ is injective). Note that since the torsion subgroup of $A$ is cyclic, this forces $H$ to be cyclic. We can assume that the $A$-action on $X_{m,s}$ is realized by a free action (remove a finite subset if necessary).

Consider the disjoint union $X_{m,s}\times (F/H)$ of $|F/H|$ copies of $X_{m,s}$. Let $F$ act on it by $g\cdot (x,g'+H)=(u(g)x,g+g'+H)$. Then this action is free: indeed, if $(u(g)x,g+g'+H)=(x,g'+H)$ then $g\in H$ and $u(g)x=x$, if $g$ were $\neq 0$, this implies $u(g)\neq 0$, and since $A$ acts freely on $X_{m,s}$ this yields a contradiction. This commutes with the near action of $\Z^2$ (which does not touch the second coordinate) and thus defines a near free action of $\Z^2\times F$. We denote it as $X_{m,s,u,H}$.

\begin{lem}
If $u^{-1}u'$ vanishes on $H$, then $X_{m,s,u,H}$ and $X_{m,s,u',H}$ are isomorphic as near $\Z^2\times F$-sets.
\end{lem}
\begin{proof}
Write $u'=u+q$, where $q$ factors through $G/H$. On $X_{m,s}\times F/H$, write $f(x,g'+H)=(q(g')x,g'+H)$. It commutes with the $\Z^2$-action, while for $g\in F$, we have
\[fgf^{-1}(x,g'+H)=fg(q(g')^{-1}x,g'+H)=f(u(g)q(g')^{-1},g+g'+H)\]
\[=(q(g+g')u(g)q(g')^{-1},g+g'+H)=(q(g)u(g),g+g'+H),\]
so indeed conjugates the two near actions.
\end{proof}

Define $\mathcal{C}_\ell(F)$ as the set of cyclic subgroup $H$ of $F$ such that some injective homomorphism $H\to\Z/\ell\Z$ extends to $F$.

\begin{lem}
Given a subgroup $H$ of $F$, if $H\in\mathcal{C}_\ell(F)$ then every injective homomorphism $H\to\Z/\ell\Z$ extends to $F$.
\end{lem}
\begin{proof}
The assumption implies that $H$ is cyclic. Let $f,f'$ be injective homomorphism $H\to\Z/\ell\Z$ such that $f$ extends to $F$. It turns out that $f$ and $f'$ are equal after multiplication by an invertible element of $f(H)=f'(H)$, and the latter extends to an automorphism of $\Z/\ell\Z$. This implies that $f'$ extends as well.
\end{proof}

Start again with $\ell=\gcd(m,k_s)$ as above.

\begin{thm}\label{msuh}
The quadruple $(m,s,H,u)$, where $m\ge 1$, $s\in\Z^2$, $H\in\mathcal{C}_\ell(F)$ and $u$ is an injective homomorphism $H\to A_{m,s}$, is a complete invariant of 1-ended near free near actions of $\Z^2\times F$.
\end{thm}
\begin{proof}
Let us start showing these are pairwise non-isomorphic as near actions. Let us choose an extension of $u$ to $F$, still denoted $u$. As an near action of $\Z^2$, we have $n$ copies of $X_{m,s}$, where $n=|F/H|$. Hence, $(m,s)$ is determined. Then $F$ acts on the set of $n$ ends of the $\Z^2$-action; this action is transitive and the stabilizer of any point is equal to $H$. So, $H$ is determined as well. This stabilizer acts on the given copy of $X_{m,s}$ by $u|_H$, so $u|_H$ is determined as well up to conjugation by automorphism of the near $\Z^2$-set $X_{m,s}$, but since its automorphism group is abelian (by Proposition \ref{nearaz2}), it is determined up to equality.

Now consider any 1-ended near free near action of $\Z^2\times F$ on a set $X$. On $\Z^2$, 
it is finitely-ended by Proposition \ref{fendind}. Since it is 1-ended on the whole group, the $F$-action on the set of ends is transitive, and since $\Z^2$ is central, they correspond to pairwise isomorphic 1-ended $\Z^2$-sets. So, by Theorem \ref{mscomplete}, $X$ is near-isomorphic to the disjoint of $n$ copies of $X_{m,s}$, for some $m,n\ge 1$ and $s\in\Z^2$. The near automorphism group (as near $\Z^2$-set) of this disjoint union is the permutational wreath product of $\Aut(X_{m,s})$ and the symmetric group $\mks_n$. Then $F$ maps into an abelian transitive subgroup of $\mks_n$, hence simply transitive. We can thus, after fixing one of the copies, identify this set of $n$ elements with a quotient $F/H$ of $F$ and write this disjoint union as $X_{m,s}\times (F/H)$. The set of near automorphisms acting on $F/H$ as a translation is the wreath product $A\wr F/H$. Given that $V=A^{F/H}$ is abelian (the exponent here means power, not invariant vectors), an abelian subgroup of this semidirect product $V\rtimes F/H$ with surjective projection to $F/H$ has be included in the subgroup $W\rtimes (F/H)$, where $W$ is the set of $F/H$-invariant vectors in $V$. Here, this set of invariant vectors $W$ can be identified with $A$, embedded diagonally into $A^{F/H}$. So the action of $F$ has the form $g\cdot (x,g'+H)=(u(g)x,g+g'+H)$, as required (near freeness forces injectivity of $F\to A\times (G/H)$).
\end{proof}

%Let us give another example: w

\begin{exe}
Suppose that $F$ is cyclic of prime order $p$. In this case, when $(m,s)$ is given, we have the possibilities (see Example \ref{e_msuh}):

\begin{itemize}
\item $n=1$, $p$ divides both $m$ and $s$: in this case the near automorphism group of $X_{m,s}$ has a unique cyclic subgroup of order $p$, and the possibilities are classified by the $p-1$ isomorphisms from $F$ to this cyclic subgroup.
\item $n=p$: the case of the plain product.
\end{itemize}
\end{exe}

\begin{exe}\label{e_msuh}
Fix $(m,s)$. Let us describe the possibilities for $(H,u)$
The case when $H=0$ gives a unique $u$: this is the case of the product of the $X_{m,s}$ with the simply transitive action of $S$; call it ``plain product".

When $F$ is cyclic of prime order, we have $F\in\mathcal{C}_\ell(F)$ if and only $p$ divides $\ell$ (i.e., divides both $m$ and $s$). In this case, we obtain $p-1$ nonzero injective homomorphisms $F\to A$.

When $F$ is cyclic of order $p^k$, it more generally holds that if $p^k$ does not divide $\ell$ then the only possibility is the plain product. (Indeed, if $H\neq 0$ and there is an injective homomorphism $H\to A$ that can be extended to $F$, then the extension has to be injective.) When $H\neq 0$, say has order $p^{k'}$ with $k'\ge 1$ and $p^k$ divides, we then have $p^{k'-1}(p-1)$ such choices of injective homomorphisms.

In general, when $H$ is cyclic of order $p^{k'}$ with $k'\ge 1$, there is a maximal $k\ge k'$ such that $H\subseteq p^{k-k'}F$. Then $H\in\mathcal{C}_\ell(F)$ if and only if $p^k$ divides $\ell$; in this case we have $p^{k'-1}(p-1)$ injective homomorphisms $H\to A$ (which can be extended to $F$).

The general description follows, since one can treat separately all Sylow subgroups of $F$.
\end{exe}

\begin{lem}\label{nearautx}
The near automorphism group of the near $\Z^2\times F$-set $X_{m,s,H,u}$ is $A_{m,s}\times F/H$, with $A_{m,s}$ embedded diagonally into the power $A_{m,s}^{F/H}$.
\end{lem}
\begin{proof}
Write $A=A_{m,s}$ and $X=X_{m,s,H,u}$. As a near $\Z^2$-set, the near automorphism group of $X$ is the permutational wreath product $A\wr\mks(F/H)$. We have to compute the centralizer $C$ of the image of $F$. The projection of $F$ to $\mks(F/H)$ is a transitive abelian group $F/H$, and hence is equal to its own centralizer. So, the projection of $C$ in $\mks(F/H)$ is included in $F/H$, with equality since it includes the image of $F$. Hence $C$ is included in $A\wr (F/H)$. Since the projection of $F$ on $F/H$ is surjective and central, we deduce that $C$ is included in $A\times (F/H)$, where $A$ is viewed as embedded diagonally. This is abelian and includes the image of $F$, hence is exactly the centralizer.
\end{proof}

\begin{prop}
The lest of the near $\Z^2$-set $X_{m,s}$ is equal to $k_s$.

The lest of the near $\Z^2\times F$-set $X_{m,s,H,u}$ is equal to $|F/H|k_s$.
\end{prop}
\begin{proof}
We use Proposition \ref{basiclest}(\ref{bl4}).

We have to compute the index number of $X_{m,s}$ as near $A_{m,s}$-set (recall that $A_{m,s}=\mkst_{\Z^2}(X_{m,s})$). Then $\Z^2$ maps onto $k_s\Z$. It remains to determine the image of $z$: since $z^m$ mapsto $s$ and $s$ maps to 0, $z$ mapsto zero. Hence the lest is equal to $k_s$.
%The image of the index map is $k_s\Z$, hence the lest divides $k_s$ by Lemma \ref{indexlest}. ?????????? 

In view of Lemma \ref{nearautx}, this is the index number of $X_{m,s,H,u}$ as near $A_{m,s}$, where $A_{m,s}$ is embedded diagonally. This is therefore equal to $|F/H|k_s$.
\end{proof}

\subsection{The case of $\Q$-rank $\ge 3$}
\begin{lem}\label{2inftype}
Let $G$ be an abelian group and let $A$ be a subgroup of $G$ of infinite index isomorphic to $\Z^2$. Let $X$ be an infinite near free near $G$-set. Then $X$ is not of finite type as near $A$-set.
\end{lem}
\begin{proof}
Suppose the contrary. Then $G$ acts on the set of ends of the near $A$-set $X$. By the contradictory assumption and Theorem \ref{abends}, $X$ is finitely-ended as near $A$-set. Hence, passing to a finite index subgroup of $G$ including $A$, we can suppose that $G$ fixes each end of the near $A$-set $X$. Then, passing to a $G$-commensurated subset, we can suppose that $X$ is 1-ended as near $A$-set.

Since $G$ is abelian, the near action on $X$ is by near $A$-automorphisms. Since the near action of $G$ is $\star$-faithful we obtain an injection of $G$ into the group of near $A$-automorphisms of $X$. By Theorem \ref{mscomplete}, $X$ is near isomorphic to $X_{m,s}$ for some $m\ge 1$ and $s\in\Z^2$, and by Proposition \ref{nearaz2}, it follows that the group of near $A$-automorphisms of $X$ is virtually isomorphic to $\Z^2$. Since $G$ includes $A$ with infinite index, it cannot have such an embedding, and we reach a contradiction.
\end{proof}

It would be nice to obtain a simpler proof of Lemma \ref{2inftype}, without using the classification of 1-ended near free near actions of $\Z^2$, as it would simplify the proof of Theorem \ref{3more}.

\begin{thm}\label{3more}
Let $G$ be finitely generated abelian group of $\Q$-rank $\ge 3$ and $X$ a near free 1-ended near $G$-set. Then $X$ is near isomorphic to the simply transitive action.
\end{thm}
\begin{proof}Let $A\subseteq G$ be isomorphic to $\Z^2$. By Lemma \ref{2inftype}, $X$ is not of finite type as near $A$-set. 

So we can write $X=Y\sqcup Z$ as near $A$-set, where $Y$ is an $A$-set of finite type and $Z$ is a realizable $A$-set, realizable with infinitely many orbits; we choose such an action. Then $Z$ is the disjoint union of a free $A$-action and a finite subset, which we can move into $Y$: so we can assume that $Z$ is a free $A$-set. Since $Y$ is finitely-ended as $A$-set by Theorem \ref{abends}, we can write it as $Y_1\sqcup\dots\sqcup Y_k$ with each $Y_i$ being 1-ended as $A$-set.

We can choose the ordering so that $Y_i$ is stably realizable if and only if $i>\ell$. For $g\in G$, $gY_i$ is a 1-ended near $A$-subset and hence has finite symmetric difference with either $Y_j$, or one $A$-orbit in $Z$. If $i\le\ell$, then $gY_i$ is not stably realizable, and hence has finite symmetric difference  with $Y_j$ for some $j\le\ell$. Hence $\bigsqcup_{i\le\ell}Y_i$ is commensurated by $G$. Since $Z$ is infinite and $X$ is 1-ended, this implies that $\ell=0$. Hence $X$ is stably realizable as near $A$-set. Hence after adding or removing finitely many points, $X$ is realizable as a free $A$-set.

Let us observe that this works for an arbitrary choice of $A$. Hence, if the index character were nonzero on $G$, we could have chosen $A$ not included in the kernel of the index character, and we would have obtained a contradiction (since stably realizable implies zero index). This shows that $X$ has zero index as near $G$-set.

Now, writing $G=\Z^d\times F$ with $F$ finite, choose $B=\Z^k\times\{0\}^{d-k}\subseteq G$ with $k$ maximal such that the near action of $B$ on $X$ is realizable as a free action. By the previous work, we have $k\ge 2$; let us show that $k=d$. Otherwise, choose $u=u_{k+1}$, the next basis element.

Since the index is zero, we can choose a permutation $U$ of $X$ lifting $u$. Then for all $a\in A$, $UaU^{-1}$ is a finite perturbation of $a$, and hence, by Theorem \ref{per1e}, there exits a finitely supported permutation $s$ such that $UaU'=sas^{-1}$ for all $a\in A$. Thus, $s^{-1}U$ commutes with the $A$-action. So we can replace $U$ with $s^{-1}U$ and hence the action of $\Z^{k+1}$ is realizable. Since the near action is near free, the union of non-free orbits is finite, and since the action of $\Z^k$ is free, this union is empty, so the $\Z^{k+1}$-action is free. This contradicts the maximality of $k$.
This shows that the $\Z^d$-action is realizable (as a free action). 

Choose a realization of $F$ as well. 
Fix $x\in X$. For $g\in F$, consider the value of $f_x(t)=tgt^{-1}x$ when $t$ ranges over $\Z^d$. Then, for fixed $u\in\Z^d$, we have $tugu^{-1}t^{-1}x=tgt^{-1}x$ if and only if $g^{-1}ugu^{-1}t^{-1}x=t^{-1}x$. This holds for a cofinite subset of $t$. Hence, there is a cofinite subset of $t$ such that $f(t)=f(tu)$ for every $u$ inside the standard generating subset of $\Z^d$. Since $d\ge 2$, $\Z^d$ is 1-ended and this implies that $f_x$ takes a single value, say $k(g,x)$ outside some finite subset $T_{g,x}$. Writing $T=\bigcup_{g\in F}T_g$, we have $tgt^{-1}x=k(g,x)$ for all $g\in G$ and all $t\notin T_x$. Then, choosing $t\notin T_x\cup T_{k(h,x)}$, we have
\[k(gh,x)=tght^{-1}x=(tgt^{-1})(tht^{-1})x=tgt^{-1}k(h,x)=k(g,k(h,x)).\]
So $(g,x)\mapsto k(g,x)$ defines a new action of $F$ on $X$.

It commutes with the $\Z^d$-action: choose $t$ such that $t^{-1}x$ does not belong to the finite support of $gst^{-1}x=gst^{-1}x$. 
Then 
\[k(g,sx)=tgt^{-1}sx=tgst^{-1}x=tsgt^{-1}x=sk(g,x).\]

Moreover, it is a finite perturbation of the original action: indeed, fix one generator $u$. The set of $x$ such that there exist $n_1\le 0\le n_2$ such that both $u^{n_1}x$ and $u^{n_2}x$ meet the support of $g^{-1}ugu^{-1}$ or $g^{-1}u^{-1}gu$ is finite. Otherwise, it follows that $f_x(u^n)=f_x(1)$ for all $n\ge 0$ or for all $n\ge 0$, and it follows that $f_x(1)=k(g,x)$, and hence $gx=k(g,x)$.
\end{proof}

\subsection{Application to abelian normal subgroups}\label{s_abnors}
\begin{proof}[Proof of Corollary \ref{abnors}]
In the space of ends of the near $A$-action, isolated points correspond to $A$-commensurated 1-ended subsets up to near equality (see Proposition \ref{isolend}). The $G$-action on the set of isolated ends stabilizes those finitely many ends corresponding to non-$A$-completable near subsets (there are finitely many as a consequence of Theorem \ref{nearactionfp}). This yields a decomposition of $X$ as near $G$-set: $X=Y\sqcup Z$, where $Y$ is the union of all these non-completable subsets ($Y$ is well-defined up to near equality). Some finite index subgroup $H$ of $G$ including $A$ fixes the (finite) set of ends of $Y$. Then $H$ maps into the product of the near automorphism groups of all these subsets, which is abelian, extending the $A$-action. If $Y$ is infinite, this shows that $A/(A\cap [H,H])$ has $\Q$-rank $\ge 2$, so $d\ge 2$.

Suppose now that $d=0$ and that $A$ is central. Since $d\le 1$, $X$ is completable as near $A$-set, so we can use the criterion (\ref{zsr3}) of Theorem \ref{complab} to show that $X$ is stably realizable as near $A$-set. We have to check that for every subgroup $B$ of $A$, the subset $X^B$ has index 0 as near $A$-set. Since $A$ is central, $B$ is normal and hence $X^B$ is commensurated by $G$, i.e., is a near $G$-subset. Since $d=0$, the index character of $G$ vanishes on $A$ and we can conclude.  

Now suppose that $d=0$ and the orbit assumption. By Theorem \ref{complab}, the set of subgroups $B$ (of $\Q$-corank 1) of $A$ such that $|n_B|\neq 0$ is finite. It is $G$-invariant. Hence, if non-empty, its image in the set of hyperplanes of $A\otimes\Q$ is a finite orbit, a contradiction.
\end{proof}

\section{Near automorphisms of graphs and structures}\label{s_neretin}

\subsection{Structures on sets}\label{stronset}
Let $I$ be a set (called labeling set) and $a:I\to\N$ be an arity function. The pair $\II=(I,a)$ is usually called a ``signature"; we write $I=I_\II$ and $a=a_\II$. For any set $X$, define
\[\WW_\II(X)=\{(y,i):i\in I,y\in X^{a(i)}\}.\]
It is convenient to think of it as the disjoint union of copies $\lceil X^{a(i)}\rfloor_i=X^{a(i)}\times\{i\}$ indexed by $i$.

\begin{defn}
A finitary relational structure\index{finitary relational structure}\index{structure (finitary relational)} (or ``structure") of signature $\II$ on $X$ is a subset $P$ of $\WW_\II(X)$. Define $\St_\II(X)$ as the set of such structures, namely the Boolean algebra of subsets of $\WW_\II(X)$. We also call $(X,P)$ an $\II$-structure.
\end{defn}

Writing $P=\bigsqcup_{i\in I}\lceil P_i\rfloor_i$, where $P_i\subseteq X^{a(i)}$, we can view $P$ as the family of the structures $P_i$.
For $F\subseteq X$, define 
%\[\St_\II(X)[x]=\{(y,i)\in \WW_\II:x\in\mathrm{Im}(y)\},\]
\[P_{[F]}=\{(y,i)\in P:F\cap\mathrm{Im}(y)\neq\emptyset\},\]
where $\mathrm{Im}(y)=\{y_j:1\le j\le a(i)\}$. We say that $P$ is multiproper\index{multiproper} if $P_{[F]}$ is finite for every finite subset $F\subseteq X$ (or equivalently for every singleton $F$). The set of multiproper subsets $P\subseteq \WW_\II(X)$ forms an ideal in the Boolean algebra $\WW_\II(X)$, including the ideal of finite subsets.
%For $F\subseteq X$, define $P_{[F]}=\bigcup_{x\in F}P_{[x]}$.

Every map $f:X\to Y$ induces a map $f_\circ:\WW_\II(X)\to\WW_\II(Y)$, which in restriction to $\lceil X^{a(i)}\rfloor_i$ is just the map $f^{\times a(i)}=f\times f\times\cdots\times f$ ($a(i)$ times) from $\lceil X^{a(i)}\rfloor_i$ to $\lceil Y^{a(i)}\rfloor_i$, applying $f$ to each coordinate. The inverse image map yields a Boolean algebra homomorphism $f^*:\St_\II(Y)\to\St_\II(X)$. If $f$ is proper (i.e., has finite fibers), then so is $f_\circ$, and it then follows that $f^*$ maps the ideal of multiproper subsets of $\WW_\II(Y)$ into the ideal of multiproper subsets of $\WW_\II(X)$.

\begin{defn}
Let $(X,P)$, $(Y,Q)$ be $\II$-structures. Then $f:X\to Y$ is called a homomorphism of $\II$-structures if $P\subseteq f^*Q$. This means that for every $i$, one has the inclusion $P_i\subseteq f_i^*Q_i$, which also means that $f^{\times a(i)}$ maps $P_i$ into $Q_i$. The monoid of endomorphisms of the $\II$-structure $(X,P)$ is denoted by $\End(X,P)$, and its group of automorphisms is denoted by $\Aut(X,P)$ or $\mks(X,P)$, by definition it is the group of invertible elements in $\End(X,P)$. 
\end{defn}

The group $\mks(X)$ is endowed with the topology of pointwise convergence, which is a group topology. When $X$ is infinite countable, this is a Polish, non-locally-compact topology. For every signature $\II=(I,a)$ and $\II$-structure $\St_\II(X)$ on $X$, the subgroup $\mks(X,P)\subseteq \mks(X)$ is closed. An important fact (see \cite[Theorem 2.6]{Cam}) is that every closed subgroup of $\mks(X)$ occurs this way, that is, is equal to $\mks(X,P)$ for some signature $\II$ and $\II$-structure $P$ on $X$.

%Let $X$ be a set. The usual topology on $\mks(X)$ is the topology of pointwise convergence, which is a group topology. When $X$ is infinite countable, this is a Polish, non-locally-compact topology.
 
%Let $\II=(I,a)$ be a signature. 

%\begin{defn}We say that a structure $P$ is proper if it satisfies the following two conditions:\begin{itemize}\item for each $i$, the structure $P_i\subseteq X^{a(i)}$ is proper (if $a(i)=1$ this is an empty condition);\item for each $x\in X$, the set of $i$ such that $\breve{\pi}(P_i)\neq\emptyset$ is finite. (In other words, each $x\in X$ occurs only in finitely many $P_i$. If $I$ is finite, this is an empty condition.)\end{itemize}\end{defn}

%for every $i\in I$ and every $j<a(i)$, $\pi_j$ is proper (i.e., has finite fibers) in restriction to $P_i$. We then call $(X,P)$ a proper structure.

\begin{exe}
An oriented graph (without multiple edges, allowing self-loops) on the set of vertices $X$ can be encoded in a subset $P$ of $X^2$. This can be viewed as a structure where $I$ is a singleton $\{x_0\}$ with $a(x_0)=2$, and $P$ is identified with $P\times\{x_0\}$, arity is 2. Then $P$ is multiproper if and only if every vertex has finite valency.
\end{exe}

\begin{defn}
Given two structures $(P^{(1)},I_1)$ and $(P^{(2)},I_2)$ on a set $X$, the join\index{join} $P^{(1)}\sqcup P^{(2)}$ is the structure $(P,I_1\sqcup I_2)$ on the set $X$ where $P_i=P^{(j)}_i$ for each of $j=1,2$. It amounts to endow $X$ with both structure simultaneously.
\end{defn}
 
\begin{exe}
If $P_1,P_2$ are subsets of $X^2$, defining two oriented graph structures, the join $P_1\sqcup P_2$ endows $X$ with a labeled oriented graph structure, with oriented edges labeled in $\{1,2\}$ (and no multiple oriented edges of a single label). 
\end{exe}

%We let $\mks(X)$ act coordinate-wise on $M_X\times I$ (with trivial action on $I$). Let $\mks(X,P)$ be the stabilizer of $P$ for this action. Then this is a closed subgroup of $\mks(X)$. Moreover, it is well-known (see \cite[Theorem 2.6]{Cam}) that every closed subgroup of $\mks(X)$ occurs this way.

\begin{rem}\label{tricho}
For a closed subgroup $G$ of $\mks(X)$, there is an important well-known trichotomy, directly following from the definition. Denote by $G_{(F)}$ the pointwise stabilizer of a subset $F$.
\begin{enumerate}
\item $G$ is not locally compact. This holds if and only if for every finite subset $F\subseteq X$, the pointwise stabilizer $G_{(F)}$ has an infinite orbit;
\item\label{loccpt} $G$ is locally compact and non-discrete. This holds if and only if for some finite subset $F_0\subseteq X$, the pointwise stabilizer $G_{(F_0)}$ has no infinite orbit, while for every finite subset $F$, the pointwise stabilizer $G_{(F)}$ is not reduced to the identity. 
\item $G$ is discrete. This holds if and only if for some finite subset $F_0\subseteq X$, the pointwise stabilizer $G_{(F_0)}$ is reduced to the identity.
\end{enumerate}
In \cite{BS}, Bergman and Shelah split Case (\ref{loccpt}) into two subcases, according to whether there exists a finite subset $F_0\subseteq X$ such that $G_{(F_0)}$ has orbits of uniformly bounded finite cardinal. 
\end{rem}

\begin{defn}
Denote by $\mks_{<\aleph_0}(X,P)$ the intersection $\mks_{<\aleph_0}(X)\cap\mks(X,P)$.

Say that the structure $P$ on $X$ is fillable if $\mks_{<\aleph_0}(X)$ is finite.\index{fillable}
\end{defn}

This roughly means that an automorphism of $(X,P)$ is determined, outside some fixed finite subset $F_0$, by its germ at infinity. The following proposition, due to Bergman-Shelah \cite[Lemma 20]{BS} makes this more rigorous. It is proved using a simple but tricky pigeonhole argument (already apparent in the earlier \cite[11.2.6]{Sco}):

%, or Bergman-Shelah, Lemma 20 in \cite{BS}):

%This means that every automorphism of $(X,P)$ is determined, outside some fixed finite subset $F_0$, by its germ at infinity. Actually this has the following characterization, proved using a simple but tricky pigeonhole argument (see \cite[11.2.6]{Sco}, or Bergman-Shelah, Lemma 20 in \cite{BS}):

\begin{prop}
Every closed subgroup of $\mks(X)$ included in $\mks_{<\aleph_0}(X)$ is finite.\qed
\end{prop}

\begin{cor}\label{bershe}
For every finitary structure $P$ on $X$, the following are equivalent:
\begin{enumerate}
\item $\mks_{<\aleph_0}(X,P)$ is finite (i.e., $P$ is fillable);
\item $\mks_{<\aleph_0}(X,P)$ is discrete in $\mks(X,P)$;
\item $\mks_{<\aleph_0}(X,P)$ is closed in $\mks(X,P)$.\qed
\end{enumerate}
\end{cor}
%normal subgroup of finitely supported automorphisms $\mks_{<\aleph_0}(X,P)$ is finite (i.e.\ $P$ is fillable) if and only if it is closed in $\mks(X,P)$.\qed

\subsection{Almost automorphism group of a finitary relational structure}

\begin{defn}
Let $P$ be a finitary relational structure on a set $X$ (see \S\ref{stronset}). Let us denote by $\lcc P\rcc$\index{$\lcc P\rcc$} the class of $P$ modulo perturbation by finite subsets. 

Define the almost automorphism group $\mks(X,\lcc P\rcc)$ of $P$ (or of $\lcc P\rcc$) as the subgroup of those $f\in\mks(X)$ such that $f^*\lcc P\rcc=\lcc P\rcc$. In other words, this is the subgroup of $\mks(X)$ of those elements commensurating $P$. Also denote it by $\mathrm{AAut}(X,P)$.
\end{defn}

%of invertible elements of $\mk{E}(X,\lcc P\rcc)$, that is, the subgroup of $\mks(X)$ of those elements commensurating $P$.

\begin{rem}
If $I$ is finite, or if $P$ is multiproper, then $\mks(X,\lcc P\rcc)$ includes the subgroup of finitely supported permutations $\mks_{<\aleph_0}(X)$; in particular in these cases it is dense in $\mks(X)$. This is not true in general. For instance, if $I$ is infinite, the arity is constant and $P_i=P$ for all $i$, then $\mks(X,\lcc P\rcc)=\mks(X,P)$, which in many cases can fail to contain all finitely supported permutations.
\end{rem}

\begin{rem}\label{aavsna}
The concept of almost automorphism group seems to first appear in work of Truss \cite{Tru87,Tru89} in the context of the countable universal random graph (also known as Erd\H os-R\'enyi or Rado graph, although discovered much earlier by Ackermann), see also \cite{MSS} and \cite[\S 6.1,\S 7.2]{Tar}. Beware that in more recent work in the context of trees, ``almost automorphism group" sometimes denotes the near automorphism group, introduced in a wider generality in \S\ref{s_products}, which in most studied cases coincides with the quotient of the almost automorphism group (in its original definition) by its subgroup of finitely supported permutations. 

%Almost automorphisms and near automorphisms of the infinite random graph appear in the monograph \cite{Tar}, in \S 6.1. %and \S 7.2 respectively.

The almost automorphism group (of unlabeled graphs) has been rediscovered (as ``quasi-automorphism group") in \cite{Leh,LehS}, and notably studied in the context of a colored binary tree, where it has been considered since then in \cite{BMN} and \cite{NS} (with no reference to Truss).
\end{rem}

% TARZI PREND DES PERMUTATIONS
%The best-known studied example is the case of trees, discussed below.  Graphs consisting of infinite disjoint unions of finite complete graphs are also considered in a close setting in several papers, notably \cite[\S 3]{Wi} (see \cite[\S 5]{CC} where the formulation makes the link more explicit).

%sometimes denotes its quotient by finitely supported permutations which we call here (in a more general context) near automorphism group.

\begin{rem}
The subgroup $\mks(X,\lcc P\rcc)$ of $\mks(X)$ depends on $P$, but does not only depend on the subgroup $\mks(X,P)$. Indeed, let $P$ encode the structure of a rooted tree $T_{2,2}$ where each vertex has 2 successors. Define $P'\supset P$ by enriching the structure to label vertices by their distance to the root. Then $\mks(X,P)=\mks(X,P')$ while $\mks(X,\lcc P\rcc)\neq\mks(X,\lcc P'\rcc)$.

Nevertheless, in various interesting cases, the group $G=\mks(X,P)$ can sometimes be used to define groups equal or related to $\mks(X,\lcc P\rcc)$, notably considering the commensurator group of a compact open subgroup of $G$, see \cite{CdM}.
\end{rem}

%; however, this terminology has been widely used in the context of regular trees, and in this case, or more generally in the case of trees without leaves, the automorphism 

The following lemma shows that in the study of almost automorphism groups on countable sets, countable labeling sets are enough.

\begin{lem}\label{coulacou}
Let $X$ be a countable set, and consider a signature $\II=(I,a)$. For every $\II$-structure $P$ on $X$, there exists a countable subset $J\subseteq I$ such that, denoting $P_K=P\cap\WW_{(K,a)}(X)$,
 we have $\mks(X,\lcc P_J\rcc)\subseteq \mks(X,P_{I\smallsetminus J})$. In particular, $\mks(X,\lcc P\rcc)=\mks(X,\lcc P_J\rcc)$
% we have $\mk{E}(X,\lcc P_J\rcc)\subseteq \mk{E}(X,P_{I\smallsetminus J})$. In particular, $\mk{E}(X,\lcc P\rcc)=\mk{E}(X,\lcc P_J\rcc)$ and $\mks(X,\lcc P\rcc)=\mks(X,\lcc P_J\rcc)$
\end{lem}
\begin{proof}
For $d\in\N$, denote $I_d=a^{-1}(\{d\})$. Since $X$ and hence $X^d$ is countable, the power set $W_d=\mathcal{P}(X^d)$ is compact metrizable, and hence all its subsets are separable. Hence there exists a countable subset $J_d$ of $I_d$ such that $D=\{P_i:i\in J_d\}$ is dense in $L_d=\{P_i:i\in I_d\}\subseteq W_d$. Define $J=\bigcup_{d\in\N}J_d\subseteq I$.

Take $f\in\mks(X,\lcc P_J\rcc)$. Now consider $i\notin J$, say $i\in I_d\smallsetminus J_d$. Hence $P_i$ is not an isolated point in $L_d$ and there is an injective sequence $(j_n)$ in $J_d$ such that $P_{j_n}$ tends to $P_i$ in $W_d$. Since $f\in\mks(X,\lcc P_J\rcc)$, by definition of the latter, we have $f\in\mks(X,P_{j_n})$ for $n$ large enough. That is, $f^{-1}(P_{j_n})$, which means that $f$ fixes $P_{j_n}$ for its contravariant action on $W_d=\mathcal{P}(X^d)$. Since the action of $\mks(X)$ on $\mathcal{P}(X^d)$ is by self-homeomorphisms, this implies that $P_i$ is preserved by $f$, that is, $f\in\mks(X,P_i)$. Hence $f\in\mks(X,P_{I\smallsetminus J})$.

Hence $\mk{E}(X,\lcc P_J\rcc)\subseteq \mk{E}(X,\lcc P_J\rcc)\cap\mk{E}(X,P_{I\smallsetminus J})\subseteq \mk{E}(X,\lcc P\rcc)$.
\end{proof}

%Suppose that $f\in\mk{E}(X,\lcc P_J\rcc)$. Now consider $i\notin J$, say $i\in I_d\smallsetminus J_d$. Hence $P_i$ is not an isolated point in $L_d$ and there is an injective sequence $(j_n)$ in $J_d$ such that $P_{j_n}$ tends to $P_i$ in $W_d$. Since $f\in\mk{E}(X,\lcc P_J\rcc)$, by definition of the latter, we have $f\in\mks(X,P_{j_n})$ for $n$ large enough. That is, $f^{-1}(P_{j_n})$, which means that $f$ fixes $P_{j_n}$ for its contravariant action on $W_d=\mathcal{P}(X^d)$. Since the action of $\mk{E}(X)$ on $X^d$ is by self-homeomorphisms, this implies that $P_i$ is fixed by $f$, that is, $f\in\mks(X,P_i)$. Hence $f\in\mks(X,P_{I\smallsetminus J})$.

%Hence $\mk{E}(X,\lcc P_J\rcc)\subseteq \mk{E}(X,\lcc P_J\rcc)\cap\mk{E}(X,P_{I\smallsetminus J})\subseteq \mk{E}(X,\lcc P\rcc)$.

%Beware that the action of $\mks(X)$ on $K$ is not continuous, even for $d=1$, but anyway its point stabilizers are closed. Then for every $j\in J$, we have $f\in\mks(X,\lcc P_j\rcc)$, and $f\in\mks(X,P_i)$ for all but finitely many $i\in J$. 
%The almost automorphism group contains the subgroup $\mks_{<\aleph_0}(X)$ and hence is a dense subgroup of $\mks(X)$.

 The almost automorphism group of a structure has a topology which is generally finer and more interesting than the topology induced by its dense inclusion into $\mks(X)$. Indeed, since $\mks(X,\lcc P\rcc)$ preserves $\lcc P\rcc$, it acts on $\lcc P\rcc$ (viewed as a set!), the set of subsets of $X^2$ with finite symmetric difference with $P$. We endow $\mks(X,\lcc P\rcc)$ with the topology induced by its diagonal inclusion $f\mapsto (f,(f^{-1})*)$ into $\mks(X)\times\mks(\lcc P\rcc)$, which is the topology of pointwise convergence on $X\sqcup \lcc P\rcc$.

%By restriction, it defines a topology on the almost automorphism group, which is the topology induced by inclusion $\mks(X)\times\mks(\lcc P\rcc)$

%By restriction, it defines a topology on the almost automorphism group, which is the topology induced by inclusion $\mks(X)\times\mks(\lcc P\rcc)$

\begin{prop}\label{uniqtopo}
The given embedding in $\mks(X)\times\mks(\lcc P\rcc)$ identifies $\mks(X,\lcc P\rcc)$ with a closed subgroup of $\mks(X)\times\mks(\lcc P\rcc)$. If $X$ is countable, $\mks(X,\lcc P\rcc)$ is a Polish group.

The topology on $\mks(X,\lcc P\rcc)$ is the unique topology for which $\mks(X,P)$ is an open subgroup and carries its original topology.
%
%There is a unique group topology on $\mks(X,\lcc P\rcc)$ making $\mks(X,P)$ an open subgroup. If $X$ is countable, this is a Polish group.
%If $(X,P)$ is weakly rigid then this topology does exist.
\end{prop}
\begin{proof}
The uniqueness is trivial (it holds for all group inclusions). It remains to check that it induces on $\mks(X,P)$ its topology.

A pre-basis of the topology of $\mks(X,\lcc P\rcc)$ is given by the point stabilizers of elements of $X$ and of elements of $P$, which includes $\mks(X,P)$. Intersecting with the latter, we obtain that another pre-basis is given by the $\mks(X,P)$, the point stabilizers of $X$, and the intersections $\mks(X,P)\cap\mks(X,P')$ when $P'\in \lcc P\rcc$, $P'\neq P$. Now observe that this intersection is also equal to $\mks(X,P)\cap\mks(X,P\bigtriangleup P')$, and $\mks(X,P\bigtriangleup P')$ is a finite intersection of point stabilizers of the action on $X$. This shows that a pre-basis of the topology of $\mks(X,\lcc P\rcc)$ is given by $\mks(X,P)$ along with point-stabilizers of the action on $X$. Hence $\mks(X,P)$ is open and the induced topology is the original topology.

Let us show that the image $H$ of $\mks(X,\lcc P\rcc)$ in $\mks(X)\times\mks(\lcc P'\rcc)$ is closed. Let $K\subseteq X\times \lcc P\rcc$ be the belonging relation: $(x,P')\in K$ if $x\in P'$). Clearly $H$ preserves $K$. Conversely, suppose that $(f,g)\in\mk{E}(X)\times\mks(\lcc P'\rcc)$ belongs the stabilizer of $K$. Then, for $x\in X$ and $P'\in \lcc P\rcc$, we have $x\in P'$ $\Leftrightarrow$ $(x,P')\in K$ $\Leftrightarrow$ $(f(x),g(P'))\in K$ $\Leftrightarrow$ $f(x)\in g(P')$. Hence $f(P')=g(P')$. So $f\in\mk{E}(X,\lcc P\rcc)$. Composing $(f,g)$ with some element of $H$, yields some element $(f_0,g_0)$ with $f_0=\mathrm{id}_X$. Hence the above equivalence yields that $g_0(P')=P'$ for all $P'\in \lcc P\rcc$, which means that $g_0$ is the identity, and hence that $(f,g)\in H$. So $H$ is the stabilizer of $K$, and hence $H$ is closed in $\mks(X)\times\mks(\lcc P'\rcc)$.

Suppose that $X$ is countable. If $I$ is countable, then $\mk{E}(X,\lcc P\rcc)$, as a closed submonoid of $\mk{E}(X)\times\mk{E}(\lcc P\rcc)_{\mathrm{op}}$, is Polish and similarly $\mks(X,\lcc P\rcc)$, as a closed subgroup of $\mks(X)\times\mks(\lcc P\rcc)$, is Polish. In general, $I$ can fail to be countable; however by Lemma \ref{coulacou}, there exists a countable subset $J\subseteq I$ such that $\mk{E}(X,\lcc P_J\rcc)\subseteq \mk{E}(X,P_{I\smallsetminus J})$. In other words, $\mk{E}(X,\lcc P_J\rcc)=\mk{E}(X,P_{I\smallsetminus J})\cap\mk{E}(X,\lcc P_J\rcc)$, which, as a closed submonoid of the Polish monoid $\mk{E}(X,\lcc P_J\rcc)$, is closed. The group case is concluded in the same way, using that $\mks(X,\lcc P_J\rcc)=\mks(X,P_{I\smallsetminus J})\cap\mks(X,\lcc P_J\rcc)$.
\end{proof}

\begin{rem}
Actually, assuming that the arity function is not identically zero, it is enough to embed $\mks(X,\lcc P\rcc)$ into $\mks(\lcc P\rcc)$.
%$\mk{E}(X,\lcc P\rcc)$ into $\mk{E}(\lcc P\rcc)$ (and $\mks(X,\lcc P\rcc)$ into $\mks(\lcc P\rcc)$).
Indeed, choose $i$ with $a(i)\ge 1$ and write $P'=P\bigtriangleup\{\lceil x\rfloor_i\}$, where $\lceil x\rfloor_i=(x,\dots,x)\in \lceil X^{a(i)}\rfloor_i\subseteq \WW_\II(X)$. Then the stabilizer of $x$ is included in the intersection of the stabilizers of $P$ and $P'$. Hence the embedding into $\mks(X)$ is unnecessary in this case.
\end{rem}

Since $\mks(X,P)$ is an open subgroup of $\mks(X,\lcc P\rcc)$, whether the latter group is locally compact or discrete is therefore dictated by the trichotomy of Remark \ref{tricho}. Incidentally, this shows that this trichotomy only depends on $\lcc P\rcc$.

\begin{exe}\label{almex}
\begin{enumerate}
\item {\it (Houghton's groups)} Consider a tree $T_{2,k}$ consisting of one vertex of valency $k$ connected to $k$ diverging infinite rays, such that for each $i=1,\dots,k$, the vertices of the $i$th ray are labeled by $i$. The automorphism group of this graph (i.e., of this structure) is trivial. The almost automorphism group of this graph is thus discrete, and is known as Houghton's group $H_k$ (the interpretation of $H_k$ as particular instance of an almost automorphism group is an observation in \cite{LehS}). Its quotient by finitely supported permutations is isomorphic to $\Z^{k-1}$, and actually is canonically isomorphic to the set of $k$-tuples in $\Z^k$ with sum 0. If we forget the labeling, the resulting automorphism group is naturally a semidirect product $H_k\rtimes\mks_k$.
\item\label{almparti} Consider a set $X$ with a partition $(X_i)_{i\in I}$ into finite subsets. Consider the structure encoding this partition and the labeling of elements of $X_i$ by $i$. This is encoded in the structure $P$ with labeling set $I$, $a(i)=1$ for all $i$ and $P_i=X_i$; note that $P$ is multiproper.

Thus, the automorphism group is the product $\prod_{i\in I}\mks(X_i)$. The almost automorphism group is generated by it along with the normal subgroup finitely supported permutations (note that in general the intersection is both infinite and of infinite index in $\mks_{<\aleph_0}(X)$). Up to this terminology, this group and its topology were independently introduced by Akin-Glasner-Weiss and Willis \cite[\S 3]{Wi} (see also \cite[\S 5]{CC}).
\end{enumerate}
\end{exe}

%\begin{rem}
%
%Define the almost endomorphism monoid\index{almost endomorphism monoid} $\mk{E}(X,\lcc P\rcc)$ of $P$ (or of $\lcc P\rcc$) as the stabilizer of $\lcc P\rcc$ in the monoid $\mkrm(X)=X^X$ of self-maps of $X$. 
%\end{rem}

%Graphs consisting of infinite disjoint unions of finite complete graphs are also considered in a close setting in several papers, notably \cite[\S 3]{Wi} (see \cite[\S 5]{CC} where the formulation makes the link more explicit).suppose that $P$ is the graph of a cofinite-partial permutation $\sigma$ of $X$. That is,but finitely many vertices have exactly one edge inwards, and one edge outwards.

Let us now mention how this extends beyond the group setting.

\begin{defn}
Let $(X,P)$, $(Y,Q)$ be $\II$-structures. Then $f:X\to Y$ is called an almost homomorphism of $\II$-structures if $P\substar f^*Q$. This is equivalent to each of the following:
\begin{itemize}
\item for every $i\in I$, one has the near inclusion $P_i\substar f_i^*Q_i$, and for all but finitely many $i$, one has the inclusion $P_i\subseteq f_i^*Q_i$.
\item there exists a finite subset $F_i$ of $X^{a(i)}$, with $F_i$ nonempty for finitely many $i$, such that for every $i\in I$, the map $f^{\times a(i)}$ maps $P_i\smallsetminus F_i$ into  $Q_i$.
\end{itemize}
The set of almost homomorphisms from $(X,P)$ to $(Y,Q)$ is denoted by $\mkr(X,Y;\lcc P,Q\rcc)$.
\end{defn}

It is stable under compositions. Thus sets endowed with a $\II$-structure, with almost homomorphisms form a category $\mathrm{Alm}_{\II}$. 

\begin{defn}
The monoid $\mkr(X,X;\lcc P,P\rcc)$ of almost endomorphisms of the $\II$-structure $(X,P)$ is denoted by $\End(X,\lcc P\rcc)$, or $\mathrm{AEnd}(X,\lcc P\rcc)$, or $\mkr(X;\lcc P\rcc)$.
%, and its group of automorphisms is denoted by $\Aut(X,\lcc P\rcc)$ or $\mks(X,\lcc P\rcc)$, by definition it is the group of invertible elements in $\End(X,P)$. 
\end{defn}

We can endow these with natural topologies as well. First, for two posets $U,V$, endow the set $\mkrm(U,V)$ of maps $U\to V$ with a topology as follows. First, the poset $V$ is endowed with the topology for which the open subsets are the subsets $O$ such that $y\in O$, $x\le y$ implies $x\in O$. (This is not Hausdorff unless $V$ is a discrete poset). Thus, a net $(x_i)$ converges to $x$ if and only if $x_i\le x$ for large $i$. Then the set of maps $\mkrm{U,V}=V^U$ is endowed with the product topology (topology of pointwise convergence). Now let $\mkrm_{\le}(U,V)$ be the subset of $\le$-preserving maps, that is, satisfying $x\le x'$ implies $f(x)\le f(x')$. Then for three posets $U,V,W$, the composition map $\mkrm_{\le}(U,V)\times\mkrm_{\le}(V,W)\to \mkrm_{\le}(U,W)$ is continuous.

Next, let $(X,P)$, $(Y,Q)$ be $\II$-structures. We apply this to $U=\lcc Q\rcc$ (subposet of the power set $\St_\II(Y)$ of $\WW_\II(Y)$), and $V$ the power set $\St_\II(X)$ of $\WW_\II(X)$. We embed diagonally $\mkrm(X,Y;\lcc P,Q\rcc)$ into $\mkrm(X,Y)\times \mkrm(\lcc Q\rcc,\St_\II(X))_{\mathrm{op}}$. Here $\mkrm(X,Y)=Y^X$ is endowed with the product topology of the discrete topology on $Y$, while $\mkrm(\lcc Q\rcc,\St_\II(X))$ is endowed with the above product of topologies of the poset topologies. Here the suffix ``op" means that the latter is endowed with the opposite arrows. Hence, for the resulting topologies composition between the sets $\mkrm(X,Y;\lcc P,Q\rcc)$ is continuous. In particular, it makes $\mkrm(X;\lcc P\rcc)$ a topological monoid, and its subgroup of invertible elements (which is not closed, e.g., when $P$ is empty) is $\mks(X,\lcc P\rcc)$ with the previous topology.

\subsection{Near automorphism groups of multiproper structures}\label{s_products}

We say that a self-map $f$ of a set $X$ is finitary\index{finitary map} if $f\sim\mathrm{id}_X$ or equivalently $f\approx\mathrm{id}_X$. This means that $f$ coincides with the identity map outside a finite subset.

An issue with the notion of near actions is that the cartesian product of two near actions is not a near action. This comes from the fact that given two finitary self-maps $f_i$ of sets $X_i$, $i=1,2$, the cartesian product self-map $f_1\times f_2$ of $X_1\times X_2$ is not finitary in general. Namely, if $f_i$ is the identity outside a finite subset $F_i$, then $f_1\times f_2$ is the identity outside the subset $(F_1\times X_2)\times (X_1\times F_2)$. We call such a subset a mullion\index{mullion}. Let us introduce them in the more general setting of \S\ref{stronset}.

%the cartesian product $f_1\times f_2$ of two finitary self-maps (say, on sets $X_1,X_2$, identity outside finite subsets $F_1$ and $F_2$) is not finitary.We continue with the notation of \S\ref{stronset}. 

Given a set $X$ and a signature $\II=(I,a)$, we say that a subset $P\subseteq \WW_{\II}(X)$ is a mullion if $P=P_{[F]}$ for some finite subset $F$ (see \S\ref{stronset} for the definition of $P_{[F]}$). The set of mullions forms an ideal $\mathcal{M}_\II(X)$ (or $\mathcal{M}(X)$) of the Boolean subalgebra $\St_\II(X)$, whose intersection with the ideal of multiproper subsets is reduced to the ideal of finite subsets of $\WW_{\II}(X)$.

If $f:X\to Y$ is a proper map, then the inverse image map of the induced map $f_\circ:\WW_\II(X)\to\WW_\II(Y)$, namely the Boolean algebra homomorphism $\St_\II(Y)\to\St_\II(X)$, maps the ideal of mullions $\mathcal{M}(Y)$ into the ideal of mullions $\mathcal{M}(X)$. Moreover, if $f$ is finitary, then $f^\ast$ is the identity modulo mullions.

Therefore, the functor $X\mapsto\WW_\II(X)/\mathcal{M}(X)$ from the category $\overline{\mk{P}}$ of (nonempty) sets with proper maps to the category of Boolean algebras, maps finitary maps to identities. Accordingly, 
by Proposition \ref{canosur0}(\ref{canosur3}), it induces a contravariant functor from the category $\mk{P}^\star$ of sets with proper near maps, to the category of Boolean algebras. In particular, for every set $X$, we obtain a natural (contravariant/right) action of the monoid on the Boolean algebra $\WW_\II(X)/\mathcal{M}(X)$, and hence an action of $\mkst(X)$.

%{canosur4}

We denote by $\lclc P\rcrc$\index{$\lclc P\rcrc$ ($P$ modulo mullions)} the class of a structure $P\subseteq \mathcal{W}_\II(X)$ modulo mullions. 

\begin{defn}
Given a structure $(X,P)$, the stabilizer of $\lclc P\rcrc$ in $\mkst(X)$ is called the near automorphism group\index{near automorphism group} of $P$ (or of $\lclc P\rcrc$). It is denoted $\mkst(X,P)$, or $\mkst(X,\lclc P\rcrc)$, or $\NAut(X,\lclc P\rcrc)$.
\end{defn}

%Given a structure $(X,P)$, the stabilizer of $\lclc P\rcrc$ in $\mkst(X)$ (resp.\ $\mk{E}(X)$) is called the near automorphism group\index{near automorphism group}\index{near endomorphism monoid} (resp.\ near endomorphism monoid) of $P$ (or of $\lclc P\rcrc$). It is denoted $\mkst(X,P)$, or $\mkst(X,\lclc P\rcrc)$, or $\NAut(X,\lclc P\rcrc)$ (resp.\ $\mk{E}(X,P)$, or $\mk{E}(X,\lclc P\rcrc)$). 

Note that the near automorphism group is the group of invertible elements of the near endomorphism group. Also note that if $P$ is multiproper, its near endomorphism group is the same as the stabilizer of $\lcc P\rcc$.

Assume that $P$ is multiproper.
There is a canonical homomorphism $\mks(X,\lcc P\rcc)\to\mkst(X,\lcc P\rcc)$. Its kernel is the subgroup $\mks_{<\aleph_0}(X)$ of finitely supported permutations. Its image is the subgroup $\mkst_0(X,P)$ of near automorphism with index 0.

The topology on the almost near automorphism group, introduced in Proposition \ref{uniqtopo}, passes to the quotient to a (possibly non-Hausdorff) group topology on $\mkst_0(X,P)$. However, to extend it to the full near automorphism group, when the index is nonzero, requires an additional argument.

% ------>Laisser
%
%
%VERSION MONOIDE??? BOF. on obtient sur NEnd(X,\lcc P\rcc) une topologie (sans besoin de l'astuce!), qui se restreint a NAut... mais en une topologie a priori plus faible.
% 
% exemple T arbre regulier trivalent defini par rel d'incidence dans T^2
% incluant la diagonale
% ainsi End(T) contient plein d'elements a support fini, par exemple la retraction sur un sous-arbre.
% et tout automorphisme avec point fixe est limite de tels trucs

\begin{prop}\label{toponear}
Let $P$ be a multiproper structure on $X$. There is a unique group topology on $\mkst(X,\lcc P\rcc)$ making $\mks(X,P)/\mks_{<\aleph_0}(G,P)$ an open subgroup. On $\mkst_0(X,\lcc P\rcc)$, it coincides with the quotient topology of $\mks(X,\lcc P\rcc)/\mks_{<\aleph_0}(G,\lcc P\rcc)$. The topological group $\mkst(X,\lcc P\rcc)$ is Hausdorff if and only if $(X,P)$ is fillable. If $X$ is countable, then the Hausdorff quotient $\mkst(X,\lcc P\rcc)/\overline{\{1\}}$ is Polish.
%If $(X,P)$ is weakly rigid then this topology does exist.
\end{prop}
\begin{proof}By uniqueness of the extension of the topology, the second statement follows from the first. We start proving the first statement for $\mksh(X,\lcc P\rcc)$.

Indeed, by Proposition \ref{uniqtopo}, there is a topology on $\mks(X,\lcc P\rcc)$ with an open inclusion $\mks(X,P)\subseteq \mks(X,\lcc P\rcc)$. Passing to the quotient, it induces an open inclusion $\mks(X,P)/\mks_{<\aleph_0}(G,\lcc P\rcc)\subseteq \mksh(X,\lcc P\rcc)$.

Attempts to adapt the proof of Proposition \ref{uniqtopo} to prove the general case here leads to a significant amount of technical complications to deal with elements of nonzero index. 

Instead, let us avoid this by using the Hilbert hotel join (Remark \ref{hilbho}). Embed $X$ into $Y=X\sqcup \N$. 
Assuming that $P\subseteq M_X\times I$, define $I'=I\sqcup\{w\}$ for some additional parameter $w\notin I$ with arity 2, and define $P'=P\sqcup(Q\times\{w\})$, where $Q=\{(n,n+1):n\in\N\}$. So $Q$ defines the standard oriented graph structure on $\N$ and $Q\times\{w\}$ distinguishes $\N$ inside $Y$.

%Assuming that $P\subseteq X^d\times\N$ with $d\ge 1$, define $I'=I\sqcup\{w\}$ (for some additional parameter $w\notin I$), and define $P'=P\cup(Q\times\{w\})$, where $Q=\{(n,n+1):n\in\N\}$. So $Q$ defines the standard oriented graph structure on $\N$ and $Q\times\{w\}$ distinguishes $\N$ inside $Y$.

%=P\cup (\mathrm{diag}(X)\times\{w\})\cup Z\times I'$.

First note that $\mks(Y,P')$ stabilizes $X$ and $\N$, is reduced to $\mks(X)$ (since $\mks(\N,Q)$ is trivial). In particular, the topology on $\mks(Z,P)$ coincides with to the topology on $\mks(X,P)$. Therefore, the topology on $\mks(Z,P')/\mks_{<\aleph_0}(Z,P')$ coincides with the topology on $\mks(X,P)/\mks_{<\aleph_0}(X,P)$.

The group $\mkst(\N,\lcc Q\rcc)$ is infinite cyclic; let $s$ be its unique element of index 1 (that is, represented by $n\mapsto n+1$). Define an embedding $f$ of $\mkst(X,\lcc P\rcc)$ into $\mksh(Z,P')$ by $f(g)=gs^{-\phi_X(g)}$, where $\phi_X$ is the index on $X$. This is an injective group homomorphism. In restriction to $\mksh(X,\lcc P\rcc)$ and hence to $\mks(X,P)/\mks_{<\aleph_0}(X,P)$, it is just the plain embedding, induced by the inclusion $\mksh(X)\subseteq \mksh(Z)$ (extending by the identity outside $X$). Hence, this defines on $\mkst(X)$ a group topology whose restriction to $\mks(X,P)/\mks_{<\aleph_0}(X,P)$ is its original topology. This proves the existence assertion.

The topology Hausdorff it and only if its open subgroup $\mks(X,P)/\mks_{<\aleph_0}(X)$ is Hausdorff, which holds if and only if $\mks_{<\aleph_0}(X,P)$ is closed. By Corollary \ref{bershe}, this holds if and only if $\mks_{<\aleph_0}(X,P)$ is finite, i.e., if $(X,P)$ is fillable.

Quotients of Polish groups by closed normal subgroups are Polish (see \cite[1.2.3]{BK}). When $X$ is countable, $\mks(X,\lcc P\rcc)$ is Polish and this applies to obtain that its Hausdorff quotient $\mks(X,\lcc P\rcc)/\overline{\mks_{<\aleph_0}(X)}$ is also Polish; the latter can be identified with $\mksh(X,\lcc P\rcc)/\overline{\{1\}}$. In turn, since the latter is open of countable index (as kernel of the index character) in $\mkst(X,\lcc P\rcc)/\overline{\{1\}}$, we deduce that the latter is Polish as well.
\end{proof}

\begin{rem}
In the case of the near automorphism group of regular trees, this topology was indicated by Neretin \cite{Ne1}, although not giving details. Indeed, in this case (and more generally of trees with all vertices having valency $\ge 3$), the Neretin group is naturally a subgroup of the self-homeomorphism group of the boundary, and this can be used to make alternative definitions of the topology (albeit not the most naive one: it is not closed in the self-homeomorphism group).
A direct construction of the topology is done in \cite[\S 4]{LaGa}, proving by hand that the topology on $\mks(X,P)$ can be extended.
\end{rem}

%\begin{rem}
%One also obtains a monoid quotient topology on $\mk{E}(X,\lcc P\rcc)$, and we directly obtain, by restriction, a group topology on $\mkst(X,\lcc P\rcc)$. This can be weaker than the above-constructed topology. Indeed, in the case of an infinite leafless tree, $\mkst(X,\lcc P\rcc)$ is Hausdorff. On the other hand, the set of endomorphisms with finite image is dense, and hence the quotient topology on $\mke(X,\lcc P\rcc)$ is the indiscrete topology.\end{rem}

\begin{rem}
When $P$ is non-multiproper, we can define a topology on $\mks(X,\lclc P\rcrc)$ as the topology induced from the embedding into $\mks(\lclc P\rcrc)$. The set of subsets $\lclc P\rcrc$ being huge, this topology is, in general, by far too discrete, and does not induce the original topology on $\mks(X,P)$. 
\end{rem}

\begin{rem}
Proposition \ref{toponear} is false for non-multiproper structures, even at the level of $\mks(X)$. That is, if we define $\mks(X,\lclc P\rcrc)$ as the stabilizer of $\lclc P\rcrc$ in $\mks(X)$, then it is not always true that any element of $\mks(X,\lclc P\rcrc)$ restricts to a topological isomorphism between two open subgroups of $\mks(X,P)$. Indeed consider $X$ made of two sequences $(x_n)_{n\ge 0}$, $(y_n)_{n\ge 0}$ (injective with disjoint images). Choose $I=\{1,2\}$, arity function equal to 2, and $P_1=\{(x_0,x_n):n\ge 1\}$ and $P_2=\{(y_0,y_n):n\ge 1\}$. It thus encodes a an oriented graph structure, with two components, each being a tree with one central vertex and edges to all other vertices, and the two colors distinguishing the two trees. Hence the automorphism group identifies to $\mks(\N_{\ge 1})^2$ (acting on the leaves of each tree). Since $P$ is a mullion, $\mks(X,\lclc P\rcrc)$ is equal to $\mks(X)$. Let $f$ be the cycle
\[\dots\mapsto x_{2n}\mapsto x_{2n-2}\mapsto\dots\mapsto x_4\mapsto x_2 \;\;\mapsto\;\;  y_4\mapsto x_5\mapsto y_6\mapsto x_7\mapsto\dots\]
For every $n$, we can find $u\in\mks(X,P)$ that fixes $x_i$ and $y_i$ for all $i\le 2n$, and which, for some $m>n$, satisfies $u(x_{2m+2})=x_{2m-1}$. Since $x_{2m+2}=f^{-1}(x_{2m})$ and $x_{2m-1}=f^{-1}(y_{2m})$, this yields $uf^{-1}(x_{2m})=f^{-1}(y_{2m})$, that is, $fuf^{-1}(x_{2m})=y_{2m}$ and hence $fuf^{-1}\notin \mks(X,P)$. This proves that for every open subgroup $U$ of $\mks(X,P)$, $fUf^{-1}$ is not included in $\mks(X,P)$.
\end{rem}

%in the spirit of Proposition \ref{uniqtopo} is done in \cite[\S 4]{LaGa}

%, however there are suspicious details, notably about the meaning of $g^{-1}(o)$ and $h(o)$ before Lemma 4.3 ($o$ is a root, 

%a closed subgroup of the self-homeomorphism group

% of its boundary, and hence carries  

%For a proper subset $P$ of $X^d$, denote by $\lcc P\rcc$ its class modulo finite perturbation. If $P$ is a proper subset of $X^2$, it can be thought as an oriented graph structure on $X$ (considering symmetric subsets encompasses the case of non-oriented graphs). Then the stabilizer $\NAut(X,\lcc P\rcc)$ of $\lcc P\rcc$ is known as the near automorphism group of the given graph.

%This latter definition encompasses, notably, the definition of the near automorphism group $\mkst_G(X)$ when $X$ is a near $G$-set: indeed, one can consider the family $(E_g)_{g\in G}$, where $E_g\subseteq G^2$ is the graph of some representative of $g$; if $G$ is finitely generated, it is enough to restrict to a finite subfamily indexed by a generating subset. In the sequel, we restrict to a single subset $E\subseteq G^2$ as initially introduced; this already encompasses the case of near automorphism groups of near $\Z$-sets, whose study was initiated in \cite{ACM}, and is considered in \S\ref{nearautz}. 

%AGENCEMENT DE CET EXEMPLE DOUTEUX: ANTICIPATION SUR SECTION QUI SUIT, ETC

\begin{exe}\label{exstrune}In all examples below, when $I$ can be chosen to be a singleton $\{i\}$, we simply view the structure $P$ as a subset of $X^{a(i)}$ (where $a(i)$ is mostly equal to 2).
\begin{enumerate}
\item\label{grcentrlz} Suppose that $P$ defines an oriented locally finite graph in which all vertices but finitely many have a single inwards edge and a single outwards edge. That is, some finite perturbation of $P$ is the graph of a cofinite-partial permutation $\sigma$ of $X$. Then $\mkst(X,\lcc P\rcc)$ is the centralizer of the corresponding near permutation in $\mkst(X)$. The study (rather in $\mksh(X)$) of such centralizers was initiated in \cite{ACM}, and is extended in \S\ref{nearautz} (notably relying on results of \S\ref{s_equivariant}).

\item\label{neargr} Generalizing (\ref{grcentrlz}), consider a group generated by a finite subset $S$ and acting (resp.\ near acting) on a set $X$. Use $S$ as labeling set with arity function equal to 2; define $P_s=\{(x,sx):s\in X\}$, and $P=P_S=\bigcup_{s\in S}P_s\times\{s\}$. This is a multiproper structure. In the case of an action, $P_s$ is well-defined; in the case of a near action, only $\lcc P_s\rcc$ is well-defined and $P_s$ depends on a choice of a representative for $s$.

Then $\mkst(X,\lcc P\rcc)$ is the group of automorphisms of the near $G$-set $X$; thus we view it here as the near automorphism group of the corresponding Schreier graph. for instance, when $X=G$ is 1-ended, this near automorphism group is reduced to $G$ through its right action on itself, by Corollary \ref{ghx}.
(For $G$ not finitely generated, see Remark \ref{notsame}.)

\item\label{nearmon} (\ref{neargr}) immediately extends to the case of a monoid $G$ generating by a finite subset $S$ and near acting on $X$ by proper near maps (i.e., endowed with a monoid homomorphism $\alpha:G\to\mk{E}(X)$. Define $P$ exactly in the same way: then $\mkst(X,\lcc P\rcc)$ is the centralizer of $\alpha(G)$ in $\mkst(X)$.

In the case when $G$ is monogenic and $S$ is a singleton $\{s\}$, $P$ is an oriented forest structure (of finite valency, with exactly one outwards edge at all but finitely many vertices). See Example \ref{exner}(\ref{exprose}).

\item\label{iari1} Let $P$ be a structure with labeling set $I$ and arity 1 on all of $I$. If $I$ is finite, then $P$ is automatically multiproper and $\mkst(X,\lcc P\rcc)$ is, by definition, the pointwise stabilizer of a certain finite subset $A$ of the Boolean algebra $\mathcal{P}^\star(X)$; we can suppose that $A$ is a Boolean subalgebra. Note that $A$ corresponds to a finite partition $X=\bigsqcup_{i=1}^nX_i$ of $X$ into infinite subsets, defined up to near equality, and that $\mkst(X,\lcc P\rcc)$ is the subgroup of $\mkst(X)$ commensurating each $X_i$, and thus is naturally isomorphic to $\prod_{i=1}^n\mkst(X_i)$.

%if the arity $i$, we obtain the pointwise stabilizer of a certain subset, and hence of a certain Boolean subalgebra $A$ of $\mathcal{P}^\star(X)$. If $A$ is finite, we obtain a direct product of finitely many near permutation groups.

%\item 
%When $I$ is infinite, multiproperness of $P$ amounts to assuming that 

Note that $P$ is multiproper if and only if each $x\in X$ belongs to $P_i$ for finitely many $i$.
 This is notably the case if we have a partition $X=\bigsqcup_{i\in I} X_i$ and $X_i=P_i$. This case was already considered in Example \ref{almex}(\ref{almparti}).

\item Continue with the setting of (\ref{iari1}), and suppose that we have a partition $X=\bigsqcup_{i\in I} X_i$ and $X_i=P_i$. (This case was already considered in Example \ref{almex}(\ref{almparti})). By the previous remark, $P$ is multiproper. In this case, let us describe more precisely $\mkst(X,P)$.

% that is, suppose that $(P_i)_{i\in I}$ forms a partition. 

We need the terminology of near products from \S\ref{nearwreath}.
First, suppose that all $X_i$ are finite. Then we have a natural isomorphism $\mkst(X,\lcc P\rcc)=\prod^\ast_{i\in I}\mks(P_i)$. 

In general, we have two natural homomorphisms from $\mkst(X,\lcc P\rcc)$, namely one to $\prod^\star\mks(X_i)$, and one to $\prod\mkst(X_i)$. Both of the latter two groups have a canonical projection onto $\prod^\star\mkst(X_i)$, and the two composite homomorphisms $\mkst(X,\lcc P\rcc)\to\prod^\star\mkst(X_i)$ coincide. Thus the product homomorphism from $\mkst(X,\lcc P\rcc)$ to $\prod^\star\mks(X_i)\times\prod\mkst(X_i)$ maps into the fibre product over $\prod^\star\mkst(X_i)$. It is easy to check that this is an isomorphism from $\mkst(X,\lcc P\rcc)$ onto this fibre product. 

Note that the image into $\prod\mkst(X_i)$ is not surjective: the image is reduced to the subgroup $\prod^{(\phi)}\mkst(X_i)$ of elements $(g_i)$ in the product such that $\phi_{X_i}(g_i)=0$ for all but finitely many $i$.

%The intersection of their kernels is trivial and we thus obtain the fibre product. Note that the image into $\prod\mkst(X_i)$ is not surjective: the image is reduced to the subgroup $\prod^{(\phi)}\mkst(X_i)$ of elements $(g_i)$ in the product such that $\phi_{X_i}(g_i)=0$ for all but finitely many $i$.

\item Let $P\subseteq X^2$ be an equivalence relation with finite orbits. First assume that all orbits have cardinal $k$. Then the near automorphism group is the near wreath product $\mks_k\wr^\star\mkst(Y)$, where $Y=X/P$ is the quotient set (i.e., the set of orbits). In general, let $Y_k$ be the set of orbits of cardinal $k$. Then the near automorphism group has surjective homomorphisms into $\prod_k^{(\phi)}\mks_k\wr^\star\mkst(Y_k)$ and $\prod_k^\star\mks_k\wr\mks(Y_k)$; the intersection of kernels is trivial and thus we have the fibre product, relative to their common quotient $\prod^\star\mks_k\wr^\star\mksh(Y_k)$. 

A special case is when $Y_k$ is finite for all $k$ outside some finite set $M$: in this case we obtain $\prod_{k\in M}\mks_k\wr^\star\mkst(Y_k)\times\prod_{k\notin M}^\star\mks_k\wr\mks(Y_k)$.

\end{enumerate}
\end{exe}

%Suppose in addition that all $P_i$ are finite. Then we have a natural isomorphism from $\NAut(X=\mathcal{Z})=\AAut(X=\mathcal{Z})$ to the near product $\prod^\ast_{i\in I}\mks(P_i)$ (as defined in \S\ref{nearwreath}).

\begin{rem}\label{notsame}
Note that when $S$ is infinite, we can express the automorphism group of the near $G$-set $\mkst_G(X)$ as $\bigcap_{\mathrm{finite }S'\subseteq S}\mkst(X,\lcc P_{S'}\rcc)$, but this intersection is not expressible, at least in an obvious way, as near automorphism group of a structure. 

Also, when a partition $X=Y\sqcup Z$ is given, the subgroup of near automorphisms preserving this partition is not expressible, at least in an obvious way, as near automorphism group of a multiproper structure (it is the near automorphism group defined by the non-multiproper structure $P=Y^2\cup Z^2\subseteq X^2$). On the other hand, it can be viewed as the stabilizer of $\{Y,Z\}$ in the set of non-ordered pairs of subsets of $X$. This behaves like a multiproper structure, and suggests that the framework for defining ``near automorphism groups of a multiproper structure" should be wider than just the groups of the form $\mkst(X,\lcc P\rcc)$.
\end{rem}

Let us now mention how this extends beyond the group setting. Unlike the case of almost homomorphisms, we will not attempt to construct a topology in this generality.

\begin{defn}
Denote by $\subseteq^{\#}$ the inclusion modulo mullions. 

Let $(X,P)$, $(Y,Q)$ be $\II$-structures and $f\in\mkrst(X,Y)$. Then $f:X\to Y$ is called an near homomorphism of $\II$-structures if $P\subseteq^{\#} f^*Q$. 

The set of near homomorphisms from $(X,P)$ to $(Y,Q)$ is denoted by $\mkr(X,Y;\lcc P,Q\rcc^{\#})$
\end{defn}

It is stable under compositions. Thus sets endowed with a $\II$-structure, with near homomorphisms form a category $\mathrm{Near}_{\II}$, with a natural functor from the category $\mathrm{Alm}_{\II}$. A subcategory (with the same objects) is given by proper near homomorphisms, and a smaller subcategory is given by near isomorphisms.

\begin{defn}
The monoid $\mkrst(X,X;\lcc P,P\rcc^{\#})$ of almost endomorphisms of the $\II$-structure $(X,P)$ is denoted by $\End(X,\lcc P\rcc^{\#})$, or $\mathrm{NEnd}(X,P)$, or $\mkr(X;\lcc P\rcc^{\#})$ (the symbol $\#$ can be omitted when $P$ is multiproper).
\end{defn}

Its subgroup of invertible elements is the near automorphism group of $(X,P)$. An intermediate submonoid is given by proper near endomorphisms.

\subsection{Near automorphisms of trees}\label{nautree}

Groups of near automorphisms of trees have appeared independently: in the context of Thompson's groups, and in Neretin's work. They are now extensively studied, although essentially not in the angle of near actions (a notable exception being Kapoudjian's approach, see Remark \ref{kapou2}).
  
Let us denote the near automorphism group of a tree as $\NAut(T)$, rather than in the form $\mkst(T,\lcc P\rcc)$. We do not follow the terminology ``almost automorphism" which, in this tree context, is often used, because it is in conflict with a distinct and related older use of ``almost automorphism", which we have followed; see Remark \ref{aavsna}.

%(As we see in Example \ref{nautt}, several natural structures $P$ on the set $T$ lead to the same group.)

Let $T_{d,k}$ denote a rooted tree where the root has valency $k$ and all other vertices have valency $d+1$. 
%Note that $\NAut(T_{1,k})\simeq\Z^k)$, the index character being given by the sum of coordinates NON PRODUIT SEMIDIRECT.

\begin{exe}\label{exner}
Let us indicate that several naturally defined structures lead to near automorphism groups of trees.

\begin{enumerate}
%\item\label{ortree} {\it (Trees vs rooted trees)} Given an arbitrary tree, define a principal orientation as an orientation in which all edges point towards a given vertex $v$; let $P_v$ be the corresponding set of oriented edges. It is a standard fact that $P_v\bigtriangleup P_w$ is finite (of cardinal $2d(v,w)$) for any two vertices $v,w$. In particular, the class $\lcc P_v]$ does not depend on the choice of $v$, and thus is preserved by all near automorphisms of $T$. Thus, choosing a root (i.e., a principal orientation) does not matter when defining the almost and near automorphism groups. 

\item\label{centralself} {\it (Near automorphism group of trees as centralizers)} 
If $X$ is a set, $\mk{E}(X)$ the monoid of near proper self-maps of a set $X$ and $\mkst(X)$ its subgroup of invertible elements, namely the group of near permutations, near automorphism group of trees readily appear as centralizers of suitable elements of $\mk{E}(X)$.

This occurs as particular case of Example \ref{exstrune}(\ref{nearmon}) (which interpret centralizers of arbitrary finite subsets as near automorphism groups). Namely, suppose that $f$ is surjective, has a fixed point $v$ such that $\bigcup_{n\ge 0}f^{-1}(\{v\})=X$. This is precisely the condition that means that the graph of $f$ defines a tree endowed with a principal orientation. Conversely, to every tree $T$ of finite valency and choice $v$ of root, one can associate the self-map $f_v$ fixing $v$ and moving all vertices one step towards $v$; its class in $\mk{E}(T)$ does not depend on the choice of $v$ and its centralizer in $\mkst(T)$ is equal to $\NAut(T)$.

\item {\it (Centralizer of the shift)} 
An illustration of (\ref{centralself}) is when one considers the free semigroup $X$ on a finite alphabet $A$ on $d$ letters and $f$ is the ``shift" consisting in erasing the first letter ($f$ fixing the empty word). Since it defines the graph of a rooted tree ($T_{d,d}$), its centralizer in $\mkst(X)$ is isomorphic to $\NAut(T_{d,d})$.

%Say that an orientation on a graph is functional if at every vertex there is a single outwards edge (which can be a self-loop); call it a functional oriented graph. Call it near functional if at all vertices, there are finitely many outwards edges, and at all but finitely many vertices, there are exactly one outwards edge: this means that the set of oriented edges is a finite perturbation of a functional oriented graph. 

\item\label{exprose} More generally, consider a proper self-map $f$ of a set $X$ and suppose that it defines a graph $\mathcal{G}_f\subseteq X^2$ with finitely many components (that is, there are only finitely many subsets $Y$ with $f^{-1}(Y)=Y$). We claim that its centralizer in $\mkst(X)$ can be identified to the stabilizer of some finite subset of the boundary in $\NAut(T)$, for some locally finite tree $T$ (which is obtained from a finite  perturbation of $\mathcal{G}_f$).

Say that $x\in X$ is forward-bounded if $\{f^n(x):n\ge 0\}$ is finite: this only depends on the component of $x$. Every forward-bounded component has a unique minimal non-empty $f$-stable subset, on which $f$ acts as a cycle, called minimal cycle. Let $Z$ be the (finite) union of all these minimal cycles. Choose a point $z_0\in Z$ and perturb $f$ as $f'$ defined by $f'(z)=z_0$ for all $z\in Z$ and $f'$ coincides with $f$ elsewhere. Then for $f'$, there is a single forward-bounded component, and it has a fixed point. Then the graph of $f'$ defines a forest structure on $X$; let us add finitely many oriented edges to obtain a tree $T$ (also removing the self-loop at $z_0$). The resulting oriented tree has the property that every vertex has finitely many outwards edges, and all but finitely many have exactly one outwards edge (call this a near inwards orientation). Those boundary points of $T$ with a representative that is an oriented ray form a finite subset $F$ of $\partial T$. Obviously $F$ is canonically determined by the near equivalence class of $f$, and for a near automorphism of the given graph with forgotten orientation, the following three properties are equivalent: centralizing $f$, near preserving the orientation, and preserving $F$.

%At the level of a forest, we can define, for a subset $F$ of the boundary, a near $F$-principal orientation as near inwards orientation in which the limits of oriented rays are precisely the elements of $F$. These exist if and only if $F$ meets the boundary of each component in a finite subset, and in each component they are unique up to finite perturbation. 

\item\label{ortree} Given an arbitrary tree, define an inwards orientation as an orientation in which all edges point towards a given vertex $v$; let $P_v$ be the corresponding set of oriented edges. It is a standard fact that $P_v\bigtriangleup P_w$ is finite (of cardinal $2d(v,w)$) for any two vertices $v,w$. In particular, the class $\lcc P_v\rcc$ does not depend on the choice of $v$, and thus is preserved by all near automorphisms of $T$. Thus, choosing a root (i.e., an inwards orientation) does not matter when defining the almost and near automorphism groups. 
%
%$T$ as above, choose a vertex $v_0$ and define $P'_{T,v_0}$ as the set of pairs of adjacent vertices pointing towards $v_0$. Then $[P'_{T,v_0}]$ does not depend on $v_0$, and in particular its stabilizer is also $\NAut(T)$. Thus, in the ``near" world, each tree has a canonical (near) orientation. 

\item\label{bumo1} Fix a set $I=$ of cardinal $d+1$. Consider an action of the free product $\Gamma_I=\langle (t_i)_{i\in I}\mid t_i^2=1\rangle$ of $d+1$ cyclic groups of order 2, on a set $T$. When the $\Gamma_i$-action is simply transitive, we can think of $T$ as a non-oriented tree, where each edge is colored in $T$ and each vertex meets exactly one edge of each color. Define $t:T\to T^I$ by $t(x)=(t_i(x))_{i\in I}$.
Let $F$ be a finite subgroup of $\mks(I)$. Letting $\mks(I)$ act on $T^I$ by switching coordinates, define
\[P_F=\Big\{\Big(x,\sigma\cdot \big(t(x)\big)\Big):x\in T,\sigma\in F\Big\}\subseteq T\times T^I=T^{2+d}.\] 
When the $\Gamma_I$-action on $T$ is simply transitive (so $T\simeq T_{d,d+1}$), $\mks(T,P_F)$ is known as the Burger-Mozes\index{Burger-Mozes group} group associated to $F$, often denoted $U(F)$ (where, out of context, $F$ should be considered as a permutation group, i.e., a finite group endowed with a faithful action on a finite set). Informally, it means the group of automorphisms of $T$ that locally at each vertex permutes the colors according as an element of $F$.

The near automorphism group $\mkst(T,\lcc P_{F}\rcc)$ is considered by Lederle in \cite{Led}. Note that in this context, the definition naturally fits to a near action of $\Gamma_I$ on $T$ (which is automatically realizable then); a reasonable restricting hypothesis is that $\Gamma_I$ acts near freely and that the near Schreier graph $T$ is near isomorphic to a tree. 

For an inclusion of subgroups $F\subseteq F'$, the group $\mks(T,P_{F'})\cap\mks(T,\lcc P_F\rcc)$ is considered (and denoted $G(F,F')$) by Le Boudec in \cite{lb2,lb3}, with its topology given as closed subgroup of $\mks(T,\lcc P_F\rcc)$.

\item There is a natural generalization of Burger-Mozes groups to pairs of permutation groups, as observed by S.\ Smith \cite{Sm}. See \S\ref{bumosmi}.
\end{enumerate}
\end{exe}

%an orientation such that all vertices have finitely many outwards edges, and all but finitely many 

%We refer to Remark \ref{nautree_his} for some more detail on this.
 
%If $T$ is a tree of finite valency, an observation of Neretin \cite{Ne1} is that $\NAut(T)$ can be made a topological locally compact group, for which the subgroup of automorphisms (with its natural locally compact topology of pointwise convergence) is an open subgroup. 

% In view of Proposition \ref{partcompletable}, this shows that there are locally finite graphs whose near automorphism group  cannot be defined by a cofinite-partial action: indeed, start from any non-completable action of $\Z^2$: since it is abelian, it defines a near action by near automorphisms on the corresponding Schreier graph (for instance, the one depicted in Figure (\ref{fig1}) in the introduction, forgetting the labeling and orientation if one wishes). This is in sharp contrast with the case of trees (see \{}Proposition \ref{nauttcompleta} and the preceding 

There exist connected locally finite graphs whose near action of the near automorphism group is not completable. Indeed, start from any non-completable, $\star$-faithful near action of $\Z^2$: since it is abelian, it thus embeds as a group of near automorphisms on the corresponding near Schreier graph (for instance, the one depicted in Figure (\ref{fig1}) in the introduction, forgetting the labeling and orientation if one wishes). Nevertheless, for trees we have:
  
\begin{thm}\label{tree_com}
For every tree $T$ of finite valency, the near action of $\NAut(T)$ on $T$ is completable.
\end{thm}
\begin{proof}
Define a near ray as a partially defined function $r:\Z\to T$, whose domain of definition has finite symmetric difference with $\N$, and such that $(r_{n+1},r_n)$ is an oriented edge for all $n$ large enough. Call it a $v$-ray if the domain of definition has the form $\Z_{\ge k}$ on which $r$ is a geodesic ray with $r(k)=v$. Consider the equivalence relation between near rays: $r\sim r'$ if $r_n=r'_n$ for all $n$ large enough, or equivalently if the graphs of $r$ and $r'$ have finite symmetric difference in $\Z\times T$; denote by $\lcc r\rcc$ the equivalence class of $r$. The resulting quotient set is called the graded boundary of $T$, and denoted by $\partial^\sharp(T)$. 

The group $\NAut(T)$ obviously acts on the graded boundary: this follows from the fact that the graph of a near ray is a biproper subset of $\Z\times\NAut(T)$. 

Fix a vertex $o$. Then each $\sim$-equivalence class of near rays contains a unique $o$-ray. For every vertex $v$ define $\Omega_v$ as the set of elements in $\partial^\sharp(T)$ whose unique $o$-ray representative $r$ satisfies $r_0=v$. The map $v\mapsto\Omega_v$ is clearly injective.

Let us show that it is near $\NAut(T)$-equivariant. For $\varphi\in\NAut(T)$, let $\tilde{\varphi}$ be an isomorphism between cofinite subsets $U_1,U_2$ of $T$, representing $\varphi$. Let $U'_1$ be the cofinite subset of those $v\in U_1$ such that the image of every $v$-ray is included in $U_1$. Then for every $v\in U'_1$, we have
\begin{equation}\varphi(\Omega_v)=\Omega_{\tilde{\varphi}(v)}.\label{tva}\end{equation}
To check this, consider an element $\lcc r\rcc$ of $\Omega_v$, for which we choose its unique $o$-ray representative $r$. So $r_{0}=v$. Then $(\tilde{\varphi}(r_{k}))_{k\ge 0}$ is a $\tilde{\varphi}(v)$-ray. Hence its class, which is $\varphi(\lcc r\rcc)$, belongs to $\Omega_{\tilde{\varphi}(v)}$. This show near equivariance.
\end{proof}

\begin{rem}When $T=T_{d,k}$, $d\ge 2$, there are known constructions amounting to realize the near action of $\NAut(T)$ on $T$ as a partial action, as defined in \S\ref{s_partact}, although not termed in this way (since Exel's notion of partial action was unknown to this community). Note that Theorem \ref{tree_com} mechanically provides a partial action (by the trivial implication of Proposition \ref{partcompletable}), without regularity assumptions, but the partial bijections are possibly not graph isomorphisms and we need to slightly refine the construction to obtain a partial action in which the partial bijections are graph isomorphisms.
\end{rem}
 
% has constant valency $d+1$ except possibly at one vertex (that is, $T\simeq T_{d,k}$ in the notation below) 

%Let us introduce some notation
%\begin{itemize}
%\item A bicoloring on a tree $T$ is a function $T\to\{0,1\}$ such that no incident vertices have the same image (color). If $T\neq\emptyset$, there are exactly two colorings on $T$; they are given by taking the distance to a fixed vertex modulo 2. A bicolored tree is a tree endowed with a bicoloring.
%\item $T_{d,k}$ denotes a rooted tree where the root has valency $k$ and all other vertices have valency $d+1$
%\item $T_{d,d',k}$ denotes a rooted bicolored tree where the root has valency $k$, all vertices of nonzero even (resp.\ odd) height have valency $d+1$ (resp.\ $d'+1$), and the root has color 0. (If $d=d'$ this is the bicolored version of $T_{d,k}$. If $d\neq d'$ the bicoloring still distinguishes .)
%\end{itemize}

\begin{prop}\label{dmi1}
Fix $d\ge 2$. Let $T_1,T_2$ be two trees in which for but finitely many vertices $v$, the value $\delta(v)$ of the valency of $v$ modulo $d$ is equal to 1. Write $\theta(T_i)=\sum_{v\in T}(\delta(v)-1)$. If $T_1$ and $T_2$ are near isomorphic, then $\theta(T_1)=\theta(T_2)$.

In particular, for $d\ge 2$, if $T_{d,k}$ and $T_{d,k'}$ are near isomorphic if and only if $k-k'$ divides $d-1$.
\end{prop}
\begin{proof}
If $T$ is a rooted tree, denote by $nT$ the forest obtained by taking $n$ disjoint copies of $T$. If $o$ is the root of $T$, let $T_v$ be the subtree of those vertices $w$ such that $v$ belongs to the segment $[0,w]$.

We start with the ``if" part of the last statement. Removing the root, we see that $T_{d,k}\simeq kT_{d,d}$. In particular, $T_{d,d}\simeq dT_{d,d}$ and hence hence, by obvious induction, $kT(d,d)$ and $k'T(d,d)$ ($k,k'\ge 1$) are near isomorphic as soon as $k-k'$ divides $d-1$.

To prove the first statement (and thus the ``only if" part of the second, which follows), we need some preparatory work.

Let $T$ be a rooted tree of finite valency and $A$ an abelian group. Say that $f:T\to A$ is upper-harmonic if for every vertex $v$ we have $f(v)=\sum f(v_i)$, where $(v_i)$ is the family of successors of $v$. Say that $f$ is near upper-harmonic if this holds for all but finitely many vertices. Let $\har(T,A)\subseteq \nhar(T,A)$ be the groups of upper-harmonic and near upper-harmonic functions. $\nhar(A)$ also includes the group $A^{(T)}$ of finitely supported functions, and we especially consider the corresponding quotient $\nhar'(T,A)$. An easy argument shows that $A^{(T)}\cap\nhar(A)=\{0\}$ (consider a vertex of maximal height in the support to get a contradiction). So this induces an inclusion $\har(T,A)\subseteq \nhar'(T,A)$, that is, every class of near upper-harmonic functions has at most one finitely supported representative.

To every upper-harmonic function $f$ we can canonically associate a function function $\partial f:\mathrm{clop}(T)\to A$: namely, for every vertex $v$, define $(\partial f)(\partial T_v)=f(v)$. By a routine argument, this extends canonically to an $A$-valued measure, in the sense that is is (finitely) additive on disjoint subsets.

We can extend this to near upper-harmonic functions $f$: namely $\partial f$ is the unique $A$-valued measure on $\partial T$ such that $(\partial f)(T_v)=f(v)$ for every $v$ such that $f$ is upper-harmonic on $T_v$. Then $f\mapsto\partial f$ is an additive map and vanishes on finitely supported functions; thus $\partial f$ only depends on the class of $f$ in $\nhar'(T,A)$. We define $\int f=(\partial f)(\partial T)$. Thus $\int$ is a homomorphism  from $\nhar'(T,A)$ to $A$; if $f$ is harmonic then $\int f=f(o)$, where $o$ is the root of $T$. 

%Define $\theta(T,A)=\int 1$, where 1 is the constant function (which is .

It is immediate from the definition that any isomorphism $u:T\to T'$ induces a group isomorphism $u_*:\nhar(T)\to\nhar(T')$, commuting with $\int$, that is, writing $\int_T$ and $\int_{T'}$ for clarity, we have $\int_{T'}\circ u_*=\int_T$.

First note that if we glue two consecutive vertices of $T_i$, we do not affect the value of $\theta(T_i)$. Hence we can suppose that $\delta$ equals 1 outside the root of $T_i$, for $i=1,2$, and in particular $\delta(T_i)$ is the valency $k_i$ of its root (modulo $d-1$) (this applies in particular to $T_{d,k}$).

Choose $A=\Z/(d-1)\Z$. Note that the the function $f_i:T_i\to A$ mapping the root $o$ to $k_i$ and other vertices to 1 is upper-harmonic. 

Let $u:T_1\to T_2$ be a near isomorphism. Denote by $\lcc \cdot\rcc$ the class in $\nhar'$. Then $u_*(\lcc f_1\rcc)=u_*(1)=1=\lcc f_{k'}\rcc$. Then, in $A$, we have
\[k_2=f_{2}(o)=\int \lcc f_{2}\rcc=\int u_*(\lcc f_{1}\rcc)\int \lcc f_{1}\rcc=f_{1}(o)=k_1.\qedhere\]
\end{proof}

\begin{rem}
Write $(k)\Z_p=\{1,\dots,k\}\times\Z_p$. Proposition \ref{dmi1} generalizes the fact that for $p$ prime and $k,k'\ge 1$, $(k)\Z_p$ and $(k')\Z_p$ are diffeomorphic as $p$-adic manifolds of class $C^1$ if and only if $k-k'$ divides $p-1$. Indeed, one can view $(k)\Z_p$ as the boundary of $T_{p,k}$ and any such diffeomorphism extends to a near isomorphism.

The phenomenon of ``invariance modulo $d-1$" also appears in \cite[p.53]{Br} in the closely related context of Thompson's groups.
\end{rem}

\begin{prop}\label{regindex0}
Suppose that the tree $T$ is eventually regular of valency $d+1\ge 3$ (there exists $d$ such that all vertices but finitely many have valency $d+1$). Then the index character of the near action of $\NAut(T)$ on $T$ is zero. That is, $\NAut(T)=\NAut_0(T)$.

The same conclusion holds when $T$ is eventually 
cyclically regular, i.e., for some choice of root, there exists an eventually periodic sequence $(d_n)$ with $\limsup d_n\ge 2$ and such that all vertices of large enough height $n$ have valency $d_n+1$.
%biregular of degrees $d+1\ge 2$, $d'+1\ge 3$, in the sense that all vertices but finitely many have valency $d+1$ or $d'+1$, and all but finitely many vertices of valency $d+1$ (resp.\ $d'+1$) has all its neighbors of valency $d'+1$ (resp.\ $d+1$).
\end{prop}
\begin{proof}
We start with the eventually regular case. Up to rearranging edges and removing finitely many vertices, we can suppose that only the root has degree $m\neq d$. Then for every finite subset $F$ of vertices containing the root, the forest obtained by removing $F$ consists of $m+d(|F|-1)$ regular rooted trees of degree $d$ (where the regular rooted tree of degree $d$ is understood to have the root has valency $d$ and all other vertices of degree $d+1$). Thus, since $m,d$ are fixed, its isomorphism type determines $|F|$. This shows that the index character is zero.

This immediately extends to the case of a forest with finitely many components and in which all but finitely many vertices have the same valency $\ge 3$ (since such a forest is near isomorphic to an eventually regular tree).

Now consider the case when $T$ is eventually cyclically regular. Let $k\ge 1$ be the (eventual) minimal period of $(d_n)$. Define a new graph structure on $T$ by joining any vertex of height $n$ to all its descendants of height $n+k$. For every $m\in\Z/k\Z$, let $T(m)$ be the set of vertices whose height modulo $k$ is $m$; these are the components of the new graph structure, which make it a forest with finitely many components, eventually regular of degree $d$, where $d-1=\prod_{j=0}^{k-1}(d_{j+q}-1)\ge 2$ for $q$ large enough.
Then one readily checks that $\NAut(T)$ commensurates each $T(k)$, and its action on $T(m)$ is by near automorphisms. Hence, by the eventually regular case, its near action on $T(m)$ has index 0. Summing over $m\in\Z/k\Z$, we deduce that the near action on $T$ has index 0.
\end{proof}

\begin{rem}
Note that the case of biregular trees is covered by Proposition \ref{regindex0} (where $d_n$ has eventual period 2).

The vanishing of the index character of $\NAut(T_{d,k})$ can be predicted by the use of simplicity results. For $d\ge 2$, the group $\NAut(T_{d,d+1})\simeq\NAut(T_{d,2})$ is known to be simple \cite{Ka99}, which directly implies the vanishing of the index character. Also, $\NAut(T_{d,k})$ is known to have a dense copy of the Higman-Thompson group $V_{d,k}$, which has a simple subgroup of index $\le 2$ and observing that the index character is continuous it follows that it is zero. In any case, the previous 6-line proof of the first assertion of Proposition \ref{regindex0} is much simpler than using these simplicity results.
\end{rem}

 %(this is observed, for instance in \cite{LB}),

\begin{rem}
There exists a tree $T$ of bounded valency $\ge 3$, such that the near action of $\NAut(T)$ on $T$ has nonzero index. Namely, let $T$ be the rooted tree which, viewed as a planar tree, has the property that every vertex of the left branch has 3 successors, while other have 2 successors. Another interpretation of this tree: all words in $\{a,b,c\}$ such that if the letter $c$ appears, then it is followed by only letters $a$; that is, the union of the set of all words $w(a,b)$ in $a,b$ and the union over all $n\ge 0$ of the set of all words of the form $w(a,b)ca^n$.

Write $o$ for the root. Removing either $\{o,a\}$ or $\{o,c,bc\}$ results in a forest consisting of 1 copy of $T$ and 4 copies of the regular binary rooted tree $T_{2,2}$. (In the first case, their roots are $a^2,ba,ca,b,c$; in the second case their roots are $a,b,ac,abc,bbc$.) Since in one case we remove 2 elements and 3 in the second case, any isomorphism between these two forests has an index value equal to 3-2=1.

Note that the trees obtained as union of $k\ge 1$ copies of $\N$ have a nonzero index character; namely the near automorphism group is isomorphic to the wreath product $\Z\wr\mks_k$ so that the index character is given by $(f,\sigma)\mapsto\sum_{i=1}^kf(i)$. For $k=2$ this is just the regular tree of valency 2.
\end{rem}

\begin{rem}\label{kapou1}
For $d\ge 2$, the near action of $\NAut(T_{d,d})$ was studied by Kapoudjian \cite{Ka}. He established that its Kapoudjian class (\S\ref{s_omega}) is nonzero. 

For $k\ge 1$, viewing $T_{d,d}$ as the subtree of $T_{d,k}$ rooted at some vertex of height 1 of the latter, the near action of $\NAut(T_{d,d})$, extended as the identity elsewhere, is a near action on $T_{d,k}$. Since the trivial near action has a zero Kapoudjian class and both are balanced, it follows from Proposition \ref{propomegaa}(\ref{omega2}) that the near action of $\NAut(T_{d,k})$ on $T_{d,k}$ has a nonzero Kapoudjian class for any $k\ge 1$.
\end{rem}

\subsection{Thompson's groups}

%Consider  

\begin{rem}[Continuation of Remark \ref{kapou1}]\label{kapou2} Kapoudjian \cite{Ka} established that the near action of the Higman-Thompson group $V_{d,2}$ on $T_{d,d}$ has its Kapoudjian class is nonzero if and only $d$ is odd.

In \cite[Remark 2.9]{KaS}, Kapoudjian and Sergiescu observe that the near action of Thompson's group of the circle ``$T$" on the set of vertices of the regular rooted binary tree $T_{2,2}$ is stably realizable, more precisely that it is realizable after adding one vertex; the resulting tree can be identified to $T_{2,1}$. Actually the argument adapts to Thompson's group $V$: its near action on $T_{2,1}$ is realizable. Let us explain it. Let $v_0$ be the root, $v_1$ the unique vertex of height 1, and for any $v\neq v_0$ let $v^+$ and $v^-$ be its two successors (in their natural order). Define a bijection $\iota$ (called labeling) from the set of vertices of $T_{2,1}$ to $\Z[1/2]/\Z$ (identified with $\Z[1/2]\cap [0,1\mathclose[$) as follows: $\iota(v_{0})=0$, $\iota(v_1)=1/2$, and then for every $v$ of height $n-1\ge 1$, $\iota(v^{\pm})=\iota(v)\pm 2^{-n}$. We readily see that the $2^{n-1}$ vertices of height $n\ge 1$ achieve all labels $m/2^n$, where $m$ ranges over odd integers in $[0,2^n]$; hence the labeling is indeed a bijection. The essential observation is then that the bijection $\iota$ near conjugates the near action of $V$ on $T_{2,1}$ to its near action on the set of dyadic numbers, viewed as a subset of the circle $\R/\Z$. The latter action is indeed realizable, since every element can be represented by its left-continuous representative.

Note that Thompson's group ``$T$" includes a quasi-cyclic subgroup $C_{2^\infty}$ acting freely on the set of dyadic numbers. By Corollary \ref{qcynr}, removing a point yields a non-realizable action. Hence, by the above, the near action of ``$T$" (and hence of Thompson's group $V_{2,2}$) on $T_{2,2}$ is not realizable: it is necessary to add a point to produce a realizable action.

For odd $d\ge 3$ such a construction cannot be carried out for the Higman-Thompson group $V_{d,2}$, since its near action on the set of vertices of the tree is not stably realizable, by the above-mentioned computation of its Kapoudjian class. Indeed, what always holds is that the near action of $V_{d,2}$ on the set of $d$-adic intervals (i.e., interval of the form $[md^{-n},(m+1)d^{-n}]$ is balanceably near isomorphic to the near action on the set of vertices of the rooted tree. For $d=2$, mapping a dyadic interval to its middle yields a near equivariant near bijection to the set of dyadic numbers modulo one (it is undefined on $[0,1]$ and misses the images $0$ and $1/2$, whence the additional vertex). For $d>2$ there is no such analogue. Or, we can try an analogue: for instance, for $d=3$, one can map any triadic interval $[m3^{-n},(m+1)3^{-n}]$ to its ``one-third point" $(3m+1)3^{-n-1}$. This is defined everywhere, near equivariant, but unlike the middle-point map when $d=2$, its image does not have finite complement: the image is the set of triadic numbers whose triadic expansion finishes with 1. 
Note that by construction, it is a $V_{3,2}$-commensurated subset.
\end{rem}

\subsection{Generalization of Burger-Mozes groups}\label{bumosmi}
Here we call graph a set $X$ endowed with a symmetric subset of $X^2$.

 A bimodal graph is just a graph endowed with a labeling of vertices (called bimoding) into $\{0,1\}$ (thus vertices with label 0, resp.\ 1, are called even, resp.\ odd), such that each edge involves one vertex of each parity. Note that any nonempty graph admits exactly zero or two bimodings, permuted by post-composition by the transposition $0\leftrightarrow 1$: indeed, given an even vertex $v$, the parity of other vertices is given by the distance to $v$ modulo 2. A bimoding exists in case of a tree, and more generally if and only if there are no loops of odd size (in particular, this prevents self-loops); the graph is then called bipartite. (``Bimodal" is sometimes called ``bicolored", but this would make the sequel confusing.) An isomorphism between bimodal graphs is called bimode-preserving if it conjugates the bimodings; for an automorphism of a bipartite graph, this does not depend on the choice of the bimoding and defines a canonical subgroup of index $\le 2$.

Let $A,B$ be nonempty sets. An $(A,B)$-colored graph is a bimodal graph endowed with a labeling of vertices into the disjoint union $A\sqcup B$, such that even (resp.\ odd) vertices are labeled by $A$ (resp.\ $B$). In an $(A,B)$-colored graph, say that an even (resp.\ odd) vertex is regular if the labeling yields a bijection of the set $S(v)$ of neighbors of $v$ onto $B$ (resp.\ $A$); this bijection is denoted by $b_v$ for $v$ even and $a_v$ for $v$ odd.

Define an $(A,B)$-colored graph to be
\begin{itemize}
\item $(A,B)$-biregular if every vertex is regular;
\item near $(A,B)$-biregular if all but finitely many vertices are regular;
\item rooted $(A,B)$-biregular if it is endowed with a root, and all non-root vertices are regular.
\end{itemize}

There is a unique nonempty $(A,B)$-regular tree (up to isomorphism of $(A,B)$-colored graph), and we usually denote by $T_{A,B}$ such an $(A,B)$-colored tree. Indeed, one easy way to construct $T_{A,B}$ is to consider the universal covering of the bipartite graph $A\sqcup B$. This also yields a way to prove its uniqueness: if a tree $T$ is $(A,B)$-bipartite, and $G$ is its automorphism group as a labeled graph, one readily checks that $G$ acts freely and transitively, and hence the quotient by $G$ is a bipartite $(A,B)$-graph, covered by $T$.

%Define a regular graph homomorphism between two graphs (without self-loops) as a map $f$ between their sets of vertices such that for for any vertices 

Say that a map $f$ between two graphs is pre-regular at a vertex $v$ if induces a bijection from $S(v)$ into $S(f(v))$. 

Consider two finite faithful permutations groups $F,F'$, given with their actions on the sets $A,B$. Let $W$ be another group $W$ endowed with homomorphisms $F\stackrel{u}\to W\stackrel{u'}\leftarrow F$. (Since all that will matter is the fibre product $F\times_W F'=\{(s,s')\in F\times F':u(s)=u'(s')\}$, it is essentially not a restriction to assume that both $u$ and $u'$ are surjective. But we do not make this assumption, because typically one could fix $F,F'$ and let $W$ vary, or have fibre products that are simpler to express with non-surjective structural homomorphisms.)

Given two $(A,B)$-colored graphs $X,Y$ and a map $f:X\to Y$, let us say that an edge $(v,w)$, with $v$ even and $w$ odd, is $(F,F',W)$-regular for $f$ if 
%an even (resp.\ odd) vertex $v\in X$ is $(F,F')$-regular for $f$ if
\begin{itemize}
\item $f$ is pre-regular at $v$ and $w$;
\item all vertices $v,w,f(v),f(w)$ are $(A,B)$-regular;
\item denoting by $\kappa'_{f,v}$ the permutation $b_{f(v)}\circ f\circ b_v^{-1}$ of $B$ (which is well-defined by the previous two conditions) belongs to $F'$;
\item similarly denoting by $\kappa_{f,w}$ the permutation $a_{f(w)}\circ f\circ a_w^{-1}$ of $A$ belongs to $F$;
\item $u(\kappa_{f,w})=u'(\kappa'_{f,v})$.
\end{itemize}

Note that when $W=\{1\}$, the last condition is superfluous and when moreover $F=\mks(A)$, $F'=\mks(B)$ the last two conditions are superfluous.

We say that $f$ is $(F,F',W)$-regular if every edge is regular for $f$. If $f$ is bijective and $(F,F',W)$-regular, it is straightforward that its inverse is also $(F,F',W)$-regular. Let $F\boxtimes_W F'$ be the group of automorphisms of $T_{A,B}$ that are $(F,F',W)$-regular. For $W=\{1\}$ it was introduced by Smith \cite{Sm} (and denoted $F\boxtimes F'$). Note that $\{\mathrm{id}_A\}\boxtimes\{\mathrm{id}_B\}$ is the group of automorphisms of $T_{A,B}$ preserving the $(A,B)$-coloring. This group acts simply transitively on the set of edges $(v,w)$, $v$ even; it is obviously included in $F\boxtimes_W F'$ for all $F,F',W$.

%\item $f$ maps $S(v)$ into $S(f(v))$, necessarily by the map $\kappa_{f,v}=restricts to a bijection from $S(v)$ to $S(f(v))$, which is equal to $b_{f(v)}^{-1}\circ b_v$, and to a bijection from $S(w)$ to $S(f(w))$, which is equal to $a_{f(w)}^{-1}\circ a_w$

%$f$ restricts to a bijection from neighbors of $v$ to neighbors of $f(v)$ such that, after conjugation by the labeling, the resulting permutation of $B$ (resp.\ $A$) belongs to $F'$ (resp.\ $F$).

%Define the group $U(F,F')$ as the automorphism group of the $(A,B)$-biregular tree $T_{A,B}$ as $(A,B)$-colored graph. 

If $A,B$ are finite, or more generally if $F,F'$ are closed in $\mks(A)$ and $\mks(B)$, then $F\boxtimes_W F'$ is closed. In general (for $|A|,|B|\ge 2$), it is discrete if and only if for all $(a,b)\in A\times B$ we have $F_a\cap\Ker\,u=\{1\}$ and $F_b\cap\Ker\,u'=\{1\}$.

The most opposite case to the case $W=\{1\}$ is when, say $A=B$, and $F=F'=W=\mks(A)$. In this case, the vertex stabilizers achieve all possible permutations, while the pointwise stabilizers of balls of radius 1 are trivial.

Note that when $|A|=2$, the set of even vertices in $T_{A,B}$ can be viewed as a tree $T'$, by joining any two vertices at distance 2, in which edges are colored in the sense of (\ref{bumo1}).
When $F=\mks_2$ and $W=\{1\}$, $F\boxtimes F'$ thus identifies with the Burger-Mozes group $U(F')$, while when $F=\{1\}$, $F\boxtimes F'$ is its subgroup of index 2 of elements preserving some bimoding of $T'$. Finally, consider the case when $F=\mks_2$ and $u,u'\neq 1$, so we can suppose that $W=F$ and $u=\mathrm{id}_F$, then $F\boxtimes_WF'$ includes $\{\mathrm{id}_A\}\boxtimes\Ker\,u'$ as subgroup of index 2, the other coset consisting of bimode-reversing automorphisms of $T'$ that act on the coloring by an element of $F'\smallsetminus\Ker\,u'$.

While working in $T_{A,B}$ is natural to define these groups, it is more natural, to define almost and near automorphism groups, to work with near $(A,B)$-regular graphs; let us stick to the case when $A,B$ are finite. It is not hard to encode the definition of $(F,F',W)$-regularity in a structure. Then an almost $(F,F',W)$-isomorphism is a bijection between two near regular $(A,B)$-graphs that is $(F,F',W)$-regular at all but finitely many vertices. A near $(F,F',W)$-isomorphism is an $(F,F',W)$-isomorphism between two $(F,F',W)$ finite subsets (endowed with the induced $(A,B)$-colored structure), modulo coincidence outside a finite subset.

\section{Centralizers in near symmetric groups}\label{cenesy}

\subsection{Near automorphism groups of near $\Z$-sets}\label{nearautz}

The centralizers of elements of $\mksh(X)$ are studied in \cite{ACM}, notably some subquotients are exhibited in \cite[Theorem 3.3]{ACM}. Here we use the language introduced here to obtain a more natural and precise description, including the notion of near wreath product introduced in \S\ref{nearwreath}. We work with near permutations groups instead of its subgroup of balanced near permutations. This is much more convenient and provides a complete description. Note that the centralizers of elements are particular instances of near automorphism groups (introduced in \S\ref{s_neretin}), because it can be interpreted as the study of near automorphism groups of special graphs, namely oriented graphs in which all but finitely many vertices have exactly one edge inwards, and one edge outwards.

%Note that this subsection takes place in this \S\ref{s_neretin}, because it can be interpreted as the study of near automorphism groups of special graphs, namely oriented graphs in which all but finitely many vertices have exactly one edge inwards, and one edge outwards.

For simplicity, we will first describe the centralizer in $\mkst(X)$ of some element $f\in\mksh(X)$; we will describe the minor changes afterwards. So, we choose a representative $f\in\mks(X)$.

The centralizer is precisely $\mkst_{\langle f\rangle}(X)$. Decomposing $X=X_\infty\sqcup X_{\mathrm{fin}}$ as the union of infinite and finite $f$-orbits  yields a canonical decomposition $\mkst_{\langle f\rangle}(X)=\mkst_{\langle f\rangle}(X_\infty)\times\mkst_{\langle f\rangle}(X_{\mathrm{fin}})$. (Note that such a reduction cannot be stated in a such practical way when restricting to balanced near isomorphism groups!) This reduces to a separate study of each part.

So $X_{\infty}$ is just a free action of $\Z$. Let $Y_\infty$ be the quotient, i.e., the set of infinite orbits. The automorphism group $\mks_{\langle f\rangle}(X_\infty)$ is the unrestricted wreath product $W(Y_\infty)=\Z\hatwr\mks(Y_\infty)$. 

Given $\mkst_{\langle f\rangle}(X_\infty)$, it acts on the set of ``germs at $-\infty$" of $f$-orbits, yielding a homomorphism into $\mks(Y_\infty)$. Also, it acts as eventually as a translation on each such germ, so this actually yields a homomorphism $u_-$ from $\mkst_{\langle f\rangle}(X_\infty)$ into $\Z\hatwr\mks(Y_\infty)$, which is just the identity in restriction to the automorphism group $\mks_{\langle f\rangle}(X_\infty)$. Similarly, we have another homomorphism $u_+$ by looking at germs at $+\infty$. So we have a homomorphism $(u_-,u_+)$ from $\mkst_{\langle f\rangle}(X_\infty)$ into $(\Z\hatwr\mks(Y_\infty))^2$. The latter is injective: indeed, any $g$ in the kernel breaks only finitely many edges of the Schreier graph of $f$; on any component not including any broken edge, it acts as the identity near the infinity and hence is the identity. On the finitely many broken components, it still acts as the identity near both $-\infty$ and $-\infty$ and hence is the identity element of $\mkst(X)$.

Let us describe the image of $(u_-,u_+)$. Since each projection is surjective, it is enough to describe the image by $u_+$ of $\Ker(u_-)$. If $g$ is in this kernel, the above argument already shows that $g$ acts as the identity on all but finitely many components. Hence its image by $u_+$ is an element of the restricted wreath product $W^*(Y_\infty)=\Z\wr\mks_{<\aleph_0}(X_\infty)$. Conversely, each such element can be achieved. To conclude, we obtain

\begin{thm}\label{autxinfty}
Writing $Y=Y_\infty=X_{\infty}/\langle f\rangle$. The centralizer $\mkst_{\langle f\rangle}(X_\infty)$ is naturally isomorphic to the fibre product $W(Y)\times_{W(Y)/W^*(Y)}W(Y)$, that is the subgroup of $W(Y)^2$ of pairs having the same image in $W(Y)/W^*(Y)$. 
\end{thm}

There is a natural homomorphism $\phi_Y:W^*(Y_\infty)\to\Z$ mapping $((n_y)_{y\in Y},\sigma)$ to $\sum_{y\in Y}n_y$. If $Y$ is infinite, the map from the fibre product $W(Y)\times_{W(Y)/W^*(Y)}W(Y)$ to $W^*(Y)$ mapping $(s,s')$ to $s^{-1}s'$ is not a group homomorphism. However, post-composition with $\phi_Y$ yields a group homomorphism. One can easily check that $\phi_{Y_\infty}$ is precisely the index character of the near action of $\mkst_{\langle f\rangle}(X_\infty)$ on $X_\infty$. Therefore the balanced near automorphism group $\mkst_{0,\langle f\rangle}(X_\infty)$ (i.e., the centralizer of $f$ in $\mksh(X)$ is equal to the fibre product $W(Y)\times_{W(Y)/W^*_0(Y)}W(Y)$ for $Y=Y_\infty$, where $W^*_0(Y)$ is the kernel of $\phi_Y$.

Before embarking into $X_{\mathrm{fin}}$, let us start with the case of $X_k$, the union of all $k$-cycles. Let $Y_k$ be the quotient. Then one easily checks that that $\mkst_{\langle f\rangle}(X_k)$ is the near wreath product  $(\Z/k\Z)\wr^\star\mkst(Y_k)$ (see \S\ref{nearwreath}). The index character $\phi_k$ is then equal to $k(\phi\circ\pi)$, where $\pi$ is the projection to $\mkst(X_k)$ and $\phi$ is the index character of $\mkst(X_k)$.

Given a family of groups $(G_i)_{i\in I}$ and subgroups $H_i\subseteq G_i$, recall that the semirestricted product $\prod_i^{(H_i)}G_i$ is the set of families $(g_i)_{i\in I}$ in $\prod G_i$ such that $g_i\in H_i$ for all but finitely many $i$.

Then we claim that $\mkst_{\langle f\rangle}(X_{\mathrm{fin}})$ is the restricted product $\prod_k^{(\Ker\phi_k)}\mkst_{\langle f\rangle}(X_k)$.

To check this, first observe that $X_k$ is commensurated, and hence, we have, for all $k$, a homomorphism to $M'_k=\mkst_{\langle f\rangle}(X_k)$, which is obviously surjective. This yields a homomorphism $v$ from $\mkst_{\langle f\rangle}(X_{\mathrm{fin}})$ to $\prod_kM'_k$.

Second, define $M_k=(\Z/k\Z)\hatwr\mks(Y_k)=\mks_{\langle f\rangle}(X_k)$, the automorphism group of the $\langle f\rangle$-set $X$. Note that it has a canonical homomorphism into $M'_k$, whose kernel is the restricted wreath product $(\Z/k\Z)\wr\mks_{<\aleph_0}(X_k)$. Then any element in $\mkst_{\langle f\rangle}(X_{\mathrm{fin}})$ preserves $X_k$ for all but finitely many $k$, and hence we have a well-defined homomorphism $w$ from $\mkst_{\langle f\rangle}(X_{\mathrm{fin}})$ to the near product (see \S\ref{nearwreath}) $\prod^\star_{k<\infty}M_k$. Since the corresponding homomorphism $\mks_{\langle f\rangle}(X_{\mathrm{fin}})\to\prod_k M_k$ is surjective (and indeed an isomorphism), $w$ is also surjective.

\begin{thm}
The product homomorphism $(v,w)$ from from $\mkst_{\langle f\rangle}(X_{\mathrm{fin}})$ to \[\prod_kM'_k\times {\prod_{k}}^{\star}M_k\]
is injective, and its image is the set of pairs having the same image in $\prod^\star_kM'_k$.
\end{thm}
\begin{proof}
Let $g$ belong to the kernel. Being in the kernel of $w$, it acts as the identity on $X_k$ for all large $k$. Being in the kernel of $v$, it near acts trivially on the remaining finitely many $k$, so $g$ near acts as the identity.

It follows from the definition that the image is included in the set of pairs having the same image in $\prod^\star_kM'_k$. Conversely, consider such a pair and let us show that it belong to the image. Multiplying by the image of an automorphism, we can suppose that it belongs to the first factor. This means that we have a pair $(x,1)$, where $x$ belongs to the kernel of $\prod_kM'_k\to\prod^\star_kM'_k$. That is, $x=(x_k)$ with $x_k\in M'_k$ and $x_k=1$ for large $k$. This can obviously be achieved (just define a near automorphism as equal to $x_k$ on $X_k$). 
\end{proof}

\begin{rem}
One particular case is when $Y_k$ is finite for all $k$. In this case, the near automorphism group is just the near product $\prod^\star_k\Z/k\Z\wr\mks(Y_k)$. Hence, if $Y_k$ is finite for all $k\ge 2$, we obtain the product $\mks^\star(Y_1)\times \prod^\star_{k>1}\Z/k\Z\wr\mks(Y_k)$. We emphasize this latter case because such elements play a significant role in \cite{ACM}.
\end{rem}

Recall that the image of $v$ is the restricted product $\prod_k^{(\Ker\phi_k)}M'_k$. Define $\Phi((x_k)_k)=\sum_k\phi_k(x_k)$. Then the index character is given by $\Phi\circ v$.

%\begin{cor}$\mkst_{\langle f\rangle}(X_{\mathrm{fin}})$ has a normal composition series with successive subquotients: $\bigoplus_kM'_k$, $\prod_k^\star M_k$. It has a finer normal composition series with successive subquotients: $\bigoplus_{k}(\Z/k\Z)^{)Y_k(}$, $\bigoplus_k\mkst(Y_k)$, $\prod^\star_k(\Z/k\Z)^{Y_k}$, $\prod^\star_k\mkst(Y_k)$.
 
%It also has a normal composition series with successive subquotients: $\prod^\star_k(\Z/k\Z)\wr\mks_{<\aleph_0}(X_k)$, $\prod_k^{(\Ker\phi_k)}M'_k$. It has a finer normal composition series with successive subquotients: $\prod^\star_k(\Z/k\Z)^{(X_k)}$, $\prod^\star_k\mks_{<\aleph_0}(X_k)$, $\prod_k(\Z/k\Z)^{)Y_k(}$, $\prod_k\mksh(Y_k)$, $\Z^{(J_f)}$, where $J_f$ is the set of $k$ such that $Y_k$ is infinite.
%\end{cor}

Finally, in the case $f$ has nonzero index, we can suppose that it has positive index and hence can choose a representative that consists of a injective map. So in addition to the 2-sided orbits above, there are some infinite one-sided orbits. This does not affect the description of the near automorphism group $X_{\mathrm{fin}}$.

For the description of the near automorphism group of $X_\infty$, we now define $Y_{-\infty}$ and $Y_{\infty}$ as the set of germs at $-\infty$ and $+\infty$; they should be considered as two sets whose intersection is cofinite in both (in the above choice, we $Y_{-\infty}$ is a cofinite subset of $Y_{+\infty}$). Then, by the same arguments, the near automorphism group naturally identifies to the fibre product $W(Y_{-\infty})\times_{W(Y_\infty)/W^*(Y_\infty)}W(Y_{+\infty})$.

We now come back to $X=X_\infty$:

\begin{defn}
Consider a permutation of a set $X$ given by a free action of $\Z$ with $\kappa$ orbits ($\kappa$ is any cardinal). Define $G_\kappa$ as the near automorphism group of the corresponding $\Z$-set, that is, the centralizer in $\mkst(X)$ of the given permutation.
\end{defn}

%The near automorphism group of $X$ then only depends on the number $\kappa$ of orbits; we call it $G_\kappa$. 

Then $G_\kappa$ is described in Theorem \ref{autxinfty}, where $Y$ is a set of cardinal $\kappa$, $X=\Z\times Y$ and the action of $\Z$ is the obvious ``horizontal" action. If $\kappa$ is finite, $G_\kappa$ is then isomorphic to the square $(\Z\wr\mks(\kappa))^2$ and is not very mysterious, so we assume that $\kappa$ is infinite.

%(except in Remark \ref{remacm} where we will only assume that $\kappa\ge 2$).

\begin{cor}\label{abelgkappa}
The derived subgroup of $G_\kappa$ has index 2 in the kernel of the index character, and admits no proper normal subgroup of index $<2^\kappa$. 
%The abelianization of $G_\kappa$ is isomorphic to $\Z\times (\Z/2\Z)$ and is the largest countable quotient of $G_\kappa$.
\end{cor}
\begin{proof}
We use the fibre product of Theorem \ref{autxinfty}. An obvious quotient is the same fibre product killing the abelian kernels, that is, the fibre product $\mks(\kappa)\times_{\mkst_0(\kappa)}\mks(\kappa)$. Define a homomorphism of the latter into $\Z/2\Z$ by mapping $(f_1,f_2)$ to the signature of the finitely supported permutation $f_1^{-1}f_2$ (that is is a homomorphism follows from the signature being invariant by conjugation by arbitrary permutations).

The fibre product decomposition also yields a semidirect decomposition of $G_\kappa$ (using the diagonal as a splitting):
\[(\Z\wr\mks_{<\aleph_0}(\kappa))\rtimes (\Z\hatwr\mks(\kappa)).\] 
Since $\mks(\kappa)$ is a perfect group, and since every element of $\Z^\kappa$ can be written as difference between one element and a translate by some permutation with infinite cycle, the group $\Z\hatwr\mks(\kappa)$ is perfect. In addition, the abelianization map of the group $(\Z\wr\mks_{<\aleph_0}(\kappa))$ is the map $(\xi,\sigma)\mapsto(\sum_t\xi(t),\eps(\sigma))$, valued in $\Z\times (\Z/2\Z)$, where $\eps$ is the signature. Since this abelianization homomorphism extends to $G_\kappa$ (the $\Z$ part is the index map and the $\Z/2\Z$ part is constructed above), we deduce that this is the abelianization homomorphism of $G_\kappa$ as well. This also shows that in the above semidirect product decomposition $\Lambda\rtimes S$ of $G_\kappa$ above, we have $[G_\kappa,G_\kappa]=\Lambda'\rtimes S$, where $\Lambda'=[\Lambda,\Lambda]$.

Now let $N$ be a normal subgroup of index $<2^\kappa$ in the derived subgroup. The group $\mks(\kappa)$ is a perfect group and, by Baer's theorem (or the earlier Ulam-Schreier theorem when $\kappa$ is countable), all its nontrivial quotients are quotients of $\mks(\kappa)/\mks_{<\kappa}(\kappa$ and hence have cardinal $2^\kappa$. Hence the diagonal $\mks(\kappa)$ is included in $N$. The commutator of the diagonal $S$ and $\Lambda'$ includes $[\Lambda',\Lambda']$. An easy argument shows that $\Lambda'$ is perfect, and hence we deduce that $[G_\kappa,G_\kappa]=\Lambda'\rtimes S$ is included in $N$. Hence the quotient by $N$ is abelian. This concludes the proof.
\end{proof}

%and the left-hand factor $\Z\wr\mks_{<\aleph_0}(\kappa)$ contains (actually, is equal) to the derived subgroup of $\Z\wr\mks_{<\aleph_0}(\kappa)$, which is therefore also included in $N$. Thus $N$ includwe have a semidirect product $A\rtimes B$ where $N$ contains both $B$ and $[A,A]$, and h

\begin{rem}\label{remacm}
Set $\kappa=\aleph_0$ here. Corollary \ref{abelgkappa} yields a contradiction in \cite[Theorem 3.3]{ACM} which asserts the existence of a certain normal series for $G_\kappa$. Namely, they define two subgroups $N_3$ and $N_4$ ($N_4$ is denoted $SC_G(\bar{g})$ there) in $G_\kappa$ with the assertions that they are normal subgroups, that $G_\kappa/N_4$ is free abelian, and that $N_4/N_3$ is isomorphic to $\mks_{<\aleph_0}(\kappa)$. Assume that this could hold. Corollary \ref{abelgkappa} then implies that $G_\kappa/N_4$ is trivial or infinite cyclic, and hence $N_3$ has countable index. Since $N_4/N_3$ is not abelian, we have a contradiction.

The error in \cite{ACM} is actually the claim in \cite[Theorem 3.3]{ACM} that $N_3$ is a normal subgroup. The gap is somewhat easy to locate, namely the proof of \cite[Theorem 3.3]{ACM} is ``All parts are easy, so we omit details". 

Indeed, $N_3$ is generally not a normal subgroup. It is defined as equal to $BN_2$, where $B,N_2$ are two more subgroups, $N_2$ being indeed normal (defined as a suitable kernel, actually of the diagonal homomorphism into $\mks(\kappa)^2$, in the language above), while $B$ (denoted $\mathrm{OC}_G$ there) is defined as the image of the centralizer the $f$, which we here assume to be a permutation with only infinite cycles. They actually claim that $B$ is normal. Let us check directly that this is false. Let $(a_n)_{n\in\Z}$ and $(b_n)_{n\in\Z}$ be two $f$-cycles. Let $\phi$ be the cycle along the cycle $(a_n)$. It commutes with $f$ and hence maps to an element of $B$. Let $\tau$ be the involution exchanging $a_n$ and $b_n$ for all $n\in\N$ and fixing all other points. Then $\tau$ commutes with $f$ modulo finitely supported permutation, but $\tau^{-1}\psi\tau$ is the cycle $(\dots,a_{-2},a_{-1},b_0,b_1,\dots)$ and is not a finite perturbation of any element commuting with $f$. So the image of $\tau^{-1}\psi\tau$ does not belong to $B$. Hence $B$ is not a normal subgroup.
\end{rem}
%\newpage
\subsection{More centralizers}\label{morecen}
Let us use the results of \S\ref{s_equivariant} to obtain further results about centralizers. Recall that for a near $G$-set $X$, its near automorphism group, i.e., the centralizer of the image of $G$ in $\mkst(X)$ is denoted $\mkst_G(X)$. If $X$ is a $G$-set, its automorphism group as $G$-set is denoted $\mks_G(X)$. Also, recall that $[G]$ denotes the $G$-set $G$ under left translation. 

%Let $G$ be a group in which every finite subset is contained in a finitely generated 1-ended subgroup. Let $Y$ be a set. Consider the free action of $G$ on $X=G\times Y$ given by horizontal left translations, that is, $g\cdot(h,y)=(gh,y)$. Then the obvious map $\mks_G(X)\to\mkst_G(X)$ is an isomorphism. In other words, the centralizer is reduced to the unrestricted wreath product $G\hatwr\mks(Y)$, where $G^Y$ acts by right translations on horizontal strata.

\begin{prop}\label{free1end}
Let $G$ be a (not necessarily finitely generated) 1-ended group.
Let $Y$ be a set. Consider the free action of $G$ on $X=G\times Y$ given by horizontal left translations, that is, $g\cdot(h,y)=(gh,y)$. Then the obvious map $\mks_G(X)\to\mkst_G(X)$ is an isomorphism. In other words, the centralizer is reduced to the unrestricted wreath product $G\hatwr_Y\mks(Y)$, where $G^Y$ acts by right translations on horizontal strata.
\end{prop}
\begin{proof}
Since $G$ is infinite and the right action is free, the action of $G\hatwr\mks(Y)$ is clearly $\star$-faithful and hence the given homomorphism is injective.

%By Proposition \ref{lim1e}, the $G$-set $[G]$ is 1-ended. 

The centralizer $C$ permutes the set of isolated ends of the $G$-set $X$, which can therefore be identified to $Y$ (see Proposition \ref{isolend}). Since the automorphism group already achieves all possible permutations of $Y$, we are reduced to understand the kernel $K$ of this near action. Namely, the near action of this kernel commensurates each stratum $G\times\{y\}$. Since $G$ is 1-ended, by Lemma \ref{AE1e}, its near action of $K$ on each such stratum is given by a right translation. Since all families of right translations are achieved by horizontal automorphisms, we are reduced to consider near automorphisms $f$ that near act trivially on each stratum. Since any representative $\tilde{f}$ then acts as the identity on all but finitely many strata, it is finitely supported and hence $f$ has to be the identity.
\end{proof}

\begin{cor}\label{centr1end}
Let $G$ be a 1-ended group. Then the centralizer in $\mkst(G)$ of the left $G$-action on itself is reduced to the isomorphic image of $G$ acting on the right.

%Let $G$ be a group in which every finite subset is contained in a finitely generated 1-ended subgroup. 
In particular, if $G$ is any abelian group that is neither virtually cyclic nor countable locally finite, it is equal to its own centralizer in $\mkst(G)$.
\end{cor}
\begin{proof}
The first assertion is the particular case of Proposition \ref{free1end} when $Y$ is a singleton. The second follows from general results on ends, namely any abelian group of the given form is 1-ended. This follows from the general Theorem IV.6.10 in \cite{DD}, but actually follows from much easier earlier results:
\begin{itemize}
\item if $G$ is abelian, neither virtually cyclic nor locally finite, then it has a normal infinite finitely generated (cyclic) subgroup of infinite index, which implies that it is 1-ended by \cite[Satz 7.3]{Ab74} or \cite[Theorem 4.3]{Ho1};
\item if $G$ is abelian, uncountable and locally finite, then it is 1-ended by \cite[Theorem 1]{ScS} (later generalized to non-abelian groups by Holt \cite{Hol}).\qedhere
\end{itemize} 
\end{proof}

%(because it has a normal infinite finitely generated (cyclic) subgroup of infinite index: see \cite[Satz 7.3]{Ab74} or \cite[Theorem 4.3]{Ho1}).
%Theorem 6.10 in \cite{DD}).

% of $\Q$-rank $\ge 2$, it is equal to its own centralizer in $\mkst(G)$.

This provides a wealth of examples of maximal abelian subgroups in $\mksh(\N)$ that are infinitely generated countable groups; the question about their existence was suggested by Shelah and Stepr\={a}ns \cite{ShS}. Let us more precisely write the answer.
We start with a lemma.
\begin{lem}\label{inftycen}
Let $A$ be an abelian group and $\alpha:A\to\mkst(X)$ a near action, with $X$ infinite. Suppose that
\begin{itemize}
\item either $|A|<|X|$, 
\item or that $A$ is finitely generated and the action is not of finite type.  
\end{itemize}
Then the centralizer of $\alpha(A)$ in $\mkst(X)$ (and hence in $\mksh(X)$) has cardinal $2^{|X|}$.
%Let $A$ be a finitely generated abelian group and $\alpha:A\to\mkst(X)$ a near action, not of finite type. Then the centralizer of $\alpha(A)$ in $\mkst(X)$ (and hence in $\mksh(X)$) has cardinal $2^{|X|}$.
\end{lem}
\begin{proof}
In the first case, we use Proposition \ref{kapparea} to write the near action as $X=X_0\sqcup Y$, with $X_0$ of cardinal $\le |A|$ and $Y$ realizable as an action with $|A|$ orbits. In the second case, we use Theorem \ref{nearactionfp}, with $X_0$ of finite type and $Y$ realizable, necessarily as an action with $|X|$ orbits. In both cases, we fix a realization on $Y$, and denote by $(Z_i)_{i\in I}$ the family of $A$-orbits, with $|I|=|X|$. Denote by $A_0$ the image of $A$ in $\mkst(X_0)$.

Let $A_i$ be the image of $A$ in $\mks(Z_i)$. Then the image of $A$ in $\mks(Y)$ is included in the abelian subgroup $B=\prod_i A_i$. If the set of $i$ such that $|Z_i|\ge 2$ has cardinal $|X|$, then this subgroup, as well as its image $B'$ in $\mksh(X)$, has cardinal $2^{|X|}$. Moreover, the image of $A$ in $\mkst(X)$ is included in $A_0\times B'$.

It remains to consider the case when the set $J$ of $i$ such that $|Z_i|=1$ has complement of cardinal $<|X|$, it is a somewhat straightforward adaptation. In this case, gathering orbits by pairs, we instead use a family $(Z'_j)_{j\in J}$ of pairwise disjoint $A$-invariant subsets of $Y$, each including at least two points, and with $|J|=|X|$, and choose $A_j$ as a nontrivial abelian subgroup of $\mks(Z'_j)$ including the image $A'_j$ of $A$ in $\mks(Z'_j)$ (if $A'_j\neq\{1\}$ we just choose $A'_j=A_j$; when $A'_j=\{1\}$ we use the fact that $Z'_j$ has at least two elements and define $A_j$ as an arbitrary nontrivial abelian subgroup of $\mks(Z'_j)$. 
\end{proof}

%We decompose $X=X_0\sqcup Y$ with $X_0$ of finite type and $Y$ realizable (using Theorem \ref{nearactionfp}). Then $X_0$ is countable and $Y$ is infinite; it follows that $X$ and $Y$ have the same cardinal.

%Realize $Y$ as an $A$-action; then has $|X|$ orbits. If it includes a subset $Z$ of fixed points of cardinal $|X|$, then the centralizer of $\alpha(X)$ includes the subgroup $\mks(Z)$, which has cardinal $2^{|X|}$. Otherwise, there is a family $(Z_i)_{i\in I}$, indexed by a set $I$ with $|I|=|X|$, of distinct nontrivial $A$-orbits in $Y$, which can be chosen either all finite, or all infinite. The automorphism group of the $A$-set $Z_i$ is a nontrivial quotient $A_i$ of $A$. If all $Z_i$ are infinite, we obtain an embedding of the product $\prod A_i$ into the centralizer of $\alpha(A)$ in $\mkst(X)$. If all are finite, we obtain an embedding of the near product $\prod^\star A_i$ into the centralizer in $\mkst(X)$. In all cases, we deduce that the centralizer (both in $\mkst(X)$ and in its subgroup of countable index $\mksh(X)$) has cardinality $\ge 2^{|X|}$.

\begin{cor}
Let $A$ be a non-locally-finite countable abelian group, and $D$ an infinite countable set. Then
\begin{itemize}
\item $\mkst(D)$ includes a maximal abelian subgroup isomorphic to $A$;
\item
$\mksh(D)$ includes a maximal abelian subgroup isomorphic to $A$, unless $A$ is virtually cyclic and not cyclic.
\end{itemize}
\end{cor}
\begin{proof}
First suppose that $A$ is not virtually cyclic. Hence it is 1-ended, and then Corollary \ref{centr1end} applies to both assertions: $A$ is a maximal abelian subgroup $\mkst(A)$ (and hence in $\mksh(A)$).

Now suppose that $A$ is virtually cyclic, namely isomorphic to $\Z\times F$ for some finite abelian group, and consider an embedding into $\mkst(D)$.

%First suppose that $D$ is not a near $A$-set of finite type, and let us show that the centralizer is uncountable (here we just use that $A$ is a finitely generated abelian group). We decompose $D=X\sqcup Y$ with $X$ of finite type and $Y$ realizable (using Theorem \ref{nearactionfp}); realize $Y$ as an $A$-action; then it has infinitely many orbits. If it includes an infinite subset $Z$ of fixed points, then the centralizer includes the uncountable subgroup $\mks(Z)$. Otherwise, there is an infinite family $(Z_n)$ of distinct nontrivial $A$-orbits in $A$, which can be chosen either all finite, or all infinite. The automorphism group $A_n$ of the $A$-set $Z_n$ is a nontrivial quotient of $A$. If all $Z_n$ are infinite, we obtain an embedding of the product $\prod A_n$ into the centralizer in $\mkst(D)$. If all are finite, we obtain an embedding of the near product $\prod^\star A_n$ into the centralizer in $\mkst(D)$. In all cases, we deduce that the centralizer (both in $\mkst(D)$ and in $\mksh(D)$) has cardinality $\ge 2^{\aleph_0}$.

By Lemma \ref{inftycen}, if $D$ is a near $A$-set of infinite type, the centralizer of $A$ is uncountable. So, we can suppose that $D$ is a near $A$-set of finite type. Then $D$ is a finitely-ended near $A$-set by Proposition \ref{finitimesfree}. Write it as a disjoint union $X_1\sqcup\dots\sqcup X_n$ of 1-ended near $A$-sets, with $n\ge 1$. The near automorphism group of each $X_i$ includes an isomorphic copy of $\Z$; hence the centralizer in $\mkst(X)$ includes an isomorphic copy of $\Z^n$.

Assuming that $A\subseteq \mkst(D)$ is its own centralizer, we deduce that $n=1$. Then Theorem \ref{classnear1} implies that for some nonzero homomorphism $\phi$ of $A$ onto $\Z$, the near $A$-set $D$ is near isomorphic to $\phi^{-1}(\N)\subseteq A$. The near automorphism group of this is indeed reduced to $A$.

Now assume that $A\subseteq \mksh(D)$ is its own centralizer in $\mksh(D)$. Since the centralizer in $\mkst(D)$ includes a copy of $\Z^n$, the centralizer in $\mksh(D)$ includes an isomorphic copy of $\Z^{n-1}$. Hence $n\le 2$. If $n=1$, we are in the previous case, which has a nonzero index number, so $n=2$. Theorem \ref{classnear1} implies that for some surjective homomorphism $\phi:A\to\Z$, some subgroups $F_1,F_2$ with $F_1\cap F_2=0$, and some choice $\eps$ of sign, and denoting by $\pi_i$ the projection $A\to A/F_i$, the near $A$-set $D$ is isomorphic to the disjoint union $\pi_1(\phi^{-1}(\N))\sqcup\pi_2(\phi^{-1}(\eps\N))$. The index character of this is equal to $(|F/F_1|+\eps|F/F_2|)\phi$ and should be zero, which implies $|F_1|=|F_2|$ and $\eps=-1$. 

The centralizer in $\mksh(X)$ always contains the image of the (free, hence $\star$-faithful) action of $F/F_1\times F/F_2$, and as soon as $|F/F_1|,|F/F_2|$ are both $\ge 2$, this implies that $A$ is not equal to its centralizer. Otherwise since $|F_1|=|F_2|$ and $F_1\cap F_2=0$ we deduce $F_1=F_2=0$. This implies that $A$ is cyclic. Conversely, the simply transitive action of $\Z$ is equal to its own centralizer in $\mksh(\Z))$ (and by the above, these are, up to conjugation, the only maximal abelian subgroups in groups $\mksh(X)$ that are virtually cyclic). Note that the proof also shows that for $X$ uncountable, a maximal abelian subgroup of $\mksh(X)$ or $\mkst(X)$ cannot be finitely generated.
\end{proof}

The following proposition now addresses the locally finite case.

%We do not know whether $\mksh(\N)$ or $\mkst(\N)$ contains maximal abelian subgroups that are countable and locally finite. Clearly, there are no finite ones, and Theorem \ref{loficom}(\ref{lofi2}) implies that there are none of the form $\bigoplus F_n$ with $F_n$ finite.

\begin{prop}
Every countable locally finite abelian subgroup of $\mkst(X)$, $X$ arbitrary infinite set, has a centralizer of cardinal $2^{|X|}$. In particular,
\begin{itemize}
\item every locally finite abelian subgroup of $\mkst(X)$, $X$ arbitrary infinite set, has an uncountable centralizer;
\item no maximal abelian subgroup of $\mkst(X)$ or $\mksh(X)$ is countable and locally finite.
\end{itemize}

%Every locally finite abelian subgroup of $\mkst(X)$, $X$ arbitrary infinite set, has an uncountable centralizer. In particular, no maximal abelian subgroup of $\mkst(X)$ or $\mksh(X)$ is countable and locally finite.
\end{prop}
\begin{proof}
Let $G$ be a locally finite abelian subgroup of $\mkst(X)$. If $G$ is uncountable, it is contained in its own centralizer. Hence, the last two assertion follow from the first.

To deal with the case when $G$ is countable and $X$ is uncountable, we prove a more general fact, namely: if $H$ is an abelian subgroup of $\mkst(X)$, $X$ infinite, and $H$ has cardinal $<|X|$, then the centralizer of $X$ has cardinal $2^{|X|}$, and includes an abelian subgroup of the same cardinal including $H$. Indeed, one can write $X$ as disjoint union, as near $H$-set, of a near $H$-set of cardinal $<|X|$ and a realizable subset (by Proposition \ref{kapparea}), necessarily of cardinal $|X|$ and splitting into $|X|$ $H$-orbits $X_i$, after choice of a realization. On each $X_i$, $H$ acts as a group $H_i$ and then the abelian $\prod H_i$ acts faithfully, centralizing the $H$-action. This concludes unless the set of $i$ such that $X_i$ is not a singleton has cardinal $<|X|$, but in this case $H$ acts trivially on a subset of cardinal $|X|$ and then the centralizer contains a subgroup isomorphic to $\mkst(X)$, as well as an abelian subgroup of the same cardinal $2^{|X|}$; this has trivial intersection with $H$ (so that their product is the desired subgroup).

% and $X$ is uncountable, then we can write $X$ as $Y\sqcup Z$, both commensurated subset, with realizable action on $Y$ (which we now fix) and $Z$ countable. If $Y^G$ is infinite, the centralizer includes $\mkst(Y^G)$ is uncountable. Otherwise, there are infinitely many non-singleton orbits in $Y$, and hence we can produce an uncountable number of elements in the stabilizer.

Now suppose that both $X$ and $G$ are countable. By Theorem \ref{loficom}(\ref{lofi1}), the near action of $G$ is completable: view $X$ as a commensurated subset of a $G$-set $Y$, such that $X$ meets every orbit on a finite subset. In particular, $G$ has infinitely many orbits on $Y$. Let $Y=\bigsqcup Y_n$ be the orbit decomposition, and write $X_n=Y_n\cap X$, it is finite. Since the case $G=\{1\}$ is trivial, consider $g\neq 1$ in $G$, and realize it on $Y$. Then for $n$ large enough, $X_n$ is $g$-invariant, and for an infinite set $I$ of values of $n$, $g$ acts nontrivially on $X_n$. Then for any subset $J$ of $I$, the element $g_J$ acting on $Y_n$ as $g$ when $n\in J$ and as 1 when $n\notin I$, stabilizes $X$ and commutes with the $G$-action on $Y$. There are uncountably pairwise distinct elements of $\mks(Y)$ and hence induce uncountably many elements of $\mkst(Y)$ (since the kernel of $\mks(Y)\to\mkst(Y)$ is countable).
\end{proof}

%The symmetric group $\mks(X)$ is separable if and only if $X$ is countable, by a straightforward verification. 

Let us say that a group is $\kappa$-centerless (resp.\ $<\kappa$-centerless) if it admits a subset of cardinal $\kappa$ (resp.\ $<\kappa$) with trivial centralizer. Note that being $k$-centerless is given by a $\exists\forall$-formula in the language of groups.

Note that a group $G$ is $<\aleph_0$-centerless if and only if the group topology induced by the action of $G$ on itself by conjugation, is discrete.

The following fact is possibly very standard in the case of $\mks(X)$.

\begin{prop}
Let $X$ be a set. If $|X|\le 2^{\aleph_0}$, then every quotient (except $\mks(X)/\mka(X)$ which has nontrivial center) of $\mks(X)$, as well as $\mkst(X)$, is $2$-centerless.
 
Conversely, if $|X|>2^{\aleph_0}$, then no nontrivial quotient of $\mks(X)$, and no non-cyclic quotient of $\mkst(X)$, is $\aleph_0$-centerless. More generally, whenever $2^\kappa<|X|$, no non-cyclic quotient of $\mks(X)$ or $\mkst(X)$ is $\kappa$-centerless.
\end{prop}
\begin{proof}
We start with the converse, showing more generally that it it not $\kappa$-centerless whenever $2^\kappa<|X|$; we can suppose that $\kappa$ is infinite. Indeed, consider a subgroup $G$ of cardinal $\kappa$ in a quotient $\mks(X)/N$. Write as image of a subgroup $H$ of $\mks(X)$, of cardinal $\kappa$, surjecting onto it. The $H$-set $X$ splits as a disjoint union $\bigsqcup_{L\le H}\beta_LH/L$, where $\beta_L$ is some cardinal. Since $2^\kappa<|X|$, there exists a subgroup $L\le H$ such that $\beta_L=|X|$. Hence the centralizer of $H$ in $\mks(X)$ contains a copy of $\mks(X)$, permuting the $|X|$ copies of $H/L$, and this has cardinal $2^X$, and so does its image in $H/N$, which centralizes $G$.

In the case of $\mkst(X)$, embed the latter into $\mksh(X\sqcup\N)$ as in \S\ref{s_omega} (the homomorphism $j_X$ defined before Definition \ref{omeganear}). For $2^\kappa<|X|$, consider a subgroup $G$ of $\mkst(X)$ of cardinal $\kappa$. Write $G'=j_X(G)$. It follows from the previous paragraph that the centralizer of $G'$ has a subgroup of cardinal $2^X$ acting trivially on $\N$. Its image in $\mkst(X)$ centralizes $G$ and is nontrivial, contradiction.

Let us pass to the more interesting direction; write $|X|=\lambda$; when $X$ is finite $\mks(X)$ has a 2-element generating subset so we assume $X$ infinite. Since every 4-generated subgroup of $\mks(X)$ is included in a 2-generated one (Galvin \cite{Ga}), it is enough to produce a 4-generated subgroup with trivial centralizer.

We fix a 2-generated group $F$ (e.g., free) with an injective family $(N_i)_{i\in I}$ of normal subgroups, $|I|=\lambda$, such that $F/N_i$ has trivial center for every $i$ (the existence of such a family of centerless quotients is a routine fact, for instance Neumann's groups \cite{Neu} yield such a family). Consider the disjoint union $\bigsqcup_{i\in I}F/N_i$, which we identify to $X$ (since they both have cardinal $\lambda$). Then $F$ naturally acts on this disjoint union on the left, and on the right, both preserving each $F/N_i$; let $F_1$ and $F_2$ be the images of these actions (these are subgroups of $\mks(X)$). Each of these actions of $F$ naturally extends to the product $\prod_i F/N_i$, and we denote by $G_1$ and $G_2$ the images of this product by each of these action (for instance, $G_1$ is the set of permutations of $X$, preserving each $F/N_i$ and acting as a left-translation on it). Since these action commute, they define an action of $F\times F$; let $G\subseteq \mks(X)$ be its image.

Since the $F/N_i$ are pairwise non-isomorphic, the centralizer of $F_1$ (and $G_1$) in $\mks(X)$ in $\mks(X)$ is equal to $G_2$, and vice versa. Since each $F/N_i$ has a trivial center, we have $G_1\cap G_2=\{1\}$. Hence, the centralizer of $G=G_1G_2$ in $\mks(X)$ is trivial.

This settles the case of $\mks(X)$ itself. (When $X$ is countable, an alternative consists in using that $\mks(X)$ has a dense 2-generated subgroup. However, $\mks(X)$ is not separable when $X$ is uncountable, for any Hausdorff group topology for which the action on $X$ is continuous, since a dense subgroup should then act transitively.)

Now, the next easier case is that of the quotients $\mks(X)/\mks_{<\kappa}(X)$ when $\lambda\ge\kappa>\aleph_0$. Let us show that the image of $G$ in $\mks(X)/\mks_{<\kappa}(X)$ has a trivial centralizer. Indeed, if $f\in\mks(X)$ and centralizes $G$ modulo $\mks_{<\kappa}(X)$, define $I(f)$ as the set of $i$ such that for each $g\in G$ and $x\in F/N_i$, we have $f(g(x))=g(f(x))$. Then $I\smallsetminus I(f)$ has cardinal $<\kappa$. For $i\in I(f)$, $f$ is $G$-equivariant on $F/N_i$, hence it maps it to another orbit $F/N_j$, and necessarily $j=i$ since the $F/N_i$ are pairwise non-isomorphic; moreover, on $F/N_i$, the centralizer of $G$ is trivial as we already used in the case of $\mks(X)$. Hence, $f$ is the identity on $F/N_i$ for all $i\in I(f)$. Hence $f$ has support included in the union of $F/N_i$ for $i\notin I(f)$, which has cardinal $<\kappa$, and hence the image of $f$ in $\mks(X)/\mks_{<\kappa}(X)$ is trivial. 

It remains to consider the ``modulo finitary", and most interesting case. In this case, we need to make the additional assumption that each $F/N_i$ is 1-ended. We choose $f:X\to X$, representative of an element of $\mkst(X)$, that is closely $(F\times F)$-equivariant. We define $I(f)$ as in the previous paragraph, it is cofinite in $I$, with finite complement denoted by $J$. As in the previous paragraph, we conclude that $f$ is the identity on $F/N_i$ for each $i\in I(f)$. For $i\in J$ and $g\in F$ (acting on the left), the set of fixed points of $g$ is commensurated by $f$; each $F/N_i$ is, up to a finite error, a Boolean combination of such commensurated subsets (because $J$ is finite and the $N_j$ are pairwise distinct). Hence each $F/N_i$ is commensurated by $f$. Now using Corollary \ref{ghx} to $F/N_i$, we deduce that on $F/N_i$, $f$ near equals a right translation $r_g$ ; also applying it using that $f$ commutes with the right near action of $F$, we deduce that $f$ near equals a left translation $\ell_h$. Then we have $xg=hx$ for all but finitely many $x\in F/N_i$. By the easy claim below, we deduce that $g=h$ is central in $F/N_i$, and since the latter has trivial center, we have $g=h=1$. That is, $f$ near equals the identity on $F/N_i$ for each $i\in J$, and eventually we deduce that $f$ near equals the identity.

Claim: {\it Let $W$ be an infinite group with elements $g,h$ such that $xg=hx$ for all but finitely many $x$. Then $g=h$ is central.} Indeed, first fix $x_0$ for which the equality holds, and multiply the equality $x_0g=hx_0$ by the equality $g^{-1}x^{-1}=x^{-1}h^{-1}$ for all but finitely many $x$. Then $x_0x^{-1}=hx_0x^{-1}h^{-1}$ for all but finitely many $x$. So the centralizer of $h$ has finite complement, and hence is the whole group. Hence $h$ is central, and similarly $g$ is central, and $x_0g=hx_0$ then forces $g=h$.
\end{proof}

\printindex

\end{document}